\documentclass[12pt]{article}

\usepackage{indentfirst}
\usepackage{amsfonts}
\usepackage{amsmath}
\usepackage{amssymb}
\usepackage{amsbsy, amsthm}
\usepackage[margin=1in]{geometry}
\usepackage{imakeidx,xcolor}
\usepackage{hyperref}

\newtheorem{dfn}{Definition} [section]
\newtheorem{theorem}[dfn]{Theorem}
\newtheorem{lemma}[dfn]{Lemma}
\newtheorem{proposition}[dfn]{Proposition}

\newtheorem{corollary}[dfn]{Corollary}
\newtheorem{conjecture}[dfn]{Conjecture}
\newtheorem{question}[dfn]{Question}
\newenvironment{pf}{\noindent{\bf Proof.}}
{\enspace\vrule height5pt depth0pt width5pt}

\newcommand{\defn}[1]{\emph{#1}}{\index{#1}}

\def\X {{\mathcal X}}

\def\T {{\mathcal T}}

\def\Se {{\mathcal S}}
\def\LL {{\mathcal L}}
\def\F {{\mathcal F}}

\def\ins {{\rm ins}}
\def\C {{\mathcal C}}

\def\M {{\mathcal M}}

\def\W {{\mathcal W}}

\def\A {{\mathcal A}}
\def\Y {{\mathcal Y}}
\def\HH {{\mathcal H}}
\def\dist {{\rm dist}}
\def\PP {{\mathcal P}}

\begin{document}

\title{Coarse Menger property of quasi-minor excluded graphs and length spaces}
\author{Chun-Hung Liu\footnote{Email:chliu@tamu.edu. Partially supported by NSF under CAREER award DMS-2144042.} \\
\small Department of Mathematics, \\
\small Texas A\&M University,\\
\small College Station, TX 77843-3368, USA}
\maketitle

\begin{abstract}
Menger's theorem is an important building block of numerous results in the study of graph structure. 
We consider a variant in terms of coarse geometry.
We say that a set of graphs has the weak coarse Menger property if there exist functions $f$ and $g$ such that for any graph $G$ in this set, subsets $X$ and $Y$ of vertices of $G$, and positive integers $k$ and $r$, either there exist $k$ paths between $X$ and $Y$ pairwise at distance at least $r$, or there exists a union of at most $f(k,r)$ balls of radius at most $g(k,r)$ intersecting all paths between $X$ and $Y$.
Nguyen, Scott and Seymour proved that the set of all graphs does not have the weak coarse Menger property and asked whether every proper minor-closed family of finite graphs has it.

In this paper, we provide a positive answer to this question in a stronger form: it is true for the set of locally finite graphs with an excluded finite minor, and the functions $f$ and $g$ can be chosen so that $f$ only depends on the number $k$ of the paths in the packing and the function $g$ is a linear function of the distance threshold $r$ and is independent of $k$, which is optimal up to a constant factor.
Our result extends to every length space quasi-isometric to a locally finite graph or metric graph with an excluded finite minor, such as complete Riemannian surfaces of finite Euler genus, string graphs, and Cayley graphs of finitely generated minor-excluded groups.
\end{abstract}

%\tableofcontents

\section{Introduction} \label{sec:intro}

Menger's theorem, arguably one of the most important theorems in graph theory, states that for any graph\footnote{Graphs in this paper are finite unless otherwise specified. Locally finite graphs can be finite or infinite.} $G$ and subsets $X$ and $Y$ of the set of vertices, the maximum number of disjoint paths in $G$ between $X$ and $Y$ equals the minimum size of a set of vertices in $G$ intersecting all paths between $X$ and $Y$.
It is not only a ubiquitous building block in the study of graph structure but also shows the optimal duality between packing and covering for paths between two prescribed sets.
The packing problem and the covering problem are two of the most fundamental problems in combinatorial optimization.
The former asks for the maximum number of pairwise disjoint``interesting objects'' contained in the host space; the latter asks for the minimum size of a set of ``basic elements'' intersecting all ``interesting objects''.
In this paper we study an analog of packing and covering along the line of coarse geometry.
More precisely, we study an analog of Menger's theorem by addressing packing and covering paths in graphs or length spaces coarsely.

Coarse geometry or large-scale geometry studies properties of metric spaces involving large distance.
For example, a single point and a ball of small radius can be considered as ``equivalent'' when we only consider large-scale behaviors.
Hence the large-scale analog of a packing of disjoint objects is a set of objects that are pairwise far apart, and the large-scale analog of a covering is a set of balls of bounded radius intersecting all interesting objects.

(Finite or infinite) graphs can be treated as metric spaces since we can define distance between any two vertices. 
Formally, for any two vertices $x$ and $y$ of a (finite or infinite) graph $G$, we define \defn{$\dist_G(x,y)$} to be the infimum length of a path in $G$ between $x$ and $y$, where the \defn{length} of a path is the number of its edges.
This function $\dist_G$ is a metric on the set $V(G)$ of vertices of $G$ as long as $G$ is connected.

In spite of being special cases of metric spaces, graphs play an important role in the study of general metric spaces.
For example, Gromov \cite{g_asdim} introduced asymptotic dimension of metric spaces to study finitely generated groups via studying asymptotic dimension of their Cayley graphs; Yu \cite{y} used this notion to prove that the Novikov higher signature conjecture holds for manifolds whose fundamental group has finite asymptotic dimension.
Asymptotic dimension is a large-scale analog of topological dimension of topological spaces and is a large-scale analog of chromatic number of graphs.
The chromatic number of a graph $G$ is the infimum $k$ such that its set of vertices can be written as a union of at most $k$ sets of pairwise non-adjacent vertices (i.e.\ $k$ sets each consisting of balls of radius 0 that are pairwise at distance greater than 1).
The asymptotic dimension of a metric space is the infimum $k$ such that for every $r>0$, its set of points can be written as a union of at most $k$ sets $X_1,...,X_k$, where each $X_i$ is a union of sets $X_{i,1},X_{i,2},...$ that are pairwise at distance greater than $r$ such that each $X_{i,j}$ is contained in a ball of radius bounded in terms of $r$.

After a sequence of papers \cite{fp,jl,bbegps,l_asdim,bbeglps} about asymptotic dimension of surfaces and minor-closed families of graphs appeared, several groups of scholars independently considered very similar questions about large-scale analogs of known results in graph theory.
In particular, a natural conjecture (Conjecture \ref{strong_menger_conj}) that is a large-scale analog of Menger's theorem was proposed independently by Albrechtsen, Huynh, Jacobs, Knappe and Wollan \cite{ahjkw} and Georgakopoulos and Papasoglu \cite{gp}.

Before stating Conjecture \ref{strong_menger_conj}, we introduce some notions to simplify the description.

Let $G$ be a (finite or infinite) graph.
Let $X$ and $Y$ be (not necessarily disjoint) subsets of $V(G)$.
An \defn{$X$-$Y$-path} in $G$ is a path that has at least one end in $X$ and at least one end in $Y$.
Note that every vertex in $X \cap Y$ forms an $X$-$Y$ path.
Moreover, some internal vertex of an $X$-$Y$ path might be in $X \cup Y$.

We shall find a packing of pairwise far apart $X$-$Y$ paths.
For any sets $S,T \subseteq V(G)$, we define \defn{$\dist_G(S,T)$} $= \inf_{s \in S,t \in T}\dist_G(s,t)$.
The \defn{distance between two subgraphs of $G$} is defined to be the distance between their vertex-sets.

We will cover $X$-$Y$ paths by using balls of bounded radius.
Let $S \subseteq V(G)$, and let $k$ and $r$ be real numbers.
For every vertex $x \in V(G)$, we define \defn{$\dist_G(x,S)$} $= \inf_{s \in S}\dist_G(x,s)$.
We define \defn{$N_G^{\leq r}[S]$} $= \{v \in V(G): \dist_G(v,S) \leq r\}$; note that $N_G^{\leq r}[S]$ is the union of the balls of radius at most $r$ centered at points in $S$.
A set $Z \subseteq V(G)$ is \defn{$(k,r)$-centered in $G$} if $Z \subseteq N_G^{\leq r}[W]$ for some $W \subseteq V(G)$ with $|W| \leq k$.
Note that every $(k,r)$-centered set in $G$ is a set contained in a union of at most $k$ balls of radius at most $r$.

\begin{conjecture}[{{\cite[Conjecture 3]{ahjkw}}},{{\cite[Conjecture 1.4]{gp}}}] \label{strong_menger_conj}
There exists a real number $c>0$ such that for any graph $G$, subsets $X,Y$ of $V(G)$ and positive integers $k,r$, either $G$ contains $k$ $X$-$Y$ paths pairwise at distance at least $r$, or there exists a set $(k-1,cr)$-centered in $G$ intersecting all $X$-$Y$ paths.
\end{conjecture}

The case $k=2$ with arbitrary $r$ of Conjecture \ref{strong_menger_conj} was independently proved by Albrechtsen et al.\ \cite{ahjkw} and Georgakopoulos and Papasoglu \cite{gp}.
However, Nguyen, Scott and Seymour \cite{nss_tight} disproved the case $k=r=3$ of Conjecture \ref{strong_menger_conj} even when $G$ has maximum degree at most 3.
Their counterexamples also disprove the cases for any pair $(k,r)$ with $k \geq 3$ and $r \geq 3$ and imply that Conjecture \ref{strong_menger_conj} remains false even if the bound $cr$ for the radius of balls is changed to any function of $k$ and $r$.

On the other hand, the cases $r=2$ with arbitrary $k \geq 3$ of Conjecture \ref{strong_menger_conj} remain open.
If the graph $G$ has bounded maximum degree, then every ball of radius $cr$ contains a bounded number of vertices, so every $(k-1,cr)$-centered set is a set of bounded number of vertices not far away from $k-1$ particular vertices.
Gartland, Korhonen, Lokshtanov \cite{gkl} and Hendrey, Norin, Steiner and Turcotte \cite{hnst} independently proved the $r=2$ case of Conjecture \ref{strong_menger_conj} if $G$ has bounded maximum degree and we only require the hitting set to have bounded size but do not require the vertices in the hitting set to be not far away from $k-1$ particular vertices. 

Some results related to the restriction of Conjecture \ref{strong_menger_conj} to minor-closed families are known.
Nguyen, Scott and Seymour \cite{nss_path_width} pointed out that the tree-width\footnote{We do not need the formal definition of tree-width and path-width in this paper. Graphs of bounded tree-width and bounded path-width are the graphs that have a tree-decomposition and a path-decomposition, respectively, with bounded bag size. Tree-decomposition and path-decomposition are defined in Section \ref{subsec:tree_decomp}.} of the counterexamples to Conjecture \ref{strong_menger_conj} in \cite{nss_tight} is at most 6, showing that Conjecture \ref{strong_menger_conj} remains false even if we restrict it to some minor-closed families with simple structure and allow the bound for the radius of balls to be dependent on both $k$ and $r$.
On the other hand, they \cite{nss_path_width} proved that Conjecture \ref{strong_menger_conj} holds for graphs of bounded path-width if we allow the bound $cr$ for the radius of balls to be a function of $k,r$ and the bound for the path-width.
In addition, Nguyen, Scott and Seymour \cite{nss_planar_boundary} conjectured that Conjecture \ref{strong_menger_conj} is true if we restrict it to graphs embeddable in a fixed surface $\Sigma$ and allow the bound for the radius to be a function of $k,r,\Sigma$; they \cite{nss_planar_boundary} proved the case of this weaker conjecture when $G$ can be embedded in the plane such that all vertices in $X \cup Y$ are incident with the infinite face. 
We remark that, unlike the original form of Conjecture \ref{strong_menger_conj}, if the radius of balls is allowed to be dependent on both $k$ and $r$, then it is unclear whether at least $k-1$ balls are still required.

Nguyen, Scott and Seymour \cite{nss_tight} proposed the following weakening of the false Conjecture \ref{strong_menger_conj} by relaxing the number of balls and the radius of balls in the hitting set.

\begin{conjecture}[{{\cite[Conjecture 4.2]{nss_tight}}}] \label{weak_coarse_Menger_conj}
There exist functions $f,g: {\mathbb N} \rightarrow {\mathbb R}$ such that for any graph $G$, subsets $X,Y$ of $V(G)$ and positive integers $k,r$, either $G$ contains $k$ $X$-$Y$ paths pairwise at distance at least $r$, or there exists an $(f(k,r),g(k,r))$-centered set intersecting all $X$-$Y$ paths.
\end{conjecture}

More recently, Nguyen, Scott and Seymour \cite{nss_coarse} further disproved Conjecture \ref{weak_coarse_Menger_conj}.
Because their counterexamples contain arbitrary graphs as a minor, they asked the following question:

\begin{question}[{{\cite[Section 3]{nss_coarse}}}] \label{question_minor_weak_Menger}
Does Conjecture \ref{weak_coarse_Menger_conj} hold for graphs in a fixed proper minor-closed family?
\end{question}

A graph $H$ is a \defn{minor} of another graph $G$ if $H$ is isomorphic to a graph that can be obtained from $G$ by repeatedly deleting a vertex or an edge or contracting an edge.
A \defn{minor-closed family} is a set $\F$ of graphs such that for every graph $G$ in $\F$, every minor of $G$ is also in $\F$. 
A minor-closed family is \defn{proper} if it does not contain all graphs.
Proper minor-closed families are general.
In particular, collecting the graphs that satisfy any graph property that is preserved under vertex-deletion, edge-deletion and edge-contraction gives a proper minor-closed family.
Typical examples include, but are not limited to, the set of graphs that can be drawn in a fixed surface with no crossing, the set of linkless embeddable graphs, the set of knotless embeddable graphs, the set of the graphs whose Colin de Verdi\`{e}re parameter is at most a fixed constant, and the sets of graphs of bounded tree-width or path-width.

Question \ref{question_minor_weak_Menger} cannot be relaxed to many natural graph classes that are more general than minor-closed families, such as:
    \begin{itemize}
	    \item Graph classes of asymptotic dimension or Assouad-Nagata dimension at most 1: The asymptotic dimension of a graph class is defined to be the asymptotic dimension of the (possibly infinite) graph that is the union of disjoint copies of all graphs in this class.
        The asymptotic dimension of any proper minor-closed family is at most 2 \cite{bbeglps,l_asdim}.
	This result was later generalized to Assouad-Nagata dimension of proper minor-closed families \cite{d_andim,l_andim}.
        The Assouad-Nagata dimension addresses both large-scale and small-scale behaviors of the space and is at least the asymptotic dimension.
	The Assouad-Nagata dimension of classes of graphs of bounded tree-width is at most 1 \cite{d_andim,l_andim}.     
	An infinite class of counterexamples of Conjecture \ref{weak_coarse_Menger_conj} provided in \cite{nss_coarse} (called $(\ell,m)$-blocks) has the property that for every graph $G$ in this class, there exists an integer $\ell$ such that $G$ is obtained from a graph $W$ of tree-width at most 2 with all components having radius at most $2\ell+1$ by adding internally disjoint paths of length $2\ell+1$ between vertices in $W$.
        It is straightforward to show that this class has Assouad-Nagata dimension (and hence asymptotic dimension) at most 1 by using the fact that bounded tree-width graphs have Assouad-Nagata dimension at most 1.
		\item General proper topological minor-closed classes: A graph $H$ is a \defn{topological minor} of another graph $G$ if some graph obtained by subdividing edges of $H$ is isomorphic to a subgraph of $G$.
	Every topological minor of $G$ is a minor of $G$.
		    Another infinite class of counterexamples of Conjecture \ref{weak_coarse_Menger_conj} provided in \cite{nss_coarse} has the property that every graph in this class has exactly one vertex of degree greater than 3.
        Every graph that has at most one vertex of degree greater than 3 does not contain any graph that has at least two vertices of degree greater than 3 as a topological minor.
        Hence Conjecture \ref{weak_coarse_Menger_conj} does not hold for any topological minor-closed family that does not forbid any graph with at most one vertex of degree greater than 3.
        Note that Question \ref{question_minor_weak_Menger} is equivalent to asking whether Conjecture \ref{weak_coarse_Menger_conj} holds for every topological minor-closed family that forbids at least one graph that has maximum degree at most 3.
		\item Graph classes that forbid a fat minor: Fat minors is a coarse version of graph minors implicitly introduced by Fujiwara and Papasoglu in \cite{fp} when studying asymptotic dimension of geodesic metric spaces homeomorphic to ${\mathbb R}^2$ and explicitly defined by Bonamy et al.\ \cite{bbeglps,bbegps} when studying Assouad-Nagata dimension of graphs embedded in surfaces.
        A graph contains another graph $H$ as a minor if and only if it contains $H$ as a 1-fat minor; if a graph contains $H$ as a $k$-fat minor, then it contains $H$ as a $k'$-fat minor for every $0 \leq k' \leq k$.
        Albrechtsen and Davies \cite{ad} constructed counterexamples to Conjecture \ref{weak_coarse_Menger_conj} not containing the $(154 \times 154)$-grid as a 3-fat minor by modifying the counterexamples in \cite{nss_coarse}. 
    \end{itemize}

\subsection{Main result (the basic form)}

The first main result of this paper is Theorem \ref{main_intro}, which provides a positive answer to Question \ref{question_minor_weak_Menger} in the following stronger form:
First, our bound $g$ for the radius of balls in the hitting set is a linear function of $r$ independent of $k$. 
If we combine Theorem \ref{main_intro} and a result in \cite{ddh}, then we can further remove the dependency of $r$ from the function $f$ for the number of balls (see Theorem \ref{strongest_intro}), which will be optimal since the number of the balls must depend on the number of objects $k$ in the packing, and the radius of balls must depend on the distance threshold $r$ between the objects in the packing.
As far as we know, all known positive results that are true for all $r \geq 1$ related to the false Conjecture \ref{weak_coarse_Menger_conj} in the literature require the radius of balls to be dependent on both $k$ and $r$, except those on specific graph classes that can be easily proved via a special case of Proposition \ref{rooted_fat_EP_planar_intro} for graphs of bounded tree-width.
Second, our result holds if we pack and cover $X$-$Y$ paths whose ends are far apart, which can be restated in terms of the coarse Erd\H{o}s-P\'{o}sa property for rooted fat $K_2$-minors and imply a result about packing and covering A-paths coarsely in minor-closed families.
Third, our result holds for locally finite infinite graphs, which allows us to extend our result to some continuous spaces such as Riemannian surfaces via quasi-isometry (see Section \ref{subsec:weak_coarse_quasi_intro}).

For a graph $G$, subsets $X,Y \subseteq V(G)$ and a real number $\ell \geq 0$, an \defn{$(\ell,X,Y)$-path} in $G$ is an $X$-$Y$ path in $G$ such that $\dist_G(a,b) \geq \ell$, where $a$ and $b$ are its ends.
Note that $(0,X,Y)$-paths are exactly $X$-$Y$ paths.
Since every proper minor-closed family is a subset of the set of $H$-minor free graphs for some graph $H$, the $\ell=0$ case of the following theorem answers Question \ref{question_minor_weak_Menger} affirmatively.

\begin{theorem} \label{main_intro}
For every (finite) graph $H$, there exist a function $f: {\mathbb N}^2 \times {\mathbb N}_0 \rightarrow {\mathbb R}$ and an integer $c_H$ such that if $G$ is an $H$-minor free locally finite graph, then for any subsets $X,Y$ of $V(G)$ and integers $k,r \in {\mathbb N}$ and $\ell \in {\mathbb N}_0$, either $G$ contains $k$ $(\ell,X,Y)$-paths pairwise at distance in $G$ at least $r$, or there exists a set $(f(k,r,\ell),c_Hr+\ell)$-centered in $G$ intersecting all $(\ell,X,Y)$-paths in $G$.
\end{theorem}

We did not put effort to optimize $f$ and $c_H$ in Theorem \ref{main_intro}.
Our proof of Theorem \ref{main_intro} shows that the integer $c_H$ satisfies the following:
    \begin{itemize}
	\item $c_H \leq 8\gamma_H+44$, where $\gamma_H$ is the maximum nonnegative integer such that $K_{\max\{|V(H)|+1,5\}}$ cannot be embedded in some surface of Euler genus at most $\gamma_H$,
        \item If $G$ is finite and $H$ is planar, then the bound $c_Hr+\ell$ can be replaced by $r/2$,
        \item If $G$ is finite and $H$ is non-planar, then $c_H \leq 4\gamma_H+22$, where $\gamma_H$ is the maximum nonnegative integer such that $H$ cannot be embedded in some surface of Euler genus at most $\gamma_H$,
        \item If $G$ is finite and $H$ is an apex-graph, then $c_H \leq 14$.
    \end{itemize}
A graph is an \defn{apex-graph} if one can delete at most one vertex to make it planar.
For every surface $\Sigma$, Euler's formula implies that every graph embeddable in $\Sigma$ is $K_{3,2g+3}$-minor free, where $g$ is the Euler genus of $\Sigma$.
Since $K_{3,2g+3}$ is an apex-graph, it implies that the integer $c_H$ in Theorem \ref{main_intro} for graphs embeddable in $\Sigma$ can be chosen to be $14$ so that it is independent of the surface $\Sigma$.
Bla\v{z}ej, Pilipczuk and Protopapas \cite{bpp} independently proved that $c_H$ can be chosen to be 1 for finite graphs $G$ embedded in a surface $\Sigma$ with $\ell=0$; it is unknown whether $c_H=1$ is optimal for this special case, and their result is for this special case only and does not apply to general proper minor-closed families.
Apex-minor free graphs are beyond graphs embedded in surfaces of bounded Euler genus.
For example, if $H$ is obtained from disjoint copies of $K_{3,3}$ by identifying one vertex from each copy into a vertex, then $H$ is an apex-graph, but there exists no single surface such that all $H$-minor free graphs can be embedded.

We remark that the bound for the radius of balls must grow at least as fast as linear in $r$ if the number of balls $f$ is independent of $r$:
Consider the $(r \times n)$-grid, where $n$ is an arbitrarily large integer.
Let $X$ be the first column and $Y$ be the last column of this grid.
Then there are no two $X$-$Y$ paths pairwise at distance at least $r$ since the distance between any two distinct vertices in $X$ is at most $r-1$. 
On the other hand, to hit all $X$-$Y$ paths, every row must be intersected by a ball.
Every ball with radius $s$ can only intersect at most $2s+1 \leq 3s$ rows.
So $s \geq r/3f$.

Theorem \ref{main_intro} achieves the linear bound in $r$ for the radius of balls, but the number of balls $f$ depends on both $k$ and $r$.
Instead of putting effort to remove the dependency of $r$ from $f$ by using a more complicated proof, in Section \ref{subsec:weak_coarse_quasi_intro} we will show that by combining Theorem \ref{main_intro} with a simple trick about scaling (Lemma \ref{scaling_trick}) and a result of Davies, Distel and Hickingbotham \cite{ddh} announced in \cite{d_string} about quasi-isometry, we can remove the dependency of $r$ from $f$ while keeping the bound for the radius linear in $r$ (see Theorem \ref{strongest_intro}).
So our bound for the radius is optimal up to a constant factor.

We state Theorem \ref{main_intro} and Theorem \ref{strongest_intro} separately for a few reasons.
First, Theorem \ref{main_intro} is the key for this argument and its proof does not rely on any recent result.
Second, Theorem \ref{strongest_intro} only works for some $\ell$.
Third, the integer $c_H$ in Theorem \ref{main_intro} is more explicit.
The explicit integer $c_H$ could be beneficial for studying some concrete minor-closed families.
For example, linkless embeddable graphs and knotless embeddable graphs are $K_6$-minor free \cite{cg,s} and $K_7$-minor free \cite{cg}, respectively, so $c_H$ can be chosen to be $22$ and $30$, respectively (when $G$ is finite), and to be $60$ and $68$, respectively (when $G$ is locally finite). 

Theorem \ref{main_intro} can also be applied to a conjecture of Geelen about A-paths.
For a (finite or infinite) graph $G$ and a subset $A$ of $V(G)$, an \defn{$A$-path} is a path in $G$ between two distinct vertices in $A$.
Gallai \cite{g_Apaths} proved an analog of Menger's theorem for packing and covering $A$-paths: for any graph $G$, $A \subseteq V(G)$ and $k \in {\mathbb N}$, either $G$ contains $k$ disjoint $A$-paths, or there exists $Z \subseteq V(G)$ with $|Z| \leq 2k-2$ intersecting all $A$-paths.
Jim Geelen proposed the following conjecture on a coarse version of Gallai's theorem:

\begin{conjecture}[Geelen] \label{question_A_path}
\footnote{Geelen proposed this conjecture verbally at the 2024 Barbados Graph Theory Workshop. The exact formulation of Geelen is lost. We state this conjecture in the form as stated in \cite{hj}. Distel et al.\ \cite{dghlm} stated Geelen's conjecture as a weaker form in which the bound $g$ for the radius of balls can be dependent on $k$ and stated Conjecture \ref{question_A_path} as their own conjecture.}
There exist functions $f,g: {\mathbb N} \rightarrow {\mathbb R}$ such that for any $k,r \in {\mathbb N}$, graph $G$ and $A \subseteq V(G)$, either there exist $k$ $A$-paths pairwise at distance in $G$ at least $r$, or there exists a set $(f(k),g(r))$-centered in $G$ intersecting all $A$-paths.
\end{conjecture}

Note that $(1,A,A)$-paths are exactly $A$-paths.
So Theorem \ref{main_intro} immediately implies the following partial result for Conjecture \ref{question_A_path} for proper minor-closed families.

\begin{corollary} \label{weak_A_path_intro}
For every (finite) graph $H$, there exist a function $f: {\mathbb N}^3 \rightarrow {\mathbb R}$ and an integer $c_H$ such that if $G$ is an $H$-minor free locally finite graph, then for any subset $A$ of $V(G)$ and integers $k,r,\ell \in {\mathbb N}$, either $G$ contains $k$ $(\ell,A,A)$-paths pairwise at distance at least $r$, or there exists a set $(f(k,r,\ell),c_Hr+\ell)$-centered in $G$ intersecting all $(\ell,A,A)$-paths in $G$.
\end{corollary}

The integer $c_H$ in Corollary \ref{weak_A_path_intro} is the same as the integer $c_H$ in Theorem \ref{main_intro}.
In particular, $c_H=14$ when $H$ is an apex-graph and $G$ is finite.
Moreover, we can further remove the dependency of $r$ from $f$ by using Lemma \ref{scaling_trick} and a result in \cite{ddh} (see Theorem \ref{strongest_intro}).

We remark that Hickingbotham and Joret \cite{hj} solved the $r=2$ case of Conjecture \ref{question_A_path} with $(f(k),g(2))=(78k-78,1)$ and with $(f(k),g(2))=(4k-4,4)$.
When this paper was under preparation, Distel, Giocanti, Hodor, Legrand-Duchesne and Micek \cite{dghlm} proved a weaker version of Conjecture \ref{question_A_path} for all $r$ with $(f(k),g(k,r))=(4k-4,256^k r)$.
The result in \cite{dghlm} is incomparable with the $\ell=1$ and finite graph case of Corollary \ref{weak_A_path_intro} since their bound $g(k,r)$ for the radius of balls depends on both $k$ and $r$. 

One can restate $X$-$Y$ paths and $A$-paths in terms of rooted $K_2$-minors and restate $(\ell,X,Y)$-paths and $(\ell,A,A)$-paths in terms of rooted fat $K_2$-minors.
Fat minors are a large-scale analog of minors.
For a simple graph $H$ and $\ell>0$, an \defn{$\ell$-fat $H$-minor model} consists of a set $(S_h: h \in V(H))$ of disjoint connected subgraphs of $G$ and a set $(P_e: e \in E(H))$ of paths in $G$ such that for every edge $uv \in E(H)$, $P_{uv}$ is a $V(S_u)$-$V(S_v)$ path in $G$, and for any distinct $x,y \in V(H) \cup E(H)$, $\dist_G(S_x,S_y) \geq \ell$ unless one of $x,y$ is a vertex of $H$ and the other is an edge of $H$ such that $x$ is incident with $y$.
The \defn{distance in $G$ between two $\ell$-fat $H$-minor models} $\{(S_h: h \in V(H)),(P_e: e \in E(H))\}$ and $\{(S'_h: h \in V(H)),(P'_e: e \in E(H))\}$ is $\dist_G(\bigcup_{h \in V(H)}S_h \cup \bigcup_{e \in E(H)}P_e, \bigcup_{h \in V(H)}S'_h \cup \bigcup_{e \in E(H)}P'_e)$.
For a collection $(R_h: h \in V(H))$ of subsets of $V(G)$, we say that the $\ell$-fat $H$-minor model $\{(S_h: h \in V(H)),(P_e: e \in E(H))\}$ is \defn{$(R_h: h \in V(H))$-rooted} if $V(S_h) \cap R_h \neq \emptyset$ for every $h \in V(H)$.  
Note that every $(\ell,X,Y)$-path gives an $(X,Y)$-rooted $\ell$-fat $K_2$-minor model in $G$; every $(\ell,A,A)$-path gives an $(A,A)$-rooted $\ell$-fat $K_2$-minor model in $G$.
Hence Theorem \ref{main_intro} implies the following:

\begin{corollary} \label{rooted_fat_K2_EP_intro}
For every (finite) graph $H$, there exist a function $f: {\mathbb N}^3 \rightarrow {\mathbb R}$ and an integer $c_H$ such that if $G$ is an $H$-minor free locally finite graph, then for any collection $(R_v: v \in V(K_2))$ of subsets of $V(G)$ and integers $k,r,\ell \in {\mathbb N}$, either $G$ contains $k$ $(R_v: v \in V(K_2))$-rooted $\ell$-fat $K_2$-minor models pairwise at distance in $G$ at least $r$, or there exists a set $(f(k,r,\ell),c_Hr+\ell)$-centered in $G$ intersecting all $(R_v: v \in V(K_2))$-rooted $\ell$-fat $K_2$-minor models in $G$.
\end{corollary}

Corollary \ref{rooted_fat_K2_EP_intro} is a large-scale and rooted fat minors analog of a classical result in \cite{rs V} about the Erd\H{o}s-P\'{o}sa property for minor models.
We remark that Corollary \ref{rooted_fat_K2_EP_intro} cannot be generalized to rooted (fat or not) $P_3$-minor models with three distinct roots, where $P_3$ is the 3-vertex path: take a $w \times w$-grid for an arbitrarily large $w$, let $R_1$ be the first column, $R_2$ be the first row, and $R_3$ be the last column, then it is easy to see that there do not exist two disjoint $(R_1,R_2,R_3)$-rooted $P_3$-minor models, but any set intersecting all $(R_1,R_2,R_3)$-rooted $P_3$-minor models requires size $\Omega(w)$ and hence requires an unbounded number of balls as grids have maximum degree 4.
This example shows that if $G$ contains an arbitrarily large grid, then rooted $P_3$-minor models cannot satisfy the coarse Erd\H{o}s-P\'{o}sa property.
We can show that the large grids are essentially the only obstructions for the coarse Erd\H{o}s-P\'{o}sa property for packing and covering rooted minor model models, as follows:

\begin{proposition} \label{rooted_fat_EP_planar_intro}
For any (finite) planar graph $H$ and (finite) graph $W$, there exists $\beta \in {\mathbb N}$ such that if $G$ is an $H$-minor free (finite) graph, then for any collection $(R_v: v \in V(W))$ of subsets of $V(G)$ and integers $k,r,\ell \in {\mathbb N}$, either $G$ contains $k$ $(R_v: v \in V(W))$-rooted $\ell$-fat $W$-minor models pairwise at distance at least $r$, or there exists a set $(\beta k,\max\{\lceil \frac{r-1}{2} \rceil, \frac{\ell}{2}\})$-centered in $G$ intersecting all $(R_v: v \in V(W))$-rooted $\ell$-fat $W$-minor models.
\end{proposition}

Proposition \ref{rooted_fat_EP_planar_intro} is a simple corollary of a lemma in this paper (Lemma \ref{easy_tree}) developed on the way of proving Theorem \ref{main_intro}.
Some special cases of Proposition \ref{rooted_fat_EP_planar_intro} are probably folklore.
Our Lemma \ref{easy_tree} in fact deals with a setting that is more general than Proposition \ref{rooted_fat_EP_planar_intro}.

\subsection{Weak coarse Menger property and quasi-isometry} \label{subsec:weak_coarse_quasi_intro}

In this section we extend Theorem \ref{main_intro} to more general metric spaces and remove dependency of $r$ from the number of balls $f$.

A \defn{length space} is a metric space $(X,d)$ such that for any $x,y \in X$ and $\epsilon>0$, there exists a path in $X$ from $x$ to $y$ with length at most $d(x,y)+\epsilon$; here a \defn{path from $x$ to $y$} is a rectifiable curve from $x$ to $y$ (i.e. a continuous function from $[0,1]$ to $X$ such that $f(0)=x$ and $f(1)=y$).
Note that a graph is not a length space since the notion of paths in graphs is different from the one in length spaces.
However, the properties considered in this paper (i.e.\ paths that are pairwise far apart and unions of balls intersecting paths) are irrelevant to this difference. 

Let $(X,d_X)$ be a length space, and let $A,B$ be subsets of $X$.
For $\ell \geq 0$, an \defn{$(\ell,A,B)$-path in $(X,d_X)$} is a path from $a$ to $b$ for some $a,b \in X$ such that $d_X(a,b) \geq \ell$.
For any $\alpha,\beta \geq 0$, we say that a set $Z \subseteq X$ is \defn{$(\alpha,\beta)$-centered in $d_X$} if $Z$ is contained in the union of at most $\alpha$ balls of radius at most $\beta$.

We say that a (finite or infinite) graph or a length space $(X,d_X)$ has the \defn{weak coarse Menger property} if there exist functions $f,g: {\mathbb N} \times {\mathbb R}_{>0} \rightarrow {\mathbb R}$ such that for any $k \in {\mathbb N},r \in {\mathbb R}_{> 0}$ and subsets $A,B$ of $X$, either there exist $k$ paths in $(X,d_X)$ from $A$ to $B$ pairwise at distance in $d_X$ at least $r$, or there exists $Z \subseteq X$ such that $Z$ is $(f(k,r),g(k,r))$-centered and intersects every path in $(X,d_X)$ from $A$ to $B$.

Quasi-isometry is an equivalence relation that preserves several coarse properties, including the weak coarse Menger property.
(See Lemma \ref{quasi_iso} for precise statements.)
Formally, for real numbers $m \geq 1$ and $a \geq 0$, an \defn{$(m,a)$-quasi-isometry} from a metric space $(X,d_X)$ to another metric space $(Y,d_Y)$ is a function $\iota: X \rightarrow Y$ such that 
	\begin{itemize}
		\item $\frac{1}{m}\cdot d_X(s,t)-a \leq d_Y(\iota(s),\iota(t)) \leq m \cdot d_X(s,t)+a$ for any $s,t \in X$, and
		\item for every $y \in Y$, there exists $x \in X$ such that $d_Y(\iota(x),y) \leq a$.
	\end{itemize}
We say that $(X,d_X)$ is \defn{quasi-isometry} to $(Y,d_Y)$ if there exists an $(m',a')$-quasi-isometry from $(X,d_X)$ to $(Y,d_Y)$ for some $m' \geq 1$ and $a' \geq 0$.

Hence Theorem \ref{main_intro} can be generalized to other classes of graphs and length spaces via quasi-isometry.
Several interesting length spaces are known to be quasi-isometric to $H$-minor free (locally finite) graphs. 
For example, Davies \cite{d_string} proved the following two results:

\begin{theorem}[{{\cite[Theorem 1.7]{d_string}}}]
For every nonnegative integer $g$, there exist $m \geq 1$ and $a \geq 0$ such that every complete Riemannian surface $\Sigma$ of Euler genus $g$ is $(m,a)$-quasi-isometric to a locally finite graph embeddable in $\Sigma$.
\end{theorem}

\begin{theorem}[{{\cite[Theorem 1.1]{d_string}}}]
There exist $m \geq 1$ and $a \geq 0$ such that every string graph\footnote{A graph is a \defn{string graph} if it is the intersection graph of path-connected subsets of the plane. That is, there exists a bijection from the set of vertices to a set of path-connected subsets of the plane such that two vertices are adjacent in the string graph if and only if the corresponding two path-connected subsets intersect.} is $(m,a)$-quasi-isometric to a planar graph.
\end{theorem}

Graphs that are quasi-isometric to graphs of bounded tree-width (i.e.\ $H$-minor free graphs for some planar graph $H$) have been extensively studied.
For example, the following planar graphs $H$ satisfy that for any fixed $\ell \in {\mathbb N}$, graphs with no $\ell$-fat $H$-minor are quasi-isometric to $H$-minor free graphs: $H$ is a cycle \cite{gp}, $H=K_{2,t}$ for any fixed $t \geq 1$ \cite{adg} (improving earlier results of \cite{cdnrv,fp_cacti,gp}) and $H=K_4$ \cite{ajkw} (improving an earlier result of \cite{fp_cacti}).
In addition, graphs that admit a tree-decomposition whose every bag is $(\alpha,\beta)$-centered (for any fixed $\alpha,\beta \geq 0$) are quasi-isometric to an $H$-minor free graph for some planar $H$ \cite{d_quasi,h_quasi,nss_coarse_tw}; examples of such graphs include graphs of bounded rank-width, chordal graphs, and graphs of bounded tree-independence number.
Hence Theorem \ref{main_intro} implies that those graphs have the weak coarse Menger property once we apply quasi-isometry.
However, we remark that proving the weak coarse Menger property for those graphs is much easier and does not require the full strength of Theorem \ref{main_intro}: it is folklore that graphs that are quasi-isometric to graphs of bounded tree-width have a tree-decomposition whose every bag is $(\alpha,\beta)$-centered for some $\alpha,\beta \geq 0$, so our Lemma \ref{easy_tree} can be directly applied to obtain even stronger results.

We can also extend results from graphs to weighted graphs.
For $I \subseteq {\mathbb R}_{>0}$, an \defn{$I$-weighted (finite or infinite) graph} is a pair $(G,\phi)$, where $G$ is a (finite or infinite) graph and $\phi:E(G) \rightarrow I$ is a function.
Note that $\phi(e)$ can be thought as the length of $e$ for every edge $e$ of $G$.
For every path $P$ in an $I$-weighted graph $(G,\phi)$, we define the \defn{length of $P$} to be $\sum_{e \in E(P)}\phi(e)$.
It defines a metric \defn{$\dist_{(G,\phi)}$} such that for every $u,v \in V(G)$, $\dist_{(G,\phi)}(u,v)$ is the infimum length of a path in $(G,\phi)$ between $u$ and $v$.
Hence $(G,\phi)$ is a metric space with metric $\dist_{(G,\phi)}$.
Note that every graph is equivalent to a $\{1\}$-weighted graph.
We say that an ${\mathbb R}_{>0}$-weighted graph $(G,\phi)$ is \defn{locally finite} if every ball of finite radius contains only finitely many vertices of $G$. 
For every graph $H$, we say that an ${\mathbb R}_{>0}$-weighted graph $(G,\phi)$ is \defn{$H$-minor free} if $G$ is $H$-minor free.

Davies \cite[Theorem 1.5]{d_string} proved that there exist $m \geq 1$ and $a \geq 0$ such that every $(0,1]$-weighted planar graph is $(m,a)$-quasi-isometric to a planar graph.
This result was generalized to $H$-minor free $(0,1]$-weighted graphs for any fixed graph $H$ in \cite{ddh} (see \cite{d_string}), as follows:

\begin{theorem}[\cite{ddh}] \label{01-weighted_quasi}
For every (finite) graph $H$, there exist $m \geq 1$ and $a \geq 0$ such that every $H$-minor free locally finite $(0,1]$-weighted graph is $(m,a)$-quasi-isometric to an $H$-minor free locally finite graph.
\end{theorem}

Note that $I$-weighted graphs are not length spaces because they are discrete spaces.
We can make an $I$-weighted graph a length space by viewing each edge as a topological space homeomorphic to $[0,1]$.
More formally, a \defn{metric graph} is a metric space $(X,d_X)$ such that its underlying space is a 1-dimensional complex that is the union of all edges of a ${\mathbb R}_{>0}$-weighted graph $(G,\phi)$, and $d_X$ is an extension of $\dist_{(G,\phi)}$ such that for every $e=xy \in E(G)$, there exists a homeomorphism $f_e$ from $e$ to $[0,1]$ with $f_e(x)=0$ and $f_e(y)=1$ such that for every $z \in e-\{x,y\}$, there exists $t_z \in [0,1]$ with $f_e(z)=t_z$, $d_X(x,z)=t_z \phi(e)$ and $d_X(y,z)=(1-t_z)\phi(e)$; we say that $(X,d_X)$ is \defn{defined by} $(G,\phi)$ in this case. 
We say that a metric graph is \defn{locally finite} if it is defined by a locally finite ${\mathbb R}_{>0}$-weighted graph. 
For every graph $H$, a metric graph is \defn{$H$-minor free} if it is defined by an $H$-minor free ${\mathbb R}_{>0}$-weighted graph.

If $G$ is a graph, then it is also a $(0,1]$-weighted graph and is $(1,1)$-quasi-isometric to the metric graph defines by it. 
On the other hand, if $G'$ is a metric graph defined by a ${\mathbb R}_{>0}$-weighted graph $(G,\phi)$, then $G'$ is not necessarily quasi-isometric to $G$ unless there exists $c$ such that $G$ is a $(0,c]$-weighted graph. 
This does not create a serious issue when studying the weak coarse Menger property since we can subdivide edges to reduce the length of each edge.
See Section \ref{sec:quasi-isometry} for details.
Hence our result can also be applied to metric graphs.

With an additional trick about scaling (Lemma \ref{scaling_trick}), we can extend Theorem \ref{main_intro} and Corollary \ref{weak_A_path_intro} (for some $\ell$) to more general spaces and remove the dependency of $r$ from the number of balls, as follows:

\begin{theorem} \label{strongest_intro}
Let $H$ be a (finite) graph, and let $m \geq 1,a \geq 0$ be real numbers. 
Let $\F$ be the set consisting of
	\begin{enumerate}
		\item the locally finite $H$-minor free ${\mathbb R}_{>0}$-weighted graphs,
		\item the locally finite $H$-minor free metric graphs, and
		\item the length spaces that are $(m,a)$-quasi-isometric to a locally finite $H$-minor free metric graph.
	\end{enumerate}
Let $\F'$ be the set of length spaces that are $(m,a)$-quasi-isometric to a locally finite $H$-minor free $(0,1]$-weighted graph.
Then the following hold:
    \begin{enumerate}
	    \item There exist a function $f: {\mathbb N} \rightarrow {\mathbb R}$ and real numbers $c,c' \geq 0$ such that if $(W,d_W) \in \F \cup \F'$, then for any $X,Y \subseteq W$, $k \in {\mathbb N}$ and $r \in {\mathbb R}_{>0}$, either there exist $k$ paths in $(W,d_W)$ from $X$ to $Y$ pairwise at distance in $d_W$ at least $r$, or there exists an $(f(k),cr+c')$-centered set in $(W,d_W)$ intersecting all paths in $(W,d_W)$ from $X$ to $Y$.
	    \item There exist a function $f: {\mathbb N} \rightarrow {\mathbb R}$ and real numbers $c,c' \geq 0$ such that if $(W,d_W) \in \F \cup \F'$, then for any $A \subseteq W$, $k \in {\mathbb N}$ and $r \in {\mathbb R}_{>0}$, either there exist $k$ $A$-paths in $(W,d_W)$ pairwise at distance in $d_W$ at least $r$, or there exists an $(f(k),cr+c')$-centered set in $(W,d_W)$ intersecting all $A$-paths in $(W,d_W)$.
	    \item For every $\alpha \geq 0$, there exist a function $f: {\mathbb N} \rightarrow {\mathbb R}$ and a real number $c \geq 0$ such that if $a=0$ and $(W,d_W) \in \F$, then for any $X,Y \subseteq W$, $k \in {\mathbb N}$ and $r \in {\mathbb R}_{>0}$, either there exist $k$ $(\alpha r,X,Y)$-paths in $(W,d_W)$ pairwise at distance in $d_W$ at least $r$, or there exists an $(f(k),cr)$-centered set in $(W,d_W)$ intersecting all $(\alpha r, X,Y)$-paths in $(W,d_W)$.
    \end{enumerate}
Moreover, if $a=0$ and $(W,d_W) \in \F$, then the real number $c'$ in Statements 1 and 2 can be chosen to be $0$.
\end{theorem}

Note that the dependency in Theorem \ref{strongest_intro} is optimal since the number of balls must depend on the number $k$ of paths that we pack.
Moreover, recall that we showed that the radius of balls in the hitting set must grow at least as fast as a linear function in the required distance threshold $r$ for the paths in the packing when the number of balls only depends on $k$.
So the bound for the radius of balls in Theorem \ref{strongest_intro} is optimal up to a constant factor.

\section{Proof sketch and organization of this paper}

Theorem \ref{main_intro} is the most technical result in this paper.
Once Theorem \ref{main_intro} is proved, Theorem \ref{strongest_intro} follows from a combination of results about quasi-isometry and scaling, and the argument is fairly simple. 
We show how to derive Theorem \ref{strongest_intro} in Section \ref{sec:quasi-isometry} and then conclude this paper by introducing other open problems in Section \ref{sec:concluding_remarks}. 
The rest of the paper is dedicated to the proof of Theorem \ref{main_intro}.
In this section, we sketch a proof of Theorem \ref{main_intro} and explain the organization of this paper.

The crucial ingredient of the proof of Theorem \ref{main_intro} is the following technical form for finite graphs, and the motivation for considering this technical form will be clear as soon as we start sketching its proof.
(See Section \ref{subsec:notation} for notation.)

\begin{theorem} \label{finite_minor_induction_intro}
For any graph $H$ and integers $k,r \in {\mathbb N}$ and $\ell,\xi,\eta \in {\mathbb N}_0$, there exist $\xi^*,\eta^* \in {\mathbb N}$ such that the following holds.
Let $G$ be an $H$-minor free graph, and let $L$ be an induced subgraph of $G$.
Let $Z \subseteq V(G)$ be a $(\xi,\eta)$-centered set in $G$ such that $N_G(V(L)) \subseteq Z$.
Let $X$ and $Y$ be subsets of $V(L)$. 
If $L$ does not contain $k$ $(\ell,X,Y)$-paths in $G$ pairwise at distance in $G$ at least $r$, then there exists $Z^* \subseteq V(G)$ such that $Z^*$ is $(\xi^*,\eta^*)$-centered in $G$ and intersects all $(\ell,X,Y)$-paths in $G$ contained in $L$.
\end{theorem}

The value $\eta^*$ will be explicit when we formally prove it, but we omit it in this proof sketch for simplicity.
The finite graph case of Theorem \ref{main_intro} immediately follows from Theorem \ref{finite_minor_induction_intro} (with the explicit description of $\eta^*$) by taking $L=G$ and $Z=\emptyset$.
The locally finite case of Theorem \ref{main_intro} follows from Theorem \ref{finite_minor_induction_intro} by using a compactness argument, and we will prove it in Section \ref{subsec:infinite}.
We remark that the bound for $c_H$ in Theorem \ref{main_intro} promised in Section \ref{sec:intro} follows the bound for $\eta^*$ in Theorem \ref{finite_minor_induction_intro}, except for the case that $H$ is an apex-graph.
To obtain the bound $c_H \leq 14$ for Theorem \ref{main_intro}, we will replace part of the proof of Theorem \ref{finite_minor_induction_intro} by a simpler one to prove the apex-graph $H$ case for Theorem \ref{main_intro} directly in Section \ref{subsec:apex_minor_free}.

So it suffices to sketch a proof of Theorem \ref{finite_minor_induction_intro}.

\subsection{First step in the proof of Theorem \ref{finite_minor_induction_intro}} \label{subsec:proof_skectch_tangle}

We shall prove Theorem \ref{finite_minor_induction_intro} by induction on $k$ (i.e.\ the maximum number of paths in a packing).
The strategy is to find a set $Z_0$ which is a union of a bounded number of balls of bounded radius and find a bounded number of induced subgraphs $L_1,L_2,...,L_t$ of $G$ such that each $L_i$ does not contain $k-1$ $(\ell,X,Y)$-paths in $G$ and all $(\ell,X,Y)$-paths in $G$ contained in $L$ disjoint from $Z_0$ are contained in one of $L_1,...,L_t$; so we can apply the inductive hypothesis to each $L_i$ to obtain a bounded number of balls of bounded radius to hit all $(\ell,X,Y)$-paths contained in each $L_i$, and hence the union of $Z_0$ and those balls hits all $(\ell,X,Y)$-paths contained in $L$.
Unlike many proofs in graph theory, we cannot just take induced subgraphs to apply this strategy since the metric given by $G$ is different from the metric given by an induced subgraph. 
So we have to keep track the metric given by the whole graph $G$.
It is the reason why we consider both the graphs $L$ and $G$ in Theorem \ref{finite_minor_induction_intro} instead of just the whole graph $G$ like in Theorem \ref{main_intro}.
Moreover, we need the additional technical assumption that $N_G(V(L))$ is contained in a $(\xi,\eta)$-centered set in order to make the rest of the argument work.
Indeed, this additional technical assumption is natural based on the how we will find the aforementioned induced subgraphs $L_1,L_2,...,L_t$ with smaller packing number, which we will describe below.

A \defn{separation} of a graph $G$ is a pair $(A,B)$ of edge-disjoint subgraphs such that $A \cup B=G$.
In other words, no edge of $G$ is between $V(A)-V(B)$ and $V(B)-V(A)$, so $V(A \cap B)$ separates the graph; we call $|V(A \cap B)|$ the \defn{order} of $(A,B)$.

If there exists a separation $(A,B)$ of $L$ of small order such that no $(\ell,X,Y)$-path contained in $L$ is far away from $V(A \cap B) \cup N_G(V(L))$, then the union of balls centered at $V(A \cap B) \cup N_G(V(L))$ with small radius hits all $(\ell,X,Y)$-paths contained in $L$, and this union is contained in a union of a bounded number of balls of small radius because $N_G(V(L)) \subseteq Z$ is $(\xi,\eta)$-centered, so we are done.
Hence we may assume that for every separation $(A,B)$ of $L$ of small order, $A$ or $B$ contains an $(\ell,X,Y)$-path far away from $V(A \cap B) \cup N_G(V(L))$.
If there exists a separation $(A,B)$ of $L$ of small order such that $A$ contains an $(\ell,X,Y)$-path far away from $V(A \cap B) \cup N_G(V(L))$, then $B$ does not contain $k-1$ $(\ell,X,Y)$-paths far away from $V(A \cap B) \cup N_G(V(L))$, for otherwise those $k-1$ $(\ell,X,Y)$-paths contained in $B$ together with the one in $A$ are $k$ $(\ell,X,Y)$-paths pairwise at distance at least $r$ in $G$.
If we additionally know that $B$ also contains an $(\ell,X,Y)$-path far away from $V(A \cap B) \cup N_G(V(L))$, then $A$ also does not contain $k-1$ $(\ell,X,Y)$-paths far away from $V(A \cap B) \cup N_G(V(L))$.
In this situation, a union of a small number of balls of small radius that contains $V(A \cap B) \cup N_G(V(L))$ serves the set $Z_0$, and $A-Z_0$ and $B-Z_0$ serve the subgraphs $L_1,L_2$ in the argument in beginning of this subsection and we are done.

Hence we may assume that for every separation $(A,B)$ of $L$ of small order, exactly one of $A$ and $B$ contains an $(\ell,X,Y)$-path that is far away from $V(A \cap B) \cup N_G(V(L))$.
In other words, for every separation $(A,B)$ of small order, exactly one of $A$ and $B$ is ``important''.
This allows us to define a tangle, which is a fundamental notion in the Graph Minors series of Robertson and Seymour. 
Roughly speaking, a tangle records the important side of each separation of small order.

We will formally define tangles in Section \ref{subsec:tangles} and formally prove that either we obtain a desired hitting set for $(\ell,X,Y)$-paths or we obtain a tangle, as sketched above.
In fact, we have to be careful when we implement this strategy in order to make sure that the bound of the radius of balls in the hitting set is independent of the number $k$ of the paths in the packing.

The machinery developed in Section \ref{subsec:tangles} is general and is expected to be applicable for solving other problems related to the coarse Erd\H{o}s-P\'{o}sa property about packing objects that are pairwise far apart and covering objects by using a bounded number of balls of bounded radius.

Moreover, the machinery we have sketched so far is sufficient to prove the case that $H$ is planar in Theorem \ref{finite_minor_induction_intro} due to the Grid Minor Theorem of Robertson and Seymour.
So we may assume that $H$ is non-planar from now on in the proof sketch.

\subsection{Structure theorem}

Robertson and Seymour's structure theorem \cite{rs XVI} (Theorem \ref{minor_structual_thm} in this paper) states that every $H$-minor free graph has a ``near embedding with respect to the given tangle'' in a surface in which $H$ cannot be drawn.
The formal definition of a near embedding is complicated.
In Sections \ref{subsec:conformal}-\ref{subsec:arrangement}, we will define necessary notions to formally define near embeddings and prove simple lemmas related to those notions that will be used in the rest of the proof.

Roughly speaking, a graph $G$ is ``near embeddable'' in a surface $\Sigma$ if we can first remove a bounded number of vertices (called ``apices'') from $G$, and then partition the rest of the graph into edge-disjoint subgraphs $A_1,A_2,..A_n$ (for some integer $n$) such that all but a bounded number of those $A_i$'s share at most three vertices with $A_j$ for $j \neq i$, and the ``central piece'' of this partition can be ``arranged'' in $\Sigma$.
Note that the partition $\{A_1,A_2,...,A_n\}$ is equivalent to a set of separations $\{(A_1,B_1),(A_2,B_2),...,(A_n,B_n)\}$ with $E(A_i) \cap E(A_j)=\emptyset$ for $i \neq j$; such a set of separations is called a location and will be formally defined in Section \ref{subsec:tree_decomp}.
We also need a cyclic ordering of $V(A_i \cap B_i)$ for each $i$ in order to define a near embedding.
Each $A_i$ equipped with such a cyclic ordering is called a ``society''. 
Each of those $A_i$'s allowed to have more than three vertices in $V(A_i \cap B_i)$ is called a ``vortex'' and is required to satisfy other structural requirements. 
Societies and vortices will be defined in Section \ref{subsec:society}.
This near embedding is ``with respect to a given tangle'' if for each non-vortex $(A_i,B_i)$, $B_i$ is the important side of $(A_i,B_i)$ defined by the tangle.
The location $\{(A_i,B_i): i \in [n]\}$ together with the cyclic ordering of each $V(A_i \cap B_i)$ is called a ``segregation'', which will be defined in Section \ref{subsec:segregation}.
The ``central piece'' of a segregation $\{(A_i,B_i): i \in [n]\}$ is the graph obtained from $\bigcap_{i=1}^nB_i$ by adding edges to make $V(A_i \cap B_i)$ a clique for each non-vortex $(A_i,B_i)$ and make $V(A_i \cap B_i)$ the vertex-set of a cycle passing through its vertices according to the cyclic ordering of $V(A_i \cap B_i)$.
This ``central piece'' can be ``arranged'' in $\Sigma$ if it can be drawn in $\Sigma$ with no crossing such that the cycle given by each vortex bounds a disk in $\Sigma$.
We will formally define arrangements in Section \ref{subsec:arrangement} and define surfaces and drawings in Section \ref{subsec:surface}.

It is known that none of the crucial ingredients (apices, surfaces, vortices and tangles) in the definition of a near embedding can be dropped.
The tangle produced in the proof of Theorem \ref{finite_minor_induction_intro} discussed in the previous subsection allows us to continue the proof of Theorem \ref{finite_minor_induction_intro} by reducing the problem on $H$-minor free graphs to near embeddings.
Proving Theorem \ref{finite_minor_induction_intro} for near embeddings is the most technical part of this paper, and we will explain the strategy in the rest of this section.
The proof requires several tools proved in the literature as well as new ones developed in this paper.
Those existing and new tools will be proved in the rest of Section \ref{sec:preparation}.

\subsection{Near embeddings without apices} \label{subsec:sketch_no_apices}

In this subsection, we show how to prove Theorem \ref{finite_minor_induction_intro} for near embeddings if there do not exist apices; namely we do not need to remove vertices before partitioning the graph when we define a near embedding.
We will show how to deal with general near embeddings in the next subsection.

Near embeddings with no apices are easier to deal with because topological constraints of the embeddings restrict the metric of the graph.
For example, if we can find a subdivision of a $\theta \times \theta$-grid for some large $\theta$ contained in the near embedding such that it is disjoint from the vortices, and the drawing of this subdivision of a grid coming from the near embedding has the property that all but possibly the infinite region of the grid is homeomorphic to a disk, then for any integers $i,j$ with $r+\ell \leq i <j \leq \theta-r-\ell$ and $j-i \geq r$, the distance between $P_i$ and $P_j$ is at least $r$ due to the planarity of the drawing, and the lengths of $P_i$ and $P_j$ are at least $\max\{r+\ell,\theta-2r\} \geq \ell$, where for every $\alpha \in [\theta]$, the path $P_\alpha$ is the subdivision of the path obtained from the $\alpha$-th column of the grid by removing the first $r$ rows and the last $r$ rows; so if many $P_i$'s are from $X$ to $Y$, then we obtain $k$ $(\ell,X,Y)$-paths pairwise at distance in $G$ at least $r$.
In fact, we only require a ``grid-like'' graph instead of a subdivision of a grid to construct those $(\ell,X,Y)$-paths.
Using machinery developed in Section \ref{sec:preparation}, we can show that either we can find such a grid-like graph to complete the proof, or we can modify the near embedding by sweeping some vertices in the surface (i.e.\ some vertices in the ``central piece'') into a bounded number of vortices to obtain a new near embedding in the same surface such that all vertices in $X$ are contained in vortices or all vertices in $Y$ are contained in vortices.

Hence, we may without loss of generality assume that all vertices in $X$ are contained in vortices.
Results in the literature show that every vortex is ``essentially of bounded tree-width'', which roughly means that it can be viewed as a graph that has no tangle of large order.
Hence the argument in Section \ref{subsec:proof_skectch_tangle} can be used to show that for every vortex $W$, either there are sufficiently many paths contained in $W$ from $X$ to the boundary of $W$ pairwise at distance at least $r$, or there exists a set $Z_W$ which is a union of a bounded number of balls of bounded radius hitting all such paths.
(In fact, we will prove a variant of this argument in Section \ref{subsec:tree_decomp}, which is more flexible for future steps.)

We first assume that there exists a vortex such that we can find sufficiently many paths from $X$ to the boundary of this vortex.
Then again we can use the machinery in Section \ref{sec:preparation} to show that either we can find a grid-like graph to obtain sufficiently many $(\ell,X,Y)$-paths that are pairwise far apart, or we can further sweep vertices in the surface into a bounded number of vortices to create a new near embedding in the same surface such that all vertices in $X \cup Y$ are contained in vortices.
The former completes the proof, so we may assume that the latter holds.
We then again use results in Section \ref{subsec:tree_decomp} to show that for each vortex, either there exist at least $k$ $(\ell,X,Y)$-paths pairwise at distance at least $r$, or there exists a bounded number of balls of bounded radius hitting all $(\ell,X,Y)$-paths with an end in this vortex.
We complete the proof once the former holds for some vortex, so we may assume that the latter holds for all vortices.
But every vertex in $X \cup Y$ is contained in some vortex, so the union of the hitting set of each vortex is a hitting set for all $(\ell,X,Y)$-paths, and the proof is complete.

Hence we may assume that for every vortex $W$, there exists a set $Z_W$ which is a union of a bounded number of balls of bounded radius hitting all paths in $W$ from $X$ to the boundary of $W$.
Since all vertices in $X$ are contained in the union of the vortices, we know that the only $(\ell,X,Y)$-paths disjoint from $\bigcup_W Z_W$, where the union is over all vortices $W$, must be contained in vortices.
Therefore, the results in Section \ref{subsec:tree_decomp} again completes the proof.

We will provide a formal proof for near embeddings with no apices in Sections \ref{sec:one_vortex}-\ref{sec:roots_not_vortices}.
The machinery developed in Sections \ref{sec:one_vortex}-\ref{sec:roots_not_vortices} is in fact more general and technical because we will use them to deal with the more general situation coming from the apices sketched in the next subsection.

\subsection{Handling apices}

Finally we show how to deal with general near embeddings.
The main challenge is that apices can have neighbors everywhere in the graph, so the metric is no longer restricted by the topological constraint given by surfaces.
We remark that if $H$ is an apex-graph, then existing results in the literature imply that all the neighbors of the apices are essentially contained in vortices, so the metric remains under control by the topological constraint.
Hence the strategy stated in the previous subsection is essentially sufficient to prove the stronger form of the apex-graph $H$ case of Theorem \ref{main_intro} so that the bound for the radius of balls in the hitting set is an absolute constant.

Now we assume that the near embedding has a bounded number of apices, but there is no constraint for the neighbors of the apices.
We remove the balls of radius $r$ centered at the apices from $L$ to create a smaller induced subgraph $L'$.
If we can find a hitting set for all $(\ell,X,Y)$-paths contained in $L'$, then we can obtain a hitting set of all $(\ell,X,Y)$-paths contained in $L$ by adding those balls centered at apices.
So it suffices to work with $L'$. 

Note that we can obtain a near embedding for $L'$ by restricting the near embedding for $L$.
So we obtain a near embedding with no apices, and it looks like we can apply the strategy in the previous subsection.
However, there several issues.
First, the number of vertices removed from $L$ to obtain $L'$ is unbounded, so the tangle for $L'$ obtained in the first step of the proof is different from the one for $L$.
Second, even though the near embedding for $L$ represents the essence of the tangle for $L$, the near embedding for $L'$ obtained by restricting the near embedding for $L$ possibly does not present the essence of the tangle for $L'$, which forbids us to apply machinery developed in Section \ref{sec:preparation} to implement the strategy sketched in Section \ref{subsec:sketch_no_apices}.
Third, the topological constraint provided by the near embedding for $L'$ does not restrict the metric in $G$ because there might be a short-cut between two vertices in $L'$ by using vertices in $L-V(L')$. 

The first issue is not hard to resolve because the tangle in $L'$ can create a tangle in $L$.

To deal with the third issue, we not only consider a near embedding of $L'$ but also ``remember'' how vertices in $L-V(L')$ could create short-cuts between vertices in $L'$.
Note that the main reason why we worry about short-cuts is that two paths that are at distance in $L'$ greater than $r$ could be at distance in $G$ at most $r$, so the short-cuts we worry about have length at most $r$.
As we remove balls of radius $r$ centered at apices, the short-cuts we worry about are disjoint from the apices and hence are restricted by the topological constraints coming from the near embedding of $L$. 
So it is possible to ``remember'' those short-cuts in a near embedding of $L'$.
It is essentially the motivation of the ``protected arrangements'' defined in Section \ref{subsec:protected_arrangements}.

The second issue is the most complicated to resolve.
In order to apply the strategy sketched in Section \ref{subsec:sketch_no_apices}, we require the near embedding to be informative enough to ``represent'' the tangle.
So we have to ``repair'' the near embedding after we delete the balls centered in the apices.
As there could be an unbounded number of vertices in those balls, it is possible that we have to further partition $A_i$ into more subgraphs for an unbounded number of the separations $(A_i,B_i)$ in the location in the definition of a near embedding when repairing the near embedding. 
So the modified near embedding is far from a subgraph of the original near embedding. 
We also require the drawing of the ``central piece'' of the location to have high representativity in the surface in order to apply machinery developed in Section \ref{sec:preparation}.
So we possibly have to further remove a bounded number of balls of bounded radius from $L'$ to obtain an even smaller graph $L''$ and modify the near embedding of $L''$ to make it embedded in a surface of smaller Euler genus, while still remembering the information of vertices in the short-cuts between vertices in $L''$ like we mentioned in the previous paragraph.
This repairing process is technical, and a large portion of the new lemmas proved in Section \ref{sec:preparation} is due to this step. 

We present a formal proof for handling these issues in Section \ref{sec:avoid_apices}.
In Section \ref{subsec:finite_minor_closed}, we combine all we have discussed so far to prove Theorem \ref{finite_minor_induction_intro}.

\section{Preparation} \label{sec:preparation}

\subsection{Notation} \label{subsec:notation}

We introduce notation that will be used in this paper.

We denote the set of positive integers by ${\mathbb N}$ and denote the set of nonnegative integers by ${\mathbb N}_0$.
For every $t \in {\mathbb N}$, we denote the set $\{1,2,...,t\}$ by $[t]$.

Let $G$ be a graph.
For every $S \subseteq V(G)$, we define $N_G[S] = N_G^{\leq 1}[S]$ and $N_G(S)=N_G[S]-S$.
Let $L$ be a subgraph of $G$, and let $\F$ be a set of subgraphs of $G$.
We define $\F \cap L$ to be the set $\{F \in \F: F \subseteq L\}$.
For every $Z \subseteq V(G)$, we define $\F-Z = \F \cap (G-Z)$, which is the set of members of $\F$ disjoint from $Z$.

\subsection{Tangles} \label{subsec:tangles}

We show how to use tangles to deal with coarse Erd\H{o}s-P\'{o}sa type problems in this section.

Let $G$ be a graph.
A \defn{separation} of a graph $G$ is a pair $(A,B)$ of edge-disjoint subgraphs with $A \cup B=G$; the \defn{order} of $(A,B)$ is $\lvert V(A) \cap V(B) \rvert$.
A \defn{tangle} $\T$ in $G$ of \defn{order} $\theta$ is a set of separations of $G$, each of order less than $\theta$ such that
\begin{enumerate}
	\item[(T1)] for every separation $(A,B)$ of $G$ of order less than $\theta$, either $(A,B) \in \T$ or $(B,A) \in \T$;
	\item[(T2)] if $(A_1, B_1), (A_2,B_2), (A_3,B_3) \in \T$, then $A_1 \cup A_2 \cup A_3 \neq G$;
	\item[(T3)] if $(A,B) \in \T$, then $V(A) \neq V(G)$.
\end{enumerate}

\begin{lemma} \label{easy_tangle}
Let $G$ be a graph, and let $L$ be an induced subgraph of $G$.
Let $\F$ be a set of connected subgraphs of $L$.
Let $k,\theta \in {\mathbb N}$ and $r,r',\xi,\eta \in {\mathbb R}_{\geq 0}$ with $r' \geq r/2$.
Let $Z$ be a subset of $V(G)$ with $N_G(V(L)) \subseteq Z$ such that $Z$ is $(\xi,\eta)$-centered in $G$.
If $L$ does not contain $k$ members of $\F$ pairwise at distance in $G$ at least $r$, then either
	\begin{enumerate}
		\item there exists $Z^* \subseteq V(G)$ with $N_G^{\leq r'}[Z] \subseteq Z^* \subseteq V(L) \cup N_G^{\leq r'}[Z]$ such that $Z^*$ is $(\xi+3\theta-3,\eta+r')$-centered in $G$ intersecting all members of $\F$, and $Z^*-N_G^{\leq r'}[Z]$ is $(3\theta-3,r')$-centered in $G$, or 
		\item there exists a separation $(A^*,B^*)$ of $L$ of order less than $\theta$ such that each of $A^*-N_G^{\leq r'}[V(A^* \cap B^*)]$ and $B^*-N_G^{\leq r'}[V(A^* \cap B^*)]$ does not contain $k-1$ members of $\F-N_G^{\leq r'}[Z]$ pairwise at distance in $G$ at least $r$, or
		\item the set $\T$ consisting of all separations $(A,B)$ of order less than $\theta$ of $L$ satisfying that $A-N_G^{\leq r'}[V(A \cap B)]$ does not contain a member of $\F-N_G^{\leq r'}[Z]$ but $B-N_G^{\leq r'}[V(A \cap B)]$ contains a member of $\F-N_G^{\leq r'}[Z]$ is a tangle in $L$ of order $\theta$.
	\end{enumerate}
\end{lemma}

\begin{pf}
Let $\F' = \F-N_G^{\leq r'}[Z]$. 
Note that every member of $\F-\F'$ intersects the $(\xi,\eta+r')$-centered set $N_G^{\leq r'}[Z]$.
We may assume that Statement 1 of this lemma does not hold, for otherwise we are done.
Hence
	\begin{itemize}
		\item[(i)] there does not exist $S \subseteq V(L)$ such that $S$ is $(3\theta-3,r')$-centered in $G$ and intersects all members of $\F'$.
	\end{itemize}
By (i), for every separation $(A,B)$ of $L$ of order less than $\theta$, $N_G^{\leq r'}[V(A \cap B)]$ does not intersect all members of $\F'$. 
Since every member of $\F' \subseteq \F$ is connected, we have 
	\begin{itemize}
		\item[(ii)] for every separation $(A,B)$ of $L$ of order less than $\theta$, at least one of $A-N_G^{\leq r'}[V(A \cap B)]$ and $B-N_G^{\leq r'}[V(A \cap B)]$ contains a member of $\F'$.
	\end{itemize}

\medskip

\noindent{\bf Claim 1:} For any $F,F' \in \F'$, if $\dist_G(F,F') \leq r$, then $\dist_L(F,F') \leq r$. 

\noindent{\bf Proof of Claim 1:}
Since $\dist_G(F,F') \leq r$, there exists a path $P$ in $G$ from $V(F)$ to $V(F')$ with length in $G$ at most $r$.
If $P \subseteq L$, then $\dist_L(F,F') \leq r$.
So we may assume $P-V(L) \neq \emptyset$.
Since $N_G(V(L)) \subseteq Z$, we know $P \cap V(Z) \neq \emptyset$.
So there exists $F'' \in \{F,F'\}$ such that $\dist_G(F'',Z) \leq r/2 \leq r'$.
Hence $F'' \not \in \F'$, a contradiction.
$\Box$

\medskip

\medskip

\noindent{\bf Claim 2:} We may assume that for every separation $(A,B)$ of $L$ of order less than $\theta$, exactly one of $A-N_G^{\leq r'}[V(A \cap B)]$ and $B-N_G^{\leq r'}[V(A \cap B)]$ contains a member of $\F'$.

\noindent{\bf Proof of Claim 2:}
Suppose that there exists a separation $(A^*,B^*)$ of $L$ of order less than $\theta$ such that both $A^*-N_G^{\leq r'}[V(A^* \cap B^*)]$ and $B^*-N_G^{\leq r'}[V(A^* \cap B^*)]$ contain a member of $\F'$.
Note that for every members $F,F'$ of $\F'$ with $F \subseteq A^*-N_G^{\leq r'}[V(A^* \cap B^*)]$ and $F' \subseteq B^*-N_G^{\leq r'}[V(A^* \cap B^*)]$, we have $\dist_L(F,F') > 2r' \geq r$, so Claim 1 implies that $\dist_G(F,F') > r$.
Hence, if $B^*-N_G^{\leq r'}[V(A^* \cap B^*)]$ contains $k-1$ members of $\F'$ pairwise at distance in $G$ at least $r$, then a member of $\F'$ in $A^*-N_G^{\leq r'}[V(A^* \cap B^*)]$ and those $k-1$ members in $B^*-N_G^{\leq r'}[V(A^* \cap B^*)]$ show that $L$ contains $k$ members of $\F'$ pairwise at distance in $G$ at least $r$, a contradiction.
So $B^*-N_G^{\leq r'}[V(A^* \cap B^*)]$ does not contain $k-1$ members of $\F'$ pairwise at distance in $G$ at least $r$.
Similarly, $A^*-N_G^{\leq r'}[V(A^* \cap B^*)]$ does not contain $k-1$ members of $\F'$ pairwise at distance in $G$ at least $r$.
So Statement 2 holds.

Hence we may assume that no such $(A^*,B^*)$ exists.
This together with (ii) imply that we may assume that for every separation $(A,B)$ of $L$ of order less than $\theta$, exactly one of $A-N_G^{\leq r'}[V(A \cap B)]$ and $B-N_G^{\leq r'}[V(A \cap B)]$ contains a member of $\F'$.
$\Box$

\medskip

By Claim 2, the set $\T$ stated in Statement 3 satisfies (T1).

If there exist $(A_1, B_1), (A_2,B_2), (A_3,B_3) \in \T$ such that $A_1 \cup A_2 \cup A_3 = L$, then $\bigcap_{i=1}^3 B_i-\bigcup_{i=1}^3 V(A_i \cap B_i) = \emptyset$, so $N_G^{\leq r'}[\bigcup_{i=1}^3 V(A_i \cap B_i)] \cap V(L)$ is $(3\theta-3,r')$-centered in $G$ intersecting all members of $\F'$, contradicting (i). 
So $\T$ satisfies (T2).

If there exists $(A,B) \in \T$ such that $V(A)=V(L)$, then $N_G^{\leq r'}[V(A \cap B)] \cap V(L)$ intersects all members of $\F'$, contradicting (i). 
Hence $\T$ satisfies (T3).
This shows Statement 3.
\end{pf}

\bigskip

We call the tangle $\T$ in Statement 3 of Lemma \ref{easy_tangle} the \defn{$(G,\F,r',\theta,Z)$-tangle} in $L$.

By repeatedly applying Lemma \ref{easy_tangle}, we can obtain the following convenient variant of Lemma \ref{easy_tangle} that allows us to focus on induced subgraphs that have a tangle of large order.

\begin{lemma} \label{easy_tangle_multifold}
Let $G$ be a graph, and let $L$ be an induced subgraph of $G$.
Let $\F$ be a set of connected subgraphs of $L$.
Let $k \in {\mathbb N}$, and let $(\theta_i: i \in [k])$ be a nondecreasing sequence over ${\mathbb N}$.
Let $r,r',\xi,\eta \in {\mathbb R}_{\geq 0}$ with $r' \geq r/2$.
Let $Z$ be a subset of $V(G)$ with $N_G(V(L)) \subseteq Z$ such that $Z$ is $(\xi,\eta)$-centered in $G$.
If $L$ does not contain $k$ members of $\F$ pairwise at distance in $G$ at least $r$, then there exist $Z^* \subseteq V(G)$ with $N_G^{\leq r'}[Z] \subseteq Z^* \subseteq N_G^{\leq r'}[Z] \cup V(L)$ and a (possibly empty) set $\HH$ of disjoint connected induced subgraphs of $L$ with $|\HH| \leq k-1$ such that 
	\begin{enumerate}
		\item for every $H \in \HH$, there exist $i_H \in [\nu(H,\F)] \cup \{0\}$ and $Z_H \subseteq Z^*$ with $N_G(V(H)) \cup (Z^* \cap V(H)-N_G^{\leq r'}[Z]) \subseteq Z_H$, where $\nu(H,\F)$ is the maximum number of members of $\F \cap H$ pairwise at distance in $G$ at least $r$, such that 
			\begin{enumerate}
				\item $Z_H-N_G^{\leq 2r'}[Z]$ is $(3\theta_{1}-3+2\sum_{i=2}^{i_H}(3\theta_i-3),2r')$-centered in $G$, and if $i_H=0$, then $Z_H-N_G^{\leq 2r'}[Z]=\emptyset$, 
				\item there exists a $(G,\F \cap H,r',\theta_{i_H+1},Z^*)$-tangle in $H$, and 
				\item $H$ contains every component $C$ of $L-Z^*$ with $V(H) \cap V(C) \neq \emptyset$,
			\end{enumerate}
		\item $Z^*-N_G^{\leq 2r'}[Z]$ is $(3\theta_1-3+2\sum_{i=2}^{k-1}(3\theta_i-3),2r')$-centered in $G$, and  
		\item $Z^*$ is $(\xi+3\theta_1-3+2\sum_{i=2}^{k-1}(3\theta_i-3),\eta+2r')$-centered in $G$ and intersects all members of $\F-(\F \cap \bigcup_{H \in \HH}H)$. 
	\end{enumerate}
\end{lemma}

\begin{pf}
We shall prove this lemma by induction on $k$.
When $k=1$, we know $\F = \emptyset$, so this lemma follows from taking $Z^*=N_G^{\leq r'}[Z]$ and $\HH=\emptyset$.
So we may assume that $k \geq 2$ and the lemma holds when $k$ is smaller.

Let $\nu(L,\F)$ be the maximum number of members of $\F$ pairwise at distance in $G$ at least $r$.
Note that $\nu(L,\F) \leq k-1$ by assumption of this lemma.

Suppose to the contrary that this lemma does not hold.

\medskip

\noindent{\bf Claim 1:} There does not exist $Z_0 \subseteq V(G)$ with $N_G^{\leq r'}[Z] \subseteq Z_0 \subseteq N_G^{\leq r'}[Z] \cup V(L)$ such that $Z_0$ intersects all members of $\F$, and $Z_0-N_G^{\leq r'}[Z]$ is $(3\theta_1-3,r')$-centered in $G$.

\noindent{\bf Proof of Claim 1:} 
It $Z_0$ exists, then this lemma holds by taking $Z^*=Z_0$ and $\HH=\emptyset$, a contradiction.
$\Box$

\medskip

\noindent{\bf Claim 2:} There does not exist a $(G,\F,r',\theta_1,Z)$-tangle in $L$.  

\noindent{\bf Proof of Claim 2:} 
Suppose to the contrary that there exists a $(G,\F,r',\theta_{1},Z)$-tangle $\T$ in $L$.
Then there exists a unique component $L_0$ of $L$ such that $(L-V(L_0),L_0) \in \T$.
So $L-V(L_0)$ does not contain any member of $\F-N_G^{\leq r'}[Z]$.

If there exists a $(G,\F,r',\theta_{1},N_G^{\leq r'}[Z])$-tangle in $L$, then this lemma holds by taking $Z^*=N_G^{\leq r'}[Z]$ and $\HH=\{L_0\}$ with $i_{L_0}=0$ and $Z_{L_0}=N_G^{\leq r'}[Z]$.
So we may assume that $\T$ is not a $(G,\F,r',\theta_{1},N_G^{\leq r'}[Z])$-tangle in $L$.
Hence there exists a separation $(A,B) \in \T$ of $L$ of order less than $\theta_{1}$ such that either $A-N_G^{\leq r'}[V(A \cap B)]$ contains a member of $\F-N_G^{\leq 2r'}[Z]$, or $B-N_G^{\leq r'}[V(A \cap B)]$ does not contain a member of $\F-N_G^{\leq 2r'}[Z]$.
Since $\T$ is a $(G,\F,r',\theta_{1},Z)$-tangle $\T$ in $L$, $A-N_G^{\leq r'}[V(A \cap B)]$ does not contain a member of $\F-N_G^{\leq r'}[Z]$ and hence does not contain a member of $\F-N_G^{\leq 2r'}[Z]$.
So $B-N_G^{\leq r'}[V(A \cap B)]$ does not contain a member of $\F-N_G^{\leq 2r'}[Z]$.
Since $\F$ is connected and $N_G(V(L)) \subseteq Z$, we know that $N_G^{\leq 2r'}[Z] \cup N_G^{\leq r'}[V(A \cap B)]$ intersects all members of $\F$.
So this lemma holds by taking $Z^* = ((N_G^{\leq 2r'}[Z] \cup N_G^{\leq r'}[V(A \cap B)]) \cap V(L)) \cup N_G^{\leq r'}[Z]$ and $\HH=\emptyset$. 
$\Box$

\medskip

By Claims 1 and 2, Lemma \ref{easy_tangle} (taking $(k,\theta)=(\nu(L,\F)+1,\theta_1$) implies that 
	\begin{itemize}
		\item[(i)] there exists a separation $(A^*,B^*)$ of $L$ of order less than $\theta_1$ such that each of $A^*-N_G^{\leq r'}[V(A^* \cap B^*)]$ and $B^*-N_G^{\leq r'}[V(A^* \cap B^*)]$ does not contain $\nu(L,\F) \leq k-1$ members of $\F-N_G^{\leq r'}[Z]$ pairwise at distance in $G$ at least $r$. 
	\end{itemize}

\medskip

Let $Z' = Z \cup V(A^* \cap B^*)$.

Let $\F_A = (\F-N_G^{\leq r'}[Z']) \cap (A^*-V(A^* \cap B^*))$.
Let $a$ be the maximum number of members of $\F_A$ pairwise at distance in $G$ at least $r$.
So $a \leq \nu(L,\F)-1 \leq k-2$ by (i).
Note that $A^*-V(A^* \cap B^*)$ is an induced subgraph of $G$ with $N_G(V(A^*-V(A^* \cap B^*))) \subseteq Z'$, and $\F_A$ is a set of connected subgraphs of $A^*-V(A^* \cap B^*)$ such that $A^*-V(A^* \cap B^*)$ does not contain $a+1 \leq k-1$ members of $\F_A$ pairwise at distance in $G$ at least $r$.
By the inductive hypothesis (taking $(G,L,\F,k,(\theta_i: i \in [k]),Z)=(G,A^*-V(A^* \cap B^*),\F_A,a+1,(\theta_2,\theta_3,...,\theta_{a+2}),Z')$), there exists $Z_A^* \subseteq V(G)$ with $N_G^{\leq r'}[Z'] \subseteq Z_A^* \subseteq N_G^{\leq r'}[Z'] \cup (V(A^*)-V(A^* \cap B^*))$ and a set $\HH_A$ of disjoint connected induced subgraphs of $A^*-V(A^* \cap B^*)$ with $|\HH_A| \leq a$ such that 
	\begin{itemize}
		\item[(ii)] for every $H \in \HH_A$, there exist $i_H' \in [a] \cup \{0\}$ and $Z_H \subseteq Z_A^*$ with $N_G(V(H)) \cup (Z_A^* \cap V(H)-N_G^{\leq r'}[Z']) \subseteq Z_H$ such that 
			\begin{itemize}
				\item $Z_H-N_G^{\leq 2r'}[Z']$ is $(3\theta_{2}-3+2\sum_{i=2}^{i_H'}(3\theta_{i+1}-3),2r')$-centered in $G$,
				\item there exists a $(G,\F_A \cap H,r',\theta_{(i_H'+1)+1},Z_A^*)$-tangle in $H$, and 
				\item $H$ contains every component $C$ of $(A^*-V(A^* \cap B^*))-Z_A^*$ with $V(H) \cap V(C) \neq \emptyset$. 
			\end{itemize}
		\item[(iii)] $Z_A^*-N_G^{\leq 2r'}[Z']$ is $(3\theta_2-3+2\sum_{i=2}^{a}(3\theta_{i+1}-3),2r')$-centered in $G$, and 
		\item[(iv)] $Z_A^*$ intersects all members of $\F_A-(\F_A \cap \bigcup_{H \in \HH_A}H)$. 
	\end{itemize}
Moreover, we may assume that 
	\begin{itemize}
		\item[(v)] if $a=0$, then $\HH_A=\emptyset$ and $Z_A^* = N_G^{\leq r'}[Z']$,
	\end{itemize}
since this choice of $\HH_A$ and $Z_A^*$ satisfy (ii)-(iv).

Let $\F_B = (\F-N_G^{\leq r'}[Z']) \cap (B^*-V(A^* \cap B^*))$.
Let $b$ be the maximum number of members of $\F_B$ pairwise at distance in $G$ at least $r$.
So $b \leq \nu(L,\F)-1 \leq k-2$ by (i).
Similarly, by the inductive hypothesis, there exists $Z_B^* \subseteq V(G)$ with $N_G^{\leq r'}[Z'] \cup (V(B^*)-V(A^* \cap B^*)) \supseteq Z_B^* \supseteq N_G^{\leq r'}[Z']$ and a set $\HH_B$ of disjoint connected induced subgraphs of $B^*-V(A^* \cap B^*)$ with $|\HH_B| \leq b$ such that 
	\begin{itemize}
		\item[(vi)] for every $H \in \HH_B$, there exist $i_H' \in [b] \cup \{0\}$ and $Z_H \subseteq Z_B^*$ with $N_G(V(H)) \cup (Z_B^* \cap V(H)-N_G^{\leq r'}[Z']) \subseteq Z_H$ such that 
			\begin{itemize}
				\item $Z_H-N_G^{\leq 2r'}[Z']$ is $(3\theta_{2}-3+2\sum_{i=2}^{i_H'}(3\theta_{i+1}-3),2r')$-centered in $G$,
				\item there exists a $(G,\F_B \cap H,r',\theta_{(i_H'+1)+1},Z_B^*)$-tangle in $H$, and 
				\item $H$ contains every component $C$ of $(B^*-V(A^* \cap B^*))-Z_B^*$ with $V(H) \cap V(C) \neq \emptyset$. 
			\end{itemize}
		\item[(vii)] $Z_B^*-N_G^{\leq 2r'}[Z']$ is $(3\theta_2-3+2\sum_{i=2}^{b}(3\theta_{i+1}-3),2r')$-centered in $G$, 
		\item[(viii)] $Z_B^*$ intersects all members of $\F_B-(\F_B \cap \bigcup_{H \in \HH_B}H)$, and 
		\item[(ix)] if $b=0$, then $Z_B^* = N_G^{\leq r'}[Z']$,
	\end{itemize}

\medskip

\noindent{\bf Claim 3:} $a+b \leq \nu(L,\F)-1 \leq k-1$.

\noindent{\bf Proof of Claim 3:} 
Since $r' \geq r/2$ and every member of $\F_A \cup \F_B$ is disjoint from $N_G^{\leq r'}[Z \cup V(A^* \cup B^*)]$, the distance in $G$ between every member of $\F_A$ and every member of $\F_B$ is at least $r$.
So $a+b \leq \nu(L,\F)-1 \leq k-1$.
$\Box$

\medskip

Let $Z^* = Z_A^* \cup Z_B^*$.
Then $N_G^{\leq r'}[Z] \subseteq N_G^{\leq r'}[Z] \cup V(L)$.

\medskip

\noindent{\bf Claim 4:} $Z^*-N_G^{\leq 2r'}[Z]$ is $(3\theta_1-3+2\sum_{i=2}^{k-1}(3\theta_i-3),2r')$-centered in $G$, and $Z^*$ is $(\xi+3\theta_1-3+2\sum_{i=2}^{k-1}(3\theta_i-3),\eta+2r')$-centered in $G$.

\noindent{\bf Proof of Claim 4:} 
It suffices to show $Z^*-N_G^{\leq 2r'}[Z]$ is $(3\theta_1-3+2\sum_{i=2}^{k-1}(3\theta_i-3),2r')$-centered in $G$ since $Z$ is $(\xi,\eta)$-centered in $G$.
Since $Z'=Z \cup V(A^* \cap B^*)$, we know that $Z^*-N_G^{\leq 2r'}[Z] \subseteq (Z_A^*-N_G^{\leq 2r'}[Z']) \cup (Z_B^*-N_G^{\leq 2r'}[Z']) \cup N_G^{\leq 2r'}[V(A^* \cap B^*)]$.
Hence it suffices to show that $(Z_A^*-N_G^{\leq 2r'}[Z']) \cup (Z_B^*-N_G^{\leq 2r'}[Z'])$ is $(2\sum_{i=2}^{k-1}(3\theta_i-3),2r')$-centered in $G$. 

If $a=0$, then $Z_A^* -N_G^{\leq 2r'}[Z']=\emptyset$ by (v), so $(Z_A^*-N_G^{\leq 2r'}[Z']) \cup (Z_B^*-N_G^{\leq 2r'}[Z']) = Z_B^*-N_G^{\leq 2r'}[Z']$ is $(3\theta_2-3+2\sum_{i=2}^{b}(3\theta_{i+1}-3),2r')$-centered in $G$ by (vi) and hence is $(2\sum_{i=2}^{k-1}(3\theta_i-3),2r')$-centered in $G$ since $b+1 \leq k-1$.

So we may assume $a \geq 1$.
Similarly, we may assume $b \geq 1$.
By (iii) and (vii), $(Z_A^*-N_G^{\leq 2r'}[Z']) \cup (Z_B^*-N_G^{\leq 2r'}[Z'])$ is $(3\theta_2-3+2\sum_{i=2}^{a}(3\theta_{i+1}-3) + 3\theta_2-3+2\sum_{i=2}^{b}(3\theta_{i+1}-3),2r')$-centered in $G$.
Since $a \geq 1$ and $a+b \geq 2$ and $(\theta_i: i \in [k])$ is nondecreasing,
	\begin{align*}
		 & 3\theta_2-3+2\sum_{i=2}^{a}(3\theta_{i+1}-3) + 3\theta_2-3+2\sum_{i=2}^{b}(3\theta_i-3) \\
		 = & 2(3\theta_2-3)+2\sum_{i=3}^{a+1}(3\theta_{i+1}-3) + 2\sum_{i=3}^{b+1}(3\theta_i-3) \\
		 \leq & 2(3\theta_2-3)+2\sum_{i=3}^{a+1}(3\theta_{i+1}-3) + 2\sum_{i=a+2}^{a+b}(3\theta_i-3) \\
		 = & 2\sum_{i=2}^{a+b}(3\theta_{i+1}-3).
	\end{align*}
By Claim 3, $a+b \leq k-1$.
Therefore, $(Z_A^*-N_G^{\leq 2r'}[Z']) \cup (Z_B^*-N_G^{\leq 2r'}[Z'])$ is $(2\sum_{i=2}^{k-1}(3\theta_i-3),2r')$-centered in $G$.  
$\Box$

\medskip

Statement 2 of this lemma follows from Claim 4.

Let $\HH = \HH_A \cup \HH_B$.
Since every member in $\HH_A$ is contained in $A^*-V(A^* \cap B^*)$, and every member in $\HH_B$ is contained in $B^*-V(A^* \cap B^*)$, we know that $\HH$ is a set of disjoint connected induced subgraphs of $L$.
By Claim 3, $|\HH| \leq |\HH_A|+|\HH_B| \leq a+b \leq k-1$.

\medskip

\noindent{\bf Claim 5:} Statement 3 of this lemma holds.

\noindent{\bf Proof of Claim 5:} 
By Claim 4, it suffices to show that $Z^*$ intersects all members of $\F-(\F \cap \bigcup_{H \in \HH}H)$.
Suppose to the contrary that $Z^*$ does not intersect all members of $\F-(\F \cap \bigcup_{H \in \HH}H)$.
Then there exists $M \in \F-(\F \cap \bigcup_{H \in \HH}H)$ with $V(M) \cap Z^*=\emptyset$.
So $V(M) \cap N_G^{\leq r'}[Z']=\emptyset$.
Since every member of $\F$ is connected, $M \in \F_A \cup \F_B$.
By (iv) and (viii), $Z^*$ intersects all members of $(\F_A-(\F_A \cap \bigcup_{H \in \HH_A}H)) \cup (\F_B-(\F_B \cap \bigcup_{H \in \HH_B}H))$.
Hence $M \in (\F_A \cap \bigcup_{H \in \HH_A}H) \cup (\F_B \cap \bigcup_{H \in \HH_B}H) \subseteq \F \cap \bigcup_{H \in \HH}H$, a contradiction.
$\Box$

\medskip

For every $H \in \HH=\HH_A \cup \HH_B$, let $i_H=i_H'+1$; note that $1 \leq i_H \leq \max\{a,b\}+1 \leq (\nu(L,\F)-1)+1 = \nu(L,\F)$.
Note that (ii) and (vi) imply that $Z_H \subseteq Z_A^* \cup Z_B^* =Z^*$ and $N_G(V(H)) \subseteq Z_H$ for every $H \in \HH$.

\medskip

\noindent{\bf Claim 6:} For every $H \in \HH$, $Z^* \cap V(H)-N_G^{\leq r'}[Z] \subseteq Z_H$. 

\noindent{\bf Proof of Claim 6:}
By symmetry, we may assume $H \in \HH_A$.
So $H$ is an induced subgraph of $A^*-N_G^{\leq r'}[V(A^* \cap B^*)]$.

Hence $V(H) \cap N_G^{\leq r'}[V(A^* \cap B^*)]=\emptyset$. 
So $V(H) \cap N_G^{\leq r'}[Z'] = (V(H) \cap N_G^{\leq r'}[Z]) \cup (V(H) \cap N_G^{\leq r'}[V(A^* \cap B^*)]) = V(H) \cap N_G^{\leq r'}[Z]$.
Hence $V(H)-N_G^{\leq r'}[Z]= V(H)-N_G^{\leq r'}[Z']$.
So $Z_A^* \cap V(H)-N_G^{\leq r'}[Z] = Z_A^* \cap V(H)-N_G^{\leq r'}[Z'] \subseteq Z_H$ by (ii).

Since $V(H) \subseteq V(A^*)-V(A^* \cap B^*)$ and $Z_B^* \subseteq N_G^{\leq r'}[Z'] \cup (V(B^*)-V(A^* \cap B^*))$, we know that $Z_B^* \cap V(H)-N_G^{\leq r'}[Z] \subseteq (N_G^{\leq r'}[Z'] \cup (V(B^*)-V(A^* \cap B^*))) \cap V(A^*)-N_G^{\leq r'}[Z]=\emptyset$.
So $Z^* \cap V(H)-N_G^{\leq r'}[Z] = (Z_A^* \cap V(H)-N_G^{\leq r'}[Z]) \cup (Z_B^* \cap V(H)-N_G^{\leq r'}[Z]) \subseteq Z_H$.
$\Box$

\medskip

Since we suppose that this lemma does not hold, Claims 4 and 5 implies that Statement 1 of this lemma does not hold. 
So there exists $H \in \HH$ showing that one of Statements 1(a)-1(c) does not hold by Claim 6. 

By (ii) and (vi), $Z_H-N_G^{\leq 2r'}[Z']$ is $(3\theta_{2}-3+2\sum_{i=2}^{i_H'}(3\theta_{i+1}-3),2r')$-centered in $G$.
Note that $3\theta_{2}-3+2\sum_{i=2}^{i_H'}(3\theta_{i+1}-3) = 2\sum_{i=2}^{i_H'+1}(3\theta_{i}-3) = 2\sum_{i=2}^{i_H}(3\theta_{i}-3)$.
So Statement 1(a) holds.

Since $V(A^* \cap B^*) \subseteq Z'$, every component of $L-Z^*$ is contained in a component of exactly one of $(A^*-V(A^* \cap B^*))-Z_A^*$ or $(B^*-V(A^* \cap B^*))-Z_B^*$.
So $H$ contains every component $C$ of $L-Z^*$ with $V(H) \cap V(C) \neq \emptyset$ by (ii) and (vi).
Hence Statement 1(c) holds.

Therefore, there does not exist a $(G,\F \cap H,r',\theta_{i_H+1},Z^*)$-tangle in $H$.

By symmetry, we may assume that $H \in \HH_B$, so $H \subseteq B^*-V(A^* \cap B^*)$.
By (vi), there exists a $(G,\F_B \cap H,r',\theta_{i_H+1},Z_B^*)$-tangle $\T$ in $H$ since $i_H+1=i'_H+2$.   
Since $Z_A^* \subseteq N_G^{\leq r'}[Z'] \cup V(A^*)$, we know that $Z_A^* \cap V(B^*) \subseteq (N_G^{\leq r'}[Z'] \cup V(A^*)) \cap V(B^*) \subseteq Z_B^* \cap V(B^*)$.
Hence $Z^* \cap V(H) = Z_B^* \cap V(H)$.
So $(\F_B \cap H)-N_G^{\leq r'}[Z_B^*] = (\F \cap H)-N_G^{\leq r'}[Z^*]$.
Hence for every $(X,Y) \in \T$, since $\T$ is a $(G,\F_B \cap H,r',\theta_{i_H+1},Z_B^*)$-tangle in $H$, we know that $(X,Y)$ is a separation of $H \subseteq B^*$ such that $X-N_G^{\leq r'}[V(X \cap Y)]$ does not contain a member of $(\F_B \cap H)-N_G^{\leq r'}[Z_B^*]=(\F \cap H)-N_G^{\leq r'}[Z^*]$, but $Y-N_G^{\leq r'}[V(X \cap Y)]$ contains a member of $(\F_B \cap H)-N_G^{\leq r'}[Z_B^*] = (\F \cap H)-N_G^{\leq r'}[Z^*]$.
So $\T$ is a $(G,\F \cap H,r',\theta_{i_H+1},Z^*)$-tangle in $H$, a contradiction.
This proves the lemma.
\end{pf}

\bigskip

We will use the following special case of Lemma \ref{easy_tangle_multifold}. 

\begin{lemma} \label{easy_tangle_1}
Let $G$ be a graph, and let $L$ be an induced subgraph of $G$.
Let $\F$ be a set of connected subgraphs of $L$.
Let $k,\theta \in {\mathbb N}$ and $r,r',\xi,\eta \in {\mathbb R}_{\geq 0}$ with $r' \geq r/2$.
Let $Z$ be a subset of $V(G)$ with $N_G(V(L)) \subseteq Z$ such that $Z$ is $(\xi,\eta)$-centered in $G$.
If $L$ does not contain $k$ members of $\F$ pairwise at distance in $G$ at least $r$, then there exist $Z^* \subseteq V(G)$ with $N_G^{\leq r'}[Z] \subseteq Z^* \subseteq N_G^{\leq r'}[Z] \cup V(L)$ and a (possibly empty) set $\HH$ of disjoint connected induced subgraphs of $L$ with $|\HH| \leq k-1$ such that 
	\begin{enumerate}
		\item $Z^*$ is $(\xi+\max\{2k-3,1\} \cdot (3\theta-3),\eta+2r')$-centered in $G$ and intersects all members of $\F-(\F \cap \bigcup_{H \in \HH}H)$,  
		\item $Z^*-N_G^{\leq 2r'}[Z]$ is $(\max\{2k-3,1\} \cdot (3\theta-3),2r')$-centered in $G$, and 
		\item for every $H \in \HH$, there exists a $(G,\F \cap H,r',\theta,Z^*)$-tangle in $H$, $N_G(V(H)) \subseteq Z^*$ and $H$ contains every component $C$ of $L-Z^*$ with $V(H) \cap V(C) \neq \emptyset$. 
	\end{enumerate}
\end{lemma}

\begin{pf}
Apply Lemma \ref{easy_tangle_multifold} by taking $(\theta_i: i \in [k])$ to be the constant sequence whose every entry equals $\theta$.
Then Statements 1,2,3 of this lemma follows from Statements 3,2,1 of Lemma \ref{easy_tangle_multifold}, respectively.
\end{pf}

\bigskip

By a version of the well-known Grid Minor Theorem (such as Theorems \cite[(1.5)]{rst} and \cite[(2.2)]{rst}), for every planar graph $H$, there exists $\theta \in {\mathbb N}$ such that no $H$-minor free graph has a tangle of order at least $\theta$.
Hence if the graph $G$ is $H$-minor free for some planar graph $H$, and $\theta$ is a sufficiently large integer, then the set $\HH$ in the conclusion of Lemma \ref{easy_tangle_1} must be empty by Statement 3, and hence $Z^*$ intersects all members of $\F$ by Statement 1.
Therefore, Lemma \ref{easy_tangle_1} implies the planar $H$ case of Theorem \ref{finite_minor_induction_intro}.

We introduce other notions related to tangles that will be used in this paper.

Let $\T$ be a tangle in a graph $G$.
For $Z \subseteq V(G)$ with $\lvert Z \rvert<\theta$, we define $\T-Z$ to be the set of all separations $(A',B')$ of $G-Z$ of order less than $\theta-\lvert Z \rvert$ such that there exists $(A,B) \in \T$ with $Z \subseteq V(A \cap B)$, $A'=A-Z$ and $B'=B-Z$.
It is proved in \cite[Theorem 8.5]{rs X} that $\T-Z$ is a tangle in $G-Z$ of order $\theta-\lvert Z \rvert$.

Let $\theta$ be a positive integer.
A \defn{pretangle} of order $\theta$ in a graph $G$ is a set of separations of $G$ of order less than $\theta$ satisfying (T1) and (T2).

\subsection{Minors and conformal tangles} \label{subsec:conformal}

Let $H$ be a graph.
An \defn{$H$-minor} in a graph $G$ is a function $\mu$ with domain $V(H)\cup E(H)$ such that 
	\begin{itemize}
		\item for every $v \in V(H)$, $\mu(v)$ is a non-empty connected subgraph of $G$, called the \defn{branch set for $v$},
		\item $\mu(u)\cap \mu(v)=\emptyset$ for all distinct $u,v\in V(H)$,  
		\item for every edge $uv \in E(H)$, $\mu(uv)$ is an edge of $G$ between $\mu(u)$ and $\mu(v)$, and
		\item if $e_1$ and $e_2$ are distinct edges, then $\mu(e_1) \neq \mu(e_2)$.
	\end{itemize}
A tangle $\T$ in $G$ \defn{controls} an $H$-minor $\mu$ if $\mu$ is an $H$-minor in $G$ such that there do not exist $(A,B) \in \T$ of order less than $\lvert V(H) \rvert$ and $h \in V(H)$ such that $V(\mu(h)) \subseteq V(A)$.

Let $G$ and $H$ be graphs.
Let $\T_H$ be a pretangle in $H$ of order $\theta \geq 2$.
Let $\mu$ be an $H$-minor in $G$.
Let $\T_G$ be the set of separations $(A,B)$ of $G$ of order less than $\theta$ such that there exists $(A',B') \in \T_H$ with $\mu(E(A')) = E(A) \cap \mu(E(H))$.
By \cite[(6.1)]{rs X}, if $\T_H$ is a tangle, then $\T_G$ is a tangle in $G$ of order $\theta$.
The same argument shows that $\T_G$ is a pretangle in $G$ of order $\theta$.
We call $\T_G$ the \defn{(pre)tangle induced by $\T_H$}.
We say that $\T_H$ is \defn{conformal} with a tangle $\T$ in $G$ if $\T_G \subseteq \T$.

\subsection{Tree-decompositions and locations} \label{subsec:tree_decomp}

Now we show how to use tree-decompositions to deal with coarse Erd\H{o}s-P\'{o}sa property.
Tree-decompositions allow us to encode subgraphs into trees.
Before stating related notions, we first mention some results about trees in the literature.

\begin{theorem}[{{\cite[(8.7)]{rs V}}}] \label{Helly_trees}
Let $T$ be a tree, and let $\F$ be a family of subtrees of $T$.
Then for any integer $k \geq 0$, either there are $k$ disjoint members of $\F$, or there is a subset $X \subseteq V(T)$ with $|X| \leq k-1$ intersecting all members of $\F$.
\end{theorem}

\begin{theorem}[{{\cite[(8.6)]{rs V}}}] \label{EP_multi_trees}
Let $T$ be a tree, and let $\F_1,\F_2,...,\F_m$ be families of subtrees of $T$, where $m$ is a positive integer.
Let $x_1,x_2,...,x_m \geq 0$ be integers, and let $k=x_1+x_2+...+x_m$.
If for every $i \in [m]$, $\F_i$ contains $k$ disjoint members, then for every $i \in [m]$, there exist $x_i$ members $T^i_1,T^i_2,...,T^i_{x_i}$ of $\F_i$ such that the members of $\{T^\alpha_\beta: \alpha \in [m], \beta \in [x_\alpha]\}$ are pairwise disjoint.
\end{theorem}

Let $G$ be a graph.
A \defn{tree-decomposition of $G$} is a pair $(T,\X)$, where $T$ is a tree and $\X=\{X_t: t \in V(T)\}$ such that the following hold.
	\begin{itemize}
		\item $\bigcup_{t \in V(T)} X_t = V(G)$.
		\item For every edge $e$ of $G$, there exists $t \in V(T)$ such that $X_t$ contains all ends of $e$.
		\item For every vertex $v$ of $G$, the subgraph of $T$ induced by the set $\{t \in V(T): v \in X_t\}$ is connected.
	\end{itemize}
For every $t \in V(T)$, the set $X_t$ is called the \defn{bag} at $t$.
The \defn{adhesion} of $(T,\X)$ is $\sup\{\lvert X_t \cap X_{t'} \rvert: tt' \in E(T)\}$.

We say that a tree-decomposition $(T,\X)$ is a \defn{path-decomposition} if $T$ is path.

A \defn{location} in a graph $G$ is a set $\LL$ of separations such that $A \subseteq B'$ and $A' \subseteq B$ for any distinct $(A,B),(A',B') \in \LL$.
Let $\LL$ be a location in a graph $G$.
A \defn{tree-decomposition of $\LL$} is a tree-decomposition of $G[\bigcap_{(A,B) \in \LL}V(B)]$ such that for every $(A,B) \in \LL$, there exists $t \in V(T)$ such that $V(A \cap B) \subseteq X_t$.

Let $\F$ be a set of subgraphs of $G$.
We say that $\F$ is \defn{component exchangeable} if 
	\begin{itemize}
		\item there exists $c \in {\mathbb N}$ such that every member of $\F$ has exactly $c$ components, and 
		\item for every $F \in \F$, the $c$ components of $F$ can be denoted by $F(1),F(2),...,F(c)$ such that $\bigcup_{i=1}^cF_i(i) \in \F$ for any (not necessarily distinct) $F_1,F_2,...,F_c \in \F$ with $F_1(1),F_2(2),...,F_c(c)$ pairwise disjoint.
	\end{itemize}
Note that if all members of $\F$ are connected, then $\F$ is component exchangeable.
If $W$ is a graph with $c$ components, then the set of all $W$-minors in $G$ is component exchangeable.
The set of rooted $\ell$-fat $W$-minor models is not component exchangeable, but it can be slightly modified to make it component exchangeable, and our next lemma (Lemma \ref{easy_tree}) can still be applied (see Proposition \ref{rooted_fat_EP_planar}).

Now we prove the main result of this subsection.

\begin{lemma} \label{easy_tree}
Let $G$ be a graph, and let $L$ be a subgraph of $G$.
Let $c,k \in {\mathbb N}$ and $r,\xi,\eta \in {\mathbb N}_0$.
Let $\F$ be a component exchangeable set of subgraphs of $L$ such that every member of $\F$ has exactly $c$ components, and for every component $C$ of a member of $\F$, $L[N_G^{\leq r}[V(C)]]$ is connected.
Let $\LL$ be a location in $L$ such that for every $(A,B) \in \LL$, no component of a member of $\F$ is contained in $A-N_G^{\leq r}[V(A \cap B)]$. 
Let $(T,(X_t: t \in V(T)))$ be a tree-decomposition of $\LL$ such that every bag $X_t$ is $(\xi,\eta)$-centered in $G$.
If $L$ does not contain $k$ members of $\F$ pairwise at distance in $G$ greater than $2r$, then there exists a $((ck-1)\xi,\eta+r)$-centered set $Z$ in $G$ intersecting all members of $\F$.
\end{lemma}

\begin{pf}
Since $(T,(X_t: t \in V(T)))$ is a tree-decomposition of $\LL$, we know that for every $(A,B) \in \LL$, there exists $t_A \in V(T)$ such that $V(A \cap B) \subseteq X_{t_A}$.
Let $T'$ be the tree obtained from $T$ by, for each $(A,B) \in \LL$, adding a leaf $t'_A$ adjacent to $t_A$.
For every $(A,B) \in \LL$, let $X_{t'_A} = V(A)$.
Let $\X' = (X_t: t \in V(T'))$.
Then $(T',\X')$ is a tree-decomposition of $L$.

Since $\F$ is component exchangeable, for every $F \in \F$, we can denote the components of $F$ by $F(1),F(2),...,F(c)$ such that $\bigcup_{i=1}^cF_i(i) \in \F$ for any (not necessarily distinct) $F_1,F_2,...,F_c \in \F$ with $F_1(1),F_2(2),...,F_c(c)$ pairwise disjoint.
For any $F \in \F$ and $j \in [c]$, let $F(j)^+ = L[N_G^{\leq r}[V(F(j))]]$, and let $T_{F(j)}$ be the subgraph of $T'$ induced by $\{t \in V(T'): X_t \cap V(F(j)^+) \neq \emptyset\}$.
For any $F \in \F$ and $j \in [c]$, we know that $F(j)^+$ is connected by the assumption of this lemma, so $T_{F(j)}$ is also connected.
For every $j \in [c]$, let $\T^j = \{T_{F(j)}: F \in \F\}$, so $\T^j$ is a family of subtrees of $T'$.

\medskip

\noindent{\bf Claim 1:} $V(T_{F(j)}) \cap V(T) \neq \emptyset$ for any $F \in \F$ and $j \in [c]$.

\noindent{\bf Proof of Claim 1:} 
Suppose to the contrary that there exist $F \in \F$ and $j \in [c]$ such that $V(T_{F(j)}) \cap V(T) = \emptyset$.
Then there exists $(A,B) \in \LL$ such that $V(T_{F(j)})=\{t'_A\}$.
So $V(F(j)) \subseteq V(F(j)^+) \subseteq V(A)-V(B)$.
By the assumption of this lemma, $F(j)$ is not contained in $A-N_G^{\leq r}[V(A \cap B)]$.
So there exists $v \in V(F(j)) \cap N_G^{\leq r}[V(A \cap B)]$.
Then $N_G^{\leq r}[v] \cap V(A \cap B) \neq \emptyset$. 
But $N_G^{\leq r}[v] \subseteq V(F(j)^+)$.
So $V(F(j)^+) \cap X_{t_A} \supseteq V(F(j)^+) \cap V(A \cap B) \neq \emptyset$, implying $V(T_{F(j)}) \cap V(T) \neq \emptyset$, a contradiction.
$\Box$

\medskip

Now we suppose that this lemma does not hold and will derive a contradiction. 

\medskip

\noindent{\bf Claim 2:} For every $j \in [c]$, there exists a subset $\T_j'$ of $\T^j$ with $|\T_j'|=ck$ such that members of $\T_j'$ are pairwise disjoint.

\noindent{\bf Proof of Claim 2:} 
Suppose to the contrary that there exists $j \in [c]$ such that no subset of $\T^j$ contains $ck$ pairwise disjoint members.
By Theorem \ref{Helly_trees}, there exists $X \subseteq V(T')$ with $|X| \leq ck-1$ intersecting all members of $\T^j$.
Let $X'$ be the subset of $V(T)$ obtained from $X$ by replacing each $t'_A \in X$ by $t_A$.
Then $|X'| \leq |X| \leq ck-1$ and $X'$ intersects all members of $\T^j$ by Claim 1.

Let $Z = \bigcup_{t \in X'}X_t$.
Since $X_t$ is $(\xi,\eta)$-centered in $G$ for every $t \in V(T) \supseteq X'$, we know that $Z$ is $((ck-1)\xi,\eta)$-centered in $G$.
Since $X'$ intersects all members of $\T^j$, we know that $Z \cap V(F(j)^+) \neq \emptyset$ for every $F \in \F$.
Hence $N_G^{\leq r}[Z]$ intersects $F(j)$ (and hence $F$) for every $F \in \F$.
Note that $N_G^{\leq r}[Z]$ is $((ck-1)\xi,\eta+r)$-centered in $G$.
So this lemma holds, a contradiction.
$\Box$

\medskip

By Claim 2 and Theorem \ref{EP_multi_trees}, for every $j \in [c]$, there exist $k$ members $T^j_1,T^j_2,...,T^j_k$ of $\T_j' \subseteq \T^j$ such that the members of $\{T^\alpha_\beta: \alpha \in [c], \beta \in [k]\}$ are pairwise disjoint.
Note that for any $\alpha \in [c]$ and $\beta \in [k]$, there exists $F_{\alpha,\beta} \in \F$ such that $T^\alpha_\beta$ is induced by $\{t \in V(T'): X_t \cap V(F_{\alpha,\beta}(\alpha)^+) \neq \emptyset\}$.
Since the members of $\{T^\alpha_\beta: \alpha \in [c], \beta \in [k]\}$ are pairwise disjoint, the members of $\{F_{\alpha,\beta}(\alpha)^+: \alpha \in [c],\beta \in [k]\}$ are pairwise disjoint.

For every $\beta \in [k]$, let $M_\beta = \bigcup_{\alpha \in [c]}F_{\alpha,\beta}(\alpha)$ and $M_\beta^+ = \bigcup_{\alpha \in [c]}F_{\alpha,\beta}(\alpha)^+$.
Since the members of $\{F_{\alpha,\beta}(\alpha)^+: \alpha \in [c],\beta \in [k]\}$ are pairwise disjoint, $M_1^+,M_2^+,...,M_k^+$ are pairwise disjoint subgraphs of $L$.
So $M_1,M_2,...,M_k$ are subgraphs of $L$ with pairwise distance in $G$ greater than $2r$.
But for every $i \in [k]$, $M_i$ is a member of $\F$ since $\F$ is component exchangeable.
Hence $L$ contains $k$ members of $\F$ pairwise at distance greater than $2r$, a contradiction.
\end{pf}

\bigskip

As a side application of Lemma \ref{easy_tree}, we prove Proposition \ref{rooted_fat_EP_planar_intro}.
The following is a restatement.

\begin{proposition} \label{rooted_fat_EP_planar}
For any (finite) planar graph $H$ and (finite) graph $W$, there exists $\beta \in {\mathbb N}$ such that if $G$ is an $H$-minor free (finite) graph, then for any collection $(R_v: v \in V(W))$ of subsets of $V(G)$ and integers $k,r,\ell \in {\mathbb N}$, either $G$ contains $k$ $(R_v: v \in V(W))$-rooted $\ell$-fat $W$-minor models pairwise at distance at least $r$, or there exists a set $(\beta k,\max\{\lceil \frac{r-1}{2} \rceil, \frac{\ell}{2}\})$-centered in $G$ intersecting all $(R_v: v \in V(W))$-rooted $\ell$-fat $W$-minor models.
\end{proposition}

\begin{pf}
Let $c$ be the number of components of $W$.
Each $(R_v: v \in V(W))$-rooted $\ell$-fat $W$-minor model gives a subgraph $S$ of $G$ with exactly $c$ components, and the subgraph $G[N_G^{\leq \ell/2-0.1}[V(S)]]$ of $G$ still has exactly $c$ components.
Let $\F$ be the set consisting of all such $G[N_G^{\leq \ell/2-0.1}[V(S)]]$ (for all possible $S$).
Then $\F$ is also component exchangeable.
If there exist $k$ members of $\F$ pairwise at distance in $G$ greater than $2 \cdot \max\{\lceil \frac{r-1}{2} \rceil -\frac{\ell}{2},0\}$, then there exist $k$ $(R_v: v \in V(W))$-rooted $\ell$-fat $W$-minor models pairwise at distance in $G$ at least $r$; if there exists an $(\alpha,\max\{\lceil \frac{r-1}{2} \rceil -\frac{\ell}{2},0\})$-centered set intersecting all members of $\F$ for some $\alpha \geq 0$, then there exists an $(\alpha,\max\{\lceil \frac{r-1}{2} \rceil, \frac{\ell}{2}\})$-centered set intersecting all $(R_v: v \in V(W))$-rooted $\ell$-fat $W$-minor models.
Hence this proposition follows from Lemma \ref{easy_tree} (taking $(G,L,c,k,r,\xi,\eta,\LL)=(G,G,c,k,\max\{\lceil \frac{r-1}{2} \rceil -\frac{\ell}{2},0\},w,0,\{(\emptyset,G)\})$), where $w$ is an integer such that every $H$-minor free graph has a tree-decomposition whose every bag has size at most $w$.
Note that the existence of $w$ follows from the Grid Minor Theorem \cite{rs V}.
\end{pf}

\subsection{Societies and vortices} \label{subsec:society}

A \defn{society} is a pair $(S,\Omega)$, where $S$ is a graph and $\Omega$ is a cyclic ordering of a subset $\overline{\Omega}$ of $V(S)$.
Let $Z \subseteq V(S)$.
We define $$(S,\Omega)-Z = (S-Z,\Omega-Z),$$ where $\Omega-Z$ is the cyclic ordering obtained from $\Omega$ by deleting all vertices in $\overline{\Omega} \cap Z$.
We define $$(S,\Omega)[Z] = (S[Z],\Omega[Z]),$$ where $\Omega[Z] = \Omega-(\overline{\Omega}-Z)$.

Let $(S,\Omega)$ be a society.
We say $(S,\Omega)$ is \defn{effective} if for every $v \in \overline{\Omega}$, there exist $\lvert \overline{\Omega} \rvert-1$ paths in $S$ from $v$ to $\overline{\Omega}-\{v\}$ only sharing $v$.
Note that if $(S,\Omega)$ is effective and $|\overline{\Omega}| \leq 3$, then there exists a $K_{|\overline{\Omega}|}$-minor in $S$ such that each branch set contains exactly one element of $\overline{\Omega}$.

There are simple ways to split a non-effective societies into a set of effective societies.
We state one way here.
An \defn{adaption} of a society $(S,\Omega)$ with $|\overline{\Omega}| \leq 3$ is a set of societies defined recursively as follows: 
	\begin{itemize}
		\item If $(S,\Omega)$ is effective, then $\A=\{(S,\Omega)\}$; 
		\item if $|\overline{\Omega}|=2$ and there exists no path in $S$ between the two vertices, say $a,b$, in $\overline{\Omega}$, then $\A = \{(A,\{a\}),(B,\{b\})\}$, where $(A,B)$ is a separation of $S$ of order 0 such that $a \in V(A)$ and $b \in V(B)$;
		\item if $|\overline{\Omega}|=3$ and there exists $a \in \overline{\Omega}$ such that no path in $S$ is from $a$ to $\overline{\Omega}-\{a\}$, then $\A = \{(A,\{a\})\} \cup \A'$, where $(A,B)$ is a separation of $S$ of order 0 such that $a \in V(A)$ and $\overline{\Omega}-\{a\} \subseteq V(B)$, and $\A'$ is an adaption of the society $(B,\Omega-\{a\})$;
		\item if $|\overline{\Omega}|=3$, some component of $S$ contains all vertices in $\overline{\Omega}$, but there exists $a \in \overline{\Omega}$ such that there exist no two paths in $S$ from $a$ to $\overline{\Omega}-\{a\}$ only sharing $a$, then $\A = \{(A,\{a\} \cup V(A \cap B))\} \cup \A'$, where 
			\begin{itemize}
				\item $(A,B)$ is a separation of $S$ of order 1 such that $a \in V(A)$ and there exist $|\overline{\Omega}-(\{a\} \cup V(A \cap B))|$ paths in $B$ from $V(A \cap B)$ to $\overline{\Omega}-(\{a\} \cup V(A \cap B))$ only sharing the vertex in $V(A \cap B)$, and 
				\item $\A'$ is an adaption of the society $(B,\Omega')$, where $\Omega'$ is a cyclic ordering on $V(A \cap B) \cup (\overline{\Omega}-\{a\})$ such that $\Omega'-V(A \cap B) = \Omega-\{a\}$.  
			\end{itemize}
	\end{itemize}
Note that the separations $(A,B)$ mentioned above exist by Menger's theorem.
Moreover, all societies in an adaption of $(S,\Omega)$ are effective.

\begin{lemma} \label{adaption_size}
Let $(S,\Omega)$ be a society with $|\overline{\Omega}| \leq 3$.
If $|\overline{\Omega}| \leq 2$, then every adaption of $(S,\Omega)$ has size at most $\max\{|\overline{\Omega}|,1\}$.
If $|\overline{\Omega}|=3$, then every adaption of $(S,\Omega)$ has size at most $1+t$, where $t$ is the number of vertices $x \in \overline{\Omega}$ such that no two paths in $S$ from $x$ to $\overline{\Omega}-\{x\}$ only sharing $x$.  
\end{lemma}

\begin{pf}
Let $\A$ be an adaption of $(S,\Omega)$.
If $(S,\Omega)$ is effective, then $|\A|=1$.
So we may assume that $(S,\Omega)$ is not effective.
Hence $|\overline{\Omega}| \geq 2$.
If $|\overline{\Omega}| = 2$, then $|\A|=2$ by definition and the lemma holds.

So we may assume $|\overline{\Omega}| =3$ and denote the vertices in $\overline{\Omega}$ by $\{a,b,c\}$.
Let $t$ be the number of vertices $x \in \overline{\Omega}$ such that no two paths in $S$ from $x$ to $\overline{\Omega}-\{x\}$ only sharing $x$.
We shall prove this lemma by induction on $t$.

When $t=0$, $(S,\Omega)$ is effective, so $|\A|=1$.
So we may assume $t \geq 1$ and assume that the lemma holds when $t$ is smaller.

If no component of $S$ contains all $a,b,c$, then there exists no path in $S$ from some vertex in $\overline{\Omega}$, say $a$, to the other two vertices in $\overline{\Omega}$, so $\A$ is a union of a set of size 1 and an adaption $\A'$ of a society $(S',\Omega-\{a\})$; since $|\overline{\Omega}-\{a\}| \leq 2$, $|\A| = 1+|\A'| \leq 3$; note that every vertex $x \in \overline{\Omega}$ satisfies that no two paths in $S$ from $x$ to $\overline{\Omega}-\{x\}$ only sharing $x$, so $t=3$.

Hence we may assume that some component of $S$ contains all $a,b,c$.
Since $t \geq 1$, by the definition of $\A$, there exist $a \in \overline{\Omega}$ and a separation $(A,B)$ of $S$ of order 1 such that $a \in V(A)$ and there exist $|\{b,c\}-V(A \cap B))|$ paths in $B$ from $V(A \cap B)$ to $\{b,c\}-V(A \cap B))$ only sharing the vertex in $V(A \cap B)$, and $|\A|=1+|\A'|$, where $\A'$ is an adaption of the society $(B,\Omega')$, where $\Omega'$ is a cyclic ordering on $V(A \cap B) \cup \{b,c\}$ such that $\Omega'-V(A \cap B) = \Omega-\{a\}$.

If $|\{b,c\} \cup V(A \cap B))|=2$, then $|\overline{\Omega'}|=2$ and $(B,\Omega')$ is effective since some component contains $\{a,b,c\}$, so $|\A|=1+|\A'|=2 \leq 1+t$.

So we may assume $|\{b,c\} \cup V(A \cap B))|=3$.
Let $v$ be the vertex in $V(A \cap B)$.
Note that $v \not \in \{b,c\}$.
Moreover, for every $u \in \{b,c\}$, if there exist two paths in $S$ from $u$ to $\overline{\Omega}-\{u\}$ only sharing $u$, then there exist two paths in $B$ from $u$ to $\{v\} \cup (\overline{\Omega}-\{a,u\})$ only sharing $u$.
Since there exist two paths in $B$ from $v$ to $\{b,c\}$ only sharing $v$, the number of vertices $x \in \overline{\Omega}$ such that no two paths in $B$ from $x$ to $\overline{\Omega'}-\{x\}$ only sharing $x$ is at most $t-1$.
Hence $|\A| = 1+|\A'| \leq 1 + (1+(t-1)) \leq 1+t$ by the inductive hypothesis.
\end{pf}

\bigskip

Let $\rho$ be a nonnegative integer.
A society $(S,\Omega)$ is a \defn{$\rho$-vortex} if for all distinct $u,v \in \overline{\Omega}$, there do not exist $\rho+1$ mutually disjoint paths in $S$ between $I \cup \{u\}$ and $J \cup \{v\}$, where $I$ is the set of vertices in $\overline{\Omega}$ after $u$ and before $v$ in $\Omega$, and $J$ is the set of vertices in $\overline{\Omega}$ after $v$ and before $u$ in $\Omega$.

Let $(S,\Omega)$ be a society with $\overline{\Omega} = \{v_1,v_2,...,v_{\lvert\overline{\Omega}\rvert}\}$ in order.
A \defn{vortical decomposition} of $(S,\Omega)$ is a path-decomposition $(t_1t_2...t_{\lvert \overline{\Omega} \rvert}, (X_{t_i}:1 \leq i \leq \lvert \overline{\Omega} \rvert))$ of $S$ such that $v_i \in X_{t_i}$ for $1 \leq i \leq  \lvert \overline{\Omega} \rvert$.
The following theorem ensures the existence of a vortical decomposition.

\begin{theorem}[{\cite[(8.1)]{rs IX}}] \label{path decomp of vortex}
For every positive integer $\rho$, every $\rho$-vortex has a vortical decomposition of adhesion at most $\rho$.
\end{theorem}

\subsection{Segregations} \label{subsec:segregation}

A \defn{segregation} of a graph $G$ is a set $\Se$ of societies such that the following conditions hold. 
\begin{itemize}
	\item $S$ is a subgraph of $G$ for every $(S, \Omega) \in \Se$, and $\bigcup_{(S,\Omega) \in \Se}S=G$.
	\item $V(S \cap S') \subseteq \overline{\Omega} \cap \overline{\Omega'}$ and $E(S \cap S') = \emptyset$ for any distinct $(S,\Omega)$ and $(S', \Omega') \in \Se$.
\end{itemize}
We write $V(\Se) = \bigcup_{(S,\Omega) \in \Se}\overline{\Omega}$.
For a tangle $\T$ in $G$, we say that a segregation $\Se$ of $G$ is \defn{$\T$-central} if for every $(S,\Omega) \in \Se$, there is no $(A,B) \in \T$ of order at most half of the order of $\T$ with $B \subseteq S$.

Let $\kappa,\rho$ be nonnegative integers.
For subsets $\Se_1,\Se_2$ of a segregation $\Se$, we say that $(\Se_1,\Se_2)$ is a \defn{$(\kappa,\rho)$-witness} of $\Se$ if $\Se_1 \cup \Se_2=\Se$, $\Se_1 \cap \Se_2=\emptyset$, $\lvert \overline{\Omega} \rvert \leq 3$ for every $(S,\Omega) \in \Se_1$, $\lvert \Se_2 \rvert \leq \kappa$ and every member of $\Se_2$ is a $\rho$-vortex.
(Note that we do not require that $\lvert \overline{\Omega} \rvert >3$ for every member $(S,\Omega) \in \Se_2$.)
We say that $\Se$ is a \defn{$(\kappa,\rho)$-segregation} if it has a $(\kappa,\rho)$-witness.

Let $\Se$ be a segregation of a graph $G$.
Let $\Se_1,\Se_2$ be subsets of $\Se$ with $\Se_1 \cap \Se_2=\emptyset$ and $\Se_1 \cup \Se_2 = \Se$ such that $\lvert \overline{\Omega} \rvert \leq 3$ for every $(S,\Omega) \in \Se_1$.
We say $\Se$ is \defn{effective with respect to $(\Se_1,\Se_2)$} if every member of $\Se_1$ is effective. 

Let $\Se$ be a segregation of a graph $G$.
Let $Z \subseteq V(G)$, and let $\Se' \subseteq \Se$. 
Then we define $$\Se'[Z] = \{(S,\Omega)[Z]: (S,\Omega) \in \Se'\}.$$
Note that $\Se[Z]$ is a segregation of $G[Z]$.

Let $G$ be a graph, and let $L$ be an induced subgraph of $G$.
Let $r$ be a real number, and let $L^+ = G[N_G^{\leq r}[V(L)]]$.
Let $\Se^+$ be a segregation of $L^+$ with a $(\kappa,\rho)$-witness $(\Se_1^+,\Se_2^+)$ for some $\kappa,\rho \in {\mathbb N}_0$.
An \defn{$(L,\Se_1^+,\Se_2^+)$-adaption} of $\Se^+$ is a segregation of $\Se$ of $L$ with a $(\kappa,\rho)$-witness $(\Se_1,\Se_2)$ such that $\Se_1 = \bigcup_{(S,\Omega) \in \Se_1^+[V(L)]}\A_{S,\Omega}$ and $\Se_2 = \Se_2^+[V(L)]$, where $\A_{S,\Omega}$ is an adaption of $(S,\Omega)$.
Note that $\Se$ is effective with respect to $(\Se_1,\Se_2)$.
We call $(\Se_1,\Se_2)$ the \defn{witness} of this $(L,\Se_1^+,\Se_2^+)$-adaption of $\Se^+$.

\subsection{Surfaces and drawings} \label{subsec:surface}

In this paper, a \defn{surface} is a nonnull compact $2$-manifold without boundary.
For every subset $\Delta$ of a surface $\Sigma$, we denote the closure of $\Delta$ by $\overline{\Delta}$, and the boundary of $\Delta$ by $\partial\Delta$.
An \defn{O-arc} in a surface $\Sigma$ is a subset of $\Sigma$ homeomorphic to a circle.
A \defn{line} in a surface $\Sigma$ is a subset of $\Sigma$ homeomorphic to the interval $[0,1]$.

A \defn{drawing} in a surface $\Sigma$ is a pair $(U,V)$, where $V \subseteq U \subseteq \Sigma$, $U$ is closed, $V$ is finite, $U-V$ has only finitely many arc-wise connected components, called \defn{edges}, and for every edge $e$, either the closure $\bar{e}$ of $e$ is a line whose set of ends is $\bar{e} \cap V$, or $\bar{e}$ is an O-arc and $\lvert \bar{e} \cap V \rvert =1$.
Every component of $\Sigma-U$ is called a \defn{region}.
Every member of $V$ is called a \defn{vertex}.
For a drawing $\Gamma=(U,V)$, we define $U(\Gamma) = U$, $V(\Gamma) = V$, and define the set of edges to be $E(\Gamma)$.

Let $\Sigma$ be a surface and $\Gamma$ a drawing in $\Sigma$.
For a vertex $v$ of $\Gamma$ and an edge or a region $e$ of $\Gamma$, we say that $e$ is \defn{incident with} $v$ if $v \in \overline{e}$.
The incidence relation between $V(\Gamma)$ and $E(\Gamma)$ defines a graph, and we say that $\Gamma$ is a \defn{drawing of $G$} in $\Sigma$ if $G$ is a graph defined by this incidence relation.
In this case, we say that $G$ is \defn{embeddable} in $\Sigma$, or $G$ can be \defn{drawn} in $\Sigma$.
An \defn{atom} of $\Gamma$ is an edge of $\Gamma$, a region of $\Gamma$, or a set $\{v\}$ for some $v \in V(\Gamma)$.

Let $\Sigma$ be a surface and $\Gamma$ a drawing in $\Sigma$.
A drawing $\Gamma'$ is a \defn{subdrawing} of $\Gamma$ if $V(\Gamma') \subseteq V(\Gamma)$ and $E(\Gamma') \subseteq E(\Gamma)$.
We write \defn{$\Gamma' \subseteq \Gamma$} if $\Gamma'$ is a subdrawing of $\Gamma$.
For a closed set $\Delta \subseteq \Sigma$ with either $\bar{e} \subseteq \Delta$ or $e \cap \Delta=\emptyset$ for each $e \in E(\Gamma)$, we define \defn{$\Gamma \cap \Delta$} to be the drawing $(U(\Gamma) \cap \Delta, V(\Gamma) \cap \Delta)$.

Let $\Sigma$ be a surface and $\Gamma$ a drawing in $\Sigma$.
A subset $Z$ of $\Sigma$ is \defn{$\Gamma$-normal} if $Z \cap U(\Gamma) \subseteq V(\Gamma)$.
If $\Sigma$ is connected and not a sphere, we say that $\Gamma$ is \defn{$\theta$-representative} if $\lvert F \cap V(\Gamma) \rvert \geq \theta$ for every non-null-homotopic $\Gamma$-normal O-arc $F$ in $\Sigma$.
We say $\Gamma$ is \defn{$2$-cell} if $\Sigma$ is connected and every region of $\Gamma$ is an open disk.

Let $\Gamma$ be a $2$-cell drawing in a surface $\Sigma$.
We say that a drawing $K$ in $\Sigma$ is a \defn{radial drawing} of $\Gamma$ if it satisfies the following.
	\begin{itemize}
		\item $U(\Gamma) \cap U(K) = V(\Gamma) \subseteq V(K)$.
		\item Each region $r$ of $\Gamma$ contains a unique vertex of $K$.
		\item $K$ is a drawing of a bipartite graph, and $(V(\Gamma), V(K)-V(\Gamma))$ is its bipartition.
		\item For every $v \in V(\Gamma)$, the edges of $K \cup \Gamma$ incident with $v$ belong alternately to $\Gamma$ and to $K$ (in their cyclic order around $v$).
	\end{itemize}

\subsection{Arrangements} \label{subsec:arrangement}

Let $\Sigma$ be a surface and $\Se = \{(S_1, \Omega_1), ..., (S_k, \Omega_k)\}$ a segregation of $G$.
An \defn{arrangement} of $\Se$ in $\Sigma$ is a function $\alpha$ with domain $\Se \cup V(\Se)$, such that the following conditions hold.
\begin{itemize}
	\item For $1 \leq i \leq k$, $\alpha(S_i, \Omega_i)$ is a closed disk $\Delta_i \subseteq \Sigma$, and $\alpha(x) \in \partial\Delta_i$ for each $x \in \overline{\Omega_i}$.
	\item For $1 \leq i<j \leq k$, if $x \in \Delta_i \cap \Delta_j$, then $x=\alpha(v)$ for some $v \in \overline{\Omega_i} \cap \overline{\Omega_j}$.
	\item For all distinct $x,y \in V(\Se)$, $\alpha(x) \neq \alpha(y)$.
	\item For $1 \leq i \leq k$, $\Omega_i$ is mapped by $\alpha$ to a natural order of $\alpha(\overline{\Omega_i})$ determined by $\partial\Delta_i$.
\end{itemize}
An arrangement is \defn{proper} if $\Delta_i \cap \Delta_j = \emptyset$ for all $1 \leq i < j \leq k$ with $\lvert \overline{\Omega_i} \rvert, \lvert \overline{\Omega_j} \rvert >3$.

We will use the following celebrated structure theorem of Robertson and Seymour.

\begin{theorem}[{{\cite[(3.1)]{rs XVI}}}] \label{minor_structual_thm}
For every graph $H$, there exist nonnegative integers $\kappa,\rho,\xi$ and a positive integer $\theta$ such that if $\T$ is a tangle in a graph $G$ of order at least $\theta$ controlling no $H$-minor of $G$, then there exist $Z \subseteq V(G)$ with $|Z| \leq \xi$ and a $(\T-Z)$-central $(\kappa,\rho)$-segregation of $G-Z$ which has a proper arrangement in a surface in which $H$ cannot be drawn.
\end{theorem}

Let $G$ be a graph, and let $L$ be an induced subgraph of $G$.
Let $r$ be a real number, and let $L^+ = G[N_G^{\leq r}[V(L)]]$.
Let $\Se^+$ be a segregation of $L^+$ with a $(\kappa,\rho)$-witness $(\Se_1^+,\Se_2^+)$ for some $\kappa,\rho \in {\mathbb N}_0$.
Let $\Se$ be an $(L,\Se_1^+,\Se_2^+)$-adaption of $\Se^+$, and let $(\Se_1,\Se_2)$ be a witness.
Let $\alpha^+$ be a proper arrangement of $\Se^+$ in a surface $\Sigma$.
An \defn{$\Se$-adjustment} of $\alpha^+$ is a proper arrangement of $\Se$ in $\Sigma$ such that
	\begin{itemize}
		\item for any $j \in [2]$ and $(S,\Omega) \in \Se_j$, there exists $(S',\Omega') \in \Se_j^+$ such that $S \subseteq S'$ and $\alpha((S,\Omega)) \subseteq \alpha^+((S',\Omega'))$, and
		\item $\alpha^+(v)=\alpha(v)$ for every $v \in V(\Se) \cap V(\Se^+)$.
	\end{itemize}
Clearly, an $\Se$-adjustment of $\alpha^+$ exists by slightly perturbing the disks in the image of $\alpha^+$.

\subsection{Skeletons}
Let $\Se$ be a segregation of a graph $G$.
Let $\Se_1,\Se_2$ be subsets of $\Se$ with $\Se_1 \cap \Se_2 = \emptyset$ and $\Se_1 \cup \Se_2=\Se$ such that $\lvert \overline{\Omega} \rvert \leq 3$ for every $(S,\Omega) \in \Se_1$.
Let $\alpha$ be a proper arrangement of $\Se$ in a surface $\Sigma$.
The \defn{skeleton of $\alpha$ with respect to $(\Se_1,\Se_2)$} is the drawing $\Gamma = (U,V)$ in $\Sigma$ with $V(\Gamma)=\{\alpha(v): v \in V(\Se)\}$ such that $U(\Gamma)$ consists of the boundary of $\alpha(S,\Omega)$ for each $(S,\Omega) \in \Se_1$ with $\lvert \overline{\Omega} \rvert = 3$, and a line in the boundary of $\alpha(S',\Omega')$ with ends $\overline{\Omega'}$ for each $(S',\Omega') \in \Se_1$ with $\lvert \overline{\Omega'} \rvert=2$.
Note that no edge of the skeleton is created by $(S,\Omega) \in \Se_1$ with $\lvert \overline{\Omega} \rvert \leq 1$ or by $(S,\Omega) \in \Se_2$.
It is easy to see that if $\Se$ is an effective segregation with respect to $(\Se_1,\Se_2)$, then the skeleton of $\alpha$ with respect to $(\Se_1,\Se_2)$ is a minor of $G$; more precisely, there exists an $H$-minor $\mu$ in $G$, where $H$ is the skeleton of $\alpha$ with respect to $(\Se_1,\Se_2)$, such that
	\begin{itemize}
		\item for every $v \in V(H)$, $\mu(v)$ is a subgraph containing $v_G$, where $v_G$ is the vertex of $G$ with $\alpha(v_G)=v$, and
		\item for every $e \in E(H)$, $\mu(e)$ is an edge of $S_e$, where $(S_e,\Omega_e)$ is the member of $\Se_1$ such that $e$ is contained in $\partial\alpha(S_e,\Omega_e)$;
	\end{itemize}
we say that the skeleton of $\alpha$ with respect to $(\Se_1,\Se_2)$ is a \defn{natural minor} of $G$ if such $\mu$ exists. 

Now we show how to use Lemma \ref{easy_tree} to find a hitting set for objects contained in vortices.

\begin{lemma} \label{hitting_vortices}
For any $k,r,c \in {\mathbb N}$ and $\xi,\eta,\kappa,\rho,g \in {\mathbb N}_0$, there exist $\xi^*,\eta^* \in {\mathbb N}$ with $\xi^* \geq \xi$ and $\eta^* =\eta+r$ such that the following hold.
Let $G$ be a graph, and let $L$ be a subgraph of $G$.
Let $\F$ be a component exchangeable set of subgraphs of $L$ such that every member of $\F$ has exactly $c$ components. 
Let $Z \subseteq V(G)$ be a $(\xi,\eta)$-centered set in $G$ such that $N_G(V(L)) \subseteq Z$.
Let $\Se$ be a segregation of $L$ with a $(\kappa,\rho)$-witness $(\Se_1,\Se_2)$ such that there exists a proper arrangement of $\Se$ in a surface of Euler genus at most $g$.
For every $(S,\Omega) \in \Se_2$, let $(P^S,(X^S_t: t \in V(P^S)))$ be a vortical decomposition of $(S,\Omega)$ of adhesion at most $\rho$, where $V(P^S)=\overline{\Omega}$ and $t \in X^S_t$ for every $t \in V(P^S)=\overline{\Omega}$.
If no $k$ members of $\F$ have pairwise distance in $G$ at least $r$, and 
	\begin{enumerate}
		\item $S-N_G^{\leq r}[\overline{\Omega} \cup Z]$ does not contain a component of a member of $\F$ for every $(S,\Omega) \in \Se_1$, and
		\item $L[X^S_t]-N_G^{\leq r}[\{t\} \cup (X_t \cap \bigcup_{t' \in V(P^S)-\{t\}}X^S_{t'}) \cup Z]$ does not contain a component of a member of $\F$ for any $(S,\Omega) \in \Se_2$ and $t \in V(P^S)$,
	\end{enumerate}
then there exists $Z^* \subseteq V(G)$ with $Z^* \supseteq N_G^{\leq r}[Z]$ such that $Z^*$ is $(\xi^*,\eta^*)$-centered in $G$ and intersects all members of $\F$ contained in $\bigcup_{(S,\Omega) \in \Se_2}S$. 
\end{lemma}

\begin{pf}
Let $k,r,c \in {\mathbb N}$ and $\xi,\eta,\kappa,\rho,g \in {\mathbb N}_0$. 
Let $\lambda = (2g+3)(2(r+1)(2\kappa+1)+1)(2\rho+1)$.
Define $\xi^*=\xi+(ck-1)\lambda$ and $\eta^*=\eta+r$.

Let $G,L,\F,Z,\Se,\Se_1,\Se_2,(P^S,(X^S_t: t \in V(P^S)))$ be as stated in the lemma.
Let $\alpha$ be a proper arrangement of $\Se$ in a surface $\Sigma$ of Euler genus at most $g$.

Let $L'$ be the skeleton of $\alpha$ with respect to $(\Se_1,\Se_2)$.
Let $L''$ be the graph obtained from $L'$ by, for each $(S,\Omega) \in \Se_2$, adding edges such that there exists a cycle in $L''$ with vertex-set $\overline{\Omega}$ and passing though $\overline{\Omega}$ in the order $\Omega$, and adding a new vertex $u_S$ adjacent to all vertices in $\overline{\Omega}$.
Note that $L''$ can be drawn in $\Sigma$.

Let $U = \{u_S: (S,\Omega) \in \Se_2\}$.

Let $H = \bigcup_{(S,\Omega) \in \Se, \dist_{L''}(\overline{\Omega},U) \leq r}S$.
Since $r \geq 1$, we know $\bigcup_{(S,\Omega) \in \Se_2}S \subseteq H$.
For any $(S,\Omega) \in \Se_2$ and $t \in V(P^S)$, there exists a separation $(A_t,B_t)$ of $H$ such that $V(A_t)=X^S_t$ and $V(A_t \cap B_t) = \{t\} \cup (X^S_t \cap \bigcup_{t' \in V(P^S)-\{t\}}X^S_{t'})$, and subject to this, $A_t$ is minimal; by the assumption of this lemma, $(A_t,B_t)$ is a separation of $H$ of order at most $2\rho+1$ such that $A_t-N_G^{\leq r}[V(A_t \cap B_t) \cup Z]$ does not contain any component of a member of $\F$.
For every $(S,\Omega) \in \Se_1$ with $S \subseteq H$, there exists a separation $(A_S,B_S)$ of $H$ such that $V(A_S)=V(S)$ and $V(A_S \cap B_S) = \overline{\Omega}$, and subject to this, $A_S$ is minimal; by the assumption of this lemma, $(A_S,B_S)$ is a separation of $H$ of order at most $3 \leq 2\rho+3$ such that $A_S-N_G^{\leq r}[V(A_S \cap B_S) \cup Z]$ does not contain any component of a member of $\F$.
Let $\LL = \{(A_t,B_t): t \in \bigcup_{(S,\Omega) \in \Se_2}V(P^S)\} \cup \{(A_S,B_S): (S,\Omega) \in \Se_1, S \subseteq H\}$.
Note that $\LL$ is a location in $H$ such that every member $(A,B)$ of $\LL$ has order at most $2\rho+3$ and satisfies that 
	\begin{itemize}
		\item[(i)] $A-N_G^{\leq r}[V(A \cap B) \cup Z]$ does not contain a component of any member of $\F$.
	\end{itemize}

\medskip

\noindent{\bf Claim 1:} There exists a tree-decomposition of $\LL$ such that every bag is $(\lambda,0)$-centered in $G$.

\noindent{\bf Proof of Claim 1:}
Let $H_1 = L''[U \cup \bigcup_{(S,\Omega) \in \Se, \dist_{L''}(\overline{\Omega},U) \leq r}\overline{\Omega}]$.
Note that $V(H_1)$ is $(\kappa,r+1)$-centered in $L''$ since $\overline{\Omega}$ is a clique for every $(S,\Omega) \in \Se_1$. 
Moreover, $H_1$ can be drawn in $\Sigma$ such that every component has radius at most $(r+1)(2\kappa+1)$.
By \cite[Corollary 4.3]{cl}, there exists a tree-decomposition $(T^1,(Y^1_t: t \in V(T^1)))$ of $H_1$ such that $|Y^1_t| \leq (2g+3)(2(r+1)(2\kappa+1)+1)$ for every $t \in V(T^1)$.

Let $H_2$ be the graph obtained from $H_1$ by, for any $(S,\Omega) \in \Se_2$ and $p \in V(P^S)=\overline{\Omega}$, adding the vertices in $V(A_p \cap B_p)$ and adding edges to make $V(A_p \cap B_p)$ be a clique in $H_2$. 
For every $t \in V(T^1)$, let $Y^2_t = Y^1_t \cup \bigcup_{p \in Y^1_t \cap \bigcup_{(S,\Omega) \in \Se_2}V(P^S)}V(A_p \cap B_p)$.
Since $r \geq 1$, we know that $\overline{\Omega} \subseteq H_1$ for every $(S,\Omega) \in \Se_2$.
Then $(T^1, (Y^2_t: t \in V(T^1)))$ is a tree-decomposition of $H_2$ such that $|Y^2_t| \leq |Y^1_t| \cdot (2\rho+1) \leq (2g+3)(2(r+1)(2\kappa+1)+1)(2\rho+1) \leq \lambda$ since for every $(S,\Omega) \in \Se_2$, $(P^S,(X^S_p: p \in V(P^S))$ is a vortical decomposition and there exists a cycle in $H_1$ with vertex-set $\overline{\Omega}$ and passing though $\overline{\Omega}$ in the order $\Omega$.

Note that for every $(A,B) \in \LL$, $V(A \cap B)$ is a clique in $H_2$ and hence is contained in a bag of $(T^1, (Y^2_t: t \in V(T^1)))$.
Moreover, $V(H_2) \supseteq \bigcap_{(A,B) \in \LL}V(B)$.
So $(T^1,(Y^2_t \cap V(H_2): t \in V(T^1)))$ is a tree-decomposition of $\LL$ such that every bag has size at most $\lambda$ and hence is $(\lambda,0)$-centered in $G$.
$\Box$

\medskip

Let $\F_1$ consists of the members of $\F$ contained in $\bigcup_{(S,\Omega) \in \Se_2}S$ but disjoint from $N_G^{\leq r}[Z]$. 

Since $\bigcup_{(S,\Omega) \in \Se_2}S \subseteq H$, every member of $\F_1$ is contained in $H$.
For any $F \in \F_1$ and component $C$ of $F$, we know $H[N_G^{\leq r}[V(C)]]$ is connected since $V(F) \cap N_G^{\leq r}[Z] = \emptyset$ and $N_G(V(L)) \subseteq Z$.
By (i), for every $(A,B) \in \LL$, $A-N_G^{\leq r}[V(A \cap B) \cup Z]$ does not contain a component of any member of $\F$, so $A-N_G^{\leq r}[V(A \cap B)]$ does not contain a component of any member of $\F_1$.
Since $L$ does not contain $k$ members of $\F$ pairwise at distance in $G$ at least $r$, $H$ does not contain $k$ members of $\F_1$ pairwise at distance in $G$ at least $r$.
So Lemma \ref{easy_tree} implies that there exists a $((ck-1)\lambda,r)$-centered set $Z_1$ intersecting all members of $\F_1$.
Hence $N_G^{\leq r}[Z] \cup Z_1$ is a $(\xi+(ck-1)\lambda,\eta+r)$-centered set in $G$ intersecting all members of $\F$ contained in $\bigcup_{(S,\Omega) \in \Se_2}S$.
\end{pf}

\subsection{Protected arrangements} \label{subsec:protected_arrangements}

Given an effective segregation $\Se$ of an induced subgraph $L$ of a graph $G$, we can ``project'' every path in $L$ to a path in the skeleton $L'$ of $\Se$.
So given two vertices $u,v$ in $L$, $\dist_L(u,v)$ can be somehow told from the distance between some vertices in $L'$.
But we have to consider $\dist_G(u,v)$ instead of $\dist_L(u,v)$.
So we need some notions to deal with paths in $G$ between two vertices in $L$.

Let $G$ be a graph, and let $L$ be an induced subgraph of $G$.
Let $\Se$ be a segregation of $L$ with a $(\kappa,\rho)$-witness $(\Se_1,\Se_2)$ for some integers $\kappa,\rho$.
Let $c,r$ be positive integers, and let $\Sigma$ be a surface.
Let $\alpha$ be an arrangement of $\Se$ in $\Sigma$.
Let $L'$ be the skeleton of $\alpha$ with respect to $(\Se_1,\Se_2)$, and let $K$ be a radial drawing of $L'$.
We say that $\alpha$ is \defn{$(c,r,\Se_1,\Se_2)$-protected} if for any $u,v \in V(L)$ with $\dist_G(u,v) \leq r$, there exists a path in $L' \cup K$ with at most $cr$ edges between $\overline{\Omega_u}$ and $\overline{\Omega_v}$ for some $(S_u,\Omega_u)$ and $(S_v,\Omega_v) \in \Se$ with $u \in V(S_u)$ and $v \in V(S_v)$.

\begin{lemma} \label{adjustment_co_distance}
Let $G$ be a graph, and let $L$ be an induced subgraph of $G$.
Let $r \in {\mathbb N}$, and let $L^+ = G[N_G^{\leq r}[V(L)]]$.
Let $\Se^+$ be a segregation of $L^+$ with a $(\kappa,\rho)$-witness $(\Se_1^+,\Se_2^+)$ for some $\kappa,\rho \in {\mathbb N}_0$.
Let $\Se$ be an $(L,\Se_1^+,\Se_2^+)$-adaption of $\Se^+$, and let $(\Se_1,\Se_2)$ be its witness.
Let $\alpha^+$ be a proper arrangement of $\Se^+$ in a surface $\Sigma$, and let $\alpha$ be an $\Se$-adjustment of $\alpha^+$.
Then $\alpha$ is $(8,r,\Se_1,\Se_2)$-protected.
\end{lemma}

\begin{pf}
Let $u,v \in V(L)$ with $\dist_G(u,v) \leq r$.
Since $u,v \in V(L)$, there exist $(S_u,\Omega_u)$ and $(S_v,\Omega_v) \in \Se$ with $u \in V(S_u)$ and $v \in V(S_v)$.
Since $\Se$ is an $(L,\Se_1^+,\Se_2^+)$-adaption of $\Se^+$, there exist $(S_u^+,\Omega_u^+)$ and $(S_v^+,\Omega_v^+) \in \Se^+$ such that $S_u \subseteq S_u^+$ and $S_v \subseteq S_v^+$.
Since $\dist_G(u,v) \leq r$ and $L^+ = G[N_G^{\leq r}[V(L)]]$, there exists a path $P$ in $L^+$ of length at most $r$ from $u$ and $v$.
Hence there exist an integer $\ell \leq r$ and (not necessarily distinct) members $(S_1,\Omega_1),(S_2,\Omega_2),...,(S_\ell,\Omega_\ell)$ of $\Se^+$ such that $(S_1,\Omega_1)=(S_u^+,\Omega_u^+)$, $(S_\ell,\Omega_\ell)=(S_v^+,\Omega_v^+)$ and $P$ is the concatenation of edge-disjoint paths $P_1,P_2,...,P_\ell$, where each $P_i$ is a maximal subpath of $P \cap S_i$.
For every $i \in [\ell]$, let $a_i$ and $b_i$ be the ends of $P_i$.
Note that for every $2 \leq i \leq \ell-1$, $a_i$ and $b_i$ are in $\overline{\Omega_i}$.

For every $x \in V(P) \cap \bigcup_{i=1}^\ell \overline{\Omega_\ell}$, if $x \in V(L)$, then $x \in V(\Se) \subseteq V(L')$, and we let $\iota(x)=x$; otherwise, there exists a region of $L'$ containing $\alpha^+(x)$, and we let $\iota(x)$ be the vertex of $K$ corresponding to this region.
For every end $x$ of $P$ not in $\bigcup_{i=1}^\ell \overline{\Omega_\ell}$, we know that $x \in \{u,v\}$, and we let $\iota(x)$ be a vertex in $\overline{\Omega_x}$.

For every $(S,\Omega) \in \Se^+_1$, let $\A_{S,\Omega}$ be the adaption of $(S,\Omega)[V(L)]$ contained in $\Se_1^+$.
By Lemma \ref{adaption_size}, $|\A_{S,\Omega}| \leq 4$ for every $(S,\Omega) \in \Se^+_1$.
So for every $i \in [\ell]$ with $(S_i,\Omega_i) \in \Se_1^+$, there exists a path $Q_i$ in $L' \cup K$ between $\iota(a_i)$ and $\iota(b_i)$ with length at most 8.
For every $i \in [\ell]$ with $(S_i,\Omega_i) \in \Se_2^+$, there exists a path $Q_i$ in $L' \cup K$ between $\iota(a_i)$ and $\iota(b_i)$ with length at most 2.
Hence $\bigcup_{i=1}^\ell Q_i$ contains a path in $L' \cup K$ between $\overline{\Omega_u}$ and $\overline{\Omega_v}$ with length at most $8\ell \leq 8r$.
\end{pf}

\subsection{Respectful tangles}

Let $\Sigma$ be a surface, and let $\Gamma$ be a drawing of a graph $G$ in $\Sigma$.
A \defn{tangle} (or \defn{pretangle}) in $\Gamma$ and a \defn{separation} of $\Gamma$ are a tangle (or pretangle) in $G$ and a separation of $G$, respectively.
A tangle (or pretangle, respectively) $\T$ in $\Gamma$ of order $\theta$ is \defn{respectful (towards $\Sigma$)} if $\Sigma$ is connected, and for every $\Gamma$-normal O-arc $F$ in $\Sigma$ with $|F \cap V(\Gamma)| < \theta$, there is a closed disk $\Delta \subseteq \Sigma$ with $\partial\Delta=F$ such that $(\Gamma \cap \Delta, \Gamma \cap \overline{\Sigma-\Delta}) \in \T$; since $\T$ satisfies (T1), the disk $\Delta$ is unique, and we write $\Delta = \ins(F)$; the function \defn{ins} is called the \defn{inside function} of $\T$.

\begin{theorem}[{{\cite[(4.1)]{rs XI}}}] \label{representative_respectful}
Let $\Sigma$ be a connected surface, not a sphere.
Let $\theta \geq 1$, and let $\Gamma$ be a $\theta$-representative 2-cell drawing in $\Sigma$.
Then there is a unique respectful tangle in $\Gamma$ of order $\theta$.
\end{theorem}

Let $\Gamma$ be a drawing in a connected surface $\Sigma$.
Let $\theta \geq 1$ such that $2\theta$ is an integer.
A \defn{slope} in $\Gamma$ of order $\theta$ is a function $f$ that assigns each cycle $C$ in $\Gamma$ of length less than $2\theta$ a closed disk $f(C)$ in $\Sigma$ bounded by $U(C)$ such that 
	\begin{itemize}
		\item if $C_1$ and $C_2$ are cycles in $\Gamma$ of length less than $2\theta$ such that $U(C_1) \subseteq f(C_2)$, then $f(C_1) \subseteq f(C_2)$, and
		\item if $P_0,P_1,P_2$ are three non-trivial pairwise disjoint paths in $\Gamma$ with same ends such that the cycles $P_0 \cup P_1$, $P_1 \cup P_2$ and $P_2 \cup P_0$ have length less than $2\theta$, then one of the three disks $f(P_0 \cup P_1), f(P_1 \cup P_2), f(P_2 \cup P_0)$ contains the other two.
	\end{itemize}

Let $\Gamma$ be a $2$-cell drawing in a surface $\Sigma$, and let $K$ be a radial drawing of $\Gamma$.
Let $\T$ be a respectful pretangle in $\Gamma$ of order $\theta$ for some integer $\theta$.
For a closed walk $W$ of $K$, we define $K|W$ to be the subdrawing of $K$ formed by the vertices and the edges in $W$; note that if $C$ is a cycle of $K|W$ of length less than $2\theta$, then $C$ is a $\Gamma$-normal O-arc in $\Sigma$ and $\ins(C)$ is defined by the inside function of $\T$.
For every closed walk $W$ of $K$ of length less than $2\theta$, we define $\ins(W)$ to be the union of $U(K|W)$ and $\ins(C)$, taken over all cycles $C$ of $K|W$.
Note that the restriction of $\ins$ on the set of all cycles of $K$ of length less than $2\theta$ is a slope of order $\theta$, and we call it \defn{the slope derived from $\T$}.

The following is a consequence of \cite[(6.3)]{rs XI} that provides a membership test for a pretangle $\T$ in terms of the slope derived from $\T$.

\begin{theorem}[{{\cite[(6.3)]{rs XI}}}] \label{repsectful_cycle}
Let $\Sigma$ be a connected surface, and let $\Gamma$ be a 2-cell drawing in $\Sigma$ with $E(\Gamma) \neq \emptyset$.
Let $K$ be a radial drawing of $\Gamma$ in $\Sigma$.
Let $\T$ be a respectful pretangle of order $\theta \geq 1$ in $\Gamma$ toward $\Sigma$, and let $f$ be the slope derived from $\T$.
Let $(A,B)$ be a separation of $\Gamma$ of order less than $\theta$.
Then $(A,B) \in \T$ if and only if for every $e \in E(A)$, there exists a cycle $C$ of $K$ with $V(C) \cap V(\Gamma) \subseteq V(A \cap B)$\footnote{In \cite[(6.3)]{rs XI}, it was stated as $V(C) \subseteq V(\Gamma) \subseteq V(A \cap B)$, but we believe that it should be $V(C) \cap V(\Gamma) \subseteq V(A \cap B)$.} and with $e \subseteq f(C)$.
\end{theorem}

We have seen that a pretangle can derive a slope.
The converse statement is also true in the sense that a slope can define a pretangle, as shown by the following simpler form of \cite[(6.1)]{rs XI}.

\begin{theorem}[{{\cite[(6.1)]{rs XI}}}] \label{slope_pretangle}
Let $\Gamma$ be a 2-cell drawing in a connected surface $\Sigma$ with $E(\Gamma) \neq \emptyset$, and let $K$ be a radial drawing of $\Gamma$.
Let $\theta \geq 1$ be an integer, and let $f$ be a slope in $K$ of order $\theta$.
Then there exists a respectful pretangle $\T$ in $\Gamma$ of order $\theta$ such that $f$ is the slope derived from $\T$.
\end{theorem}

Now we prove the key lemma of this subsection.

\begin{lemma} \label{segregation_respectful_central}
For any positive integers $\theta^*>3,\kappa,\rho$, there exists a positive integer $\theta \geq \theta^*$ such that the following holds.
Let $G$ be a graph, and let $\T$ be a tangle in $G$ of order at least $\theta$. 
Let $\Se$ be a $\T$-central segregation of $G$ effective with respect to a $(\kappa,\rho)$-witness $(\Se_1,\Se_2)$ such that there exists a proper arrangement $\alpha$ of $\Se$ in the sphere. 
Let $G'$ be the skeleton of $\alpha$ with respect to $(\Se_1,\Se_2)$.
If $G'$ is 2-cell, then there exists a respectful tangle $\T'$ in $G'$ of order $\theta^*$ such that $\T'$ is conformal with $\T$.
\end{lemma}

\begin{pf}
Let $\theta^*,\kappa,\rho$ be positive integers.
Define $\theta = 3(2\rho+1)(\theta^*+9\kappa)$.

Let $G,\T,\Se,\Se_1,\Se_2,\alpha,G'$ be as stated in the lemma.

\medskip

\noindent{\bf Claim 1:} $G'$ has no isolated vertex and $E(G') \neq \emptyset$.

\noindent{\bf Proof of Claim 1:} 
Since $G'$ is a 2-cell embedding in the sphere, either $G'$ has no vertex, or $G'$ is connected.
Hence it suffices to show that $|V(G')| \geq 2$. 

Suppose to the contrary that $|V(G)'| \leq 1$.
Note that $\overline{\Omega} \subseteq V(G')$ for every $(S,\Omega) \in \Se$.
For every $(S,\Omega) \in \Se$, there exists a separation $(A_S,B_S)$ of $G$ such that $V(A_S \cap B_S) = \overline{\Omega}$ and $A_S=S$.
Since $\Se$ is $\T$-central and $\theta>1 \geq |V(G')| \geq |\overline{\Omega}| = |V(A_S \cap B_S)|$, we know $(A_S,B_S) \in \T$.
But for every $\Se' \subseteq \Se$, we know that $(\bigcup_{(S,\Omega) \in \Se'}A_S, \bigcap_{(S,\Omega) \in \Se'}B_S)$ is a separation with $\bigcup_{(S,\Omega) \in \Se'}A_S \cap \bigcap_{(S,\Omega) \in \Se'}B_S \subseteq V(G')$ and hence has order at most 1.
Then we can prove that for every $\Se' \subseteq \Se$, $(\bigcup_{(S,\Omega) \in \Se'}A_S, \bigcap_{(S,\Omega) \in \Se'}B_S) \in \T$ by induction on $|\Se'|$ and by (T1) and (T2).
So $(\bigcup_{(S,\Omega) \in \Se}A_S, \bigcap_{(S,\Omega) \in \Se}B_S) \in \T$.
But $\bigcup_{(S,\Omega) \in \Se}A_S = G$, contradicting (T3).
$\Box$

\medskip

For every $(S,\Omega) \in \Se_2$, since $(S,\Omega)$ is a $\rho$-vortex, there exists a vortical decomposition $(P^S,(X^S_v: v \in \overline{\Omega}))$ of adhesion at most $\rho$ such that $V(P^S)=\overline{\Omega}$ and $v \in X^S_v$ for every $v \in \overline{\Omega}$ by Theorem \ref{path decomp of vortex}; for every $v \in \overline{\Omega}$, let $Y^S_v = X^S_v \cap \bigcup_{u \in \overline{\Omega}-\{v\}}X^S_u$, and we note that $|Y^S_v| \leq 2\rho$.
For any separation $(A,B)$ of $G'$ and $Z \in \{A,B\}$, let $Z_G^+ = \bigcup_{(S,\Omega) \in \Se_1, \overline{\Omega} \subseteq V(Z)}V(S) \cup \bigcup_{(S,\Omega) \in \Se_2}\bigcup_{v \in \overline{\Omega} \cap V(Z)}X^S_v$, and let $Z_G^- = Z_G^+ - (\bigcup_{(S,\Omega) \in \Se_1, \overline{\Omega} \subseteq V(Z')}(V(S)-\overline{\Omega}) \cup \bigcup_{(S,\Omega) \in \Se_2}\bigcup_{v \in \overline{\Omega} \cap V(Z')}(X^S_v-Y^S_v))$, where $Z'$ is the unique member of $\{A,B\}-\{Z\}$. 

\medskip

\noindent{\bf Claim 2:} For every separation $(A,B)$ of $G'$, there exists a separation $(A_G,B_G)$ of $G$ such that $V(A_G) = A_G^+$ and $V(B_G) = B_G^-$.

\noindent{\bf Proof of Claim 2:} 
For every $(S,\Omega) \in \Se_1$, since $\overline{\Omega}$ forms a clique in $G'$, either $\overline{\Omega} \subseteq V(A)$ or $\overline{\Omega} \subseteq V(B)$; if $\overline{\Omega} \subseteq V(A)$, then $V(S) \subseteq A_G^+$; otherwise, $V(S) \subseteq B_G^-$; so for every edge $e \in E(S)$, either both ends are in $A_G^+$ or both ends are in $B_G^-$.
For any $(S,\Omega) \in \Se_2$ and $v \in \overline{\Omega}$, we have $X^S_v \subseteq A_G^+$ or $X^S_v \subseteq B_G^-$.
Hence every vertex of $V(G)$ is contained in $A_G^+$ or in $B_G^-$, and for every edge $e$ of $G$, both ends of $e$ are contained in $A_G^+$ or both ends of $e$ are contained in $B_G^-$.
This shows that there exists a separation $(A_G,B_G)$ of $G$ such that $V(A_G) = A_G^+$ and $V(B_G) = B_G^-$.
$\Box$

\medskip

Let $K$ be the radial drawing of $G'$.
For every cycle $C$ in $K$, there exist two distinct closed disks $\Delta_{C,X}$ and $\Delta_{C,Y}$ in the sphere bounded by $C$, and let $X_C = \{v \in V(G'): \alpha(v) \in \Delta_{C,X}\}$ and $Y_C = \{v \in V(G'): \alpha(v) \in \Delta_{C,Y}\}$; clearly, there exists a separation $(A^C,B^C)$ of $G'$ of order $\frac{1}{2}|E(C)|$ such that $V(A^C)=X_C$ and $V(B^C)=Y_C$.
For simplicity of the notation, we denote $(A^C)_G$ and $(B^C)_G$ by $A^C_G$ and $B^C_G$, respectively.

\medskip

\noindent{\bf Claim 3:} For every cycle $C$ in $K$ of length less than $2\theta^*$, the separation $({A^C_G},B^C_G)$ of $G$ has order less than $2(\rho+1)(\theta^*+6\kappa)$.

\noindent{\bf Proof of Claim 3:}
For every $(S,\Omega) \in \Se_2$, since $C$ is a cycle in $K$, there are at most two edges of $C$ incident with the vertex of $K$ corresponding to the region of $G'$ containing $\alpha(S,\Omega)$; so $\Omega$ can be written as a union of at most three intervals each contained in $V(A^C) \cap \overline{\Omega}$ or contained in $V(B^C) \cap \overline{\Omega}$; let $D_{S,1}$ be the set of the endpoints of those intervals; so $|D_{S,1}| \leq 6$ and $D_{S,1} \subseteq \overline{\Omega}$.
Let $D = \bigcup_{(S,\Omega) \in \Se_2}(D_{S,1} \cup (V(C) \cap \overline{\Omega}))$.
So $|D| \leq \sum_{(S,\Omega) \in \Se_2}(|D_{S,1}| + |V(C) \cap \overline{\Omega}|) < 6\kappa+\theta^*$.

Let $u \in V(A^C_G \cap B^C_G)-V(A^C \cap B^C) = (A^C)_G^+ \cap (B^C)_G^- - V(A^C \cap B^C)$.
By the definition of $(A^C)_G^+$ and $(B^C)_G^-$, we know that $u \not \in \bigcup_{(S,\Omega) \in \Se_1}(V(S)-\overline{\Omega})$.
So $u \in V(S)$ for some $(S,\Omega) \in \Se_2$.
Since $u \in (A^C)_G^+ \cap (B^C)_G^-$, there exist $a_u \in V(A^C) \cap \overline{\Omega}$ and $b_u \in V(B^C) \cap \overline{\Omega}$ such that $u \in X^S_{a_u} \cap X^S_{b_u}$.
If we can choose $a_u=b_u$, then $a_u=b_u \in D \cap V(A^C \cap B^C)$, so $u \in Y^S_{b_u}$ by the definition of $(B^C)_G^-$.
If we cannot choose $a_u=b_u$, then there exists $c_u \in D_1$ such that $c_u$ is between $a_u$ and $b_u$ in $\Omega$, and $u \in Y^S_{c_u}$.

Hence $V(A^C_G \cap B^C_G)-V(A^C \cap B^C) \subseteq \bigcup_{(S,\Omega) \in \Se_2}\bigcup_{d \in D \cap \overline{\Omega}}Y^S_d$.
Therefore, $|V(A^C_G \cap B^C_G)| \leq |V(A^C \cap B^C)|+ |D| \cdot (2\rho+1) < \theta^*+(2\rho+1)(\theta^*+6\kappa) \leq 2(\rho+1)(\theta^*+6\kappa)$.
$\Box$

\medskip

For every cycle $C$ in $K$ of length less than $2\theta^*$, Claim 3 implies that $|V(A^C_G \cap B^C_G)| < 2(\rho+1)(\theta^*+6\kappa) \leq \theta$, so either $(A^C_G,B^C_G) \in \T$ or $(B^C_G,A^C_G) \in \T$, and not both since $\T$ satisfies (T1) and (T2); if $(A^C_G,B^C_G) \in \T$, then we define $f(C)$ to be the disk $\Delta_{C,X}$; otherwise, we define $f(C)$ to be the disk $\Delta_{C,Y}$.
Since $\T$ satisfies (T2), we know that $f$ is a slope in $K$ of order $\theta^*$.
Since $G'$ is 2-cell and $E(G') \neq \emptyset$ by Claim 1, Theorem \ref{slope_pretangle} implies that there exists a respectful pretangle $\T_{G'}$ in $G'$ of order $\theta^*$ such that $f$ is the slope derived from $\T_{G'}$.

Since $\Se$ is effective with respect to $(\Se_1,\Se_2)$, there exists a natural $G'$-minor $\mu$ in $G$.
Then $\T_{G'}$ induces a pretangle $\T_G$ of order $\theta^*$ in $G$.

For a separation $(A,B)$ of $G$, we consider the following property (*).
	\begin{itemize}
		\item[(*)] For every $(S,\Omega) \in \Se_1$, if $\overline{\Omega}-V(B) \neq \emptyset$, then $\overline{\Omega} \subseteq V(A)$.
	\end{itemize}
Note that (*) is a non-trivial property since $\overline{\Omega}$ is not necessary a clique in $G$ though $\overline{\Omega}$ is a clique in $G'$.

Let $(A,B) \in \T_G$ satisfying (*).

Since $(A,B) \in \T_G$, there exists $(A_{G'},B_{G'}) \in \T_{G'}$, where $\mu(E(A_{G'})) = E(A) \cap \mu(E(G'))$ and $\mu(E(B_{G'})) = E(B) \cap \mu(E(G'))$.
By Theorem \ref{repsectful_cycle}, for every $e \in E(A_{G'})$, there exists a cycle $C_e$ in $K$ with $V(C_e) \cap V(G') \subseteq V(A_{G'} \cap B_{G'})$ and $e \subseteq f(C_e)$; it implies that if not both ends of $e$ are in $V(A_{G'} \cap B_{G'})$, then one end of $e$ is contained in the interior of $f(C_e)$.
Moreover, for every $e \in E(A_{G'}) \subseteq E(G')$, there exists $(S_e,\Omega_e) \in \Se_1$ such that $e \subseteq \partial\alpha(S_e,\Omega_e)$; if some end of $e$ is not in $V(A_{G'} \cap B_{G'})$, then $\overline{\Omega_e} \subseteq V(A_{G'})$ (by the property (*)) and $\overline{\Omega_e} \not \subseteq V(B_{G'})$, so $S_e \subseteq A^{C_e}_G$ (by the definition of $A^{C_e}_G$) and $(A^{C_e}_G,B^{C_e}_G) \in \T$ (by the definition of $f$).
So we have 
	\begin{itemize}
		\item[(i)] $S_e \subseteq A^{C_e}_G$ and $(A^{C_e}_G,B^{C_e}_G) \in \T$ for every $e \in E(A_{G'})$ not having both ends in $V(A_{G'} \cap B_{G'})$, where $(S_e,\Omega_e)$ is the member of $\Se_1$ such that $e \subseteq \partial\alpha(S_e,\Omega_e)$. 
	\end{itemize}
Since for every $e \in E(A_{G'})$, $V(C_e) \cap V(G') \subseteq V(A_{G'} \cap B_{G'})$, there exists a subset $R$ of $V(G)$ of size less than $2(\rho+1)(\theta^*+6\kappa)$ such that $V(A^{C_{e'}}_G \cap B^{C_{e'}}_G) \subseteq R$ for all $e' \in E(A_{G'})$.
Let $W = \{e \in E(A_{G'}):$ both ends of $e$ are in $V(A_{G'} \cap B_{G'})\}$.
Then for every $Q \subseteq E(A_{G'})-W$, $(\bigcup_{e \in Q}A^{C_e}_G, \bigcap_{e \in Q}B^{C_e}_G)$ has order at most $|R|<2(\rho+1)(\theta^*+6\kappa)$.
Since $\T$ satisfies (T1) and (T2), induction on $|Q|$ shows that
	\begin{itemize}
		\item[(ii)] $(\bigcup_{e \in E(A_{G'})-W}A^{C_e}_G, \bigcap_{e \in E(A_{G'})-W}B^{C_e}_G)$ has order less than less than $2(\rho+1)(\theta^*+6\kappa)$ and it is in $\T$. 
	\end{itemize}

\medskip

\noindent{\bf Claim 4:} If $S \subseteq B$ for every $(S,\Omega) \in \Se_1$ with $\overline{\Omega} \subseteq V(B)$, then $E(A) \cap \bigcup_{(S,\Omega) \in \Se_1}E(S) \subseteq E(\bigcup_{e \in E(A_{G'})-W}A^{C_e}_G)$.

\noindent{\bf Proof of Claim 4:}
Let $z \in E(A) \cap \bigcup_{(S,\Omega) \in \Se_1}E(S)$.
Let $(S,\Omega) \in \Se_1$ such that $z \in E(S)$.
By assumption of this claim, $\overline{\Omega}-V(B) \neq \emptyset$.
Since $(A,B)$ satisfies (*), $\overline{\Omega} \subseteq V(A)$ and $\overline{\Omega}-V(B) \neq \emptyset$.
So there exists $e \in E(A_{G'})$ with $z \in E(S_e) = E(S)$ such that some end of $e$ is not in $V(A_{G'} \cap B_{G'})$.
By (i), $z \in E(S_e) \subseteq E(A^{C_e}_G)$.
This proves the claim since $e \in E(A_{G'})-W$.
$\Box$

\medskip

\noindent{\bf Claim 5:} $\T_{G'}$ is a tangle in $G'$ of order $\theta^*$.

\noindent{\bf Proof of Claim 5:}
Suppose to the contrary that $\T_{G'}$ is not a tangle in $G'$ of order $\theta^*$.
So there exists a separation $(A',B') \in \T_{G'}$ of $G'$ with $V(A')=V(G')$.
This implies that $(G',\emptyset) \in \T_{G'}$ since $\T_{G'}$ satisfies (T1) and (T2).
Then the separation $(G,\emptyset) \in \T_G$.
Note that $(G,\emptyset)$ satisfies (*).
So we can take $(A,B)$ to be $(G,\emptyset)$; note that the corresponding $(A_{G'},B_{G'})=(G',\emptyset)$, so (ii) implies that $(\bigcup_{e \in E(G')-Q}A^{C_e}_G, \bigcap_{e \in E(G')-Q}B^{C_e}_G) \in \T$, where $Q = \{e \in E(G'):$ both ends of $e$ are in $V(G' \cap \emptyset)\}$. 
Indeed, $Q=\emptyset$, so $(\bigcup_{e \in E(G')}A^{C_e}_G, \bigcap_{e \in E(G')}B^{C_e}_G) \in \T$.

Since $G'$ has no isolated vertex, $\bigcup_{e \in E(G')}A^{C_e}_G$ contains $V(G')$.
Note that the order of $(\bigcup_{e \in E(G')}A^{C_e}_G, \bigcap_{e \in E(G')}B^{C_e}_G)$ is less than $2(\rho+1)(\theta^*+6\kappa)$ by (ii).
So there exists a separation $(A_1,B_1) \in \T$ of order less than $2(\rho+1)(\theta^*+6\kappa)$ such that $V(A_1) \supseteq V(G')$.
Since $\Se$ is $\T$-central, for any $(S,\Omega) \in \Se_2$ and $v \in \overline{\Omega}$, there exists a separation $(A_{S,v},B_{S,v})$ of $G$ in $\T$ such that $A_{S,v} = G[X^S_v]$.
For every $(S,\Omega) \in \Se_2$, let $v_S$ be the first vertex in $\Omega$. 
Since $2(\rho+1)(\theta^*+6\kappa) + (2\rho+1)\kappa \leq \theta$, we know $(A_1 \cup \bigcup_{(S,\Omega) \in \Se_2}A_{S,v_S}, B_1 \cap \bigcap_{(S,\Omega) \in \Se_2}B_{S,v_S}) \in \T$ is a separation of order less than $2(\rho+1)(\theta^*+6\kappa) + (2\rho+1)\kappa$.
Since $2(\rho+1)(\theta^*+6\kappa) + (2\rho+1)\kappa + (\rho+1)\kappa \leq \theta$, we can repeatedly moving $B_{S,y}$ (for $(S,\Omega) \in \Se_2$ and $y \in \overline{\Omega}-\{v_S\}$) into this the first part of this separation in the order $\Omega$ while keeping it a member of $\T$.
Therefore, there exists a separation $(A^*,B^*) \in \T$ with $V(A^*)=V(G)$, contradicting that $\T$ is a tangle in $G$.
$\Box$

\medskip

To complete the proof of this lemma, it suffices to show $\T_G \subseteq \T$, which implies that $\T_{G'}$ is conformal with $\T$.

Suppose to the contrary that $\T_G \not \subseteq \T$.
Hence there exists $(A_0,B_0) \in \T_G-\T$ such that among all separations in $\T_G-\T$, $|V(A_0 \cap B_0)|$ is minimum, and subject to this, $A_0$ is minimal. 

\medskip

\noindent{\bf Claim 6:} The following statements hold.
	\begin{itemize}
		\item $S \subseteq B_0$ for every $(S,\Omega) \in \Se$ with $\overline{\Omega} \subseteq V(A_0 \cap B_0)$.
		\item $S \subseteq B_0$ for every $(S,\Omega) \in \Se_1$ with $\overline{\Omega} \subseteq V(B_0)$.
		\item $(A_0,B_0)$ satisfies (*). 
	\end{itemize}

\noindent{\bf Proof of Claim 6:}
For every $(S,\Omega) \in \Se$, there exists a subgraph $Z$ of $G$ with $V(S \cap Z)=\overline{\Omega}$ such that $(S,Z)$ is separation of $G$.

We first assume $(S,\Omega) \in \Se$ with $\overline{\Omega} \subseteq V(A_0 \cap B_0)$ and prove the first statement of this claim.
Then the order of $(S,Z)$ is at most $|V(A_0 \cap B_0)|<\theta^* \leq \theta/2$, so either $(S,Z) \in \T$ or $(Z,S) \in \T$.
Since $\Se$ is $\T$-central, $(S,Z) \in \T$.
Since $(A_0,B_0) \in \T_G-\T$, $(B_0,A_0) \in \T$.
Since $V(S \cap Z) = \overline{\Omega} \subseteq V(A_0 \cap B_0)$, we know that $(B_0 \cup S, A_0 \cap Z)$ has order at most $|V(A_0 \cap B_0)|$ and hence is in $\T$.
So $(A_0 \cap Z, B_0 \cup S) \not \in \T$.
Since $A_0 \cap Z \subseteq A_0$ and $(A_0,B_0) \in \T_G$, we have $(A_0 \cap Z, B_0 \cup S) \in \T_G-\T$.
So $S \subseteq B_0$, for otherwise $(A_0 \cap Z, B_0 \cup S) \in \T_G-\T$ contradicts the minimality of $A_0$.

Now we assume $(S,\Omega) \in \Se_1$ with $\overline{\Omega} \subseteq V(B_0)$ and prove the second statement of this claim. 
Since $(S,\Omega) \in \Se_1$, the order of $(S,Z)$ is at most 3, so either $(S,Z) \in \T$ or $(Z,S) \in \T$.
Since $\Se$ is $\T$-central, $(S,Z) \in \T$.
Since $(A_0,B_0) \in \T_G-\T$, $(B_0,A_0) \in \T$.
Since $(B_0 \cup S, A_0 \cap Z)$ has order at most $|V(A_0 \cap B_0) \cup V(A_0 \cap S \cap Z)| = |V(A_0 \cap B_0) \cup V(A_0 \cap \overline{\Omega})| \leq |V(A_0 \cap B_0)|<\theta$, it is in $\T$.
Hence $(A_0 \cap Z, B_0 \cup S) \not \in \T$.
Since $A_0 \cap Z \subseteq A_0$ and $(A_0,B_0) \in \T_G$, we have $(A_0 \cap Z, B_0 \cup S) \in \T_G-\T$. 
So $S \subseteq B_0$, for otherwise $(A_0 \cap Z, B_0 \cup S) \in \T_G-\T$ contradicts the minimality of $A_0$.

Finally, we prove the third statement of this claim.
Suppose to the contrary that there exists $(S,\Omega) \in \Se_1$ with $\overline{\Omega}-V(A_0) \neq \emptyset$ and $\overline{\Omega}-V(B_0) \neq \emptyset$.
Since $|V(S) \cap V(G')| \leq 3$ and $\T_{G'}$ is a tangle of order $\theta^*>3$ by Claim 5, we know $(S,Z) \in \T_G$.

Recall that $|\overline{\Omega}-V(A_0)| \geq 1$ and $|\overline{\Omega}-V(B_0)| \geq 1$.
If $|\overline{\Omega}-V(B_0)| \geq 2$, then $|\overline{\Omega}-V(A_0)| =1$ and $\overline{\Omega} \cap V(A_0 \cap B_0)=\emptyset$, and since $\Se$ is effective with respect to $(\Se_1,\Se_2)$, we know $|V(S) \cap V(A_0 \cap B_0)| \geq 
|\overline{\Omega}-V(B_0)| \geq 2$, so $(A_0 \cup S, B_0 \cap Z)$ is a separation of $G$ of order at most $|V(A_0 \cap B_0) - V(S)|+|\overline{\Omega} \cap V(B_0)| \leq |V(A_0 \cap B_0)| - |V(S) \cap V(A_0 \cap B_0)|+1 \leq |V(A_0 \cap B_0)|-1$ and hence is in $\T_G-\T$ and contradicts the choice of $(A_0,B_0)$.

So $|\overline{\Omega}-V(B_0)| =1$, then since $|\overline{\Omega}-V(A_0)| \geq 1$ and $\Se$ is effective with respect to $(\Se_1,\Se_2)$, we know $|V(S) \cap V(A_0 \cap B_0)-\overline{\Omega}| \geq 1 = |\overline{\Omega}-V(B_0)|$, so $(A_0 \cap Z, B_0 \cup S)$ is a separation of $G$ of order at most $|\overline{\Omega} \cap V(A_0)|+|V(A_0 \cap B_0) - V(S)| = (|\overline{\Omega}-V(B_0)| + |\overline{\Omega} \cap V(A_0 \cap B_0)|) + (|V(A_0 \cap B_0)|-(|V(A_0 \cap B_0) \cap V(S) - \overline{\Omega}| + |V(A_0 \cap B_0) \cap \overline{\Omega}|) \leq |V(A_0 \cap B_0)|$.
So $(A_0 \cap Z, B_0 \cup S) \in \T_G$.
Since $\Se$ is $\T$-central, $(S,Z) \in \T$, so $(B_0 \cup S, A_0 \cap Z) \in \T$.
Hence $(A_0 \cap Z, B_0 \cup S) \in \T_G-\T$.
Since $\Se$ is effective with respect to $(\Se_1,\Se_2)$ and $\overline{\Omega}-V(B_0) \neq \emptyset$, there exists an edge in $A_0 \cap S$.
So $A_0 \cap Z \subset A_0$.
It contradicts the choice of $(A_0,B_0)$.
So $(A_0,B_0)$ satisfies (*).
$\Box$

\medskip

Since $(A_0,B_0) \in \T_G$ satisfies (*) by Claim 6, and $(A,B)$ is an arbitrary separation in $\T_G$ satisfying (*), now we can denote $(A_0,B_0)$ by $(A,B)$.
Let $(A_1,B_1) = (\bigcup_{e \in E(A_{G'})-W}A^{C_e}_G, \allowbreak \bigcap_{e \in E(A_{G'})-W}B^{C_e}_G)$.
By (ii) and Claims 4 and 6, we know that $(A_1,B_1)$ is a separation of $G$ of order less than less than $2(\rho+1)(\theta^*+6\kappa)$ in $\T$ such that $E(A_1) \supseteq E(A) \cap \bigcup_{(S,\Omega) \in \Se_1}E(S)$.
Since $\Se$ is effective with respect to $(\Se_1,\Se_2)$, and $G'$ has no isolated vertex by Claim 1, we know that for every vertex $v \in V(G') \cap V(A)$, there exists $(S,\Omega) \in \Se_1$ such that some edge of $S$ is incident with $v$.
Hence $V(G') \cap V(A) \subseteq V(A_1)$.

Since $(A,B) \not \in \T_G-\T$, we know $(B,A) \in \T$.
Since the order of $(A_1 \cup B, B_1 \cap A)$ is less than $|V(A_1 \cap B_1)|+|V(A \cap B)| \leq 2(\rho+1)(\theta^*+6\kappa)+\theta^* < \theta$, we know $(A_1 \cup B, B_1 \cap A) \in \T$.
Since $V(G') \cap V(A) \subseteq V(A_1)$, we know $V(G') \subseteq V(A_1) \cup V(B)$.
Since $|V((A_1 \cup B) \cap (B_1 \cap A)| + 2\rho+1 < \theta$, for each $(S,\Omega) \in \Se_2$, we can repeatedly add $X^S_v$ (for $v \in \overline{\Omega})$) into $A_1 \cup B$ in the order $\Omega$ to eventually create a separation in $\T$ of at most the order of $(A_1 \cup B, B_1 \cap A)$ such that the first part of this separation contains $A_1 \cup B \cup S$.
By repeating this operation for all members of $\Se_2$, we obtain a separation $(A^*,B^*) \in \T$ such that $V(A^*)=V(G)$.
But $\T$ satisfies (T3), a contradiction.
This proves the lemma.
\end{pf}

\subsection{A metric in a drawing and zones} \label{subsec:metric_drawing}

Let $\Gamma$ be a $2$-cell drawing in a surface $\Sigma$, and let $K$ be a radial drawing of $\Gamma$.
Let $\T$ be a respectful tangle in $\Gamma$ of order $\theta$ for some integer $\theta$, and let $\ins$ be the inside function of $\T$
For any two atoms $a,b$ of $K$, define a function $m_\T(a,b)$ as follows:
	\begin{itemize}
		\item if $a=b$, then $m_\T(a,b)=0$;
		\item if $a \neq b$ and $a,b \subseteq \ins(W)$ for some closed walk $W$ of $K$ of length less than $2 \theta$, then $m_\T(a,b) = \min\frac{1}{2} \lvert E(W) \rvert$, taking over all such closed walks $W$;
		\item otherwise, $m_\T(a,b) = \theta$.
	\end{itemize}
Since $K$ is bipartite, $m_\T$ is integral.
In addition, for every atom $c$ of $\Gamma$, we define $a(c)$ to be an atom of $K$ such that
	\begin{itemize}
		\item $a(c)=c$ if $c \subseteq V(\Gamma)$; 
		\item $a(c)$ is the region of $K$ including $c$ if $c$ is an edge of $\Gamma$; 
		\item $a(c)=\{v\}$, where $v$ is the vertex of $K$ in $c$, if $c$ is a region of $\Gamma$.
	\end{itemize}
For atoms $b,c$ of $\Gamma$, we define $m_\T(b,c) = m_\T(a(b),a(c))$.
Note that \cite[(9.1)]{rs XI} implies that the following: 

\begin{theorem}[\cite{rs XI}]
Let $\Gamma$ be a $2$-cell drawing in a surface $\Sigma$.
If $\T$ is a respectful tangle in $\Gamma$, then $m_\T$ is a metric on the atoms of $\Gamma$.
\end{theorem}

Let $\T$ be a respectful tangle in a 2-cell drawing $G$ in a surface $\Sigma$.
If $X,Y$ are sets of atoms of $G$, then we define $m_\T(X,Y)=\min\{m_\T(x,y): x \in X, y \in Y\}$.
When one of $X$ and $Y$, say $Y$, has size one, we denote $m_\T(X,Y)$ by $m_\T(X,y)$, where $y$ is the unique element of $Y$.

The following is a restatement of a result in \cite{lt}.

\begin{theorem}[{\cite[Theorem 5.3]{lt}}] \label{A distance} 
Let $\Sigma$ be a connected surface, and let $\Gamma$ be a $2$-cell drawing of a graph in $\Sigma$ with $E(\Gamma) \neq \emptyset$.
Let $\T$ be a respectful tangle of order $\theta$ in $\Gamma$.
Let $x \in V(\Gamma)$.
If $(A,B) \in \T$ is a separation of $\Gamma$ such that $x \in V(A)-V(B)$ and there exists a path $P$ in $A$ from $x$ to a vertex $y \in V(A)$ internally disjoint from $V(B)$, then $m_\T(x,y) \leq \lvert V(A) \cap V(B) \rvert$.
\end{theorem}

We also need the following variant:

\begin{lemma}[{\cite[Lemma 5.2]{l}}] \label{A distance set}
Let $\Sigma$ be a surface, and let $\Gamma$ be a $2$-cell drawing of a graph in $\Sigma$ with $E(\Gamma) \neq \emptyset$.
Let $\T$ be a respectful tangle of order $\theta$ in $\Gamma$.
Let $X \subseteq V(\Gamma)$.
If $(A,B) \in \T$ is a separation of $\Gamma$ of order less than $\lvert X \rvert$ such that $X \subseteq V(A)$, and subject to this, $A$ is minimal, then $m_\T(X,y) \leq \lvert V(A) \cap V(B) \rvert$ for every $y \in V(A)$.
\end{lemma}

Let $\Gamma$ be a $2$-cell drawing in a surface $\Sigma$, and let $\T$ be a respectful tangle of order $\theta$ in $\Gamma$.
Let $x$ be an atom of $\Gamma$.
A \defn{$\lambda$-zone around $x$} is an open disk $\Delta$ in $\Sigma$ with $x \subseteq \Delta$ such that $\partial\Delta$ is an O-arc, $\partial\Delta \subseteq U(\Gamma)$, $m_\T(x,y) \leq \lambda$ for every atom $y$ of $\Gamma$ with $y \subseteq \overline{\Delta}$, and if $x \in E(\Gamma)$, then $\lambda \geq 2$.
A \defn{$\lambda$-zone} is a $\lambda$-zone around some atom.
Let $\Delta$ be a $\lambda$-zone.
Note that $U(\Gamma) \cap \partial\Delta$ is a cycle, and the drawing $\Gamma' = \Gamma \cap (\Sigma - \Delta)$ is $2$-cell in $\Sigma$.
We say that $\Gamma'$ is the \defn{drawing obtained from $\Gamma$ by clearing $\Delta$}.
We say that $\T'$ is a \defn{tangle of order $\theta-4\lambda-2$ obtained by clearing $\Delta$} if $\T'$ is a tangle in $\Gamma'$ of order $\theta-4\lambda-2$ satisfying the following.
\begin{itemize}
	\item $\T'$ is respectful with a metric $m_{\T'}$.
	\item $\T'$ is conformal with $\T$.
	\item If $x,y$ are atoms of $\Gamma$ and $x',y'$ are atoms of $\Gamma'$ with $x \subseteq x'$ and $y \subseteq y'$, then $m_\T(x,y) \geq m_{\T'}(x',y') \geq m_\T(x,y)-4\lambda-2$.
\end{itemize}

\begin{theorem}[{\cite[(7.10)]{rs XII}}] \label{clean a zone}
Let $\Delta$ be a $\lambda$-zone.
If $\theta \geq 4\lambda+3$, then there exists a unique respectful tangle of order $\theta-4\lambda-2$ obtained by clearing $\Delta$.
\end{theorem}

\begin{theorem}[{\cite[(9.2)]{rs XIV}}] \label{big zone contains ball}
Let $\Gamma$ be a $2$-cell drawing in a surface $\Sigma$, and let $\T$ be a respectful tangle in $\Gamma$ of order $\theta$.
Let $x$ be an atom of $\Gamma$, and $\lambda$ an integer with $2 \leq \lambda \leq \theta-4$.
Then there exists a $(\lambda+3)$-zone $\Delta$ around $x$ such that $x' \subseteq \Delta$ for every atom $x'$ of $\Gamma$ with $m_\T(x,x') \leq \lambda$.
\end{theorem}

\begin{lemma}[{\cite[Lemma 5.7]{lt}}] \label{disjoint boundary of zone}
Let $\Gamma$ be a $2$-cell drawing in a surface, $z$ an atom, and $\T$ a respectful tangle in $\Gamma$ of order $\theta$.
Let $\lambda$ be a nonnegative integer, and let $C$ be the cycle of the boundary of a $\lambda$-zone around $z$.
If $\theta \geq \lambda+8$, then there exists a $(\lambda+7)$-zone $\Lambda$ around $z$ such that the cycle bounding $\Lambda$ is disjoint from $C$, and $\Lambda$ contains the $\lambda$-zone bounded by $C$.
\end{lemma}

\begin{lemma} \label{buffer}
Let $\Gamma$ be a $2$-cell drawing in a surface, $z$ an atom of $\Gamma$, and $\T$ a respectful tangle in $\Gamma$ of order $\theta$.
Let $\lambda$ and $k$ be positive integers. 
Let $Y$ be a subset of atoms of $\Gamma$ such that $m_\T(z,y) \leq \lambda$ for every $y \in Y$. 
If $\theta \geq 10k+\lambda+4$, then there exist disjoint cycles $C_1,C_2,...,C_k$ in $\Gamma$ such that 
	\begin{enumerate}
		\item for each $i \in [k]$, $C_i$ bounds a $(\lambda+3+10i)$-zone $\Lambda_i$ around $z$, 
		\item $\Lambda_k \supseteq \Lambda_{k-1} \supseteq ... \supseteq \Lambda_1 \supseteq Y$, and  
		\item $x \subseteq \Lambda_{i}$ for every atom $x$ of $\Gamma$ with $m_\T(z,x) \leq \lambda+10i$ for every $i \in [k]$. 
	\end{enumerate}
\end{lemma}

\begin{pf}
By Theorem \ref{big zone contains ball}, there exists a $(\lambda+3)$-zone $\Lambda_0$ around $z$ such that $x \subseteq \Lambda_0$ for every atom $x$ of $\Gamma$ with $m_\T(z,x) \leq \lambda$.
So $Y \subseteq \Lambda_0$.
Let $C_0$ be the cycle in $\Gamma$ bounding $\Lambda_0$.
For every $i \in [k]$, by Theorem \ref{big zone contains ball} and Lemma \ref{disjoint boundary of zone}, there exists a $(\lambda+3+10i)$-zone $\Lambda_i$ around $z$ such that the cycle $C_i$ bounding $\Lambda_i$ is disjoint from $C_{i-1}$, and $\Lambda_i \supseteq \Lambda_{i-1}$ and $x \subseteq \Lambda_i$ for every atom $x$ of $\Gamma$ with $m_\T(z,x) \leq \lambda+3+10(i-1)+7$.
Then $C_1,C_2,...,C_k$ are cycles satisfying this lemma.
\end{pf}

\subsection{Better arrangements}

Let $\Sigma$ be a surface, $\theta$ an integer, and $\phi$ a nondecreasing function with domain ${\mathbb Z}$.
Let $\T$ be a tangle in a graph $G$.
We say that an arrangement $\alpha$ of a segregation $\Se$ of $G$ with a $(\kappa,\rho)$-witness $(\Se_1,\Se_2)$ in $\Sigma$ is a \defn{$(\Sigma, \theta,\phi,\T)$-arrangement with respect to $(\Se_1,\Se_2)$} if the following conditions hold.
	\begin{itemize}
		\item $\alpha$ is a proper arrangement of $\Se$ in $\Sigma$.
		\item The skeleton of $\alpha$ with respect to $(\Se_1,\Se_2)$ is a 2-cell drawing in $\Sigma$ and is a natural minor of $G$.
		\item There exists a respectful tangle $\T'$ of order at least $\theta$ in the skeleton of $\alpha$ with respect to $(\Se_1,\Se_2)$ conformal with $\T$.
		\item $m_{\T'}(\overline{\Omega}, \overline{\Omega'}) \geq \phi(\rho')$ for all distinct members $(S,\Omega),(S',\Omega')$ of $\Se_2$, where $\rho'$ is the smallest nonnegative integer such that every member of $\Se_2$ is a $\rho'$-vortex. 	
		\item $\Se$ is a $\T$-central segregation. 
	\end{itemize}
When $\phi$ is a constant function with $\phi(x)=c$ for some $c$, we also write a $(\Sigma, \theta,\phi,\T)$-arrangement as a $(\Sigma, \theta,c,\T)$-arrangement.

\begin{lemma}[{{\cite[Lemma 5.11]{l}}}] \label{sweeping balls into vortices}
For any\footnote{In \cite[Lemma 5.11]{l}, $\kappa,k$ are assumed to be a positive integer, but the proof in \cite[Lemma 5.11]{l} works when $\kappa=0$ or $k=0$.} $\kappa,k\in {\mathbb N}_0$, there exists\footnote{The statement in \cite[Lemma 5.11]{l} does not include $\kappa^* \leq \kappa+k$, but it is proved in \cite[Lemma 5.11]{l}.} $\kappa^*=\kappa^*(\kappa,k) \leq \kappa+k$ such that for all $\rho,\lambda \in {\mathbb N}$ and nondecreasing function $\phi$ with domain ${\mathbb Z}$, there exist $\rho^*=\rho^*(\kappa,k,\rho,\lambda,\phi)$ such that for every integer $\theta^*$, there exists a positive integer\footnote{The condition that $\theta$ is a positive integer with $\theta \geq \theta^*$ is not stated in \cite[Lemma 5.11]{l}, but it is proved in \cite[Lemma 5.11]{l}.} $\theta=\theta(\kappa,k,\rho,\lambda,\phi,\theta^*) \geq \theta^*$ such that the following hold.
Let $\T$ be a tangle in a graph $G$, and let $\Se$ be a $\T$-central segregation of $G$ with a $(\kappa,\rho)$-witness $(\Se_1,\Se_2)$ with a proper arrangement $\alpha$ in a surface $\Sigma$ such that the skeleton $G'$ of $\alpha$ with respect to $(\Se_1,\Se_2)$ is a natural minor of $G$ and is a 2-cell drawing in a surface $\Sigma$ with a respectful tangle $\T'$ with order at least $\theta$ that is conformal with $\T$.
If $\Lambda_1,\Lambda_2,...,\Lambda_k$ are $\lambda$-zones around some atoms of $G'$, then there exists a $\T$-central segregation $\Se^*$ of $G$ with a $(\kappa^*,\rho^*)$-witness $(\Se_1^*,\Se_2^*)$ such that the following hold.
	\begin{enumerate}
		\item $\Se_1^* \subseteq \Se_1$ and $\bigcup_{(S,\Omega) \in \Se_2}S \cup \bigcup_{(S,\Omega) \in \Se_1, \alpha(S,\Omega) \subseteq \bigcup_{i=1}^k \overline{\Lambda_i}}S \subseteq \bigcup_{(S,\Omega) \in \Se_2^*}S$.
		\item $\Se^*$ has a $(\Sigma,\theta^*,\phi,\T)$-arrangement with respect to $(\Se_1^*,\Se_2^*)$. 
		\item $\bigcup_{(S,\Omega) \in \Se_2}\alpha(S,\Omega) \cup \bigcup_{i=1}^k\Lambda_i \subseteq \bigcup_{(S,\Omega) \in \Se_2^*}\alpha^*(S,\Omega)$ and $\alpha^*(S,\Omega)=\alpha(S,\Omega)$ for every $(S,\Omega) \in \Se^*_1$.\footnote{Statement 3 of this lemma was not stated in \cite[Lemma 5.11]{l}, but it is proved in \cite[Lemma 5.11]{l}.}
	\end{enumerate}
\end{lemma}

What we will actually use is the following version of Lemma \ref{sweeping balls into vortices} for $(c,r,\Se_1,\Se_2)$-protected arrangements.

\begin{lemma} \label{sweeping_into_vortices_protected}
For any $\kappa,k \in {\mathbb N}_0$, there exists $\kappa^*=\kappa^*(\kappa,k) \leq \kappa+k$ such that for all $\rho,\lambda \in {\mathbb N}$ and nondecreasing function $\phi: {\mathbb Z} \rightarrow {\mathbb R}$, there exist $\rho^*=\rho^*(\kappa,k,\rho,\lambda,\phi)$ such that for every $\theta^* \in {\mathbb N}$, there exists a positive integer $\theta=\theta(\kappa,k,\rho,\lambda,\phi,\theta^*) \geq \theta^*$ such that the following hold.
Let $G$ be a graph, and let $L$ be an induced subgraph of $G$.
Let $\T_L$ be a tangle in $L$, and let $\Se$ be a $\T_L$-central segregation of $L$ with a $(\kappa,\rho)$-witness $(\Se_1,\Se_2)$ such that there exists a $(\Sigma,\theta,1,\T_L)$-arrangement $\alpha$ with respect to $(\Se_1,\Se_2)$ for some surface $\Sigma$.
Let $L'$ be the skeleton of $\alpha$ with respect to $(\Se_1,\Se_2)$.
Let $\Lambda_1,\Lambda_2,...,\Lambda_k$ be $\lambda$-zones around some atoms of $L'$.
If $\alpha$ is $(c,r,\Se_1,\Se_2)$-protected for some $c,r \in {\mathbb N}$, then there exists a $\T_L$-central segregation $\Se^*$ of $L$ with a $(\kappa^*,\rho^*)$-witness $(\Se_1^*,\Se_2^*)$ such that
	\begin{enumerate}
		\item $\Se_1^* \subseteq \Se_1$ and $\bigcup_{(S,\Omega) \in \Se_2}S \cup \bigcup_{(S,\Omega) \in \Se_1, \alpha(S,\Omega) \subseteq \bigcup_{i=1}^k \overline{\Lambda_i}}S \subseteq \bigcup_{(S,\Omega) \in \Se_2^*}S$, and
		\item $\Se^*$ has a $(c,r,\Se_1^*,\Se_2^*)$-protected $(\Sigma,\theta^*,\phi,\T_L)$-arrangement with respect to $(\Se_1^*,\Se_2^*)$. 
	\end{enumerate}
\end{lemma}

\begin{pf}
Let $\kappa,k,\rho,\lambda,\phi,\theta^*$ be as stated in the lemma.
Define $\kappa^*,\rho^*,\theta$ as in Lemma \ref{sweeping balls into vortices}.

Let $G,L,\T_L,\Se,\Se_1,\Se_2,\alpha,\Sigma,L',\Lambda_1,...,\Lambda_k$ be as stated in the lemma.
By Lemma \ref{sweeping balls into vortices}, there exists a $\T_L$-central segregation $\Se^*$ of $L$ with a $(\kappa^*,\rho^*)$-witness $(\Se_1^*,\Se_2^*)$ such that Statement 1 of this lemma holds, and $\Se^*$ has a $(\Sigma,\theta^*,\phi,\T_L)$-arrangement $\alpha^*$ with respect to $(\Se_1^*,\Se_2^*)$ such that $\bigcup_{(S,\Omega) \in \Se_2}\alpha(S,\Omega) \cup \bigcup_{i=1}^k\Lambda_i \subseteq \bigcup_{(S,\Omega) \in \Se_2^*}\alpha^*(S,\Omega)$ and $\alpha^*(S,\Omega)=\alpha(S,\Omega)$ for every $(S,\Omega) \in \Se^*_1$. 

To prove this lemma, it suffices to show that $\alpha^*$ is $(c,r,\Se_1^*,\Se_2^*)$-protected.

Let $u,v \in V(L)$ such that $\dist_G(u,v) \leq r$.
Since $\alpha$ is $(c,r,\Se_1,\Se_2)$-protected, there exists a path $P$ in $L' \cup K$ with at most $cr$ edges between $\overline{\Omega_u}$ and $\overline{\Omega_v}$ for some $(S_u,\Omega_u),(S_v,\Omega_v) \in \Se$ with $u \in V(S_u)$ and $v \in V(S_v)$, where $K$ is a radial drawing of $L'$.
Let $L^*$ be the skeleton of $\alpha^*$ with respect to $(\Se_1^*,\Se_2^*)$, and let $K^*$ be a radial drawing of $L^*$.
Note that $L^*$ is a subgraph of $L'$, and $L^* \cup K^*$ is a minor of $L' \cup K$.
If $P \subseteq L^* \cup K^*$, then $P$ is a path in $L^* \cup K^*$ with at most $cr$ edges between $\overline{\Omega_u'}$ and $\overline{\Omega_v'}$ for some $(S_u',\Omega_u'),(S_v',\Omega_v') \in \Se^*$ with $u \in V(S_u')$ and $v \in V(S_v')$.
If $P \not \subseteq L^* \cup K^*$, then there exists a walk in $L^* \cup K^*$ with at most $|E(P)| \leq cr$ edges between $\overline{\Omega_u'}$ and $\overline{\Omega_v'}$ for some $(S_u',\Omega_u'),(S_v',\Omega_v') \in \Se^*$ with $u \in V(S_u')$ and $v \in V(S_v')$.
Hence $\alpha^*$ is $(c,r,\Se_1^*,\Se_2^*)$-protected.
\end{pf}

\section{Avoiding apices} \label{sec:avoid_apices}

Let $G$ be a graph, and let $L$ be a subgraph of $G$.
Recall that for a set $\F$ of subgraphs of $G$, we define $\F \cap L$ to be the set $\{F \in \F: F \subseteq L\}$.

\begin{lemma} \label{avoid_apex_new_0}
For any $g,\kappa,\rho \in {\mathbb N}_0$ and $k,r,\xi,\eta \in {\mathbb N}$, there exist $\xi^*,\eta^* \in {\mathbb N}$ with $\eta^*=\eta+r$ such that the following holds.

Let $G$ be a graph, and let $L$ be an induced subgraph of $G$.
Let $\F$ be a set of connected subgraphs of $L$ such that no $k$ members of $\F$ have pairwise distance in $G$ at least $r$.
Let $Z \subseteq V(G)$, and let $Z^+ \subseteq V(G)$ be a $(\xi,\eta)$-centered set in $G$ with $Z^+ \supseteq Z \cup N_G(V(L))$.
Let $\T_L$ be the $(G,\F,r,\theta',Z^+)$-tangle in $L$ for some integers $\theta$ and $\theta'$ with $\theta' \geq \theta > |Z|+4\rho+6$, and let $\T_G$ be a tangle of order at least $\theta$ in $G$ induced by $\T_L$.
Let $\Se$ be a $(\T_G-Z)$-central segregation of $G-Z$ with a $(\kappa,\rho)$-witness $(\Se_1,\Se_2)$.
If there exists a proper arrangement of $\Se$ in $\Sigma$ with respect to $(\Se_1,\Se_2)$ for some surface $\Sigma$ of Euler genus at most $g$, then there exists $Z^* \subseteq V(G)$ with $Z^* \supseteq N_G^{\leq r}[Z^+]$ such that $Z^*$ is $(\xi^*,\eta^*)$-centered in $G$ intersecting all members of $\F$ contained in $\bigcup_{(S,\Omega) \in \Se_2}S$.  
\end{lemma}

\begin{pf}
Let $g,\kappa,\rho \in {\mathbb N}_0$.
Let $k,r,\xi,\eta \in {\mathbb N}$.
Define $\xi^*$ and $\eta^*$ to be the integers $\xi^*$ and $\eta^*$, respectively, mentioned in Lemma \ref{hitting_vortices} by taking $(k,r,c,\xi,\eta,\kappa,\rho,g) = (k,r,1,\xi,\eta,\kappa,\rho,g)$.

Let $G,L,\F,\T_L,\T_G,Z,Z^+,\Se,\Se_1,\Se_2,\Sigma$ be as stated in the lemma.

We will use Lemma \ref{hitting_vortices} to construct a hitting set for members of $\F$ contained in $\bigcup_{(S,\Omega) \in \Se_2}S$.

Let $U = \bigcup_{(S,\Omega) \in \Se_2}\overline{\Omega}-V(L)$.
Let $G'$ be the graph obtained from $G$ by, for every $v \in U$, adding a copy $\iota(v)$ of $v$ into $G'$ such that $\iota(v)$ is an isolated vertex in $G'$.
Let $L' = G'[V(L) \cup \{\iota(v): v \in U\}]$.
Note that $L'$ is isomorphic to a subgraph of $G$.

For every $(S,\Omega) \in \Se$, let $\iota(S)$ be the graph obtained from $S \cap L$ by adding $\{\iota(v): v \in \overline{\Omega} \cap U\}$, and let $\iota(\Omega)$ be the cyclic ordering on $(\overline{\Omega} \cap V(L)) \cup \{\iota(v): v \in \overline{\Omega} \cap U\}$ such that $\Omega$ can be obtained from $\iota(\Omega)$ by replacing each $\iota(u)$ by $u$.
For every $i \in [2]$, let $\Se_i' = \{(\iota(S), \iota(\Omega)): (S,\Omega) \in \Se_i\}$.
Let $\Se'=\Se_1' \cup \Se_2'$.
Note that $\Se'$ is a segregation of $L'$ with a $(\kappa,\rho)$-witness $(\Se_1',\Se_2')$ such that there exists a proper arrangement in a surface of Euler genus at most $g$.

Observe that every vertex in $V(L')-V(L)$ is an isolated vertex in $L'$.
So $N_{G'}(V(L')) = N_G(V(L)) \subseteq Z^+$.

\medskip

\noindent{\bf Claim 1:} For every $(\iota(S),\iota(\Omega)) \in \Se_1'$, $\iota(S) - N_{G'}^{\leq r}[\overline{\iota(\Omega)} \cup Z^+]$ does not contain a member of $\F-N_{G'}^{\leq r}[Z^+]$.

\noindent{\bf Proof of Claim 1:}
Let $(S,\Omega) \in \Se_1$.
There exists a separation $(A_S,B_S)$ of $G-Z$ with $A_S=S$ and $V(A_S \cap B_S)=\overline{\Omega}$.
Since $\Se$ is $(\T_G-Z)$-central and $|V(A_S \cap B_S)| \leq |\overline{\Omega}| \leq 3<(\theta-|Z|)/2$, $(A_S,B_S) \in \T_G-Z$.
Note that $(A_S \cap L, B_S \cap L)$ is a separation of $L$ of order at most $|\overline{\Omega}| \leq 3<(\theta-|Z|)/2$, so $(A_S \cap L, B_S \cap L) \in \T_L$ since $\T_G$ is induced by $\T_L$.
Hence $A_S \cap L - N_G^{\leq r}[V(A_S \cap B_S \cap L)]$ does not contain a member of $\F-N_G^{\leq r}[Z^+]$.
So $S \cap L - N_G^{\leq r}[(\overline{\Omega} \cap V(L)) \cup Z^+] \subseteq A_S \cap L - N_G^{\leq r}[V(A_S \cap B_S \cap L)]$ does not a member of $\F-N_G^{\leq r}[Z^+]$.
Hence $\iota(S) - N_{G'}^{\leq r}[\overline{\iota(\Omega)} \cup Z^+] = S \cap L - N_G^{\leq r}[(\overline{\Omega} \cap V(L)) \cup Z^+]$ does not a member of $\F-N_G^{\leq r}[Z^+]$.
$\Box$

\medskip

For every $(S,\Omega) \in \Se_2$, since $(S,\Omega)$ is a $\rho$-vortex, there exists a vortical decomposition $(P^S,(X^S_t: t \in V(P^S)))$ of $(S,\Omega)$ of adhesion at most $\rho$ with $V(P^S)=\overline{\Omega}$; for every $t \in V(P^S)$, let $Y_t = \{t\} \cup (X^S_t \cap \bigcup_{t' \in V(P^S)-\{t\}}X^S_{t'})$, and let $(A_t,B_t)$ be the separation of $S \cap L$ such that $A_t = S[V(L) \cap X^S_t]$ and $V(A_t \cap B_t) = V(L) \cap Y_t$. 

\medskip

\noindent{\bf Claim 2:} For any $(S,\Omega) \in \Se_2$ and $t \in V(P^S)$, we have that $L[V(L) \cap X^S_t]-N_G^{\leq r}[(V(L) \cap Y_t) \cup Z^+]$ does not contain any member of $\F-N_G^{\leq r}[Z^+]$.

\noindent{\bf Proof of Claim 2:}
Note that for every $t \in V(P^S)$, there exists a separation $(R_t,Q_t)$ of $G-Z$ such that $V(R_t)=X^S_t$ and $V(R_t \cap Q_t)=Y_t$.
For every $t \in V(P^S)$, since $|Y_t| \leq 2\rho+1 <(\theta-|Z|)/2$ and $\Se$ is $(\T_G-Z)$-central, $(R_t,Q_t) \in \T_G-Z$, so $(R_t',Q_t') \in \T_G$ for some subgraphs $R_t'$ and $Q_t'$ of $G$ with $R_t'=G[V(R_t) \cup Z]$ and $V(Q_t') = V(Q_t) \cup Z$.
Since $\T_G$ is induced by $\T_L$, for every $t \in V(P^S)$, we know $(R'_t[V(L)], Q_t'[V(L)]) \in \T_L$, so $(A_t,B_t) \in \T_L$, and hence $A_t-N_G^{\leq r}[V(A_t \cap B_t)]$ does not contain a member of $\F-N_G^{\leq r}[Z^+]$; so $L[V(L) \cap X^S_t]-N_G^{\leq r}[(V(L) \cap Y_t) \cup Z^+]$ does not contain any member of $\F-N_G^{\leq r}[Z^+]$.
$\Box$

\medskip

For every $(S,\Omega) \in \Se_2$, let $\iota(P^S)$ be the path obtained from $P^S$ by replacing each vertex $v \in U \cap V(P^S)$ by $\iota(v)$; for simplicity of notation, we do not distinguish $v$ and $\iota(v)$ for each $v \in U \cap V(P^S)$; for every $t \in V(\iota(P^S))$, let $\iota(X^S_t)$ be obtained from $X^S_t$ by replacing each vertex $v \in U \cap X^S_t$ by $\iota(v)$.
Then for every $(S,\Omega) \in \Se_2$, $(\iota(P^S),(\iota(X^S_t): t \in V(\iota(P^S))))$ is a vortical decomposition of $(\iota(S),\iota(\Omega))$ of adhesion at most $\rho$ with $V(\iota(P^S))=\overline{\iota(\Omega)}$.
By Claim 2, for any $(\iota(S),\iota(\Omega)) \in \Se_2'$ and $t \in \iota(V(P^S))$, we have that $L'[\iota(X^S_t)]-N_{G'}^{\leq r}[\{t\} \cup (\iota(X^S_t) \cap \bigcup_{t' \in V(\iota(P^S))-\{t\}}\iota(X^S_{t'})) \cup Z^+] = L[V(L) \cap X^S_t]-N_G^{\leq r}[(V(L) \cap Y_t) \cup Z^+]$ does not contain any member of $\F-N_G^{\leq r}[Z^+]$.

By Lemma \ref{hitting_vortices}, there exists a set $Z^* \supseteq N_{G'}^{\leq r}[Z^+]$ such that $Z^*$ is $(\xi^*,\eta^*)$-centered in $G'$ and intersects all members of $\F-N_G^{\leq r}[Z^+]$ contained in $\bigcup_{(\iota(S),\iota(\Omega)) \in \Se_2'}\iota(S)$.
Since $N_{G'}^{\leq r}[Z^+] = N_G^{\leq r}[Z^+]$ and every member of $\F$ is contained in $L$, we can remove vertices not in $N_G^{\leq r}[Z^+] \cup V(L)$ from $Z^*$.
So we may assume that $Z^* \supseteq N_{G}^{\leq r}[Z^+]$ such that $Z^*$ is $(\xi^*,\eta^*)$-centered in $G$ and intersects all members of $\F-N_G^{\leq r}[Z^+]$ contained in $\bigcup_{(S,\Omega) \in \Se_2}(S \cap L)$ and hence intersects all members of $\F-N_G^{\leq r}[Z^+]$ contained in $\bigcup_{(S,\Omega) \in \Se_2}S$.
Since $Z^* \supseteq N_G^{\leq r}[Z^+]$, we know that $Z^*$ intersects all members of $\F$ contained in $\bigcup_{(S,\Omega) \in \Se_2}S$.
\end{pf}

\begin{lemma} \label{avoid_apex_new}
For any $g,\kappa \in {\mathbb N}_0$, there exists $\kappa^* \in {\mathbb N}_0$ such that for any $k,r,\xi_{-1},\theta^* \in {\mathbb N}$ and $\rho \in {\mathbb N}_0$, there exist $\theta \in {\mathbb N}$ such that for any $\xi,\eta \in {\mathbb N}$, there exist $\xi^*,\eta^* \in {\mathbb N}$ with $\eta^*=\eta+(g+2)r$ such that the following holds.

Let $G$ be a graph, and let $L$ be an induced subgraph of $G$.
Let $\F$ be a set of connected subgraphs of $L$ such that no $k$ members of $\F$ have pairwise distance in $G$ at least $r$.
Let $Z \subseteq V(G)$ with $|Z| \leq \xi_{-1}$, and let $Z^+ \subseteq V(G)$ be a $(\xi,\eta)$-centered set in $G$ with $Z^+ \supseteq Z \cup N_G(V(L))$.
Let $\T_L$ be the $(G,\F,r,\theta',Z^+)$-tangle in $L$ for some integer $\theta' \geq \theta$, and let $\T_G$ be a tangle of order at least $\theta$ in $G$ induced by $\T_L$.
Let $\Se$ be a $(\T_G-Z)$-central segregation of $G-Z$ with a $(\kappa,\rho)$-witness $(\Se_1,\Se_2)$.
If there exists a proper arrangement of $\Se$ in $\Sigma$ with respect to $(\Se_1,\Se_2)$ for some surface $\Sigma$ of Euler genus at most $g$, then there exist a (possibly empty) set $\HH$ of pairwise disjoint induced subgraphs of $L-N_G^{\leq r}[Z^+]$ and a collection $(Z_H: H \in \HH)$ of subsets of $V(G)$ such that
	\begin{enumerate}
		\item $N_G^{\leq r}[\bigcup_{H \in \HH}Z_H]$ intersects all members of $\F-(\F \cap \bigcup_{H \in \HH}H)$,
		\item for every $H \in \HH$, we have $Z_H \supseteq N_G(V(H))$, $Z_H$ is $(\xi^*,\eta^*)$-centered in $G$, and $(\F \cap H)-N_G^{\leq r}[Z_H] \neq \emptyset$, and
		\item for every $H \in \HH$, either 
			\begin{enumerate}
				\item there exists $Z_H^* \subseteq V(G)$ with $Z_H^* \supseteq N_G^{\leq r}[Z_H]$ such that $Z_H^*$ is $(\xi^*,\eta^*)$-centered in $G$ and intersects all members of $\F \cap H$, or
				\item there exist a $(G,\F \cap H,r,\theta,Z_H)$-tangle $\T_H$ in $H$, a segregation $\Se^{H+}$ of $G[N_G^{\leq r}[V(H)]]$ with a $(\kappa^*,\rho)$-witness $(\Se^{H+}_1,\Se^{H+}_2)$, a proper arrangement $\alpha^{H+}$ of $\Se^{H+}$ in a surface $\Sigma_H$ of Euler genus at most $g$, an $(H,\Se_1^{H+},\Se_2^{H+})$-adaption $\Se^H$ of $\Se^{H+}$ with witness $(\Se_1^H,\Se_2^H)$, and an $\Se^H$-adjustment $\alpha_H$ of $\alpha^{H+}$ such that $\alpha_H$ is a $(\Sigma_H,\theta^*,1,\T_{H})$-arrangement of $\Se^H$ with respect to $(\Se^H_1,\Se^H_2)$. 
			\end{enumerate}
	\end{enumerate}
\end{lemma}

\begin{pf}
Let $g,\kappa \in {\mathbb N}_0$.
For every $i \in {\mathbb N}_0$, let $\kappa_i = 2^i\kappa$.
Define $\kappa^*=\kappa_{g+1}$.
Let $k,r,\xi_{-1},\theta^* \in {\mathbb N}$ and $\rho \in {\mathbb N}_0$.
Let $\theta_0 = \theta^*+\xi_{-1}+2\rho+6$.
For every $i \in {\mathbb N}$, let $\theta_i = \theta_{i-1} + 2(\kappa_i+2)(\theta^*+7)^2+4r+53 + \theta_i'$, where $\theta_i'$ is the integer $\theta$ mentioned in Lemma \ref{segregation_respectful_central} by taking $(\theta^*,\kappa,\rho)=(\theta^*+3,\kappa_i,\rho)$.
Define $\theta = \theta_{g+1}$.
Let $\xi,\eta \in {\mathbb N}$.
Let $\xi_0$ and $\eta_0$ be the integers $\xi^*$ and $\eta^*$, respectively, mentioned in Lemma \ref{avoid_apex_new_0} by taking $(g,\kappa,\rho,k,r,\xi,\eta) = (g,\kappa,\rho,k,r,\xi,\eta)$.
For every positive integer $i$, let $\xi_i = \xi_{i-1}+(2\kappa_{i-1}(2\rho+1)+4)\theta$ and $\eta_i = \eta_{i-1}+r$.
Define $\xi^* = \xi_{2k-2+g}$ and $\eta^*= \eta_{g+1}$.

Let $G,L,\F,\T_L,\T_G,Z,Z^+,\Se,\Se_1,\Se_2,\Sigma$ be as stated in the lemma.
Since $\T_L$ be the $(G,\F,r, \allowbreak \theta',Z^+)$-tangle in $L$, we know $k \geq 2$.

By Lemma \ref{avoid_apex_new_0}, there exists $Z_0 \subseteq V(G)$ with $Z_0 \supseteq N_G^{\leq r}[Z^+]$ such that $Z_0$ is a $(\xi_0,\eta_0)$-centered in $G$ and intersects all members of $\F$ contained in $\bigcup_{(S,\Omega) \in \Se_2}S$.  

Let $\HH_0$ be the set consisting of the components of $L-Z_0$ containing a member of $\F-N_G^{\leq r}[Z_0]$.

\medskip

\noindent{\bf Claim 1:} $Z_0$ intersects all members of $\F \cap \bigcup_{(S,\Omega) \in \Se_2}S$ and all members of $(\F-N_G^{\leq r}[Z_0])-(\F \cap \bigcup_{H \in \HH_0}H)$.

\noindent{\bf Proof of Claim 1:}
$Z_0$ intersects all members of $\F \cap \bigcup_{(S,\Omega) \in \Se_2}S$ by the definition of $Z_0$.
If there exists $F \in (\F-N_G^{\leq r}[Z_0])-(\F \cap \bigcup_{H \in \HH_0}H)$ disjoint from $Z_0$, then $F$ is contained in a component of $L-Z_0$ (since $F$ is connected), so $F \in (\F-N_G^{\leq r}[Z_0]) \cap H$ for some $H \in \HH_0$, a contradiction.
$\Box$

\medskip

For every $H \in \HH_0$, let $Z_0^H = Z_0$.
Since $Z_0 \supseteq N_G^{\leq r}[Z^+] \supseteq N_G^{\leq r}[Z \cup N_G(V(L))]$, we know that $N_G^{\leq r}[V(H)] \cap Z = \emptyset$ and $N_G^{\leq r}[V(H)] \subseteq V(L)$ and $N_G(V(H)) \subseteq Z_0$ for every $H \in \HH_0$.

\medskip

\noindent{\bf Claim 2:} For any $H \in \HH_0$ and $(S,\Omega) \in \Se_1[V(H)]$, $S-N_G^{\leq r}[\overline{\Omega} \cup Z_0^H]$ does not contain a member of $\F$.

\noindent{\bf Proof of Claim 2:}
For every $(S,\Omega) \in \Se_1[V(H)]$, there exists $(S',\Omega') \in \Se_1$ such that $S \subseteq S'$ and $\overline{\Omega'} \cap V(H) \subseteq \overline{\Omega}$, so $S-N_G^{\leq r}[\overline{\Omega} \cup Z_0^H] \subseteq S'-N_G^{\leq r}[\overline{\Omega'}]$ (since $N_G(V(H)) \subseteq Z_0=Z_0^H$).

For every $(S',\Omega') \in \Se_1$, there exists a separation $(A,B)$ of $G-Z$ with $A=S'$ and $|V(A \cap B)| = |\overline{\Omega'}| \leq 3 < (\theta-|Z|)/2$; since $\Se$ is $(\T_G-Z)$-central, we know $(A,B) \in \T_G-Z$.
Since $\T_G$ is induced by the $(G,\F,r,\theta',Z^+)$-tangle $\T_L$, we know that for every $(S',\Omega') \in \Se_1$, $S' \cap L - N_G^{\leq r}[\overline{\Omega'} \cap V(L)]$ does not contain a member of $\F-N_G^{\leq r}[Z^+]$, so $S'-N_G^{\leq r}[\overline{\Omega'}]$ does not contain a member of $(\F \cap L)-N_G^{\leq r}[Z^+] = \F-N_G^{\leq r}[Z^+]$.

Therefore, for every $(S,\Omega) \in \Se_1[V(H)]$, $S-N_G^{\leq r}[\overline{\Omega} \cup Z_0^H]$ does not contain a member of $\F-N_G^{\leq r}[Z^+]$; since $Z_0^H \supseteq N_G^{\leq r}[Z^+]$, $S-N_G^{\leq r}[\overline{\Omega} \cup Z_0^H]$ does not contain a member of $\F$. 
$\Box$

\medskip

For every $H \in \HH_0$, let $H^+ = G[N_G^{\leq r}[V(H)]]$, so $H^+$ is connected and disjoint from $Z$, and hence there exists a proper arrangement of $\Se[V(H^+)]$ in a connected surface of Euler genus at most $g$.
For every $H \in \HH_0$, let $\Se^{H+}=\Se[V(H^+)]$.

Let $\HH$ be a set of pairwise disjoint connected induced subgraph of $L-Z_0$ such that for every $H \in \HH$, there exists $i_H \in [g] \cup \{0\}$ such that there exist $Z_H \subseteq V(G)$ with $Z_H \supseteq Z_0$, a segregation $\Se^{H+}$ of $H^+$ with a $(\kappa_{i_H},\rho)$-witness $(\Se^{H+}_1,\Se^{H+}_2)$, where $H^+ = G[N_G^{\leq r}[V(H)]]$, a proper arrangement $\alpha_H^+$ of $\Se^{H+}$ in a connected surface $\Sigma_{H}$ of Euler genus at most $g-i_H$, an $(H,\Se^{H+}_1,\Se^{H+}_2)$-adaption $\Se^H$ of $\Se^{H+}$ with a witness $(\Se^H_1,\Se^H_2)$, and an $\Se^{H}$-adjustment $\alpha_H$ of $\alpha_H^+$ such that 
	\begin{itemize}
		\item[(i)] $N_G^{\leq r}[\bigcup_{H' \in \HH}Z_{H'}]$ intersects all members of $\F - (\F \cap \bigcup_{H' \in \HH}H')$,
		\item[(ii)] for every $(S,\Omega) \in \Se^H$, 
			\begin{itemize}
				\item if $(S,\Omega) \in \Se^H_1$, then $S-N_G^{\leq r}[\overline{\Omega} \cup Z_H]$ does not contain a member of $\F \cap H$, 
				\item if $(S,\Omega) \in \Se^H_2$, then $S-N_G^{\leq r}[Z_H]$ does not contain a member of $\F \cap H$, 
			\end{itemize}
		\item[(iii)] $N_G(V(H)) \subseteq Z_H$, 
		\item[(iv)] $H$ contains a member of $\F-N_G^{\leq r}[Z_H]$, and
		\item[(v)] $Z_H$ is $(\xi_{2(k-1-\nu(H,Z_H))+i_H+\epsilon_H},\eta_{i_H})$-centered in $G$, where 
			\begin{itemize}
				\item $\nu(H,Z_H)$ is the maximum number of members of $\F-N_G^{\leq r}[Z_H]$ contained in $H$ pairwise at distance in $G$ at least $r$, and
				\item $\epsilon_H \in \{0,1\}$ such that if $\epsilon_H=1$, then $\Sigma_H$ is the sphere and the skeleton of $\alpha_H$ with respect to $(\Se_1^H,\Se_2^H)$ is a 2-cell drawing in $\Sigma_H$.
			\end{itemize}
	\end{itemize}
Note that $\HH$ exists since $\HH_0$ with $i_H=0$ and $Z_H=Z_0^H$ for every $H \in \HH_0$ is a candidate by Claims 1 and 2.

We choose $\HH$ such that $|V(\bigcup_{H \in \HH}H)|$ is minimum, and subject to this, the sum of the Euler genus of $\Sigma_H$ over all $H \in \HH$ is minimum.

We shall prove that $\HH$ and $(Z_H: H \in \HH)$ satisfy the conclusion of this lemma.
Note that (i) implies Statement 1 of this lemma, and (iii)-(v) implies Statement 2 of this lemma.

Suppose to the contrary that $\HH$ and $(Z_H: H \in \HH)$ do not satisfy the conclusion of this lemma.
So Statement 3 of this lemma is violated.
Hence there exists $H \in \HH$ showing that Statement 3 does not hold.
In particular,
	\begin{itemize}
		\item[(vi)] $W$ does not intersect all members of $\F \cap H$ for every $(\xi^*,\eta^*)$-centered set $W \subseteq V(G)$ with $W \supseteq N_G^{\leq r}[Z_H]$.
	\end{itemize}

For simplicity of notation, let $\xi_H = \xi_{2(k-1-\nu(H,Z_H))+i_H+\epsilon_H}$.
Note that $\xi_H \leq \xi_{2k-3+g}$ and $\eta_{i_H} \leq \eta_g$.

\medskip

\noindent{\bf Claim 3:} There exists a $(G,\F \cap H, r,\theta,Z_H)$-tangle $\T_{H}$ in $H$.

\noindent{\bf Proof of Claim 3:}
Suppose to the contrary that there does not exist a $(G,\F \cap H, r,\theta,Z_H)$-tangle in $H$.
Hence there does not exist a $(G,(\F \cap H)-N_G^{\leq r}[Z_H], r,\theta,Z_H)$-tangle in $H$.

We apply Lemma \ref{easy_tangle} (with $(G,L,\F,k,\theta,r,r',\xi,\eta,Z)=(G,H,(\F \cap H)-N_G^{\leq r}[Z_H], \allowbreak \nu(H,Z_H)+1, \theta,r,r,\xi_H,\eta_{i_H},Z_H)$).
The third outcome of Lemma \ref{easy_tangle} does not hold.
If there exists $Z' \subseteq V(G)$ with $Z' \supseteq N_G^{\leq r}[Z_H]$ such that $Z'$ is $(\xi_H+3\theta-3,\eta_{i_H}+r)$-centered in $G$ intersecting all members of $(\F \cap H)-N_G^{\leq r}[Z_H]$, then $Z'$ intersects all members of $\F \cap H$ (since $Z' \supseteq N_G^{\leq r}[Z_H]$), so it contradicts (vi) (since $\xi_H+3\theta-3 \leq \xi^*$ and $\eta_{i_H}+r \leq \eta^*$).
So the second outcome of Lemma \ref{easy_tangle} holds.

That is, there exists a separation $(A,B)$ of $H$ of order less than $\theta$ such that each of $A-N_G^{\leq r}[V(A \cap B)]$ and $B-N_G^{\leq r}[V(A \cap B)]$ does not contain $\nu(H,Z_H)$ members of $(\F \cap H)-N_G^{\leq r}[Z_H]$ pairwise at distance in $G$ at least $r$.

Let $Z_H' = Z_H \cup N_G^{\leq r}[V(A \cap B)]$.
Note that $Z_H'$ is $(\xi_H+\theta-1,\max\{\eta_{i_H},r\})$-centered in $G$ and hence is $(\xi_{2(k-1-\nu(H,Z_H))+i_H+\epsilon_H+1},\eta_{i_H})$-centered in $G$.

Let $\C$ be the set of components of $H-N_G^{\leq r}[V(A \cap B)]$ containing a member of $\F-N_G^{\leq r}[Z_H']$.
Note that $\C \neq \emptyset$, for otherwise $N_G^{\leq r}[Z_H']$ is a $(\xi_H+\theta-1,\eta_{i_H}+r)$-centered set intersecting all members of $\F \cap H$, contradicting (vi).

Let $\HH' = (\HH-\{H\}) \cup \C$.
Then $\HH'$ is a set of pairwise disjoint connected induced subgraphs of $L-Z_0$.
For every $C \in \C$, let $i_{C}=i_H, \epsilon_C=0$ and $Z_{C} = Z_H'$.

Clearly, $\HH'$ satisfies (iii) and (iv).
Since $\C \neq \emptyset$, $\bigcup_{H' \in \HH'}Z_{H'} \supseteq \bigcup_{H' \in \HH}Z_{H'}$.
Since $Z_H'$ intersects all members of $(\F \cap \bigcup_{H' \in \HH}H')-(\F \cap \bigcup_{H' \in \HH'}H')$, (i) implies that $N_G^{\leq r}[\bigcup_{H' \in \HH'}Z_{H'}]$ intersects all members of $\F - (\F \cap \bigcup_{H' \in \HH'}H')$.
Hence $\HH'$ satisfies (i).

Since each of $A-N_G^{\leq r}[V(A \cap B)]$ and $B-N_G^{\leq r}[V(A \cap B)]$ does not contain $\nu(H,Z_H)$ members of $(\F \cap H)-N_G^{\leq r}[Z_H]$ pairwise at distance in $G$ at least $r$, we know $\nu(C,Z_{C}) \leq \nu(H,Z_H)-1$ for every $C \in \C$, so $2(k-1-\nu(H,Z_H))+i_H+\epsilon_H+1 \leq 2(k-1-\nu(C,Z_C))+i_C+\epsilon_C$.
Hence $\HH'$ satisfies (v).

For every $C \in \C$, since $N_G^{\leq r}[V(C)] \subseteq H^+$, there exist a segregation $\Se^{C+}=\Se^{H+}[N_G^{\leq r}[V(C)]]$ of $G[N_G^{\leq r}[V(C)]]$ with a $(\kappa_{i_{C}},\rho)$-witness $(\Se^{C+}_1,\Se^{C+}_2)$, a proper arrangement of $\Se^{C+}$ in a connected surface $\Sigma_{H}$ of Euler genus at most $g-i_H=g-i_{C}$, and an $(C,\Se^{C+}_1,\Se^{C+}_2)$-adaption $\Se^{C}$ of $\Se^{C+}$ with a witness $(\Se^C_1,\Se^C_2)$ such that for every $(S',\Omega') \in \Se^{C}$, there exists $(S,\Omega) \in \Se^H$ such that $S' \subseteq S$, so $S'-N_G^{\leq r}[\overline{\Omega'} \cup Z_{C}] \subseteq S-N_G^{\leq r}[\overline{\Omega} \cup Z_{H}]$ does not contain a member of $\F \cap H \supseteq \F \cap C$ (by (ii)), and if $(S',\Omega') \in \Se^{C}_2$, then $(S,\Omega) \in \Se^H_2$ and $S'-N_G^{\leq r}[Z_{C}] \subseteq S-N_G^{\leq r}[Z_{H}]$ does not contain a member of $\F \cap H \supseteq \F \cap C$ (by (ii)).
So $\HH'$ satisfies (ii).

However, $V(A \cap B) \cap V(H) = V(A \cap B) \neq \emptyset$ (since $H$ is connected and $\nu(C,Z_C)<\nu(H,Z_H)$ for some $C \in \C$), so $|V(\bigcup_{H' \in \HH'}H')| < |V(\bigcup_{H' \in \HH}H')|$, contradicting the choice of $\HH$.
$\Box$

\medskip

\noindent{\bf Claim 4:} $\alpha_H$ is not a $(\Sigma_H,\theta^*,1,\T_H)$-arrangement of $\Se^H$ with respect to $(\Se_1^H,\Se_2^H)$.

\noindent{\bf Proof of Claim 4:}
Since $H$ witnesses that Statement 3 of this lemma does not hold, this claim follows from Claim 3.
$\Box$

\medskip

\noindent{\bf Claim 5:} $\Se^H$ is $\T_{H}$-central.

\noindent{\bf Proof of Claim 5:}
Suppose to the contrary that there exist $(A,B) \in \T_{H}$ and $(S,\Omega) \in \Se^H$ such that $B \subseteq S$.
Let $Y=\overline{\Omega}$ if $(S,\Omega) \in \Se^H_1$, and let $Y=\emptyset$ if $(S,\Omega) \in \Se^H_2$.
So (ii) implies that $B-N_G^{\leq r}[Y \cup Z_H]$ does not contain a member of $\F \cap H$.
Since $\T_{H}$ is the $(G,\F \cap H,r,\theta,Z_H)$-tangle, $A-N_G^{\leq r}[V(A \cap B)]$ does not contain a member of $(\F \cap H)-N_G^{\leq r}[Z_H]$.
Hence $N_G^{\leq r}[Y \cup Z_H \cup V(A \cap B)]$ intersects all members of $\F \cap H$.
By (vi), $N_G^{\leq r}[Y \cup Z_H \cup V(A \cap B)]$ is not $(\xi^*,\eta^*)$-centered in $G$.
Note that $N_G^{\leq r}[Y \cup Z_H \cup V(A \cap B)]$ is $(|Y|+\xi_H+\theta-1,\eta_{i_H}+r)$-centered in $G$.
It is a contradiction since $|Y|+\xi_H+\theta-1 \leq 3+\xi_H+\theta-1 \leq \xi^*$ and $\eta_{i_H}+r \leq \eta^*$.
$\Box$

\medskip

Let $G_H'$ be the skeleton of $\alpha_H$ with respect to $(\Se_1^H,\Se_2^H)$.
Note that $G_H'$ is a natural minor of $H$ since $\Se^H$ is effective.

\medskip

\noindent{\bf Claim 6:} If $\Sigma_H$ is not the sphere, then $G_H'$ is a $\theta_{i_H}$-representative 2-cell drawing in $\Sigma_H$.

\noindent{\bf Proof of Claim 6:}
Assume that $\Sigma_H$ is not the sphere.
So $i_H \leq g-1$ and $\epsilon_H=0$.

Suppose to the contrary that this claim does not hold.
Then there exists a non-contractible O-arc $O$ in $\Sigma_H$ showing that $G_H'$ is not 2-cell or $G_H'$ is not $\theta_{i_H}$-representative.
By the 3-path condition for non-contractible O-arcs, we can choose $O$ such that for every $(S,\Omega) \in \Se_2^H$, $O \cap \alpha_H(S,\Omega)$ is either empty or a line.
If $O$ witnesses that $G_H'$ is not 2-cell, then we slightly perturb $O$ so that for every $(S,\Omega) \in \Se_2^H$, $O \cap \alpha_H(S,\Omega)$ is either empty or a line that intersects $\alpha_H(S,\Omega) \cap V(G_H')$ exactly at its endpoints.
So $|O \cap V(G_H')| \leq \max\{2\kappa_{i_H},\theta_{i_H}\}$.

For every $(S,\Omega) \in \Se_2^H$, if $O$ intersects the interior of $\alpha_H(S,\Omega)$, then let $Q_S$ be the set of the ends of the line $O \cap \alpha_H(S,\Omega)$; otherwise, let $Q_S=\emptyset$.
Note that $|Q_S| \leq 2$ for each $(S,\Omega) \in \Se_2^H$.

For every $(S,\Omega) \in \Se_2^H$, let $(P^S,(X^S_t: t \in V(P^S)))$ be a vortical decomposition of $(S,\Omega) \in \Se_2^H$ of adhesion at most $\rho$; we may assume $V(P^S)=\overline{\Omega}$ and $t \in X^S_t$ for every $t \in V(P^S)$.
For any $(S,\Omega) \in \Se_2^H$ and $t \in V(P^S)$, let $Y_t = \{t\} \cup (X^S_t \cap \bigcup_{t' \in V(P^S)-\{t\}}X^S_{t'})$.
Let $Z_H' = (O \cap V(G_H')) \cup \bigcup_{(S,\Omega) \in \Se_2^H}\bigcup_{t \in Q_S}X^S_t$.

Then $Z_H' \subseteq V(H)$ is a set with $|Z_H'| \leq 2\kappa_{i_H}(2\rho+1)+\theta_{i_H}$ such that there exists a segregation $\M$ of $H-Z_H'$ with a $(2\kappa_{i_H},\rho)$-witness $(\M_1,\M_2)$ such that   
	\begin{itemize}
		\item $\M_1 \subseteq \Se^H_1[V(H)-Z_H']$,
		\item for every $(S,\Omega) \in \M_2$, there exists $(S',\Omega') \in \Se^H_2$ such that $S \subseteq S'-Z_H'$, and
		\item $\M$ has a proper arrangement in a surface of Euler genus at most $g-i_H-1$. 
	\end{itemize}

Let $Z_H^* = N_G^{\leq r}[Z_H \cup Z_H']$.
Note that $Z_H^*$ is $(\xi_H+2\kappa_{i_H}(2\rho+1)+\theta_{i_H},\eta_{i_H}+r)$-centered in $G$ and hence is $(\xi_{k-1-\nu(H,Z_H)+i_H+\epsilon_H+1},\eta_{i_H+1})$-centered in $G$.

Let $\C$ be the components of $H-Z_H^*$ containing a member of $\F-N_G^{\leq r}[Z_H^*]$.
Note that $\C \neq \emptyset$, for otherwise $N_G^{\leq r}[Z_H^*]$ is a $(\xi^*,\eta^*)$-centered set intersecting all members of $\F \cap H$, contradicting (vi).

Let $\HH' = (\HH-\{H\}) \cup \C$.
Then $\HH'$ is a set of pairwise disjoint connected induced subgraphs of $L-Z_0$.
For every $C \in \C$, let $i_{C}=i_H+1, \epsilon_C=\epsilon_H$ and $Z_{C} = Z_H^*$.

Then $\HH'$ satisfies (iii) and (iv).
Since $\C \neq \emptyset$, $\bigcup_{H' \in \HH'}Z_{H'} \supseteq \bigcup_{H' \in \HH}Z_{H'}$.
Since $Z_H^*$ intersects all members of $(\F \cap \bigcup_{H' \in \HH}H')-(\F \cap \bigcup_{H' \in \HH'}H')$, (i) implies that $N_G^{\leq r}[\bigcup_{H' \in \HH'}Z_{H'}]$ intersects all members of $\F - (\F \cap \bigcup_{H' \in \HH'}H')$.
Hence $\HH'$ satisfies (i).

For every $C \in \C$, since $C \subseteq H$ and $Z_C = Z_H^* \supseteq Z_H$, we know $\nu(C,Z_{C}) \leq \nu(H,Z_H)$, so $Z_{C}=Z_H^*$ is $(\xi_{k-1-\nu(C,Z_{C})+i_{C}+\epsilon_C},\eta_{i_C})$-centered in $G$ since $i_{C}=i_H+1$.
Hence $(Z_{H'}: H' \in \HH')$ satisfies (v).

For every $C \in \C$, let $C^+ = G[N_G^{\leq r}[V(C)]]$; since $Z_H^* \supseteq N_G^{\leq r}[Z_H] \supseteq N_G^{\leq r}[N_G(V(H))]$, we know that $C^+ \subseteq H-Z_H^*$ and $\kappa_{i_C} = \kappa_{i_H+1} \geq 2\kappa_{i_H}$, so $\M[V(C^+)]$ is a segregation of $C^+$ with a $(\kappa_{i_C},\rho)$-witness $(\Se^{C+}_1,\Se^{C+}_2)$ such that there exists a proper arrangement of $\M[V(C^+)]$ in a surface $\Sigma_{C}$ of Euler genus at most $g-i_H-1=g-i_C$. 
Since $C$ is connected, $\Sigma_C$ can be chosen to be connected.

For every $C \in \C$, let $\Se^{C+} = \M[V(C^+)]$, and let $\Se^{C}$ be an $(C,\Se^{C+}_1,\Se^{C+}_2)$-adaption of $\Se^{C+}$. 

For any $C \in \C$ and $(S,\Omega) \in \Se^{C}$, there exists $(S',\Omega') \in \Se^H$ such that $S \subseteq S'-Z_H^*$ and $\overline{\Omega'} \subseteq \overline{\Omega} \cup N_G^{\leq r}[Z_H^*]$, so $S-N_G^{\leq r}[\overline{\Omega} \cup Z_C] \subseteq S'-N_G^{\leq r}[\overline{\Omega'} \cup Z_H]$ does not contain a member of $\F \cap H \supseteq \F \cap C$ by (ii) (since $Z_H \cup N_G^{\leq r}[Z_H^*] \subseteq Z_C$), and if $(S,\Omega) \in \Se^C_2$, then $(S',\Omega') \in \Se^H_2$ and $S-N_G^{\leq r}[Z_C] \subseteq S'-N_G^{\leq r}[Z_H]$ does not contain a member of $\F \cap H \supseteq \F \cap C$ by (ii).
Hence $\HH'$ satisfies (ii).

However, $Z_H' \cap V(H) \neq \emptyset$, so $|V(\bigcup_{H' \in \HH'}H')| < |V(\bigcup_{H' \in \HH}H')|$, contradicting the choice of $\HH$.
$\Box$

\medskip

\noindent{\bf Claim 7:} $\Sigma_H$ is not the sphere, and there exists a respectful tangle $\T_{G_H'}$ in $G_H'$ of order $\theta_{i_H}$ toward $\Sigma_H$.

\noindent{\bf Proof of Claim 7:}
If $\Sigma_H$ is not the sphere, then by Claim 6, $G_H'$ is a $\theta_{i_H}$-representative 2-cell drawing, so $G_H'$ has a respectful tangle of order $\theta_{i_H}$ by Theorem \ref{representative_respectful}.
So we may assume that $\Sigma_H$ is the sphere.
By Claims 5, $\T_{H}$ has order at least $\theta \geq \theta_{i_H}$ and $\Se^H$ is a $\T_{H}$-central $(\kappa_{i_H},\rho)$-segregation.
So if $G_H'$ is a 2-cell drawing in the sphere, then by Lemma \ref{segregation_respectful_central}, there exists a respectful tangle $\T_{G_H'}$ in $G_H'$ of order $\theta_{i_H}$ toward $\Sigma_H$ conformal with $\T_{H}$. 

Hence we may assume that $G_H'$ is not a 2-cell drawing.
Then there exists a non-contractible O-arc $O$ in $\Sigma_H$ showing that $G_H'$ is not 2-cell.
We can choose $O$ such that for every $(S,\Omega) \in \Se_2^H$, $O \cap \alpha_H(S,\Omega)$ is either empty or a line.
For every $(S,\Omega) \in \Se_2^H$, if $O$ intersects the interior of $\alpha_H(S,\Omega)$, then let $Q_S$ be the set consiting of two consecutive vertices $v_{S,1},v_{S,2}$ in $\Omega$ and two consecutive vertices $v_{S,3},v_{S,4}$ in $\Omega$ such that the endpoints of the line $O \cap \alpha_H(S,\Omega)$ are contained in the union of the interval of $\Omega$ between $v_{S,1}$ and $v_{S,2}$ and the interval between $v_{S,3}$ and $v_{S,4}$; otherwise, let $Q_S=\emptyset$.

For every $(S,\Omega) \in \Se_2^H$, let $(P^S,(X^S_t: t \in V(P^S)))$ be a vortical decomposition of $(S,\Omega) \in \Se_2^H$ of adhesion at most $\rho$; we may assume $V(P^S)=\overline{\Omega}$ and $t \in X^S_t$ for every $t \in V(P^S)$.
For any $(S,\Omega) \in \Se_2^H$ with $Q_S \neq \emptyset$, let $Y_S = ((X^S_{v_{S,1}} \cap X^S_{v_{S,2}}) \cup (X^S_{v_{S,3}} \cap X^S_{v_{S,4}}))-\overline{\Omega}$. 
Let $Z_H' = \bigcup_{(S,\Omega) \in \Se_2^H, Q_S \neq \emptyset}Y_S$.
Then $Z_H' \subseteq V(H)$ is a set with $|Z_H'| \leq 2\kappa_{i_H}\rho$ such that $Z_H' \subseteq \bigcup_{(S,\Omega) \in \Se_2^H}(V(S)-\overline{\Omega})$.
Note that there exists a separation $(A_O,B_O)$ of $H$ with $A_O \cap B_O=Z_H'$ such that each of $A_O$ and $B_O$ consists of the subgraph of $H$ lieing at one side of $O$.
We can choose $O$ such that $(A_O,B_O) \in \T_H$, and subject to this, $B_O$ is minimal. 

Let $C = B_O-Z_H'$.
Let $\HH' = (\HH-\{H\}) \cup \{C\}$.
Let $Z_C = Z_H \cup Z_H'$.
Let $i_C=i_H$ and $\epsilon_C=1$.

Note that $\Se^H[V(C)]$ is a segregation of $C$ with a $(\kappa_{i_H},\rho)$-witness $(\Se^H_1[V(C)],\Se^H_2[V(C)])$ having a proper arrangement in the sphere such that the skeleton of $\Se^H[V(C)]$ with respect to $(\Se^H_1[V(C)],\Se^H_2[V(C)])$ is a component of $G_H'$ (by the minimality of $B_O$) and hence is a 2-cell drawing in the sphere.
Clearly, $\Se^{H+}[N_G^{\leq r}[V(C)]]$ is a segregation of $G[N_G^{\leq r}[V(C)]]$ with a $(\kappa_{i_H},\rho)$-witness $(\Se^{H+}_1[N_G^{\leq r}[V(C)]], \Se^{H+}_2[N_G^{\leq r}[V(C)]])$ and has a proper arrangement $\alpha^{C+}$ in the sphere (which has Euler genus at most $g-i_H=g-i_C$) such that $\Se[V(C)]$ is a $(C,\Se^{H+}_1[N_G^{\leq r}[V(C)]], \Se^{H+}_2[N_G^{\leq r}[V(C)]])$-adaption of $\Se^{H+}[N_G^{\leq r}[V(C)]]$ with witness $(\Se_1^H[V(C)], \allowbreak \Se_2^H[V(C)])$ such that the restriction of $\alpha_H$ on $C$ is an adjustment of $\alpha^{C^+}$.

Since $(A_O,B_O) \in \T_H$, $A_O-N_G^{\leq r}[V(A_O \cap B_O)] = A_O-N_G^{\leq r}[Z_H']$ does not contain any member of $\F-N_G^{\leq r}[Z_H]$.
So $\HH'$ and $(Z_{H'}: H' \in \HH')$ satisfy (i).
Moreover, if $C$ does not contain a member of $\F-N_G^{\leq r}[Z_C]$, then $N_G^{\leq r}[Z_H \cup Z_C]$ is a $(\xi_H+\theta,\eta_{i_H}+r)$-centered set in $G$ intersecting all members of $\F \cap H$, contradicting (vi).
So $\HH'$ satisfies (iv).

For every $(S,\Omega) \in \Se^H[V(C)]$, there exists $(S',\Omega') \in \Se^H$ such that $S \subseteq S'-Z_C$ and $\overline{\Omega'} \subseteq \overline{\Omega} \cup N_G^{\leq r}[Z_C]$, so $S-N_G^{\leq r}[\overline{\Omega} \cup Z_C] \subseteq S'-N_G^{\leq r}[\overline{\Omega'} \cup Z_H]$ does not contain a member of $\F \cap H \supseteq \F \cap C$ by (ii), and if $(S,\Omega) \in \Se^C_2$, then $(S',\Omega') \in \Se^H_2$ and $S-N_G^{\leq r}[Z_C] \subseteq S'-N_G^{\leq r}[Z_H]$ does not contain a member of $\F \cap H \supseteq \F \cap C$ by (ii).
Hence $\HH'$ satisfies (ii).

Clearly, $\HH'$ satisfies (iii).

Note that $Z_C = Z_H \cup Z_{H}'$ is $(\xi_H+2\kappa_{i_H}\rho, \eta_{i_H})$-centered in $G$.
Since $G_H'$ is not 2-cell, $\epsilon_H=0$ by (v).
Since $\epsilon_C=1$, $\xi_{2(k-1-\nu(H,Z_{H}))+i_{H}+\epsilon_H} + 2\kappa_{i_H}\rho \leq \xi_{2(k-1-\nu(C,Z_{C}))+i_{C}+\epsilon_C}$.
Hence $(Z_{H'}: H' \in \HH')$ satisfies (v).
$\Box$

\medskip

Let $\T^-_{G_H'} = \{(A,B) \in \T_{G_H'}: |V(A \cap B)|<\theta^*\}$.
Note that $\T^-_{G_H'}$ is a tangle of order $\theta^*$ in $G_H'$ since $\theta_{i_H} \geq \theta^*$.

By Claims 4-7, $\T_{G_H'}^-$ is not conformal with $\T_{H}$.

We fix a natural $G_H'$-minor $\mu$ in $H$.
For every separation $(A,B)$ of $H$, let $(A,B)_{G_H'}$ be the separation $(C,D)$ of $G_H'$ such that $\mu(E(C))=E(A) \cap \mu(E(G_H'))$, and we denote $(A,B)_{G_H'}$ by $(A_{G_H'},B_{G_H'})$. 
Let $\T'$ be the tangle of order $\theta^*$ in $H$ induced by $\T^-_{G_H'}$.
Then for every separation $(A,B)$ of order less than $\theta^*$ of $H$, $(A,B) \in \T'$ if and only if $(A,B)_{G_H'} \in \T_{G_H'}$.

Since $\T^-_{G_H'}$ is not conformal with $\T_{H}$, $\T' \not \subseteq \T_{H}$.

\medskip

\noindent{\bf Claim 8:} There exists a separation $(A^*,B^*)$ of $H$ of order less than $\theta^*$ with $(A^*,B^*) \in \T'-\T_{H}$ such that $A^*-V(B^*)$ is connected and $V(A^* \cap B^*) = N_H(V(A^*)-V(B^*))$.

\noindent{\bf Proof of Claim 8:}
Since $\T' \not \subseteq \T_{H}$, there exists $(A,B) \in \T'-\T_{H}$.
We choose $(A,B)$ such that $A$ is minimal.

Suppose to the contrary that $A-V(B)$ is not connected, or $A-V(B)$ is connected but $V(A \cap B) \supset N_H(V(A)-V(B))$.
Since $(A,B) \in \T'-\T_H$, we know $(B,A) \in \T_H$.
Since $\T_H$ is a $(G,\F \cap H,r,\theta,Z_H)$-tangle in $H$, $A-N_G^{\leq r}[V(A \cap B)]$ contains a member of $(\F \cap H)-N_G^{\leq r}[Z_H]$.
Since every member of $\F$ is connected, there exists a component $Q$ of $A-V(B)$ such that $Q-N_G^{\leq r}[V(A \cap B)]$ contains a member of $(\F \cap H)-N_G^{\leq r}[Z_H]$.
Let $(A',B')$ be a separation of $H$ such that $V(A') = N_H[V(Q)]$ and $V(A' \cap B') = N_H(V(Q))$. 
Since $\T_H$ is a $(G,\F \cap H,r,\theta,Z_H)$-tangle in $H$, and $A'-N_G^{\leq r}[V(A \cap B)]$ contains a member of $(\F \cap H)-N_G^{\leq r}[Z_H]$, we know $(B',A') \in \T_H$ and hence $(A',B') \not \in \T_H$.
Since $A' \subseteq A$ and $(A,B) \in \T'$, we know $(A',B') \in \T'$.
So $(A',B') \in \T'-\T_H$, contradicting the minimality of $A$.
Hence this claim follows from choosing $(A^*,B^*)=(A,B)$.
$\Box$

\medskip

\noindent{\bf Claim 9:} There exists a $2(\kappa_{i_H}+2)(\theta^*+7)^2$-zone $\Lambda$ around a vertex $v^*$ in $G_H'$ (with respect to $m_{\T_{G_H'}}$) such that $\bigcup_{(S,\Omega) \in \Se^H,\overline{\Omega} \cap V(A^*_{G_H'}) \neq \emptyset}\alpha(S,\Omega) \subseteq \Lambda$, and for every $(S,\Omega) \in \Se^H_2$, either $\overline{\Lambda} \supseteq \alpha(S,\Omega)$ or $\overline{\Lambda} \cap \alpha(S,\Omega)=\emptyset$.

\noindent{\bf Proof of Claim 9:}
We first show that there exists a $(2(\kappa_{i_H}+2)((\theta^*)^2+2)+5)$-zone $\Lambda_1$ in $G_H'$ (with respect to $m_{\T_{G_H'}}$) such that $\bigcup_{(S,\Omega) \in \Se^H,\overline{\Omega} \cap V(A^*_{G_H'}) \neq \emptyset}\alpha(S,\Omega) \subseteq \Lambda_1$.
For any component $C$ of $A^*_{G_H'}-V(A^*_{G_H'} \cap B^*_{G_H'})$ and $v \in V(C)$, Lemma \ref{A distance} implies that $m_{\T_{G_H'}}(v,x) \leq |V(A^*_{G_H'} \cap B^*_{G_H'})| \leq |V(A^* \cap B^*)| \leq \theta^*-1$ for every $x \in N_H[V(C)]$, so $m_{\T_{G_H'}}(y,z) \leq 2(\theta^*-1)$ for any $y,z \in N_H[V(C)]$.
Hence, if $A^*_{G_H'}$ is connected, then since $|V(A^*_{G_H'} \cap B^*_{G_H'})| \leq \theta^*-1$, we know $m_{\T_{G_H'}}(y,z) \leq 2(\theta^*-1)^2+4(\theta^*-1) \leq 2(\theta^*)^2$ for any $y,z \in A^*_{G_H'}$.
So we may assume that $A^*_{G_H'}$ is not connected.
Since $A^*$ is connected, for every component $C$ of $A^*_{G_H'}$, there exists $(S_C,\Omega_C) \in \Se^H_2$ such that $V(A^*_{G_H'}) \cap \overline{\Omega_C} \neq \emptyset$; moreover, for any distinct $(S_C,\Omega_{C})$ and $(S_{C'},\Omega_{C'})$, there exist distinct components $C_1,C_2,...,C_p$ of $A^*_{G_H'}$ for some $p \leq |\Se_2^H| \leq \kappa_{i_H}$ such that $(S_C,\Omega_C)=(S_{C_1},\Omega_{C_1})$, $(S_{C'},\Omega_{C'})=(S_{C_p},\Omega_{C_p})$, and for every $i \in [p-1]$, $V(C_i)$ intersects both $\overline{\Omega_{C_i}}$ and $\overline{\Omega_{C_{i+1}}}$.
Hence there exists $v \in \overline{\Omega_C}$ for some component $C$ of $A^*_{G_H'}$ such that $m_{\T_{G_H'}}(v,x) \leq 2(\theta^*)^2+(2+2(\theta^*)^2)\kappa_{i_H}+2(\theta^*)^2 \leq 2(\kappa_{i_H}+2)((\theta^*)^2+2)$ for every $x \in A^*_{G_H'}$.
So there exists a $(2(\kappa_i+2)((\theta^*)^2+2)+5)$-zone in $G_i'$ containing $\alpha(S,\Omega)$ for every $(S,\Omega) \in \Se^H$ with $\overline{\Omega} \cap V(A^*_{G_H'}) \neq \emptyset$ by Lemma \ref{big zone contains ball}.

This shows the existence of $\Lambda_1$.
Let $v^*$ be a vertex such that $\Lambda_1$ is around $v^*$.
If $\overline{\Lambda_1} \cap \alpha(S,\Omega) \neq \emptyset$ for some $(S,\Omega) \in \Se^H_2$, then $m_{\T_{G_H'}}(v,x) \leq (2(\kappa_i+2)((\theta^*)^2+2) + 7$ for every atom $x$ intersecting $\alpha(S,\Omega)$, so there exists a $(2(\kappa_{i_H}+2)((\theta^*)^2+2) + 10)$-zone $\Lambda_2$ around $v^*$ containing all atoms $x$ intersecting $\alpha(S,\Omega)$ by Lemma \ref{big zone contains ball}.
Hence by repeating this process at most $|\Se^H_2| \leq \kappa_{i_H}$ times, we obtain a $(2(\kappa_{i_H}+2)((\theta^*)^2+2) + 10\kappa_{i_H})$-zone $\Lambda$ around $v^*$ such that $\bigcup_{(S,\Omega) \in \Se^H,\overline{\Omega} \cap V(A^*_{G_H'}) \neq \emptyset}\alpha(S,\Omega) \subseteq \Lambda$, and for every $(S,\Omega) \in \Se^H_2$, either $\overline{\Lambda} \supseteq \alpha(S,\Omega)$ or $\overline{\Lambda} \cap \alpha(S,\Omega)=\emptyset$.  
Note that $2(\kappa_{i_H}+2)((\theta^*)^2+2) + 10\kappa_{i_H} \leq 2(\kappa_{i_H}+2)(\theta^*+7)^2$.
$\Box$

\medskip

Let $Z_{i+1} = N_G^{\leq r}[V(A^* \cap B^*) \cup Z_H]$.

Since $\Sigma_H$ is not the sphere by Claim 7, $i_H \leq g-1$ and $\epsilon_H=0$.
Let $\C$ be the set consisting of the components of $A^*-N_G^{\leq r}[V(A^* \cap B^*) \cup Z_H]$ containing a member of $\F-N_G^{\leq r}[Z_{i+1}]$.
Since $(B^*,A^*) \in \T_H$, we know that $\C \neq \emptyset$ by (vi).

For every $C \in \C$, let $Z_C=Z_{i+1}$ and let $C^+=G[N_G^{\leq r}[V(C)]]$; for $j \in [2]$, let $\Se^{C+}=\Se^{i+}[V(C^+)]$ and $\Se^{i+1+}_j=\Se^{i+}_j[V(C^+)]$ for $j \in [2]$; so $(\Se^{C+}_1,\Se^{C+}_2)$ is a $(\kappa_{i_H},\rho)$-witness (and hence a $(\kappa_{i_H+1},\rho)$-witness) of $\Se^{C+}$.
For every $C \in \C$, let $\Se^{C}$ be a $(C,\Se^{C+}_1,\Se^{C+}_2)$-adaption of $\Se^{C+}$ with witness $(\Se^{C}_1,\Se^{C}_2)$.

Let $M$ be the subgraph of $H$ induced by $\bigcup_{(S,\Omega) \in \Se^H,\alpha(S,\Omega) \subseteq \overline{\Lambda}}V(S)$.
By Claims 8 and 9, $\bigcup_{C \in \C}C \subseteq A^* \subseteq M$.

Let $\M = \Se^H[V(M)]$.
Let $M^+ = N_G^{\leq r}[V(M)]$.
Let $\M^+ = \Se^{H+}[V(M^+)]$.
For every $j \in [2]$, let $\M_j = \Se^H_j[V(M)]$ and $\M^+_j = \Se^{H+}_j[V(M^+)]$.

Note that there exist a proper arrangement $\beta^+$ of $\M^+$ in $\Sigma_H$ and a proper arrangement $\beta$ of $\M$ in $\Sigma_H$ by slightly shirking the disks in the images of $\alpha_H^+$ and $\alpha_H$, respectively.
Since $\Lambda$ bounds a disk, the resulting proper arrangement of $\M$ is actually in the sphere.

\medskip

\noindent{\bf Claim 10:} $\beta^+$ is not a proper arrangement in the sphere.

\noindent{\bf Proof of Claim 10:}
Suppose to the contrary that $\beta^+$ is a proper arrangement in the sphere.
Then for every $C \in \C$, there exists a proper arrangement $\alpha_{C}^+$ of $\Se^{C+}$ in the sphere obtained by slightly shrinking the disks in the image of $\beta^+$.

Let $\HH' = (\HH-\{H\}) \cup \C$.
For every $C \in \C$, let $i_C=i_H+1$ and $\epsilon_C=\epsilon_H$, so $\alpha_C^+$ is a proper arrangement of $\Se^{C+}$ in the sphere, which is a connected surface of Euler genus at most $g-i_H-1=g-i_C$.

By Claim 8, $(B^*,A^*) \in \T_H$, so $B^* - N_G^{\leq r}[V(A^* \cap B^*)]$ does not contain any member of $(\F \cap H)-N_G^{\leq r}[Z_H]$.
So $Z_{i+1} = N_G^{\leq r}[V(A^* \cap B^*) \cup Z_H]$ intersects all members of $(\F \cap H) - (\F \cap \bigcup_{C \in \C}C)$.
Hence by (i), $N_G^{\leq r}[\bigcup_{H' \in \HH'}Z_{H'}]$ intersects all members of $\F-(\F \cap \bigcup_{H' \in \HH'}Z_{H'})$.
So $\HH'$ and $(Z_{H'}: H' \in \HH')$ satisfies (i).

For every $C \in \C$, $N_G(V(C)) \subseteq (N_G(V(C)) \cap V(H)) \cup (N_G(V(C)) - V(H)) \subseteq N_G^{\leq r}[V(A^* \cap B^*) \cup Z_H] \cup N_G(V(H)) \subseteq N_G^{\leq r}[V(A^* \cap B^*) \cup Z_H]$, where the last inclusion follows from (iii), so $N_G(V(C)) \subseteq Z_{i+1}=Z_C$.
Hence $\HH'$ satisfies (iii).

The definition of $\C$ implies that $\HH'$ satisfies (iv).

Moreover, $Z_{i+1} = N_G^{\leq r}[V(A^* \cap B^*) \cup Z_H]$ is $(\xi_H+\theta^*,\eta_{i_H}+r)$-centered in $G$ and hence $(\xi_{2(k-1-\nu(H,Z_H))+i_H+\epsilon_H+1},\eta_{i_H+1})$-centered in $G$.
So for every $C \in \C$, $Z_C=Z_{i+1}$ is \linebreak $(\xi_{2(k-1-\nu(C,Z_C))+i_C+\epsilon_C},\eta_{i_C})$-centered in $G$ since $\nu(C,Z_C) \leq \nu(H,Z_H)$.
Hence $\HH'$ satisfies (v).

For any $C \in \C$ and $(S,\Omega) \in \Se^{C}$, there exists $(S',\Omega') \in \Se^H$ such that $S - N_G^{\leq r}[\overline{\Omega} \cup N_G^{\leq r}[V(A^* \cap B^*) \cup Z_H]] \subseteq S' - N_G^{\leq r}[\overline{\Omega'} \cup Z_H]$ does not contain any member of $\F \cap H \supseteq \F \cap C$ by (ii); if $(S',\Omega') \in \Se^H_2$, then $S - N_G^{\leq r}[\overline{\Omega}] \subseteq S' - N_G^{\leq r}[\overline{\Omega'}]$ does not contain any member of $\F \cap H \supseteq \F \cap C$ by (ii).
So $\HH'$ and $(Z_{H'}: H' \in \HH')$ satisfy (ii).

Since $H$ is connected, $V(A^* \cap B^*) \cap V(H)=V(A^* \cap B^*) \neq \emptyset$.
So $|V(\bigcup_{H' \in \HH'}H')| < |V(\bigcup_{H' \in \HH}H')|$, contradicting the choice of $\HH$.
$\Box$

\medskip

Since $\beta^+$ is not a proper arrangement in the sphere by Claim 10, there exists a non-null homotopic O-arc $O_1$ in $\bigcup_{(S,\Omega) \in \M^+}\beta^+(S,\Omega) \subseteq \Sigma_{H}$ such that $O_1 \cap \partial\beta^+(S',\Omega') \subseteq \beta^+(\overline{\Omega'})$ for every $(S',\Omega') \in \M^+$.
We can choose $O_1$ such that for every $(S,\Omega) \in \M^+$, every connected component of $O_1 \cap \beta^+(S,\Omega)$ is $G_H'$-normal.
Since $\Se^H$ is an $(H,\Se^{H+}_1,\Se^{H+}_2)$-adaption of $\Se^{H+}$, for every $(S,\Omega) \in \M^+ = \Se^{H+}[V(M^+)]$, every connected component of $O_1 \cap \beta^+(S,\Omega)$ is $G_H'$-normal and contains at most 12 vertices by Lemma \ref{adaption_size}.
Note that for every vertex $v$ in $V(\M^+)-V(G_H')$, there exists a region of $G_i'$ containing $\beta^+(v)$.
So there exists a closed walk $O_2$ in the radial drawing of $G_H'$ such that some cycle in $O_2$ is non-null homotopic and there exists $R \subseteq V(O_2)$ such that every vertex in $R$ either is in $O_1 \cap V(G_H')$ or corresponds to a region of $G_H'$ containing a vertex in $O_1 \cap V(\M^+)-V(G_H')$, and every subwalk of $O_2$ between two vertices in $R$ internally disjoint from $R$ has length at most 24. 

Let $K$ be the radial drawing of $G_H'$, and let $\lambda = 2(\kappa_{i_H}+2)(\theta^*+7)^2$.
For every vertex $u$ in $R$, there exists a path in $G$ from a vertex in $V(M)$ to either $u$ or the vertex in $O_1 \cap V(\M^+)-V(G_H')$ corresponding to $u$ with length at most $r$, so there exists a path $I_u$ in $K$ from $u$ to some vertex $u_M \in V(M) \cap V(G_H')$ of length at most $2r+1$. 
For every $u \in R$, Claim 9 implies that there exists a closed walk $W^-_u$ in $K$ of length at most $\lambda$ such that $\{u_M,v^*\} \subseteq \ins(W^-_{u})$; so there exists a closed walk $W_{u}$ in $K$ containing $W$ with length at most $\lambda+2(2r+1)<\theta_{i_H}$ such that $\{u,v^*\} \subseteq \ins(W_{u})$. 
For any $u \in V(O_2)$, since every subwalk of $O_2$ between two vertices in $R$ internally disjoint from $R$ has length at most 24, we know there exists a closed walk $W_{u}$ in $K$ with length at most $\lambda+2(2r+1)+48<\theta_{i_H}$ such that $\{u,v^*\} \subseteq \ins(W_{u})$. 
Therefore, since $\lambda+2(2r+1)+51 < \theta_{i_H}$, Lemma \ref{big zone contains ball} implies that $O_2$ is contained in a $(\lambda+2(2r+1)+51)$-zone $\Lambda^*$ around $v^*$ (with respect to $m_{\T_{G_H'}}$).
Since $\Lambda^*$ is a disk, $O_2$ cannot contain a non-null homotopic cycle, a contradiction.
This proves the lemma.
\end{pf}

\bigskip

Now we strengthen the last outcome of Lemma \ref{avoid_apex_new} to make the $(\Sigma^*,\theta^*,1,\T_{L^*})$-arrangement in the last outcome a $(\Sigma^*,\theta^*,\phi,\T_{L^*})$-arrangement for a given function $\phi$ but with the price that increases the depth of the vortices and losses the connection between the segregation of $H$ and the segregation $G[N_G^{\leq r}[V(H)]]$ while still having a protected arrangement.

\begin{lemma} \label{avoid_apex_new_2}
For any $g,\kappa \in {\mathbb N}_0$, there exists $\kappa^* \in {\mathbb N}_0$ such that for any $\rho \in {\mathbb N}_0$ and nondecreasing function $\phi^*: {\mathbb Z} \rightarrow {\mathbb R}$, there exists $\rho^* \in {\mathbb N}$ such that for any $k,r,\xi_{-1},\theta^* \in {\mathbb N}$, there exists $\theta \in {\mathbb N}$ such that for any $\xi,\eta \in {\mathbb N}$, there exist $\xi^*,\eta^* \in {\mathbb N}$ with $\eta^*=\eta+(g+2)r$ such that the following holds.

Let $G$ be a graph, and let $L$ be an induced subgraph of $G$.
Let $\F$ be a set of connected subgraphs of $L$ such that no $k$ members of $\F$ have pairwise distance in $G$ at least $r$.
	Let $Z \subseteq V(G)$ with $|Z| \leq \xi_{-1}$, and let $Z^+ \subseteq V(G)$ be a $(\xi,\eta)$-centered set in $G$ with $Z^+ \supseteq Z \cup N_G(V(L))$.
Let $\T_L$ be the $(G,\F,r,\theta',Z^+)$-tangle in $L$ for some integer $\theta' \geq \theta$, and let $\T_G$ be a tangle of order at least $\theta$ in $G$ induced by $\T_L$.
Let $\Se$ be a $(\T_G-Z)$-central segregation of $G-Z$ with a $(\kappa,\rho)$-witness $(\Se_1,\Se_2)$.
If there exists a proper arrangement of $\Se$ in $\Sigma$ with respect to $(\Se_1,\Se_2)$ for some surface $\Sigma$ of Euler genus at most $g$, then there exist a (possibly empty) set $\HH$ of pairwise disjoint induced subgraphs of $L-N_G^{\leq r}[Z^+]$ and a collection $(Z_H: H \in \HH)$ of subsets of $V(G)$ such that
	\begin{enumerate}
		\item $N_G^{\leq r}[\bigcup_{H \in \HH}Z_H]$ intersects all members of $\F-(\F \cap \bigcup_{H \in \HH}H)$,
		\item for every $H \in \HH$, we have $Z_H \supseteq N_G(V(H))$, $Z_H$ is $(\xi^*,\eta^*)$-centered in $G$, and $(\F \cap H)-N_G^{\leq r}[Z_H] \neq \emptyset$, and
		\item for every $H \in \HH$, either 
			\begin{enumerate}
				\item there exists $Z_H^* \subseteq V(G)$ with $Z_H^* \supseteq N_G^{\leq r}[Z_H]$ such that $Z_H^*$ is $(\xi^*,\eta^*)$-centered in $G$ and intersects all members of $\F \cap H$, or
				\item there exist a $(G,\F \cap H,r,\theta,Z_H)$-tangle $\T_H$ in $H$, a segregation $\Se^H$ of $H$ with a $(\kappa^*,\rho^*)$-witness $(\Se^H_1,\Se^H_2)$, and an $(8,r,\Se_1^H,\Se_2^H)$-protected $(\Sigma_H,\theta^*,\phi^*,\T_H)$-arrangement of $\Se^H$ with respect to $(\Se^H_1,\Se^H_2)$ for some surface $\Sigma_H$ of Euler genus at most $g$.
			\end{enumerate}
	\end{enumerate}
\end{lemma}

\begin{pf}
Let $g,\kappa \in {\mathbb N}_0$.
Define $\kappa^*$ to be the nonnegative integer $\kappa^*$ mentioned in Lemma \ref{avoid_apex_new} by taking $(g,\kappa)=(g,\kappa)$.
Let $\rho \in {\mathbb N}_0$, and let $\phi^*: {\mathbb Z} \rightarrow {\mathbb R}$ be a nondecreasing function.
Define $\rho^*$ to be the integer $\rho^*$ mentioned in Lemma \ref{sweeping_into_vortices_protected} by taking $(\kappa,k,\rho,\lambda,\phi)=(\kappa^*,0,\max\{\rho,1\},1,\phi^*)$.
Let $k,r,\xi_{-1},\theta^* \in {\mathbb N}$.
Let $\theta_1$ to be the integer $\theta$ mentioned in Lemma \ref{sweeping_into_vortices_protected} by taking $(\kappa,k,\rho,\lambda,\phi,\theta^*)=(\kappa^*,0,\max\{\rho,1\},1,\phi^*,\theta^*)$.
Define $\theta$ be the integer $\theta$ mentioned in Lemma \ref{avoid_apex_new} by taking $(g,\kappa,k,r,\xi_{-1},\theta^*,\rho)=(g,\kappa,k,r,\xi_{-1},\theta_1,\rho)$.
Let $\xi,\eta \in {\mathbb N}$.
Define $\xi^*,\eta^*$ to be the integers $\xi^*,\eta^*$, respectively, mentioned in Lemma \ref{avoid_apex_new} by taking $(g,\kappa,k,r,\xi_{-1},\theta^*,\rho,\xi,\eta)=(g,\kappa,k,r,\xi_{-1},\theta_1,\rho,\xi,\eta)$.

Let $G,L,\F,Z,Z^+,\T_L,\T_G,\Se,\Se_1,\Se_2$ be as stated in the lemma.
By Lemma \ref{avoid_apex_new}, there exist a set $\HH$ of pairwise disjoint induced subgraphs of $L-N_G^{\leq r}[Z^+]$ and a collection $(Z_H: H \in \HH)$ of subsets of $V(G)$ such that Statements 1 and 2 of this lemma holds, and for every $H \in \HH$, either Statement 3(a) of this lemma holds, or there exist a $(G,\F \cap H,r,\theta,Z_H)$-tangle $\T_H$ in $H$, a segregation $\Se^{H+}$ of $G[N_G^{\leq r}[V(H)]]$ with a $(\kappa^*,\rho)$-witness $(\Se^{H+}_1,\Se^{H+}_2)$, a proper arrangement $\alpha^{H+}$ of $\Se^{H+}$ in a surface $\Sigma_H$ of Euler genus at most $g$, an $(H,\Se_1^{H+},\Se_2^{H+})$-adaption $\Se^0$ of $\Se^{H+}$ with witness $(\Se_1^0,\Se_2^0)$, and an $\Se^0$-adjustment $\alpha^0$ of $\alpha^{H+}$ such that $\alpha^0$ is a $(\Sigma_H,\theta_1,1,\T_H)$-arrangement of $\Se^0$ with respect to $(\Se^0_1,\Se^0_2)$ in $\Sigma_H$.
Let $H \in \HH$ such that Statement 3(a) does not hold.
Then Lemma \ref{adjustment_co_distance} implies that $\alpha^0$ is $(8,r,\Se_1^0,\Se^0_2)$-protected.
By Lemma \ref{sweeping_into_vortices_protected} (taking $k=0$), there exists a $\T_H$-central segregation $\Se^H$ of $H$ with a $(\kappa^*,\rho^*)$-witness $(\Se^H_1,\Se^H_2)$ such that $\Se^H$ has an $(8,r,\Se_1^H,\Se^H_2)$-protected $(\Sigma_H,\theta^*,\phi^*,\T_H)$-arrangement with respect to $(\Se^H_1,\Se^H_2)$.
This proves the lemma.
\end{pf}

\section{Roots in one vortex} \label{sec:one_vortex}

In this section, we deal with $X$-$Y$ paths where at least one end is in a fixed vortex. 
The main result roughly states that, assuming that there do not exist many pairwise far apart $X$-$Y$ paths in $G$, there exists a centered set hitting all $X$-$Y$ paths with both ends in a fixed vortex, and if we cannot hit all $X$-$Y$ paths with exactly one end in this vortex, then there is certain structure that will be helpful to construct $X$-$Y$ paths in the whole graph in the future.
We will prove this main result in Section \ref{subsec:one_vortex_main}.
For preparation, we introduce simple notions that relate paths in $G$ and paths in the skeleton of an arrangement of its segregation in Section \ref{subsec:projection_liftings} and state a known result in the literature in Section \ref{subsec:nested_neighborhoods}.

\subsection{Natural projections and liftings} \label{subsec:projection_liftings}

Let $\Se$ be a segregation of a graph $G$ effective with respect to a $(\kappa,\rho)$-witness $(\Se_1,\Se_2)$ for some $\kappa,\rho \in {\mathbb N}_0$.
Let $\alpha$ be a proper arrangement of $\Se$ in a surface $\Sigma$.
Let $G'$ be the skeleton of $\alpha$ with respect to $(\Se_1,\Se_2)$.
For a path $P$ in $G$, the \defn{natural projection of $P$ in $G'$} is the subgraph $P'$ of $G'$ obtained from $P-\bigcup_{(S,\Omega) \in \Se_2}E(S)$ by, for each $(S,\Omega) \in \Se_1$, replacing each maximal subpath $Q$ of $P \cap S$ internally disjoint from $\overline{\Omega}$ by $G'[V(Q) \cap \overline{\Omega}] \cap \partial\alpha(S,\Omega)$; note that $P'$ is a union of disjoint paths.
On the other hand, given a path $Q'$ in $G'$, a \defn{lifting} of $Q'$ in $G$ is a path in $G-\bigcup_{(S,\Omega) \in \Se_2}E(S)$ obtained from $Q'$ by replacing each edge $e$ of $Q'$ by a path in $S_e$ between the ends of $e$, where $(S_e,\Omega_e)$ is the member of $\Se_1$ such that $e \subseteq \partial\alpha(S_e,\Omega_e)$.
Since $\Se$ is effective with respect to $(\Se_1,\Se_2)$, every path in $G'$ has a lifting in $G$.

\subsection{Nested neighborhoods} \label{subsec:nested_neighborhoods}

We say that $(S,\Omega, \Omega_0)$ is a \defn{neighborhood} if $S$ is a graph and $\Omega, \Omega_0$ are cyclic permutations with $\overline{\Omega}, \overline{\Omega_0} \subseteq V(S)$.
A neighborhood $(S,\Omega, \Omega_0)$ is \defn{rural} if $S$ has a drawing $\Gamma$ in the plane and there are disks $\Delta_0 \subseteq \Delta$ such that 
	\begin{itemize}
		\item $\Gamma$ uses no point outside $\Delta$ and no point in the interior of $\Delta_0$,  
		\item $\overline{\Omega}$ are the vertices in $\Gamma \cap \partial\Delta$, and $\overline{\Omega_0}$ are the vertices in $\Gamma \cap \Delta_0$, and 
		\item the cyclic permutations of $\overline{\Omega}$ and $\overline{\Omega_0}$ coincide with the natural cyclic orders on $\partial\Delta$ and $\partial\Delta_0$, respectively.
	\end{itemize}
In this case, we say that $(\Gamma, \Delta, \Delta_0)$ is a \defn{presentation} of $(S,\Omega, \Omega_0)$.
For a fixed presentation $(\Gamma, \Delta, \Delta_0)$ of a neighborhood $(S,\Omega, \Omega_0)$ and an integer $s \geq 0$, an \defn{$s$-nest} for $(\Gamma, \Delta, \Delta_0)$ is a sequence $(C_1, C_2, ..., C_s)$ of pairwise disjoint cycles of $S$ such that $\Delta_0 \subseteq \Delta_1 \subseteq ... \subseteq \Delta_s \subseteq \Delta$ and $\overline{\Delta_0} \cap \partial\Delta_1=\emptyset$, where $\Delta_i$ is the closed disk in the plane bounded by $C_i$ in the drawing $\Gamma$ for each $i \in [s]$.

If $(S,\Omega, \Omega_0)$ is a neighborhood and $(S_0,\Omega_0)$ is a society, then the \defn{composition} of $(S_0,\Omega_0)$ with $(S,\Omega, \Omega_0)$ is the society $(S \cup S_0, \Omega)$.
A society $(S,\Omega)$ is \defn{$s$-nested} if it is the composition of a society with a rural neighborhood that has an $s$-nest for some its presentation.

A subgraph $F \subseteq S$ for a rural neighborhood $(S,\Omega,\Omega_0)$ with presentation $(\Gamma,\Delta,\Delta_0)$ is \defn{perpendicular} to an $s$-nest $(C_1,C_2,...,C_s)$ for $(\Gamma,\Delta,\Delta_0)$ if every component $P$ of $F$ is a path between $\overline{\Omega}$ and $\overline{\Omega_0}$ such that $P \cap C_i$ is a path for each $i \in [s]$.

We will use a special case of \cite[Lemma 8.2]{l}.

\begin{lemma}[{{\cite[Lemma 8.2]{l}}}] \label{perpendicular}
For any $k,s \in {\mathbb N}$, there exists $s' \in {\mathbb N}$ with $s' \geq s$ such that the following holds.
If $(S,\Omega)$ is a composition of a society $(S_0,\Omega_0)$ with an $s'$-nested rural neighborhood $(S',\Omega,\Omega_0)$ that has an $s'$-nest $(C_1,C_2,...,C_{s'})$ for some presentation $(\Gamma,\Delta,\Delta_0)$, and $F$ is a union of $k$ disjoint paths in $S$ each having one end in $\overline{\Omega}$ and one end in $V(S_0)$, then there exists $F'$ that is a union of $k$ disjoint paths in $S$ and there exists a society $(S_0',\Omega_0')$ with $S_0 \subseteq S_0'$ such that 
	\begin{enumerate}
		\item for any vertices $u$ and $v$, they are the ends of a component of $F$ if and only if they are the ends of a component of $F'$,
		\item $(S,\Omega)$ is a composition of $(S_0',\Omega_0')$ with an $s$-nested rural neighborhood $(S'',\Omega,\Omega_0')$ that has an $s$-nest $(C_1',C_2',...,C_s')$ for some its presentation such that $F' \cap S''$ is perpendicular to $(C_1',C_2',...,C_s')$,
		\item $S'' \cap S_0'$ is a cycle with vertex-set $\overline{\Omega_0'}$ passing through $\overline{\Omega_0'}$ in the order $\Omega_0'$, and
		\item for every component $Q$ of $F' \cap S_0$, if $u$ and $v$ are the ends of the component of $F'$ containing $Q$, then $Q$ is a subpath of the component of $F$ with ends $u$ and $v$.\footnote{This statement is not stated in \cite[Lemma 8.2]{l}, but it follows from its proof.}
	\end{enumerate}
\end{lemma}

\subsection{Main lemma} \label{subsec:one_vortex_main}

Let $G$ be a graph, and let $L$ be an induced subgraph.
Let $X$ and $Y$ be subsets of $V(L)$.
Let $\ell$ be a nonnegative integer.
An \defn{$(\ell,X,Y)$-path in $(L,G)$} is a path in $L$ between a vertex $x$ in $X$ and a vertex $y$ in $Y$ such that $\dist_G(x,y) \geq \ell$. 

An \defn{interesting tuple} is a tuple $(G,L,Z,\xi,\eta,X,Y,\F,\ell,k,r,\T_L,\theta',\theta,\Se,\kappa,\rho,\Se_1,\Se_2,\Sigma,g,\phi,\alpha)$ such that the following hold:
	\begin{itemize}
		\item $G$ is a graph, and $L$ is an induced subgraph of $G$.
		\item $Z \subseteq V(G)$ is a $(\xi,\eta)$-centered set in $G$ such that $N_G(V(L)) \subseteq Z$.
		\item $X$ and $Y$ are subsets of $V(L)$. 
		\item $\F$ is the set of all $(\ell,X,Y)$-paths in $(L,G)$. 
		\item No $k$ members of $\F$ have pairwise distance in $G$ at least $r$.
		\item $\T_L$ is the $(G,\F,r,\theta',Z)$-tangle in $L$, and $\theta' \geq \theta$.
		\item $\Se$ is a $\T_L$-central segregation of $L$ with a $(\kappa,\rho)$-witness $(\Se_1,\Se_2)$.
		\item $\Sigma$ is a surface of Euler genus at most $g$.
		\item $\phi: {\mathbb Z} \rightarrow {\mathbb R}$ is a nondecreasing function. 
		\item $\alpha$ is an $(8,r,\Se_1,\Se_2)$-protected $(\Sigma,\theta,\phi,\T_L)$-arrangement of $\Se$ in $\Sigma$ with respect to $(\Se_1,\Se_2)$. 
	\end{itemize}

\begin{lemma} \label{one_vortex}
For any $k,r,\rho,s,k' \in {\mathbb N}$, there exist $\lambda^*,\phi^* \in {\mathbb N}$ such that for every $\kappa \in {\mathbb N}_0$, there exists $\theta \in {\mathbb N}$ such that for any $\xi,\eta \in {\mathbb N}$ and $\ell,g \in {\mathbb N}_0$, there exist $\xi^*,\eta^* \in {\mathbb N}$ with $\eta^*=4\eta+6r+\ell$ such that the following hold.

If $(G,L,Z,\xi,\eta,X,Y,\F,\ell,k, \allowbreak r, \T_L,\theta',\theta,\Se,\kappa,\rho,\Se_1,\Se_2,\Sigma,g,\phi^*,\alpha)$ is an interesting tuple, then for every $(S,\Omega) \in \Se_2$, the following statements hold.
	\begin{enumerate}
		\item There exists $Z^* \subseteq V(G)$ with $Z^* \supseteq N_G^{\leq r+\ell}[Z]$ such that $Z^*$ is $(\xi^*,\eta^*)$-centered in $G$, and either 
			\begin{enumerate}
				\item $Z^*$ intersects all $(\ell,X,Y)$-paths in $(L,G)$ having an end in $V(S) \cap X$, or
				\item $Z^*$ intersects all $(\ell,X,Y)$-paths in $(L,G)$ having an end in $V(S) \cap Y$.
			\end{enumerate}
		\item There exist $2s$ disjoint cycles $D_1,D_2,...,D_{2s}$ in the skeleton $L'$ of $\alpha$ with respect to $(\Se_1,\Se_2)$ such that  
			\begin{enumerate}
				\item $\alpha(S,\Omega) \subseteq \Delta_1 \subseteq \Delta_2 \subseteq ... \subseteq \Delta_{2s}$, where $\Delta_i$ is a closed disk in $\Sigma$ bounded by $D_i$,
				\item $\Delta_{2s}$ is disjoint from $\alpha(S',\Omega')$ for every $(S',\Omega') \in \Se_2-\{(S,\Omega)\}$,
				\item $m_{\T_{L'}}(\overline{\Omega},V(D_{2s})) \leq \lambda^*$, where $\T_{L'}$ is the respectful tangle of order $\theta$ in $L'$, and  
				\item for every $\Xi \in \{X,Y\}$, if $Z^*$ does not intersect all $(\ell,X,Y)$-paths in $(L,G)$ having an end in $V(S) \cap \Xi$, then there exist $k'$ disjoint paths $P_1,P_2,...,P_{k'}$ in $L$ from $\Xi \cap V(S)-N_G^{\leq r+\ell}[Z]$ to $V(D_{2s})$ such that 
					\begin{enumerate}
						\item the intersection of $\Gamma$ and the natural projection of $\bigcup_{i=1}^{k'}P_i$ in $L'$ is perpendicular to the nest $(D_{s+1},D_{s+2},...,D_{2s})$ for the rural neighborhood $(\Gamma,\Delta_{2s}, \Delta)$ for some closed disk $\Delta$ bounded by a cycle $C$ in $\Gamma$ with $\Delta_s \subseteq \Delta \subseteq \Delta_{s+1}$, where $\Gamma$ is the drawing obtained from $L' \cap \Delta_{2s}$ by deleting the interior of $\Delta$, 
						\item $\{P_i \cap \bigcup_{(S',\Omega') \in \Se, \alpha(S',\Omega') \subseteq \overline{\Delta_s}}S': i \in [k']\}$ is a set of subgraphs of $L-N_G^{\leq r+\ell}[Z]$ pairwise at distance in $G$ at least $r$, 
						\item $m_{\T_{L'}}(\overline{\Omega},V(C)) \geq s$ and $m_{\T_{L'}}(V(C),V(D_{2s})) \geq s$.
					\end{enumerate}
			\end{enumerate}
	\end{enumerate}
\end{lemma}

\begin{pf}
Let $k,r,\rho,s,k' \in {\mathbb N}$. 
We may assume $s \geq 3$.
	\begin{itemize}
		\item Let $k_0 = R(16\rho+8,k)$, where $R$ is the Ramsey number.
		\item Let $k_1 = k_0^2$.
		\item Let $s_0 = 9rk_0 + 2s$.
		\item Let $s_1$ be the integer $s'$ mentioned in Lemma \ref{perpendicular} by taking $(k,s)=(2k_1+k',s_0)$.
		\item Let $s_2=2s_1$ and $s_3 = s_2+r$.
		\item Define $\lambda^* = 10s_3+16$ and $\phi^*=10s_3+16$.
	\end{itemize}
Let $\kappa \in {\mathbb N}_0$.
	\begin{itemize}
		\item Let $\kappa_1,\rho_1,\theta_1$ be the integers $\kappa^*,\rho^*,\theta$, respectively, mentioned in Lemma \ref{sweeping balls into vortices} by taking $(\kappa,k,\rho,\lambda,\phi,\theta^*)=(\kappa,1,\rho,10s_3+15,1,1)$.
		\item Define $\theta = \theta_1+2\rho_1+ \lambda^*+\phi^*$
	\end{itemize}
Let $\xi,\eta \in {\mathbb N}$, and let $\ell,g \in {\mathbb N}_0$.
	\begin{itemize}
		\item Let $\xi_1$ and $\eta_1$ be the integers $\xi^*$ and $\eta^*$, respectively, mentioned in Lemma \ref{hitting_vortices} by taking $(k,r,c,\xi,\eta,\kappa,\rho,g)=(k,r,1,\xi,\eta,\kappa_1,\rho_1,g)$.
		\item For $i \in \{2,3\}$, let $\xi_i$ and $\eta_i$ be the integers $\xi^*$ and $\eta^*$, respectively, mentioned in Lemma \ref{hitting_vortices} by taking $(k,r,c,\xi,\eta,\kappa,\rho,g)=(k_1,r,4-i,\xi,\eta,\kappa_1,\rho_1,g)$.
		\item Let $\xi_4$ and $\eta_4$ be the integers $\xi^*$ and $\eta^*$, respectively, mentioned in Lemma \ref{hitting_vortices} by taking $(k,r,c,\xi,\eta,\kappa,\rho,g)=(k',r,1,\xi,\eta,\kappa_1,\rho_1,g)$.
		\item Define $\xi^*=\xi_1+\xi_2+\xi_3+1+\xi_4$ and $\eta^*=\eta_1+\eta_2+\eta_3+2r+\ell+\eta_4$.
	\end{itemize}
Note that $\eta_1=\eta_2=\eta_3=\eta_4=\eta+r$, so $\eta^* = 4(\eta+r)+2r+\ell$.

Let $G,L,Z,X,Y,\F,\T_L,\theta',\Se,\Se_1,\Se_2,\alpha,\Sigma$ be as stated in the lemma.
Let $(S,\Omega)$ be a member of $\Se_2$.
Let $L'$ be the skeleton of $\alpha$ with respect to $(\Se_1,\Se_2)$.
Let $\T_{L'}$ be the respectful tangle in $L'$ of order $\theta$.

By Lemma \ref{buffer} (with $(\lambda,k)=(2,s_3+1)$, there exist disjoint cycles $C_0,C_1,C_2,...,C_{s_3}$ in $L'$ such that for every $0 \leq i \leq s_3$, $C_i$ bounds a $(10i+15)$-zone $\Lambda_i$ around some vertex $\omega$ in $\overline{\Omega}$ and such that $\overline{\Omega} \subseteq \Lambda_0 \subseteq \Lambda_1 \subseteq ... \subseteq \Lambda_{s_3}$, and $x \subseteq \overline{\Lambda_i}$ for every atom $x$ of $L'$ with $m_{\T_{L'}}(x,\omega) \leq 10i+12$. 

Since $\phi^* > 10s_3+15$, $\Lambda_{s_3}$ is disjoint from $\alpha(S',\Omega')$ for every $(S',\Omega') \in \Se_2-\{(S,\Omega)\}$.
Moreover, for any $0 \leq i \leq s_3-2$, $m_{\T_{L'}}(V(C_i),V(C_{i+2})) \geq (10(i+1)+13) - (10i+15) =8$.
So for any $0 \leq i<j \leq s_3$, $m_{\T_{L'}}(V(C_i),V(C_{j})) \geq 4(j-i-1)$.

By Lemma \ref{sweeping balls into vortices}, there exists a $\T_L$-central segregation $\Se'$ with a $(\kappa_1,\rho_1)$-witness $(\Se_1',\Se_2')$ such that $\Se'$ has a proper arrangement in $\Sigma$, and $S \cup \bigcup_{(S',\Omega') \in \Se_1, \alpha(S',\Omega') \subseteq \overline{\Lambda_{s_3}}}S' \subseteq S^+ \subseteq L$ for some $\rho_1$-vortex $(S^+,\Omega^+) \in \Se_2'$. 

For every $(S',\Omega') \in \Se_2'$, let $(P^{S'},(X^{S'}_t: t \in V(P^{S'})))$ be a vortical decomposition of $(S',\Omega')$ of adhesion at most $\rho_1$ such that $V(P^{S'})=\overline{\Omega'}$ and $t \in X^{S'}_t$ for every $t \in V(P^{S'})$.

\medskip

\noindent{\bf Claim 1:} There exists $Z_1 \subseteq V(G)$ with $Z_1 \supseteq N_G^{\leq r}[Z]$ such that $Z_1$ is $(\xi_1,\eta_1)$-centered in $G$ and intersects all $(\ell,X,Y)$-paths in $(L,G)$ contained in $S^+$.

\noindent {\bf Proof of Claim 1:}
For every $(S',\Omega') \in \Se_2'$, since $\Se'$ is $\T_L$-central and $\T_L$ is the $(G,\F,r,\theta',Z)$-tangle in $L$, we know that $L[X^{S'}_t] - N_G^{\leq r}[\{t\} \cup (X^{S'}_t \cap \bigcap_{t' \in V(P^{S'})-\{t\}}X^{S'}_t) \cup Z]$ does not contain any member of $\F-N_G^{\leq r}[Z]$.
Similarly, for every $(S',\Omega') \in \Se_1'$, $S'-N_G^{\leq r}[\overline{\Omega'} \cup Z]$ does not contain any member of $\F-N_G^{\leq r}[Z]$.

Since every member of $\F$ is connected, $\F$ is component exchangeable.
Since no $k$ members of $\F$ have pairwise distance in $G$ at least $r$, Lemma \ref{hitting_vortices} (taking $(k,r,c,\xi,\eta,\kappa,\rho,g,\Se)=(k,r,1,\xi,\eta,\kappa_1,\rho_1,g,\Se')$), there exists $Z_1 \subseteq V(G)$ with $Z_1 \supseteq N_G^{\leq r}[Z]$ such that $Z_1$ is $(\xi_1,\eta_1)$-centered set in $G$ intersecting all members of $\F-N_G^{\leq r}[Z]$ contained in $\bigcup_{(S',\Omega') \in \Se_2'}S' \supseteq S^+$.
Since $Z_1 \supseteq N_G^{\leq r}[Z]$, we know that $Z_1$ intersects all members of $\F$ contained in $S^+$.
$\Box$

\medskip

Let $H$ be $S \cup \bigcup_{(S',\Omega') \in \Se_1, \alpha(S',\Omega') \subseteq \overline{\Lambda_{s_2}}}S'$.
Let $H' = L'[\bigcup_{(S',\Omega') \in \Se, \alpha(S',\Omega') \subseteq \overline{\Lambda_{s_2}}}\overline{\Omega'}]$.

Let $\Omega_H$ be a cyclic ordering of $V(C_{S_2})$ consistent with a natural ordering of $V(C_{s_2})$.
Note that $(H,\Omega_H)$ and $(H',\Omega_H)$ are societies.

Let $H_0 = S \cup \bigcup_{(S',\Omega') \in \Se_1,\alpha(S,\Omega) \subseteq \overline{\Lambda_{s_1}}}S'$.
Let $\Omega_{H_0}$ be a cyclic ordering of $V(C_s)$ consistent with a natural ordering of $V(C_s)$.
Then $(H_0,\Omega_{H_0})$ is a society.

Let $\Gamma$ be the drawing obtained from $L' \cap \overline{\Lambda_{s_2}}$ by deleting the interior of $\overline{\Lambda_s}$.
Recall $s_2=2s_1$.
Then $(\Gamma,\Omega_H,\Omega_{H_0})$ is a rural neighborhood with an $s_1$-nest $(C_{s_1+1},C_{s_1+2},...,C_{s_2})$, and $(H',\Omega_H)$ is a composition of $(H_0,\Omega_{H_0})$ with $(\Gamma,\Omega_H,\Omega_{H_0})$.

\medskip

\noindent{\bf Claim 2:} There do not exist $k_1$ paths in $H$ from $X \cap V(S)$ to $V(C_{s_2})$ and $k_1$ paths in $H$ from $Y \cap V(S)$ to $V(C_{s_2})$ such that those $2k_1$ paths have pairwise distance in $G$ at least $r$, and $\dist_G(x,y) \geq \ell$ for any $x \in X \cap V(S)$ that is an end of some of those $k_1$ paths from $X \cap V(S)$ and any $y \in Y \cap V(S)$ that is an end of some of those $k_1$ paths from $Y \cap V(S)$. 

\noindent {\bf Proof of Claim 2:}
Suppose to the contrary that there exist $k_1$ paths $Q_{X,1},Q_{X,2},...,Q_{X,k_1}$ in $H$ from $X \cap V(S)$ to $V(C_{s_2})$ and there exist $k_1$ paths $Q_{Y,1},Q_{Y,2},...,Q_{Y,k_1}$ in $H$ from $Y \cap V(S)$ to $V(C_{s_2})$ such that those $2k_1$ paths have pairwise distance in $G$ at least $r$, and $\dist_G(x,y) \geq \ell$ for any $x \in X \cap V(S)$ that is an end of some of $Q_{X,1},Q_{X,2},...,Q_{X,k_1}$ and $y \in Y \cap V(S)$ that is an end of some of $Q_{Y,1},Q_{Y,2},...,Q_{Y,k_1}$ 

Let $F = \bigcup_{i=1}^{k_1}(Q_{X,i} \cup Q_{Y,i})$.
Note that for every component $C$ of $F-(V(S)-\overline{\Omega})$, the natural projection of $C$ is a path in $\Gamma$ with the same ends as $C$.
By Lemma \ref{perpendicular}, there exist paths $Q'_{X,1},Q'_{X,2},...,Q'_{X,k_1}, \allowbreak Q'_{Y,1},Q'_{Y,2},...,Q'_{Y,k_1}$ and a society $(H_0',\Omega_{H_0'})$ with $H_0' \supseteq H_0$ such that
	\begin{itemize}
		\item for every $i \in [k_1]$, the ends of $Q'_{X,i}$ equal the ends of $Q_{X,i}$, and the ends of $Q'_{Y,i}$ equal the ends of $Q_{Y,i}$,
		\item $(H',\Omega_H)$ is a composition of $(H_0',\Omega_{H_0'})$ with an $s_0$-nested rural neighborhood $(H'',\Omega_H, \allowbreak \Omega_{H_0'})$ that has an $s_0$-nest $(C_1',C_2',...,C_{s_0}')$ for some its presentation such that $\bigcup_{i=1}^{k_1}(Q^\Gamma_{X,i} \cup Q^\Gamma_{Y,i})$ is perpendicular to $(C_1',C_2',...,C_{s_0}')$, where $Q^\Gamma_{X,i}$ and $Q^\Gamma_{Y,i}$ are the natural projections of $Q'_{X,i}-(V(H_0')-\overline{\Omega_{H_0'}})$ and $Q'_{Y,i}-(V(H_0')-\overline{\Omega_{H_0'}})$, respectively, and 
		\item for every $i \in [k_1]$, $Q'_{X,i} \cap H_0 \subseteq Q_{X,i}$ and $Q'_{Y,i} \cap H_0 \subseteq Q_{Y,i}$.  
	\end{itemize}
Note that the second bullet implies that for each $\Xi \in \{X,Y\}$ and $i \in [k_1]$, $Q^\Gamma_{\Xi,i}$ is a path in $\Gamma$ from $\overline{\Omega_{H_0'}}$ to $\overline{\Omega_H}$, so $Q'_{\Xi,i}$ is the concatenation of a path $Q''_{\Xi,i}$ in $H_0'$ and the lifting of $Q^\Gamma_{\Xi,i}$ in $L$; the third bullet implies that the subgraphs $Q''_{X,1} \cap H_0,...,Q''_{X,k_1} \cap H_0,Q''_{Y,1} \cap H_0,...,Q''_{Y,k_1} \cap H_0$ have pairwise distance in $G$ at least $r$.

Without loss of generality, we may assume that the ends of $Q'_{X,1},Q'_{X,2},...,Q'_{X,k_0}$ in $V(C_{s_2})$ appearing in the order consistent with $\Omega_H$.
Note that those ends partition $\Omega_H$ into $k_0$ intervals.
So some of those intervals contains $k_1/k_0 \geq k_0$ ends of $Q'_{Y,i}$ (for $i \in [k_1]$).
Hence, without loss of generality, we may assume that the ends of $Q'_{X,1},Q'_{X,2},...,Q'_{X,k_0},Q'_{Y,k_0},Q'_{Y,k_0-1},...,Q'_{Y,1}$ in $V(C_{s_2})$ appearing in the order consistent with $\Omega_H$.
Since $\Gamma$ is a drawing in the plane, it implies that for every $i \in [s_0]$, $Q^\Gamma_{X,1} \cap C_i',Q^\Gamma_{X,2} \cap C_i',...,Q^\Gamma_{X,k_0} \cap C_i',Q^\Gamma_{Y,k_0} \cap C_i',Q^\Gamma_{Y,k_0-1} \cap C_i',...,Q^\Gamma_{Y,1} \cap C_i'$ are disjoint paths in $C_i'$ appearing in the cyclic order of $C_i'$ consistent with $\Omega_H$.

Note that $s_0 \geq 9rk_0$.
For every $i \in [k_0]$, let $P_i$ be a path that is the union of the subpath of $Q'_{X,9ri}$ between its end in $V(S)$ and $V(C_{9ri}')$, the subpath of $Q'_{Y,9ri}$ between its end in $V(S)$ and $V(C_{9ri}')$, and a lifting of the subpath of $C_{9ri}'$ between the ends $Q'_{X,9ri}$ and $Q'_{Y,9ri}$ in $V(C_{9i}')$.
Then $P_1,P_2,...,P_{k_0}$ are disjoint paths in $H$ each between a vertex $x_i$ in $X \cap V(S)$ and a vertex $y_i$ in $Y \cap V(S)$ with $\dist_G(x_i,y_i) \geq \ell$ such that for any $i \neq j$, $\dist_G(V(P_i \cap H_0),V(P_j \cap H_0)) \geq r$.
So each $P_i$ is an $(\ell,X,Y)$-path in $(L,G)$.

We can construct a 2-edge-coloring of the complete graph $K_{k_0}$ with $V(K_{k_0})=[k_0]$ by coloring each edge $ij$ according to whether $\dist_G(V(P_i),V(P_j))$ is at most $r$ or not.
Since $k_0 \geq R(16\rho+8,k)$ and there do not exist $k$ indices $i_1,i_2,...,i_k$ such that $P_{i_1},P_{i_2},...,P_{i_k}$ have pairwise distance in $G$ at least $r$, we know that there exist distinct $j_1<j_2<...<j_{16\rho+8}$ in $[k_0]$ such that $P_{j_1},P_{j_2},...,P_{j_{16\rho+8}}$ have pairwise distance in $G$ at most $r$. 

Let $a \in V(P_{j_1})$ and $b \in V(P_{j_2})$ with $\dist_G(a,b) \leq r$.

Recall $\dist_G(V(P_{j_1} \cap H_0),V(P_{j_2} \cap H_0)) \geq r$.
So $a \not \in V(P_{j_1} \cap H_0)$ or $b \not \in V(P_{j_2} \cap H_0)$.

Let $K$ be the radial drawing of $L'$.

Suppose that $a \in V(P_{j_1})-V(H_0)$ and $b \in V(P_{j_2})-V(H_0)$.
Since $\alpha$ is $(8,r,\Se_1,\Se_2)$-protected, there exists a path $P$ in $L' \cup K$ of length at most $8r$ from $\overline{\Omega_a}$ to $\overline{\Omega_b}$ for some $(S_a,\Omega_a),(S_b,\Omega_b) \in \Se_1$ with $a \in V(S_a)$ and $b \in V(S_b)$.
Since $s_1 \geq s_0 \geq 9r$ and $C_1,C_2,...,C_{s_1}$ are disjoint cycles contained in the natural projection of $H_0$, we know that $P$ is disjoint from the vertex of $K$ corresponding to the region of $L'$ containing $\alpha(S,\Omega)$.
By the planarity, we know that $P$ intersects the union of $C_{9rj_1+i}$ and the natural projections of $Q'_{X,9rj_1+i}-(V(S)-\overline{\Omega})$ and $Q'_{Y,9rj_1+i}-(V(S)-\overline{\Omega})$ for each $i \in [9r-1]$.
Hence $P$ has length at least $9r>8r$, a contradiction.

So either $a \not \in V(P_{j_1})-V(H_0)$ or $b \not \in V(P_{j_2})-V(H_0)$.
Note that the argument so far still works if we replace $j_1$ and $j_2$ by any two distinct elements of $\{j_1,j_2,...,j_{16\rho+8}\}$.
So we can define a tournament on $[16\rho+8]$ such that for any distinct $\beta,\gamma \in [16\rho+8]$, $(\beta,\gamma)$ is an arc if there exist $v_\beta \in V(P_{j_\beta})-V(H_0)$ and $v_\gamma \in V(P_{j_\gamma} \cap H_0)$ such that $\dist_G(v_\beta,v_\gamma) \leq r$.
Then some vertex has out-degree at least $8\rho+4$ in this tournament.
That is, there exist distinct $i^*,\beta_1,\beta_2,...,\beta_{8\rho+4} \in \{j_i: i \in [16\rho+8]\}$ such that for every $i \in [8\rho+4]$, there exist $u_i \in V(P_{i^*})-V(H_0)$ and $v_i \in V(P_{\beta_i} \cap H_0)$ such that $\dist_G(v_\beta,v_\gamma) \leq r$.

We may assume that there are at least $4\rho+2$ elements of $\{\beta_i: i \in [8\rho+4]\}$ greater than $i^*$, since the case that at least $4\rho+2$ elements of $\{\beta_i: i \in [8\rho+4]\}$ smaller than $i^*$ are analogous.
Without loss of generality we may assume that $i^* < \beta_1 < \beta_2 < ... < \beta_{4\rho+2}$.

Let $P^-_{i^*}$ be the maximal subpath of $P_{i^*}-(V(S)-\overline{\Omega})$ intersecting $V(C_1')$.
Since $(S,\Omega)$ is a $\rho$-vortex, there exists a vortical decomposition $(W,(X_t: t \in V(W))$ of $(S,\Omega)$ of adhesion at most $2\rho$ such that $V(W)=\overline{\Omega}$, $t \in X_t$ for every $t \in V(W)=\overline{\Omega}$, and the bag at the first vertex of $W$ contains an end of $P^-_{i^*}$ and the ordering of $V(W)$ is consistent with $\Omega_H$. 
Let $P_{\beta_1}^-$ be the maximal subpath of $P_{i^*}-(V(S)-\overline{\Omega})$ intersecting $V(C_1')$.
Let $t_1,t_2$ be the ends of $P_{\beta_1}^-$ such that $t_1<t_2$ in $W$.
Let $t_1'$ be the neighbor of $t_1$ contained in the subpath of $W$ between $t_1$ and the first vertex of $W$.
Let $t_2'$ be the neighbor of $t_2$ contained in the subpath of $W$ between $t_2$ and the last vertex of $W$.
For every $i \in [4\rho+2]-\{1\}$, since $u_i \in V(P_{i^*})-V(H_0)$ and the planarity, we know that $P_{\beta_i}$ intersects $(X_{t_1} \cap X_{t_1'}) \cup (X_{t_2} \cap X_{t_2'})$.
But $|X_t \cap X_{t'}| \leq 4\rho$, a contradiction.
$\Box$

\medskip

Let
	\begin{itemize}
		\item $\F_{X,Y,1}$ be the set consisting of the subgraphs $L[V(L) \cap (N_G^{\leq r}[V(P_X^-)] \cup V(P_X) \cup N_G^{\leq \ell}[x])] \cup P_Y$, where 
			\begin{itemize}
				\item $P_X$ is a subpath of an $(\ell,X,Y)$-path in $(L,G)$ not contained in $S^+$ such that $P_X$ is in $L-N_G^{\leq r+\ell}[Z]$ from a vertex $x \in X \cap V(S)$ to $\overline{\Omega^+}$,
				\item $P_X^-$ is the subpath of $P_X$ from $x$ to $V(C_{s_2})$ internally disjoint from $V(C_{s_2})$, and
				\item $P_Y$ is a subpath of an $(\ell,X,Y)$-path in $(L,G)$ such that $P_Y$ is in $L-N_G^{\leq r+\ell}[Z]$ from $Y \cap V(S)$ to $\overline{\Omega^+}$ disjoint from $N_G^{\leq r}[V(P_X^-)] \cup V(P_X) \cup N_G^{\leq \ell}[x]$.  
			\end{itemize}
		\item $\F_{X,Y,2}$ be the set consisting of the subgraphs $L[V(L) \cap (N_G^{\leq r}[V(P_X^-)] \cup V(P_X) \cup N_G^{\leq \ell}[x]) \cup P_Y$, where 
			\begin{itemize}
				\item $P_X$ is a subpath of an $(\ell,X,Y)$-path in $(L,G)$ not contained in $S^+$ such that $P_X$ is in $L-N_G^{\leq r+\ell}[Z]$ from a vertex $x \in X \cap V(S)$ to $\overline{\Omega^+}$, 
				\item $P_X^-$ is the subpath of $P_X$ from $x$ to $V(C_{s_2})$ internally disjoint from $V(C_{s_2})$, and
				\item $P_Y$ is a subpath of an $(\ell,X,Y)$-path in $(L,G)$ such that $P_Y$ is in $L-N_G^{\leq r+\ell}[Z]$ from $Y \cap V(S)- N_G^{\leq \ell}[x]$ to $\overline{\Omega^+}$ disjoint from $N_G^{\leq r}[V(P_X^-)] \cup V(P_X)$ and such that $V(P_Y) \cap N_G^{\leq \ell}[x] \neq \emptyset$.  
			\end{itemize}
	\end{itemize}
Since $Z \supseteq N_G(V(L))$, we know that $\F_{X,Y,1}$ is a component exchangeable set of subgraphs of $L$ such that every member of $\F_{X,Y,1}$ has exactly two components.
Moreover, every member of $\F_{X,Y,2}$ is a connected subgraph of $L$, so $\F_{X,Y,2}$ is component exchangeable.

\medskip

\noindent{\bf Claim 3:} There do not exist $k_1$ members of $\F_{X,Y,1}$ pairwise at distance in $G$ at least $r$, and there do not exist $k_1$ members of $\F_{X,Y,2}$ pairwise at distance in $G$ at least $r$.

\noindent {\bf Proof of Claim 3:}
If there exist $i \in [2]$ and $k_1$ members of $\F_{X,Y,i}$ pairwise at distance in $G$ at least $r$, then there exist $k_1$ paths in $H$ from $X \cap V(S)$ to $V(C_{s_2})$ and $k_1$ paths in $H$ from $Y \cap V(S)$ to $V(C_{s_2})$ such that those $2k_1$ paths contradict Claim 2. 
$\Box$

\medskip

\noindent{\bf Claim 4:} For every $i \in [2]$, there exists $Z_{X,Y,i} \subseteq V(G)$ such that $Z_{X,Y,i}$ is $(\xi_{i+1},\eta_{i+1})$-centered in $G$ intersecting all members of $\F_{X,Y,i}$. 

\noindent {\bf Proof of Claim 4:}
Let $i \in [2]$.
Note that every member of $\F_{X,Y,i}$ contains exactly $3-i$ components.
Since every component of a member of $\F_{X,Y,i}$ contains a vertex in $\overline{\Omega^+}$, no component of a member of $\F_{X,Y,i}$ is contained in $S'-N_G^{\leq r}[\overline{\Omega'} \cup Z]$ for any $(S',\Omega') \in \Se_1'$ or contained in $L[X^{S'}_t] - N_G^{\leq r}[\{t\} \cup (X^{S'}_t \cap \bigcup_{t' \in V(P^{S'})-\{t\}}X^{S'}_{t'}) \cup Z]$ for any $(S',\Omega') \in \Se_2'$.
Since no $k_1$ members of $\F_{X,Y,i}$ pairwise at distance in $G$ at least $r$ by Claim 3, Lemma \ref{hitting_vortices} implies that there exists $Z_{X,Y,i} \subseteq V(G)$ such that $Z_{X,Y,i}$ is $(\xi_{i+1},\eta_{i+1})$-centered in $G$ and intersects all members of $\F_{X,Y,i}$ contained in $S^+ \subseteq \bigcup_{(S',\Omega') \in \Se_2'}S'$.
Note that for every member $M$ of $\F_{X,Y,i}$, if $M \not \subseteq S^+$, then some of the corresponding $P_X$ and $P_Y$ is not internally disjoint from $\overline{\Omega^+}$, and the subpaths of $P_X$ and $P_Y$ internally disjoint from $\overline{\Omega^+}$ form another member $M'$ of $\F_{X,Y,i}$ such that $M' \subseteq M \cap S^+$, so $Z_{X,Y,i}$ intersects $M'$ and hence $M$.
$\Box$

\medskip

Let $Z_2 = Z_1 \cup Z_{X,Y,1} \cup Z_{X,Y,2}$.
Then $Z_2$ is $(\xi_1+\xi_2+\xi_3, \eta_1+\eta_2+\eta_3)$-centered in $G$.
Let $Z_3 = N_G^{\leq r+\ell}[Z_2]$.
Note that $Z_3$ is $(\xi_1+\xi_2+\xi_3, \eta_1+\eta_2+\eta_3+r+\ell)$-centered in $G$, and $Z_3 \supseteq N_G^{\leq r+\ell}[Z]$.

\medskip

\noindent{\bf Claim 5:} $Z_3$ intersects all $(\ell,X,Y)$-paths in $(L,G)$ from $V(S) \cap X$ to $V(S) \cap Y$.

\noindent {\bf Proof of Claim 5:}
Suppose to the contrary that there exists an $(\ell,X,Y)$-path $P$ in $(L,G)$ from $V(S) \cap X$ to $V(S) \cap Y$ such that $V(P) \cap Z_3 = \emptyset$.
Since $Z_3 \supseteq Z_1$, we have $P \not \subseteq S^+$ by Claim 1.
Let $x$ be the end of $P$ in $X \cap V(S)$, and let $y$ be the other end of $P$. 
So $y \in Y \cap V(S)$, and there exist a subpath $P_X$ of $P$ from $x$ to $\overline{\Omega^+}$ internally disjoint from $\overline{\Omega^+}$ and a subpath $P_Y$ of $P$ from $y$ to $\overline{\Omega^+}$ internally disjoint from $\overline{\Omega^+}$.
Let $P_X^-$ be the subpath of $P_X$ from $x$ to $V(C_{s_2})$ internally disjoint from $V(C_{s_2})$.

We first suppose that $N_G^{\leq r}[V(P_X^-)] \cap V(P_Y) \neq \emptyset$.
So there exists a path $Q$ in $G$ of length at most $r$ from $V(P_X^-)$ to $V(P_Y)$ internally disjoint from $V(P_X^- \cup P_y)$.
If $Q-V(L) \neq \emptyset$, then since $N_G(V(L)) \subseteq Z$ and $N_G^{\leq r}[Z] \subseteq Z_1 \subseteq Z_3$, we know that $V(Q) \subseteq N_G^{\leq r}[Z] \subseteq Z_3$, so $Z_3$ intersects $P_X^- \subseteq P$, a contradiction.
So $Q \subseteq L$.
Note that the end of $Q$ in $V(P_X^-)$ is contained in $\overline{\Lambda_{s_2}}$.
So the end of $Q$ in $V(P_y)$ is contained in $\overline{\Lambda_{s_2+r}}=\overline{\Lambda_{s_3}}$, and $Q \subseteq S^+$.
Then there exists a path $P'$ in $S^+$ from $x$ to $y$ contained in $P_X^- \cup Q \cup P_y$.
Note that $P'$ is an $(\ell,X,Y)$-path in $(L,G)$ contained in $S^+$ from $X \cap V(S)$ to $Y \cap V(S)$.
So $V(P') \cap Z_1 \neq \emptyset$ by Claim 1.
Hence $V(P) \cap Z_3 \supseteq V(P_X^- \cup P_Y) \cap N_G^{\leq r}[Z_1] \neq \emptyset$, a contradiction.

So $N_G^{\leq r}[V(P_X^-)] \cap V(P_Y) = \emptyset$.
Since $V(P) \cap Z_3 = \emptyset$, we know $V(P) \cap N_G^{\leq r+\ell}[Z] = \emptyset$.
Hence $P_X$ and $P_Y$ are in $L-N_G^{\leq r+\ell}[Z]$.
Moreover, $P_X$ and $P_Y$ are disjoint subpaths in $P$.
So $P_X \cup P_Y$ is contained in a member $M$ of $\F_{X,Y,1} \cup \F_{X,Y,2}$ with $V(M) \subseteq N_G^{\leq r+\ell}[V(P_X \cup P_Y)]$.
Hence $V(M) \cap Z_2 \neq \emptyset$ by Claim 4.
So $V(P_X \cup P_Y) \cap N_G^{\leq r+\ell}[Z_2] \neq \emptyset$.
Hence $V(P) \cap Z_3 \neq \emptyset$, a contradiction.
$\Box$

\medskip

If $Z_3$ intersects all $(\ell,X,Y)$-paths in $(L,G)$ with an end in $X \cap V(S)$, then let $Z_4=Z_3$; otherwise, let $P^*$ be an $(\ell,X,Y)$-path in $(L,G)$ with an end $x_{P^*} \in X \cap V(S)$ and with $V(P^*) \cap Z_3=\emptyset$, and let $Z_4 = N_G^{\leq r}[Z_3] \cup N_G^{\leq \ell}[x_{p^*}]$.
Note that $Z_4$ is $(\xi_1+\xi_2+\xi_3+1, \eta_1+\eta_2+\eta_3+2r+\ell)$-centered in $G$.

\medskip

\noindent{\bf Claim 6:} $Z_4$ intersects all $(\ell,X,Y)$-paths in $(L,G)$ with an end in $X \cap V(S)$ or intersects all $(\ell,X,Y)$-paths in $(L,G)$ with an end in $Y \cap V(S)$.

\noindent {\bf Proof of Claim 6:}
If $Z_3$ intersects all $(\ell,X,Y)$-paths in $(L,G)$ with an end in $X \cap V(S)$, then we are done since $Z_4=Z_3$.
So we may assume that $P^*$ exists.
Let $\Y$ be the set of all $(\ell,X,Y)$-paths in $(L,G)$ with an end in $Y \cap V(S)$ disjoint from $Z_4$.

Suppose to the contrary that $\Y \neq \emptyset$.
For every $P \in \Y$, let $P_Y$ be the subpath of $P$ from its end in $Y \cap V(S)$ (say $y$) to $\overline{\Omega^+}$ internally disjoint from $\overline{\Omega^+}$; note that $P_Y$ exists since it is disjoint from $Z_1$ and hence cannot be contained in $S^+$.
Since $P \in \Y$, $\dist_G(x_{P^*},y)>\ell$.
Since $P^*$ is disjoint from $Z_3 \supseteq Z_1$, we know that $P^*$ is not contained in $S^+$ by Claim 1, so there exists a subpath $P^*_X$ of $P^*$ from $x_{P^*}$ to $\overline{\Omega^+}$ internally disjoint from $\overline{\Omega^+}$.
If $N_G^{\leq r}[V(P^*_X)] \cap V(P_Y) \neq \emptyset$, then there exists an $(\ell,X,Y)$-path $P'$ in $(L,G)$ from $x_{P^*} \in X \cap V(S)$ to $y \in Y \cap V(S)$ such that $V(P') \subseteq P^*_X \cup N_G^{\leq r}[P_Y]$; since $P'$ intersects $Z_3$ by Claim 5, we know that either $Z_3 \cap V(P^*_X) \neq \emptyset$, or $Z_4 \supseteq N_G^{\leq r}[Z_3]$ intersects $P_Y$, a contradiction.
So $N_G^{\leq r}[V(P^*_X)] \cap V(P_Y) = \emptyset$.
Hence $P^*_X \cup P_Y$ is contained in a member of $\F_{X,Y,1} \cup \F_{X,Y,2}$.
So $Z_3 \supseteq N_G^{\leq r+\ell}[Z_2]$ intersects $P^*_X \cup P_Y \subseteq P^* \cup P_Y$ by Claim 4, a contradiction.
$\Box$

\medskip

\noindent{\bf Claim 7:} If $Z_4$ intersects all $(\ell,X,Y)$-paths in $(L,G)$ with an end in $V(S) \cap X$ and intersects all $(\ell,X,Y)$-paths in $(L,G)$ with an end in $V(S) \cap Y$, then this lemma holds.

\noindent {\bf Proof of Claim 7:}
Let $Z^*=Z_4$.
Let $D_i = C_{s+i}$ for every $i \in [2s]$.
Then Statement 1 of this lemma holds by Claim 6.
Statements 2(a)-2(d) clearly hold by the choices of $C_1,.,,,C_{s_3}$.
$\Box$

\medskip

By Claim 7, we may assume that there exists $\Xi \in \{X,Y\}$ such that $Z_4$ does not intersect all $(\ell,X,Y)$-paths in $(L,G)$ with an end in $\Xi \cap V(S)$.
Note that $\Xi$ is the unique member of $\{X,Y\}$ by Claim 6.

\medskip

\noindent{\bf Claim 8:} Either there exist $k'$ paths in $L-N_G^{\leq r+\ell}[Z]$ from $V(S) \cap \Xi$ to $V(C_{s_2})$ pairwise at distance in $G$ at least $r$, or there exists $Z_5 \subseteq V(G)$ such that $Z_5$ is $(\xi_4,\eta_4)$-centered in $G$ and $Z_4 \cup Z_5$ intersects all $(\ell,X,Y)$-paths in $(L,G)$ with an end in $\Xi \cap V(S)$.

\noindent {\bf Proof of Claim 8:}
Let $\F_\Xi$ be the set of all $(\ell,X,Y)$-paths in $(L,G)$ with an end in $\Xi \cap V(S)$ disjoint from $Z_4$.
Note that $\F_\Xi \neq \emptyset$ by assumption.
Moreover, since $Z_4 \supseteq N_G^{\leq r+\ell}[Z]$, we know that every path in $\F_\Xi$ is in $L-N_G^{\leq r+\ell}[Z]$. 

By Claim 1, every member of $\F_\Xi$ is not contained in $S^+$.
So for every $P \in \F_\Xi$, there exists a subpath $P_\Xi$ from $V(S) \cap \Xi$ to $\overline{\Omega^+}$ internally disjoint from $\overline{\Omega^+}$.
Let $\F_\Xi' = \{P_\Xi: P \in \F_\Xi\}$.
Note that every member of $\F_\Xi'$ is a connected subgraph of $L-N_G^{\leq r+\ell}[Z]$.
If $\F_\Xi'$ contains $k'$ members pairwise at distance in $G$ at least $r$, then we are done.
So we may assume that $\F_\Xi'$ does not contain $k'$ members with pairwise distance in $G$ at least $r$.

Since every member of $\F_\Xi'$ contains a vertex in $\overline{\Omega^+}$, we know $S'-N_G^{\leq r}[\overline{\Omega'} \cup Z]$ does not contain a member of $\F_\Xi'$ for every $(S',\Omega') \in \Se_1'$, and $L[X^{S'}_t]-N_G^{\leq r}[\{t\} \cup (X^{S'}_t \cap \bigcup_{t' \in V(P^{S'})}X^{S'}_{t'}) \cup Z]$ does not contain a member of $\F_\Xi'$ for every $(S',\Omega') \in \Se_2'$.
By Lemma \ref{hitting_vortices}, there exists $Z_5 \subseteq V(G)$ such that $Z_5$ is $(\xi_4,\eta_4)$-centered in $G$ intersecting all members of $\F_\Xi'$ (since every member of $\F_{\Xi}'$ is in $S^+$) and hence intersecting all members of $\F_\Xi$.
Hence $Z_4 \cup Z_5$ intersects all $(\ell,X,Y)$-paths in $(L,G)$ with an end in $\Xi \cap V(S)$.
$\Box$

\medskip

If the set $Z_5$ mentioned in Claim 8 exists, then we let $Z^* = Z_4 \cup Z_5$, so $Z^*$ is $(\xi_1+\xi_2+\xi_3+1+\xi_4, \eta_1+\eta_2+\eta_3+2r+\ell+\eta_4)$-centered in $G$ and hence $(\xi^*,\eta^*)$-centered in $G$, and this lemma holds by Claims 6 and 8 and the choice of $C_1,C_2,...,C_{s_3}$.

So we may assume that there exist $k'$ paths $R_1,R_2,...,R_{k'}$ in $L-N_G^{\leq r+\ell}[Z]$ from $V(S) \cap \Xi-N_G^{\leq r+\ell}[Z]$ to $V(C_{s_2})$ pairwise at distance in $G$ at least $r$.
Let $Z^*=Z_4$.
Then Statement 1 of this lemma holds by Claim 6.

Let $F = \bigcup_{i=1}^{k'}R_i$.
Note that for every component $C$ of $F-(V(S)-\overline{\Omega})$, the natural projections of $C$ is a path in $\Gamma$ having the same ends as $C$.
By Lemma \ref{perpendicular}, there exist paths $P^*_1,P^*_2,...,P^*_{k'}$ in $L$ and a society $(H_0',\Omega_{H_0'})$ with $H_0' \supseteq H_0$ such that
	\begin{itemize}
		\item for every $i \in [k']$, the ends of $P^*_i$ equal the ends of $R_i$,
		\item $(H',\Omega_H)$ is a composition of $(H_0',\Omega_{H_0'})$ with an $s_0$-nested rural neighborhood $(H'',\Omega_H, \allowbreak \Omega_{H_0'})$ that has an $s_0$-nest $(D_{s+1},D_{s+2},...,D_{s+s_0})$ for some its presentation such that the intersection of $\Gamma$ and the natural projection of $\bigcup_{i=1}^{k'}P^*_i$ is perpendicular to $(D_{s+1},D_{s+2}, \allowbreak ..., \allowbreak D_{s+s_0})$, 
		\item $H'' \cap H_0'$ is a cycle $C$ with vertex-set $\overline{\Omega_{H_0'}}$ passing through $\overline{\Omega_{H_0'}}$ in the order $\Omega_{H_0'}$, and
		\item for every $i \in [k']$, $P^*_i \cap H_0 \subseteq R_i$. 
	\end{itemize}
For every $i \in [s]$, let $D_s = C_s$.
Let $\Delta$ be the closed disk bounded by $C$ and bounding $\alpha(S,\Omega)$.
Note that $\Delta \supseteq \Delta_s$ since $H_0' \supseteq H_0$.
Then Statements 2(a)-2(c), 2(d)i.\ and 2(d)ii.\ of this lemma hold.

Now we show Statement 2(d)iii.
We know $m_{\T_{L'}}(\overline{\Omega},V(C)) \geq m_{\T_{L'}}(\overline{\Omega},V(C_s))>10(s-1)+12 > s$ by the choice of $C_s$.
Suppose Statement 2(d)iii.\ does not hold.
Then $m_{\T_{L'}}(V(C), \allowbreak V(D_{2s}))<s$.
Hence there exists a closed walk $W$ of length less than $2s$ in the radial drawing $K$ of $L'$ such that $\ins(W)$ contains some vertex $c \in V(C)$ and some vertex $d \in V(D_{2s})$.
If $W$ contains a vertex in $\Delta_{s+1}$, then since $d \in V(D_{2s})$ is contained in $\ins(W)$, $W$ must intersect $V(D_i)$ for every $s+1 \leq i \leq 2s$, so $W$ contains at least $2s$ vertices in $K$, a contradiction.
So $W$ is disjoint from $\Delta_{s+1}$.
Since $c \in V(C)$ is contained in $\ins(W)$, and $W$ is disjoint from $\Delta_{s+1}$, we know $V(C_s) \cup V(C_1) \subseteq \Delta$ is contained in $\ins(W)$.
Hence $m_{\T_{L'}}(V(C_1),V(C_s))<s$.
However, recall that for any $1 \leq i<j \leq s_3$, $m_{\T_{L'}}(V(C_i),V(C_{j})) \geq 4(j-i-1)$. 
So $s > 4(s-2)$, a contradiction.
\end{pf}

\section{Roots in all vortices}

Recall that in Section \ref{sec:one_vortex}, we show that for each vortex $(S,\Omega)$, there exists $\Xi_S \in \{X,Y\}$ such that either there exists a hitting set for all $(\ell,X,Y)$-paths having an end in $V(S) \cap \Xi_S$, or there exists a certain structure that will be useful for constructing $(\ell,X,Y)$-paths in the future.
Note that $\Xi_S$ depends on the vortex $(S,\Omega)$.
We will strengthen this result to make $\Xi_S$ independent of $(S,\Omega)$ in this section.

We need the following notions about homotopy.

Let $\Sigma$ be a surface and $X$ a subset of $\Sigma$.
Let $\Delta_1$ and $\Delta_2$ be open disks in $\Sigma$ with disjoint closure.
A \defn{$\Delta_1$-$\Delta_2$-line in $\Sigma-X$} is a line in $\Sigma-(X \cup \Delta_1 \cup \Delta_2)$ with one endpoint in $\overline{\Delta_1}$ and one endpoint in $\overline{\Delta_2}$ internally disjoint from $\overline{\Delta_1} \cup \overline{\Delta_2}$.
We say that two disjoint $\Delta_1$-$\Delta_2$ lines $L_1$ and $L_2$ in $\Sigma-X$ are \defn{homotopic in $\Sigma-X$ with respect to $\Delta_1$ and $\Delta_2$} if the union of $L_1 \cup L_2$ and a subset of $\partial\Delta_1 \cup \partial\Delta_2$ bounds a disk in $\Sigma-(X \cup \Delta_1 \cup \Delta_2)$.

We will use the following restatement \cite[Lemma 7.1.4]{h}, which is a simple corollary of \cite[Theorem 3.4]{jmm}.

\begin{lemma}[{{\cite[Lemma 7.1.4]{h}}}] \label{homotopic_counts_2}
For any $g,\kappa \in {\mathbb N}_0$ with $\kappa \geq 2$, there exists $\rho \in {\mathbb N}$ such that the following holds.
Let $\Sigma$ be a connected surface of Euler genus at most $g$.
Let $\Delta_1,\Delta_2,...,\Delta_\kappa$ be open disks in $\Sigma$ with disjoint closure.
If $S$ is a set of $\Delta_1$-$\Delta_2$-lines in $\Sigma-\bigcup_{i=1}^\kappa\Delta_i$ internally disjoint from $\bigcup_{i=1}^\kappa\overline{\Delta_i}$ such that members of $S$ are pairwise disjoint and pairwise non-homotopic in $\Sigma-\bigcup_{i=1}^\kappa\Delta_i$ with respect to $\Delta_1$ and $\Delta_2$, then $|S| \leq \rho$.
\end{lemma}

Now we are ready to prove the main result of this section.

\begin{lemma} \label{hitting_all_vortices}
For any $k,r,s,k' \in {\mathbb N}$ and $\ell,g \in {\mathbb N}_0$, there exists a nondecreasing function $\phi^*: {\mathbb Z} \rightarrow {\mathbb R}$ such that for every $\rho \in {\mathbb N}$, there exists $\lambda^* \in {\mathbb N}$ such that for every $\kappa \in {\mathbb N}_0$, there exists $\theta \in {\mathbb N}$ such that for any $\xi,\eta \in {\mathbb N}$, there exist $\xi^*,\eta^* \in {\mathbb N}$ with $\eta^*=4\eta+6r+\ell$ such that the following hold. 

If $(G,L,Z,\xi,\eta,X,Y,\F,\ell,k, \allowbreak r,\T_L,\theta',\theta,\Se,\kappa,\rho,\Se_1,\Se_2,\Sigma,g,\phi^*,\alpha)$ is an interesting tuple, then there exists $Z^* \subseteq V(G)$ with $Z^* \supseteq N_G^{\leq r+\ell}[Z]$ such that $Z^*$ is $(\xi^*,\eta^*)$-centered in $G$ and the following statements hold.
	\begin{enumerate}
		\item Either
			\begin{enumerate}
				\item $Z^*$ intersects all $(\ell,X,Y)$-paths in $(L,G)$ having an end in $X \cap \bigcup_{(S,\Omega) \in \Se_2}V(S)$, or
				\item $Z^*$ intersects all $(\ell,X,Y)$-paths in $(L,G)$ having an end in $Y \cap \bigcup_{(S,\Omega) \in \Se_2}V(S)$.
			\end{enumerate}
		\item For every $(S,\Omega) \in \Se_2$, if $\Xi$ is a member of $\{X,Y\}$ such that $Z^*$ does not intersect all $(\ell,X,Y)$-paths in $(L,G)$ having an end in $\Xi \cap V(S)$, then there exist $2s$ disjoint cycles $D_1,D_2,...,D_{2s}$ in the skeleton $L'$ of $\alpha$ with respect to $(\Se_1,\Se_2)$, and there exist $k'$ disjoint paths $P_1,P_2,...,P_{k'}$ in $L$ from $\Xi \cap V(S)-N_G^{\leq r+\ell}[Z]$ to $V(D_{2s})$ such that  
			\begin{enumerate}
				\item $\alpha(S,\Omega) \subseteq \Delta_1 \subseteq \Delta_2 \subseteq ... \subseteq \Delta_{2s}$, where $\Delta_i$ is a closed disk in $\Sigma$ bounded by $D_i$,
				\item $\Delta_{2s}$ is disjoint from $\alpha(S',\Omega')$ for every $(S',\Omega') \in \Se_2-\{(S,\Omega)\}$,
				\item $m_{\T_{L'}}(\overline{\Omega},V(D_{2s})) \leq \lambda^*$, where $\T_{L'}$ is the respectful tangle of order $\theta$ in $L'$, 
				\item the intersection of $\Gamma$ and the natural projection of $\bigcup_{i=1}^{k'}P_i$ in $L'$ is perpendicular to the nest $(D_{s+1},D_{s+2},...,D_{2s})$ for the rural neighborhood $(\Gamma,\Delta_{2s}, \Delta)$ for some closed disk $\Delta$ bounded by a cycle $C$ in $\Gamma$ with $\Delta_s \subseteq \Delta \subseteq \Delta_{s+1}$, where $\Gamma$ is the drawing obtained from $L' \cap \Delta_{2s}$ by deleting the interior of $\Delta$, 
				\item $\{P_i \cap \bigcup_{(S',\Omega') \in \Se, \alpha(S',\Omega') \subseteq \overline{\Delta_s}}S': i \in [k']\}$ is a set of subgraphs of $L-N_G^{\leq r+\ell}[Z]$ pairwise at distance in $G$ at least $r$, 
				\item $m_{\T_{L'}}(\overline{\Omega},V(C)) \geq s$ and $m_{\T_{L'}}(V(C),V(D_{2s})) \geq s$.
			\end{enumerate}
	\end{enumerate}
\end{lemma}

\begin{pf}
Let $k,r,s,k' \in {\mathbb N}$, and let $\ell,g \in {\mathbb N}_0$.
	\begin{itemize}
		\item Let $\rho_0$ be the integer $\rho$ mentioned in Lemma \ref{homotopic_counts_2} by taking $(g,\kappa)=(g,2)$.
		\item For every integer $x$, define the following:
			\begin{itemize}
				\item Let $k_0'(x) = R(32\max\{x,1\}+16,k)$, where $R$ is the Ramsey number.
				\item Let $k_1'(x) = 9r(k_0'(x)+1)$, $k_2'(x) = \rho_0 k_1'(x)$ and $s_1'(x) = s + k_2'(x) + 8r$.
				\item Let $\lambda'(x),\phi_1'(x)$ be the integers $\lambda^*,\phi^*$, respectively, mentioned in Lemma \ref{one_vortex} by taking $(k,r,\rho,s,k') = (k,r,\max\{x,1\},s_1'(x),k'+k_2'(x))$.
			\end{itemize}
		\item Define $\phi^*: {\mathbb Z} \rightarrow {\mathbb R}$ to be the function such that $\phi^*(x) = \phi_1'(x) + 2s_1'(x)+2(4\lambda'(x)+22)+\ell$ for every integer $x$.
	\end{itemize}
Let $\rho \in {\mathbb N}$.
Define $\lambda^* = \lambda'(\rho)$.
Let $\kappa \in {\mathbb N}_0$.
If $\kappa=0$, then $\Se_2=\emptyset$, so this lemma holds obviously.
So we may assume $\kappa \geq 1$.
	\begin{itemize}
		\item For every $i \in {\mathbb N}_0$, let $\theta_1^{(i)}$ be the integer $\theta$ mentioned in Lemma \ref{one_vortex} by taking $(k,r,\rho,s,k', \allowbreak \kappa) = (k,r,\max\{i,1\},s_1'(i),k'+k_2'(i),\kappa)$. 
		\item Let $\theta_1 = \max_{0 \leq i \leq \rho}\theta_1^{(i)}$.
		\item Define $\theta = \theta_1 + \phi^*(\rho) + k_2'(\rho)$.
	\end{itemize}
Let $\xi,\eta \in {\mathbb N}$.
	\begin{itemize}
		\item For every $i \in {\mathbb N}_0$, let $\xi_1^{(i)},\eta^{(i)}$ be the integers $\xi^*,\eta^*$, respectively, mentioned in Lemma \ref{one_vortex} by taking $(k,r,\rho,s,k',\kappa,\xi,\eta,\ell,g) = (k,r,\max\{i,1\},s_1'(i),k'+k_2'(i),\kappa,\xi,\eta,\ell,g)$. 
		\item Let $\xi_1 = \max_{0 \leq i \leq \rho}\xi_1^{(i)}$, and define $\eta^* = \max_{0 \leq i \leq \rho}\eta^{(i)}$.
		\item Define $\xi^* = \kappa\xi_1$.
	\end{itemize}
Note that $\eta^{(i)}=4\eta+6r+\ell$ for all $i \in {\mathbb N}_0$, so $\eta^*=4\eta+6r+\ell$.

Let $G,L,Z,X,Y,\F,\T_L,\theta',\Se,\Se_1,\Se_2,\alpha,\Sigma$ be as stated in the lemma.
Let $L'$ be the skeleton of $\alpha$ with respect to $(\Se_1,\Se_2)$.
Let $\T_{L'}$ be the respectful tangle of order $\theta$ in $L'$.

Let $\rho'$ be the smallest nonnegative integer such that $(\Se_1,\Se_2)$ is a $(\kappa,\rho')$-witness of $\Se$.
Let $k_0=k_0'(\rho'), k_1=k_1'(\rho'), k_2 = k_2'(\rho')$ and $s_1 = s_1'(\rho')$.

By Lemma \ref{one_vortex}, for every $(S,\Omega) \in \Se_2$, there exist $\Xi_S \in \{X,Y\}$ and $(\xi_1,\eta^*)$-centered set $Z_S \subseteq V(G)$ in $G$ with $Z_S \supseteq N_G^{\leq r+\ell}[Z]$ such that 
	\begin{itemize}
		\item[(i)] $Z_S$ intersects all $(\ell,X,Y)$-paths in $(L,G)$ having an end in $\Xi_S \cap V(S)$,
		\item[(ii)] there exist $2s_1$ cycles $D_{S,1},D_{S,2},...,D_{S,2s_1}$ in $L'$ such that 
			\begin{itemize}
				\item[(iia)] $\alpha(S,\Omega) \subseteq \Delta_{S,1} \subseteq \Delta_{S,2} \subseteq ... \subseteq \Delta_{S,2s_1}$, where $\Delta_{S,i}$ is a closed disk in $\Sigma$ bounded by $D_{S,i}$,
				\item[(iib)] $\Delta_{S,2s_1}$ is disjoint from $\alpha(S',\Omega')$ for every $(S',\Omega') \in \Se_2-\{(S,\Omega)\}$,
				\item[(iic)] $m_{\T_{L'}}(\overline{\Omega},V(D_{S,2s_1})) \leq \lambda'(\rho') \leq \lambda'(\rho) = \lambda^*$, and
				\item[(iid)] for every $\Xi \in \{X,Y\}$, if $Z_S$ does not intersect all $(\ell,X,Y)$-paths in $(L,G)$ having an end in $V(S) \cap \Xi$, then there exist $k'+k_2$ disjoint paths $P_{S,1},P_{S,2},...,P_{S,k'+k_2}$ in $L$ from $\Xi \cap V(S)-N_G^{\leq r+\ell}[Z]$ to $V(D_{S,2s_1})$ such that 
					\begin{itemize}
						\item[(iid1)] the intersection of $\Gamma_S$ and the natural projection of $\bigcup_{i=1}^{k'+k_2}P_{S,i}$ in $L'$ is perpendicular to the nest $(D_{S,s_1+1},D_{S,s_1+2},...,D_{S,2s_1})$ for the rural neighborhood $(\Gamma_S,\Delta_{S,2s_1}, \Delta_{S})$ for some closed disk $\Delta_S$ bounded by a cycle $C_S$ in $\Gamma_S$ with $\Delta_{S,s_1} \subseteq \Delta_S \subseteq \Delta_{S,s_1+1}$, where $\Gamma_S$ is the drawing obtained from $L' \cap \Delta_{S,2s_1}$ by deleting the interior of $\Delta_{S}$, 
						\item[(iid2)] $\{P_{S,i} \cap \bigcup_{(S',\Omega') \in \Se, \alpha(S',\Omega') \subseteq \overline{\Delta_{S,s_1}}}S': i \in [k'+k_2]\}$ is a set of subgraphs of $L-N_G^{\leq r+\ell}[Z]$ pairwise at distance in $G$ at least $r$,
						\item[(iid3)] $m_{\T_{L'}}(\overline{\Omega},V(C_S)) \geq s_1$ and $m_{\T_{L'}}(V(C_S),V(D_{S,2s_1})) \geq s_1$.
					\end{itemize}
			\end{itemize}
	\end{itemize}
Define $Z^* = \bigcup_{(S,\Omega) \in \Se_2}Z_S$.
So $Z^*$ is $(\kappa\xi_1,\eta^*)$-centered in $G$ such that $Z^* \supseteq N_G^{\leq r+\ell}[Z]$.

Note that (ii) implies Statement 2 of this lemma (taking $(\Delta_1,\Delta_2,...,\Delta_s,\Delta,\Delta_{s+1},...,\Delta_{2s}) = (\Delta_{S,s_1-s+1},\Delta_{S,s_1-s+2},...,\Delta_{S,s_1},\Delta_S,\Delta_{S,s_1+1},...,\Delta_{S,s_1+s})$.)
If there exists $\Xi \in \{X,Y\}$ such that $Z^*$ intersects all $(\ell,X,Y)$-paths in $(L,G)$ having an end in $\Xi \cap \bigcup_{(S,\Omega) \in \Se_2}V(S)$, then Statement 1 of this lemma holds and we are done. 

So we may assume that there exist $(S_1,\Omega_1),(S_2,\Omega_2) \in \Se_2$ such that $Z_{S_1}$ does not intersect all $(\ell,X,Y)$-paths in $(L,G)$ having an end in $X \cap V(S_1)$, and $Z_{S_2}$ does not intersect all $(\ell,X,Y)$-paths in $(L,G)$ having an end in $Y \cap V(S_2)$.
Note that (i) implies that $(S_1,\Omega_1),(S_2,\Omega_2)$ are distinct.

Moreover, (iid) implies that we may assume that for every $i \in [k'+k_2]$, $P_{S_1,i}$ is from $X \cap V(S_1)-N_G^{\leq r+\ell}[Z]$ to $V(D_{S_1,2s_1})$, and $P_{S_2,i}$ is from $Y \cap V(S_2)-N_G^{\leq r+\ell}[Z]$ to $V(D_{S_2,2s_1})$.

For every $i \in [k'+k_2]$, let $Q_{S_1,i}$ be the subpath of $P_{S_1,i}$ from $X \cap V(S_1)$ to $V(C_{S_1})$ internally disjoint from $V(C_{S_1})$, and let $Q_{S_2,i}$ be the subpath of $P_{S_2,i}$ from $Y \cap V(S_2)$ to $V(C_{S_2})$ internally disjoint from $V(C_{S_2})$.
For any $j \in [2]$ and $i \in [k'+k_2]$, let $u_{S_j,i}$ be the end of $Q_{S_j,i}$ in $V(C_{S_j})$.
For every $j \in [2]$, let $U_j = \{u_{S_j,i}: i \in [k'+k_2]\}$.

For every $j \in [2]$, since $\Delta_{S_j} \subseteq \Delta_{S_j,s_1+1}$ by (iid1), $\Delta_{S_j}$ is contained in a $(\lambda(\rho')+5)$-zone in $L'$ by (iia) and (iic), so $C_{S_j}$ bounds a $(\lambda(\rho')+5)$-zone $\Lambda_j$ in $L'$.
Moreover, since $\phi^*(\rho') > 2\lambda_1(\rho')$, we know $\Delta_{S_1,2s_1} \cap \Delta_{S_2,2s_1} = \emptyset$, so $\overline{\Lambda_1} \cap \overline{\Lambda_2} = \emptyset$.
Let $\Gamma$ be the drawing obtained from $L'$ by clearing $\Lambda_1 \cup \Lambda_2$.
Note that there exists a respectful tangle $\T_\Gamma$ in $\Gamma$ of order at least $\theta-2(4(\lambda^*+5)+2) > k_2$.

\medskip

\noindent{\bf Claim 1:} There exist $k_2$ disjoint paths in $\Gamma$ from $U_1$ to $U_2$.

\noindent{\bf Proof of Claim 1:} 
Suppose to the contrary that there do not exist $k_2$ disjoint paths in $\Gamma$ from $U_1$ to $U_2$.
Then there exists a separation $(A,B)$ of order less than $k_2$ of $\Gamma$ such that $U_1 \subseteq V(A)$ and $U_2 \subseteq V(B)$.
Since the order of $\T_\Gamma$ is at least $k_2$, either $(A,B) \in \T_\Gamma$ or $(B,A) \in \T_\Gamma$.
By symmetry, we may assume $(A,B) \in \T_\Gamma$.
Hence there exists $U_1' \subseteq U_1$ with $|U_1'|=k_2$ such that there exists a separation $(A^*,B^*) \in \T_\Gamma$ of order less than $|U_1'|$ with $U_1' \subseteq V(A^*)$.
We choose $(A^*,B^*)$ such that $A^*$ is minimal.
By Lemma \ref{A distance set}, $m_{\T_\Gamma}(U_1',a) \leq |V(A^* \cap B^*)|<k_2$ for every $a \in V(A^*)$. 
By (iid1), there exist $k_2$ disjoint paths in $\Gamma$ from $U_1'$ to $V(D_{S_1,2s_1})$.
Hence some vertex $v \in V(D_{S_1,2s_1})$ is in $V(A^*)$.
So there exists a closed walk $W$ in the radial drawing $K$ of $\Gamma$ of length less than $2k_2$ such that $\ins(W)$ contains $v$ and some vertex $u \in U_1'$.

If $W$ contains a vertex contained $\Delta_{S_1,s_1+1}$, then since $v \in V(D_{S_1,2s_1})$ is contained in $\ins(W)$, $W$ must intersect $V(D_{S_1,i})$ for every $s_1+1 \leq i \leq 2s_1$, so $W$ contains at least $2s_1>2k_2$ vertices in the radial drawing $K$, a contradiction.
So $W$ is disjoint from $\Delta_{S_1,s_1+1}$.
Since $u \in U_1' \subseteq V(C_{S_1})$ is contained in $\ins(W)$, and $W$ is disjoint from $\Delta_{S_1,s_1+1}$, we know $V(C_{S_1}) \cup \alpha(S_1,\Omega_1) \subseteq \Delta_{S_1}$ is contained in $\ins(W)$.
Since $\phi^*(\rho') > k_2 + 2(4(\lambda(\rho')+5)+2)$, $W$ is disjoint from $\alpha(S',\Omega')$ for every $(S',\Omega') \in \Se_2$.
So $W$ is a closed walk in the radial drawing of $L'$.
Hence $m_{\T_{L'}}(\overline{\Omega_1},V(C_{S_1}))<k_2<s_1$, contradicting (iid3).
$\Box$

\medskip

For every $(S,\Omega) \in \Se_2$, let $\Delta_S'$ be an open disk in $\Sigma$ with $\alpha(S,\Omega) \subseteq \Delta_{S}' \subseteq \Delta_S$ and $\partial\Delta_S' \cap \partial\Delta_S = \emptyset$.
For $i \in [2]$, let $\Delta_i' = \Delta_{S_i}'$. 
Note that every path $P$ in $\Gamma$ from $U_1$ to $U_2$ can be made a $\Delta_1'$-$\Delta_2'$ line $\widehat{P}$ in $\Sigma-(\Delta_1' \cup \Delta_2')$ internally disjoint from $\bigcup_{(S,\Omega) \in \Se_2}\alpha(S,\Omega)$ by concatenating a line in $\Delta_1'$ between $\partial\Delta_1'$ to $V(P) \cap U_1$ internally disjoint from $\Delta_1'$ and a line in $\Delta_2'$ between $\partial\Delta_2'$ to $V(P) \cap U_2$ internally disjoint from $\Delta_2'$.
So any $k_2$ disjoint paths in $\Gamma$ from $U_1$ to $U_2$ form a subset of $k_2$ disjoint $\Delta_1'$-$\Delta_2'$ lines in $\Sigma-(\Delta_1' \cup \Delta_2')$.
By Lemma \ref{homotopic_counts_2} and Claim 1, there exist $k_2/\rho_0 = k_1$ disjoint paths $Q_1,Q_2,...,Q_{k_1}$ in $\Gamma$ from $U_1$ to $U_2$ such that $\widehat{Q_1},\widehat{Q_2},...,\widehat{Q_{k_1}}$ are pairwise disjoint and pairwise homotopic $\Delta_1'$-$\Delta_2'$-lines in $\Sigma-(\Delta_1' \cup \Delta_2')$. 

Fix a natural cyclic ordering of $\partial\Delta_1'$.
We may assume that the endpoints of $\widehat{Q_1},\widehat{Q_2},...,\widehat{Q_{k_1}}$ in $\partial\Delta_1'$ appears in $\partial\Delta_1'$ in the order listed.
Since $\widehat{Q_1},\widehat{Q_2},...,\widehat{Q_{k_1}}$ are pairwise homotopic in $\Sigma-(\Delta_1' \cup \Delta_2')$, there exists a natural cyclic ordering of $\partial\Delta_2'$ such that the endpoints of $\widehat{Q_1},\widehat{Q_2},...,\widehat{Q_{k_1}}$ in $\partial\Delta_2'$ appears in $\partial\Delta_2'$ in the order listed, and the union of $\widehat{Q_1} \cup \widehat{Q_{k_1}}$ and the line in $\partial\Delta_i'$ (for $i \in [2]$) between the endpoints of $\widehat{Q_1}$ and $\widehat{Q_{k_1}}$ containing the endpoint of $\widehat{Q_2}$ bounds a disk $\Delta^*$ in $\Sigma-(\Delta_1' \cup \Delta_2')$. 
Note that it implies that there exist natural cyclic orderings of $V(C_{S_1})$ and $V(C_{S_2})$ such that for every $j \in [2]$, the ends of $Q_1,Q_2,...,Q_{k_1}$ in $V(C_{S_j})$ appear in $V(C_{S_j})$ in the order listed.

We may assume that $Q_i$ is from $u_{S_1,i}$ to $u_{S_2,i}$ without loss of generality.
For every $i \in [k_1]$, let $P_i$ be the path in $L$ that is the union of $Q_{S_1,i} \cup Q_{S_2,i}$ and a lifting of $Q_i$ in $L$.
Recall $k_1 = 9r(k_0+1)$.
For every $i \in [k_0]$, let $P_i^* = P_{9ri}$.

\medskip

\noindent{\bf Claim 2:} $P_1^*,P_2^*,...,P_{k_0}^*$ are $(\ell,X,Y)$-paths in $(L,G)$ from $V(S_1) \cap X-N_G^{\leq r+\ell}[Z]$ to $V(S_2) \cap Y-N_G^{\leq r+\ell}[Z]$.

\noindent{\bf Proof of Claim 2:} 
Clearly, each of $P_1^*,P_2^*,...,P_{k_0}^*$ is in $L$ from $V(S_1) \cap X-N_G^{\leq r+\ell}[Z]$ to $V(S_2) \cap Y-N_G^{\leq r+\ell}[Z]$.
Suppose that there exists $i \in [k_0]$ such that $\dist_G(x,y)<\ell$, where $x$ and $y$ are the ends of $P_i^*$.
Then there exists a path $Q$ in $G$ of length less than $\ell$ from $x$ to $y$.
Since $x,y \not \in N_G^{\leq \ell}[Z]$ and $N_G(V(L)) \subseteq Z$, we know that $Q \subseteq L$. 
Since $\phi^*(\rho') > \ell$, we know $m_{\T_{L'}}(\overline{\Omega_1},\overline{\Omega_2}) > \ell$, so the length of the natural projection of $Q$ in $L'$ is at least $\ell$, and hence $Q$ has length at least $\ell$, a contradiction.
$\Box$

\medskip

We can construct a 2-edge-coloring of the complete graph $K_{k_0}$ with $V(K_{k_0})=[k_0]$ by coloring each edge $ij$ according to whether $\dist_G(V(P_i^*),V(P_j^*))$ is at most $r$ or not.
Since $k_0 \geq R(32\rho'+16,k)$ and there do not exist $k$ indices $i_1,i_2,...,i_k$ such that $P^*_{i_1},P^*_{i_2},...,P^*_{i_k}$ have pairwise distance in $G$ at least $r$ (by Claim 2 and the assumption of this lemma), we know that there exist distinct $j_1<j_2<...<j_{32\rho'+16}$ in $[k_0]$ such that $P^*_{j_1},P^*_{j_2},...,P^*_{j_{32\rho'+16}}$ have pairwise distance in $G$ at most $r$. 

For any $j \in [2]$ and $i \in [k_0]$, let $I_{j,i} = P^*_i \cap \bigcup_{(S',\Omega') \in \Se, \alpha(S',\Omega') \subseteq \overline{\Delta_{S_j,s_1}}}S'$; note that $I_{j,i}$ is a subgraph of $P_{S_j,9ri} \cap \bigcup_{(S',\Omega') \in \Se, \alpha(S',\Omega') \subseteq \overline{\Delta_{S_j,s_1}}}S'$; let $M_i = P^*_i-V(I_{1,i} \cup I_{2,i})$.
\medskip

\noindent{\bf Claim 3:} For any $\beta_1,\beta_2 \in \{j_i: i \in [32\rho'+16]\}$, if $a_1 \in V(P^*_{\beta_1})$ and $a_2 \in V(P^*_{\beta_2})$ with $\dist_G(a_1,a_2) \leq r$, then there exists $i \in [2]$ such that $a_i \in V(M_{\beta_i})$ and $a_{3-i} \in V(I_{1,\beta_{3-i}} \cup I_{2,\beta_{3-i}})$.

\noindent{\bf Proof of Claim 3:} 
By symmetry, we may assume $\beta_1<\beta_2$.
Let $K$ be the radial drawing of $L'$.
Since $\alpha$ is $(8,r,\Se_1,\Se_2)$-projected, there exists a path $P$ in $L' \cup K$ of length at most $8r$ from $\overline{\Omega_{a_1}}$ to $\overline{\Omega_{a_2}}$ for some $(S_{a_1},\Omega_{a_1}),(S_{a_2},\Omega_{a_2}) \in \Se$ with $a_1 \in V(S_{a_1})$ and $a_2 \in V(S_{a_2})$.

We first suppose that $a_1 \in V(M_{\beta_1})$ and $a_2 \in V(M_{\beta_2})$.
Since $s_1>8r$, by the existence of the disjoint cycles $D_{S_j,1},D_{S_j,2},...,D_{S_j,s_1}$ (for $j \in [2]$), we know that $P$ is disjoint from the vertices of $K$ corresponding to the regions of $L'$ intersecting $\alpha(S_1,\Omega_1) \cup \alpha(S_2,\Omega_2)$.
By the planarity, we know that $P$ intersects the natural projection of $P_i$ in $L'$ either for each $i$ with $9r\beta_1 < i < 9r\beta_2$ or for each $i$ with $1 \leq i <9r\beta_1$.
Hence $P$ has length at least $9r>8r$, a contradiction.

So by symmetry, we may assume $a_1 \not \in V(M_{\beta_1})$.
Then by symmetry, we may assume $a_1 \in V(I_{1,\beta_1})$.

Now we suppose $a_2 \not \in V(M_{\beta_2})$.
So $a_2 \in V(I_{1,\beta_2} \cup I_{2,\beta_2})$.
Since $I_{j,i}$ is a subgraph of $P_{S_j,9ri} \cap \bigcup_{(S',\Omega') \in \Se, \alpha(S',\Omega') \subseteq \overline{\Delta_{S_j,s_1}}}S'$ for every $j \in [2]$ and $i \in \{\beta_1,\beta_2\}$, (iid2) implies that $a_2 \not \in V(I_{1,\beta_2})$.  
So $a_2 \in V(I_{2,\beta_2})$.
Then by the planarity, we know that $P$ intersects $D_{S_1,s_1+i}$ for each $i \in [s_1]$.
Hence $P$ has length at least $s_1>8r$, a contradiction.

So $a_2 \in V(M_{\beta_2})$.
This proves the claim.
$\Box$

\medskip

By Claim 3, we can define a tournament on $[32\rho'+16]$ such that for any distinct $\beta,\gamma$, $(\beta,\gamma)$ is an arc if there exist $v_\beta \in V(M_\beta)$ and $v_\gamma \in V(I_{1,\gamma} \cup I_{2,\gamma})$ such that $\dist_G(v_\beta,v_\gamma) \leq r$.
Then some vertex has out-degree at least $16\rho'+8$ in this tournament.
That is, there exist distinct $i^*,\beta_1,\beta_2,...,\beta_{16\rho'+8} \in \{j_i: i \in [32\rho'+16]\}$ such that for every $i \in [16\rho'+8]$, there exist $u_i \in V(M_{i^*})$ and $v_i \in V(I_{1,\beta_i} \cup I_{2,\beta_i})$ such that $\dist_G(u_i,v_i) \leq r$.
By symmetry, we may assume that $v_i \in V(I_{1,\beta_i})$ for every $i \in [8\rho'+4]$.

We may further assume that there are at least $4\rho'+2$ elements of $\{\beta_i: i \in [8\rho'+4]\}$ greater than $i^*$, since the case that at least $4\rho'+2$ elements of $\{\beta_i: i \in [8\rho'+4]\}$ smaller than $i^*$ are analogous.
Without loss of generality we may assume that $i^* < \beta_1 < \beta_2 < ... < \beta_{4\rho'+2}$.

For every $i$, let $P^-_i$ be the subpath of $P^*_i$ between $\overline{\Omega_1}$ and $\overline{\Omega_2}$ internally disjoint from $\overline{\Omega_1} \cup \overline{\Omega_2}$.
Let $t_{i^*}$ be the end of $P^-_{i^*}$ in $\overline{\Omega_1}$.
For every $i \in [2\rho'+2]$, let $t_{\beta_i}$ be the end of $P^-_{\beta_i}$ in $\overline{\Omega_1}$.
Note that $t_{i^*},t_{\beta_1},t_{\beta_2},...,t_{\beta_{2\rho'+2}}$ appear in $\Omega_1$ in this order. 

Since $(S_1,\Omega_1)$ is a $\rho'$-vortex, there exists a vortical decomposition $(W,(X_t: t \in V(W))$ of $(S_1,\Omega_1)$ of adhesion at most $\rho'$ such that $V(W)=\overline{\Omega_1}$, $t \in X_t$ for every $t \in V(W)=\overline{\Omega_1}$, the first vertex of $W$ is $t_{i^*}$, and the ordering of $V(W)$ is consistent with $\Omega_1$. 
Let $t'$ be the neighbor of $t_{\beta_1}$ in $W$ such that the interval in $\Omega_1$ between $t_{i^*}$ and $t'$ does not contain $t_{\beta_1}$; let $t''$ be the neighbor of $t_{\beta_{2\rho'+2}}$ in $W$ such that the interval in $\Omega_1$ between $t''$ and $t_{i^*}$ does not contain $t_{\beta_{2\rho'+2}}$.
For every $i \in [2\rho'+2]-\{1\}$, since $u_i \in V(M_{i^*})$, by the existence of $D_{S_1,1},D_{S_1,2},...,D_{S_1,s_1}$ and the planarity, we know that $P^*_{\beta_i}$ intersects $(X_{t'} \cap X_{t_{\beta_1}}) \cup (X_{t''} \cap X_{t_{\beta_{2\rho'+2}}})$.
So $|(X_{t'} \cap X_{t_{\beta_1}}) \cup (X_{t''} \cap X_{t_{\beta_{2\rho'+2}}})| \geq 2\rho'+1$.
But $|(X_{t'} \cap X_{t_{\beta_1}}) \cup (X_{t''} \cap X_{t_{\beta_{2\rho'+2}}})| \leq 2\rho'$ since the adhesion of $(W,(X_t: t \in V(W)))$ is at most $\rho'$, a contradiction.
This proves the lemma.
\end{pf}

\section{Roots not in vortices} \label{sec:roots_not_vortices}

In this section we deal with the $(\ell,X,Y)$-paths in $(L,G)$ having at least one end not contained in a vortex, which is the last key ingredient of our proof.
Recall that in the previous lemma, we show that if a hitting cannot be constructed for $(\ell,X,Y)$-paths rooted at vortices, then there is certain nice structure around the vortices that will be helpful for constructing other $(\ell,X,Y)$-paths in the future.
We show that similar nice structures also exist at somewhere far away from all vortices in Section \ref{subsec:freeness}.
Then we prove the main result of this section in Section \ref{subsec:main_not_vortices}.

\subsection{Freeness} \label{subsec:freeness}

Let $G$ be a graph, and let $\T$ be a tangle in $G$.
We say that a subset $X$ of $V(G)$ is \defn{free} with respect to $\T$ if there exists no $(A,B) \in \T$ of order less than $\lvert X \rvert$ such that $X \subseteq V(A)$.

\begin{lemma} \label{surface_ball_free}
For any $s,k',\lambda,t \in {\mathbb N}$, there exists $\lambda^*  \in {\mathbb N}$ with $\lambda^* \geq \lambda$ such that for any $k,\phi^*,t',\lambda' \in {\mathbb N}$, there exists $\theta \in {\mathbb N}$ such that the following hold.
Let $\Gamma$ be a 2-cell drawing in a connected surface $\Sigma$.
Let $\T$ be a respectful tangle in $\Gamma$ of order at least $\theta$.
Let $\Lambda_1,\Lambda_2,...,\Lambda_k$ be $\lambda$-zones in $\Gamma$.
Then there exist distinct vertices $v_1,v_2,...,v_t \in V(\Gamma)$ such that the following hold. 
	\begin{enumerate}
		\item For every $i \in [t]$, $\Gamma$ contains $2s+1$ disjoint cycles $D_{i,0},D_{i,1},D_{i,2},...,D_{i,2s}$ and $k'$ disjoint paths $P_{i,1},P_{i,2},...,P_{i,k'}$ from $V(D_{i,0})$ to $V(D_{i,2s})$ such that 
			\begin{enumerate}
				\item $v_i \in \Delta_{i,0} \subseteq \Delta_{i,1} \subseteq ... \subseteq \Delta_{i,2s}$, where $\Delta_{i,j}$ is a closed disk in $\Sigma$ bounded by $D_{i,j}$ for each $j \in [2s] \cup \{0\}$,
				\item $m_{\T}(v_i,V(D_{i,2s})) \leq \lambda^*$, 
				\item $\bigcup_{j=1}^{k'}P_{i,j}$ is perpendicular to the nest $(D_{i,1},D_{i,2},...,D_{i,2s})$ for the presentation $(\Gamma_i,\Delta_{i,2s},\Delta_{i,0})$ of a rural neighborhood $(\Gamma_i,\Omega_{i,2s}, \Omega_{i,0})$, where $\Gamma_i$ is the drawing obtained from $\Gamma \cap \Delta_{i,2s}$ by deleting the interior of $\Delta_{i,0}$, and $\Omega_{i,j}$ is a natural cyclic orderings of $V(D_{i,j})$ for every $j \in \{0,2s\}$. 
			\end{enumerate}
		\item For every $i \in [t]$, $m_\T(\Delta_{i,2s}, \bigcup_{j=1}^k\overline{\Lambda_j} \cup \bigcup_{j \in [t]-\{i\}}\Delta_{j,2s}) \geq \phi^* + (k+t+t')(4(\lambda^*+\lambda')+2)$. 
		\item For any $i \in [t]$ and $j \in [k']$, let $u_{i,j}$ be the end in $V(D_{i,s+1})$ of the maximal subpath of $P_{i,j}$ between $V(D_{i,s+1})$ and $V(D_{i,2s})$ (not necessarily internally disjoint from $V(D_{i,s+1})$).
			If $\Lambda_1',\Lambda_2',...,\Lambda_{t'}'$ are $t'$ $\lambda'$-zones disjoint from $\bigcup_{i=1}^t\Delta_{i,2s}$ such that $m_\T(\bigcup_{i=1}^t V(D_{i,s+1}), \allowbreak \bigcup_{j \in [t']}\overline{\Lambda_j'}) > 2k' +(t'+t)(4(\lambda'+\lambda^*)+2)$, and if $\T'$ is a tangle obtained from $\T$ by clearing $\Lambda_1',\Lambda_2',...,\Lambda_{t'}'$ and the $t$ $\lambda^*$-zones which are the interior of $\Delta_{1,s+1},\Delta_{2,s+1},...,\Delta_{t,s+1}$, then for every $i \in [t]$, the set $\{u_{i,j}: j \in [k']\}$ is free with respect to $\T'$.
	\end{enumerate}
\end{lemma}

\begin{pf}
Let $s,k',\lambda,t \in {\mathbb N}$.
Let $s_0$ be the integer $s'$ mentioned in Lemma \ref{perpendicular} by taking $(k,s)=(k',2s+k')$.
Let $s_1 = s_0+(k')^2$.
Define $\lambda^*= 2\lambda + 20s_1+4$.

Let $k,\phi^*,t',\lambda' \in {\mathbb N}$.
Let $\lambda_1 = 2\lambda^* + \phi^*+(k+t+t')(4(\lambda^*+\lambda')+2) + 2k'$.
Let $k_0 = 2t\lambda_1 + 2k$.
Let $k_1=4\lambda_1(k_0+1)$.
Define $\theta = 10(k_1+s_1)+ \lambda_1 + t'(4\lambda'+2)$.

Let $\Gamma,\Sigma,\T,\Lambda_1,...,\Lambda_k$ be as stated in the lemma.
For every $i \in [k]$, let $z_i$ be an atom such that $\Lambda_i$ is a $\lambda$-zone around $z_i$.

\medskip

\noindent{\bf Claim 1:} There exist $v_1,v_2,...,v_t \in V(\Gamma)$ such that for every $i \in [t]$, $$m_\T(v_i, \{z_j: j \in [k]\} \cup \{v_j: j \in [t]-\{i\}\}) \geq \lambda_1.$$ 

\noindent{\bf Proof of Claim 1:}
By Lemma \ref{buffer}, there exist disjoint cycles $C_{0,1},C_{0,2},...,C_{0,k_1}$ in $\Gamma$ such that each $C_{0,i}$ bounds a $(10i+4)$-zone $\Lambda_{0,i}$ around $z_1$ containing all atoms $x$ with $m_\T(z_1,x) \leq 10i+1$ such that $\Lambda_{0,1} \subseteq \Lambda_{0,2} \subseteq ... \subseteq \Lambda_{0,k_1}$.
Then for any $i<j$, we know $m_\T(V(C_i),V(C_j))\geq 8 \cdot \lfloor (j-i)/2 \rfloor$.
For any $i<j$, let $I_{i,j} = \overline{\Lambda_{0,j}}-\overline{\Lambda_{0,i}}$.
Then for any $i_1<j_1<i_2<j_2$, we know $m_\T(I_{i_1,j_1},I_{i_2,j_2}) \geq 8 \cdot \lfloor (i_2-j_1)/2 \rfloor$.
Note that $k_1 = 4\lambda_1(k_0+1)$.
For every $i \in [k_0]$, let $I_i = I_{4\lambda_1i,4\lambda_1(i+1)}$.
Note that for every $i \in [k_0]$, every $\lambda_1$-zone around a vertex in $V(C_{0,2\lambda_1(2i+1)})$ is contained in $I_i$.

Note that every $\lambda_1$-zone can intersect at most two $I_j$'s.
So at least $k_0-2k$ members of $\{I_i: i \in [k_0]\}$ are disjoint from all $\lambda_1$-zones around some of $z_1,z_2,...,z_k$.
Since $k_0-2k \geq 2t\lambda_1$, there exist $v_1,v_2,...,v_t$ satisfy the conclusion of this lemma. 
$\Box$

\medskip

By Lemma \ref{buffer}, for every $i \in [t]$, there exists disjoint cycles $C_{i,1},C_{i,2},...,C_{i,2s_1}$ in $\Gamma$ such that each $C_{i,j}$ bounds a $(10j+4)$-zone $\Lambda_{i,j}$ around $v_i$ containing all atoms $x$ with $m_\T(v_i,x) \leq 10j+1$ such that $\Lambda_{i,1} \subseteq \Lambda_{i,2} \subseteq ... \subseteq \Lambda_{i,2s_1}$.

\medskip

\noindent{\bf Claim 2:} For every $i \in [t]$, there exist $k'$ disjoint paths from $V(C_{i,s_1})$ to $V(C_{i,2s_1})$.

\noindent{\bf Proof of Claim 2:}
Suppose to the contrary that there exist $i \in [t]$ and a separation $(A,B)$ of order less than $k'$ such that $V(C_{i,s_1}) \subseteq V(A)$ and $V(C_{i,2s_1}) \subseteq V(B)$.
We choose $(A,B)$ such that $A$ is minimal.
So $A$ is connected, and $A-V(B)$ is disjoint from $V(C_{i,2s_1})$.

Since $\theta > k'$, either $(A,B) \in \T$ or $(B,A) \in \T$.
Since $A-V(B)$ is disjoint from $V(C_{i,2s_1})$ but $A$ contains $V(C_{i,s_1})$, by considering the tangle obtained from $\T$ by clearing $\Lambda_{i,2s_1}$, we know $(B,A) \not \in \T$.
So $(A,B) \in \T$.

By the minimality of $A$ and the planarity, $V(A \cap B)$ is disjoint from the interior $\Lambda_{i,s_1}$.
Since $\Gamma$ is 2-cell, there exists vertex $x \in V(C_{i,1})$ contained in $A-V(B)$. 
By Theorem \ref{A distance}, $m_\T(x,y) \leq 2|V(A \cap B)|^2 < 2(k')^2$ for every $y \in V(A)$.
But $m_\T(V(C_{i,1}),V(C_{i,s_1})) \geq 8 \cdot \lfloor (s_1-1)/2 \rfloor > 2(k')^2$, a contradiction.
$\Box$

\medskip

For every $i \in [t]$, let $H_i = \Gamma \cap \Delta_{i,2s_1}$, let $H_i^- = \Gamma \cap \Delta_{i,s_1}$, and let $\Omega_i$ and $\Omega_{i,0}$ be natural orderings of $V(C_{i,2s_1})$ and $V(C_{i,s_1})$, respectively.
Then $(H_i,\Omega_i)$ is a society that is a composition of the society $(H_i^-,\Omega_{i,0})$ and a rural neighborhood with an $s_1$-nest $(C_{i,s_1+1},C_{i,s_1+2},...,C_{i,2s_1})$ for some its presentation.
For every $i \in [t]$, Claim 2 implies that there exist $k'$ disjoint paths in $H_i$ from $V(C_{i,s_1})$ to $V(C_{i,2s_1})$, and let $F_i$ be the union of them. 
By Lemma \ref{perpendicular}, for every $i \in [t]$, there exists $F_i'$ that is a union of $k'$ paths $V(C_{i,s_1})$ to $V(C_{i,2s_1})$ and there exists a society $(H_i',\Omega_i')$ with $H_i^- \subseteq H_i'$ such that $(H_i,\Omega_i)$ is a composition of $(H_i',\Omega_i')$ and a rural neighborhood $(H_i'',\Omega_i,\Omega_i')$ that has a $(2s+k')$-nest $(D_{i,1},D_{i,2},...,D_{i,2s+k'})$ for some its presentation such that $F_i' \cap H_i''$ is perpendicular to $(D_{i,1},D_{i,2},...,D_{i,2s+k'})$ and $H_i'' \cap H_i'$ is a cycle $D_{i,0}$; note that $F_i' \cap H_i''$ consisting of $k'$ disjoint paths. 
For every $i \in [t]$, let $P_{i,1},P_{i,2},...,P_{i,k'}$ be the maximal subpaths of the components of $F'$ between $V(D_{i,0})$ to $V(D_{i,2s})$.
Then Statement 1 of this lemma holds.

Statement 2 of this lemma holds by Claim 1.

Now we prove Statement 3.

For any $i \in [t]$ and $j \in [k']$, let $u_{i,j}$ be the end in $V(D_{i,s+1})$ of the maximal subpath of $P_{i,j}$ between $V(D_{i,s+1})$ and $V(D_{i,2s})$ (not necessarily internally disjoint from $V(D_{i,s+1})$).
Let $\Lambda_1',\Lambda_2',...,\Lambda_{t'}'$ be $t'$ $\lambda'$-zones disjoint from $\bigcup_{i=1}^t\Delta_{i,2s}$ such that $m_\T(\bigcup_{i=1}^t V(D_{i,s+1}), \bigcup_{j \in [t']}\overline{\Lambda_j'}) > 2k' +(t'+t)(4(\lambda'+\lambda^*)+2)$.
Let $\Gamma'$ and $\T'$ be the drawing and the tangle obtained from $\Gamma$ and $\T$, respectively, by clearing $\Lambda_1',\Lambda_2',...,\Lambda_{t'}'$ and the $t$ $\lambda^*$-zones which are the interior of $\Delta_{1,s+1},\Delta_{2,s+1},...,\Delta_{t,s+1}$.

Suppose to the contrary that there exists $i \in [t]$ such that $\{u_{i,j}: j \in [k']\}$ is not free with respect to $\T'$.
Let $U = \{u_{i,j}: j \in [k']\}$.
So there exists a separation $(A,B) \in \T'$ of order less than $|U|=k'$ such that $U \subseteq V(A)$.
We choose $(A,B)$ such that $A$ is minimal.
By Lemma \ref{A distance set}, $m_{\T'}(U,y) \leq |U|<k'$ for every $y \in V(A)$.
Since $m_\T(V(D_{i,s+1}), \bigcup_{j \in [t']}\overline{\Lambda_j'}) > 2k' +t'(4\lambda'+2)$, the existence of $F_i'$ implies that some vertex $v$ in $V(D_{i,2s+k'})$ is in $A$.
So there exists a closed walk $W$ in the radial drawing $K'$ of $\Gamma'$ of length less than $2|U|=2k'$ such that some vertex $u \in U$ and $v$ are contained in $\ins(W)$.

Let $\Delta_{i,s+1}$ be the closed disk bounded by $D_{i,s+1}$.
If $W$ intersects $\Delta_{i,s+1}$, then since $v \in \ins(W) \cap V(D_{i,2s+k'})$, $W$ intersects $D_{i,j}$ for every $s+1 \leq j \leq 2s+k'$, so $W$ contains at least $2(s+k')>2k'$ vertices in the radial drawing $K'$, a contradiction.
So $W$ is disjoint from $\Delta_{i,s+1}$.
Since $u \in \ins(W)$, $\Delta_{i,s+1} \subseteq \ins(W)$.
By Claim 1, $m_{\T'}(\overline{\Delta_{i,s+1}}, \bigcup_{j \in [t]-\{i\}}\overline{\Delta_{j,s+1}}) \geq m_{\T}(\overline{\Delta_{i,s+1}}, \bigcup_{j \in [t]-\{i\}}\overline{\Delta_{j,s+1}})- (t'+t)(4(\lambda'+\lambda^*)+2) \geq (m_\T(v_i, \{v_j: j \in [t]-\{i\}\}) - 2\lambda^*) - (t'+t)(4(\lambda'+\lambda^*)+2) > 2k'$, so $W$ is disjoint from $\overline{\Delta_{i',s+1}}$ for every $i' \in [t]-\{i\}$.
Similarly, $m_{\T'}(V(D_{i,s+1}), \bigcup_{j \in [t']}\overline{\Lambda_j'}) \geq m_\T(V(D_{i,s+1}), \bigcup_{j \in [t']}\overline{\Lambda_j'}) -(t'+t)(4(\lambda'+\lambda^*)+2)  > 2k'$, so $W$ is disjoint from $\bigcup_{j=1}^{k'}\overline{\Lambda_j'}$.
So $W$ is also a closed walk in the radial drawing $K$ of $\Gamma$.
Hence $m_\T(v_i,u)<k'$ since $\Delta_{i,s+1} \subseteq \ins(W)$.
But $u$ is not contained in the disk bounded by $C_{i,s_1}$.
So $m_\T(v_i,u)> 10s_1+1 > k'$, a contradiction.
This proves the lemma.
\end{pf}

\begin{lemma} \label{2_free}
For any $\lambda,k,t \in {\mathbb N}$, there exists $\theta \in {\mathbb N}$ with $\theta \geq \lambda+70$ such that the following hold.
Let $\Gamma$ be a 2-cell drawing in a connected surface $\Sigma$.
Let $\T$ be a respectful tangle in $\Gamma$ of order at least $\theta$.
Let $k,t \in {\mathbb N}$.
Let $\Lambda_1,\Lambda_2,...,\Lambda_k$ be $\lambda$-zones in $\Gamma$.
If $v_1,v_2,...,v_t \in V(\Gamma)$ are distinct vertices of $\Gamma$ such that for every $i \in [t]$, $m_\T(v_i, \{v_j: j \in [t]-\{i\}\} \cup \bigcup_{j=1}^k\Lambda_j)>70$, then for every $i \in [t]$, 
	\begin{enumerate}
		\item there exists a cycle $D_i$ bounding a $25$-zone $\Lambda_i'$ around $v_i$ with $N_\Gamma[v_i] \subseteq \Lambda_i'$, and 
		\item there exists $U_i \subseteq V(D_i)$ with $|U_i| = 2$ such that 
			\begin{enumerate}
				\item for every $u \in U_i$, there exists a path from $v_i$ to $u$ internally disjoint from $V(D_i)$, and
				\item $U_i$ is free with respect to $\T'$, where $\T'$ is the tangle obtained from $\T$ by clearing the $(\lambda+25)$-zones $\Lambda_1,...,\Lambda_k,\Lambda_1',...\Lambda_t'$.
			\end{enumerate}
	\end{enumerate}
\end{lemma}

\begin{pf}
Let $\lambda,k,t \in {\mathbb N}$.
Define $\theta = 70+(k+t)(\lambda+25)$.

Let $\Gamma,\Sigma,\T,\Lambda_1,...,\Lambda_k,v_1,...,v_t$ be as stated in the lemma.

By Lemma \ref{buffer}, for every $i \in [t]$, there exist disjoint cycles $C_{i,1},C_{i,2},C_{i,3},C_{i,4}$ such that for every $j \in [4]$, $C_{i,j}$ bounds a $(10j+5)$-zone $\Lambda_{i,j}$ around $v_i$, and $y \subseteq \Lambda_{i,j}$ for every $y$ with $m_\T(y,v_i) \leq 10j+2$.
Note that for any $1 \leq i < j \leq 4$, $m_\T(V(C_i),V(C_j)) \geq 8 \cdot \lfloor (j-i)/2 \rfloor$.
For every $i \in [t]$, let $D_i = C_{i,2}$, so $D_i$ bounds a $25$-zone $\Lambda_i'$ around $v_i$ such that $N_\Gamma[v_i] \subseteq \Lambda_i'$.
Hence Statement 1 of this lemma holds.

\medskip

\noindent{\bf Claim 1:} For every $i \in [t]$, there exist two disjoint paths $P_{i,1},P_{i,2}$ from $V(C_{i,2})$ to $V(C_{i,4})$.

\noindent{\bf Proof of Claim 1:} 
Let $i \in [t]$.
Suppose that there do not exist two disjoint paths from $V(C_{i,2})$ to $V(C_{i,4})$.
Then there exists a separation $(A,B)$ of $\Gamma$ with $|V(A \cap B)| \leq 1$ such that $V(C_{i,2}) \subseteq V(A)$ and $V(C_{i,4}) \subseteq V(B)$.
By choosing such $(A,B)$ with $A$ minimal, we know $v_i \in V(A)-V(B)$ and $m_\T(v_i,V(C_{i,2})) \leq 1$ by Lemma \ref{A distance}, a contradiction.
$\Box$

\medskip

\noindent{\bf Claim 2:} For every $i \in [t]$, there exists $U_i \subseteq V(D_i)$ with $|U_i| = 2$ such that for every $u \in U_i$, there exists a path from $v_i$ to $u$ internally disjoint from $V(D_i)$.

\noindent{\bf Proof of Claim 2:} 
Let $i \in [t]$.
Let $Y$ be the set of vertices $y$ in $V(D_i)$ such that there exists a path from $v_i$ to $y$ internally disjoint from $V(D_i)$.
To prove this claim, it suffices to show $|Y| \geq 2$.

Suppose to the contrary that $|Y| \leq 1$.
Since $\Gamma$ is connected, $|Y|=1$.
Let $y$ be the unique vertex in $Y$.
By Claim 1, $|V(D_i)| \geq 2$.
Let $u$ be a vertex in $V(D_i)-\{y\}$.
So $u$ and $v_i$ are in different components of $\Gamma-y$.
By Theorem \ref{A distance}, $m_\T(v_i,y)=1$.
So $y$ is contained in $\Lambda_{i,1}$.
But $y \in V(C_{i,2})$, a contradiction.
$\Box$

\medskip

Claim 2 implies Statement 2(a) of this lemma.

Let $\Gamma'$ and $\T'$ be the drawing and the tangle obtained from $\Gamma$ and $\T$, respectively, by clearing the $(\lambda+25)$-zones $\Lambda_1,...,\Lambda_k,\Lambda_1',...\Lambda_t'$.

Suppose that there exists $i \in [t]$ such that $U_i$ is not free with respect to $\T'$.
Then there exists a separation $(A,B) \in \T'$ of order less than $|U_i|=2$ such that $U_i \subseteq V(A)$.
We choose $(A,B)$ such that $A$ is minimal.
By Lemma \ref{A distance set}, $m_{\T'}(U_i,a) \leq |V(A \cap B)| \leq 1$ for every $a \in V(A)$.
Since $m_\T(v_i, \{v_j: j \in [t]-\{i\}\} \cup \bigcup_{j=1}^k\Lambda_j)>70$, we know $\overline{\Lambda_{i,4}}$ is disjoint from $\bigcup_{j=1}^k\overline{\Lambda_j} \cup \bigcup_{j=1}^{t'}\overline{\Lambda_j'}$.
So Claim 1 implies that there exist two disjoint paths $P_{i,1},P_{i,2}$ in $\Gamma'$ from $U_i$ to $V(C_{i,4})$, where each $P_{i,j}$ is a union of a subpath of $C_{i,2}$ and a path from $V(C_{i,2})$ to $V(C_{i,4})$ internally disjoint from $V(C_{i,2})$.
Since $|V(A \cap B)|<|U_i|=2$, there exist $u \in U_i$ and $v \in V(C_{i,4})$ such that $\{u,v\} \subseteq V(A)$.
So there exists a closed walk $W$ in the radial drawing $K'$ of $\Gamma'$ with length at most 2 such that $\{u,v\} \subseteq \ins(W)$.
If $W$ intersects $\overline{\Lambda_{i,2}}$, then since $v \in \ins(W)$, $W$ intersects $C_{i,j}$ for every $3 \leq j \leq 4$, so $W$ contains at least 4 vertices in $K'$, a contradiction.
So $W$ is disjoint from $\overline{\Lambda_{i,2}}$.
Since $m_\T(v_i, \{v_j: j \in [t]-\{i\}\} \cup \bigcup_{j=1}^k\Lambda_j)>70$, $W$ is disjoint from $\bigcup_{j=1}^k\overline{\Lambda_j} \cup \bigcup_{j=1}^{t'}\overline{\Lambda_j'}$.
Hence $W$ is a walk in the radial drawing of $\Gamma$.
So $m_\T(u,v) \leq 1$, a contradiction.
\end{pf}

\bigskip

Let $\Sigma$ be a connected surface, and let $\Delta_1, ..., \Delta_t$ be pairwise disjoint closed disks in $\Sigma$.
Let $\Gamma$ be a drawing in $\Sigma$ such that $U(\Gamma) \cap \Delta_i = V(\Gamma) \cap \partial\Delta_i$ for $i \in [t]$.
Let $Z = \bigcup_{i=1}^t V(\Gamma) \cap \partial\Delta_i$.
We say that a partition $(Z_1, Z_2, ..., Z_p)$ of $Z$ satisfies the \defn{topological feasibility condition} if there exist pairwise disjoint disks $D_1, D_2, ..., D_p$ in $\Sigma$ such that $D_j \cap \bigcup_{i=1}^t \Delta_i = Z_j$ for $j \in [p]$.

\begin{theorem}[{\cite[Theorem~(3.2)]{rs XII}}] \label{linkage on surface}
For every connected surface $\Sigma$ and all integers $t \geq 0$ and $z \geq 0$, there exists a positive integer $\theta$ such that the following is true.
Let $\Delta_1, ..., \Delta_t$ be pairwise disjoint closed disks in $\Sigma$, and let $\Gamma$ be a $2$-cell drawing in $\Sigma$ such that $U(\Gamma) \cap \Delta_i = V(\Gamma) \cap \partial\Delta_i$ for $1 \leq i \leq t$.
Let $\lvert Z \rvert \leq z$, where $Z = \bigcup_{i=1}^t (V(\Gamma) \cap \partial\Delta_i)$, and let $(Z_1, Z_2,...,Z_p)$ be a partition of $Z$ satisfying the topological feasibility condition.
Let $\T$ be a respectful tangle of order at least $\theta$ in $\Gamma$ with metric $m_\T$ such that $m_\T(r_i,r_j) \geq \theta$ for $1 \leq i< j \leq t$, where $r_i$ is the region of $\Gamma$ meeting $\Delta_i$ for $1 \leq i \leq t$, and $V(\Gamma) \cap \partial\Delta_i$ is free for $1 \leq i \leq t$.
Then there are mutually disjoint connected subdrawings $\Gamma_1, \Gamma_2, ..., \Gamma_p$ of $\Gamma$ with $V(\Gamma_j) \cap Z = Z_j$ for $1 \leq j \leq p$.
\end{theorem}

\subsection{Main lemma} \label{subsec:main_not_vortices}

\begin{lemma} \label{hitting_surface}
For any $k,r,\kappa \in {\mathbb N}$ and $g,\ell \in {\mathbb N}_0$, there exist a function $\kappa_1: {\mathbb Z} \rightarrow {\mathbb R}$ such that for every nondecreasing function $\phi^*: {\mathbb Z} \rightarrow {\mathbb R}$, there exists a nondecreasing function $\phi: {\mathbb Z} \rightarrow {\mathbb R}$ such that for every $\rho \in {\mathbb N}$, there exist $\kappa^* = \kappa_1(\rho),\rho^* \in {\mathbb N}$ such that for every $\theta^* \in {\mathbb N}$, there exists $\theta \in {\mathbb N}$ such that for any $\xi,\eta \in {\mathbb N}$, there exist $\xi^*, \eta^* \in {\mathbb N}$, where $\eta^*=4\eta+6r+\ell$ and $\xi^*$ is independent of $\theta^*$, such that the following hold. 

If $(G,L,Z,\xi,\eta,X,Y,\F,\ell,k, \allowbreak r,\T_L,\theta',\theta,\Se,\kappa,\rho,\Se_1,\Se_2,\Sigma,g,\phi,\alpha)$ be an interesting tuple, then there exists $Z^* \subseteq V(G)$ with $Z^* \supseteq N_G^{\leq r+\ell}[Z]$ such that $Z^*$ is $(\xi^*,\eta^*)$-centered in $G$, and there exists a $\T_L$-central segregation $\Se^*$ of $L$ with a $(\kappa^*,\rho^*)$-witness $(\Se_1^*,\Se_2^*)$ such that the following hold.
	\begin{enumerate}
		\item $\Se_1^* \subseteq \Se_1$ and $\bigcup_{(S,\Omega) \in \Se_2^*}S \supseteq \bigcup_{(S,\Omega) \in \Se_2}S$.
		\item There exists an $(8,r,\Se_1^*,\Se_2^*)$-protected $(\Sigma,\theta^*,\phi^*,\T_L)$-arrangement of $\Se^*$ with respect to $(\Se^*_1,\Se^*_2)$ in $\Sigma$.
		\item $X-N_G^{\leq r+\ell}[Z] \subseteq \bigcup_{(S,\Omega) \in \Se_2^*}V(S)$ or $Y-N_G^{\leq r+\ell}[Z] \subseteq \bigcup_{(S,\Omega) \in \Se_2^*}V(S)$.
		\item If there exists $\Xi \in \{X,Y\}$ such that $\Xi-N_G^{\leq r+\ell}[Z] \not \subseteq \bigcup_{(S,\Omega) \in \Se_2^*}V(S)$, then $Z^*$ intersects all $(\ell,X,Y)$-paths in $(L,G)$ having an end in $\Xi' \cap \bigcup_{(S,\Omega) \in \Se_2}V(S)$, where $\Xi' \in \{X,Y\}-\{\Xi\}$.
	\end{enumerate}
\end{lemma}

\begin{pf}
Let $k,r,\kappa$ be positive integers, and let $g,\ell$ be nonnegative integers.
	\begin{itemize}
		\item Let $\rho_0$ be the integer $\rho$ mentioned in Lemma \ref{homotopic_counts_2} by taking $(g,\kappa)=(g,2)$. 
		\item For every integer $x$, define the following:
			\begin{itemize}
				\item Let $k_0'(x) = R(8\max\{x,1\}+8,k)$, where $R$ is the Ramsey number.
				\item Let $k_1'(x)=9r(k_0'(x)+1)$, $k_2'(x)=\rho_0 k_1'(x)$ and $s'(x)=9r + 2k_2'(x)$.
				\item Let $\phi_1'(x)$ be the function $\phi^*$ and $\lambda_1'(x)$ be the integers $\lambda^*$, respectively, mentioned in Lemma \ref{hitting_all_vortices} by taking $(k,r,s,k',\ell,g,\rho)=(k,r,s'(x),k_2'(x),\ell,g,\max\{x,1\})$.
				\item For every connected surface $\Sigma$ of Euler genus at most $g$, let $\theta'_\Sigma(x)$ be the integer $\theta$ mentioned in Theorem \ref{linkage on surface} by taking $(\Sigma,t,z)=(\Sigma,2k_2'(x)+2,6k_2'(x))$.
				\item Let $\phi_2'(x)=\max_\Sigma \theta'_\Sigma(x)$, where the maximum is over all connected surfaces of Euler genus at most $g$.
				\item Let $\lambda_2'(x)$ be the integer $\lambda^*$ mentioned in Lemma \ref{surface_ball_free} by taking $(s,k',\lambda,t)=(s'(x),k_2'(x),\lambda_1'(x)+5,2)$.
				\item For every integer $x$, let $\lambda_3'(x) = 70+\phi_2'(x)+2k_2'(x)+(52+2\lambda_2'(x)+\lambda_1'(x)+(\kappa+2+4k_2'(x))(4(\lambda_1'(x)+\lambda_2'(x)+25)+2)) + \ell$. 
			\end{itemize}
		\item Define $\kappa_1$ to be the function such that for every integer $x$, $\kappa_1(x)$ is the integer $\kappa^*$ mentioned in Lemma \ref{sweeping_into_vortices_protected} by taking $(\kappa,k)=(\kappa,\kappa+2k_2'(x)+1)$.
	\end{itemize}
Let $\phi^*: {\mathbb Z} \rightarrow {\mathbb R}$ be a nondecreasing function.
	\begin{itemize}
			\item Define $\phi: {\mathbb Z} \rightarrow {\mathbb R}$ to be the nondecreasing function such that $\phi(x) = \phi^*(x)+\phi_1'(x)+ \phi_2'(x) + 70+\lambda_3'(x)$ for every integer $x$. 
	\end{itemize}
Let $\rho$ be positive integers.
	\begin{itemize}
		\item Define $\kappa^* = \kappa_1(\rho)$.
		\item Let $k_0 = k_0'(\rho), k_1 = k_1'(\rho), k_2 = k_2'(\rho), s= s'(\rho), \phi_1 = \phi_1'(\rho), \phi_2=\phi_2'(\rho)$.
		\item Let $\theta_1$ be the integer $\theta$ mentioned in Lemma \ref{hitting_all_vortices} by taking $(k,r,s,k',\ell,g,\rho,\kappa)=(k,r,s'(\rho),k_2'(\rho),\ell,g,\rho,\kappa)$.  
		\item Let $\theta_2$ be the integer $\theta$ mentioned in Lemma \ref{surface_ball_free} by taking $(s,k',\lambda,\allowbreak t,k,\phi^*,t',\lambda')=(s'(\rho),k_2'(\rho),\lambda_1'(\rho)+5,2,\kappa,\phi_2'(\rho)+2k_2'(\rho),\kappa+4k_2'(\rho),\lambda_1'(\rho)+25)$.   
		\item Let $\theta_3$ be the integer $\theta$ mentioned in Lemma \ref{2_free} by taking $(\lambda,k,t) = (\lambda_1'(\rho)+\lambda_2'(\rho)+25,\kappa+2,2k_2'(\rho))$. 
		\item Define $\rho^*$ to be the integer $\rho^*$ mentioned in Lemma \ref{sweeping_into_vortices_protected} by taking $(\kappa,k,\rho,\lambda,\phi)=(\kappa,\kappa+2k_2'(\rho)+1,\rho,2\lambda_3'(\rho),\phi^*)$. 
	\end{itemize}
Let $\theta^*$ be a positive integer.
	\begin{itemize}
		\item Let $\theta_4$ be the integer $\theta$ mentioned in Lemma \ref{sweeping_into_vortices_protected} by taking $(\kappa,k,\rho,\lambda,\phi,\theta^*)=(\kappa,\kappa+2k_2'(\rho)+1,\rho,2\lambda_3'(\rho),\phi^*,\theta^*)$.  
		\item Define $\theta = \theta_1+\phi_2'(\rho)+\theta_2+\theta_3+\theta_4$.  
	\end{itemize}
Let $\xi,\eta$ be positive integers.
Define $\xi^*,\eta^*$ to be the integers $\xi^*,\eta^*$, respectively, mentioned in Lemma \ref{hitting_all_vortices} by taking $(k,r,s,k',\ell,g,\rho,\kappa,\xi,\eta)=(k,r,s'(\rho),k_2'(\rho),\ell,g,\rho,\kappa,\xi,\eta)$.
Note that $\eta^*=4\eta+6r+\ell$.
This completes the definition of the parameters.

Let $G,L,Z,X,Y,\F,\T_L,\Se,\Se_1,\Se_2,\alpha,\Sigma$ be as stated in the lemma.
Let $L'$ be the skeleton of $\alpha$ with respect to $(\Se_1,\Se_2)$.
Let $\T_{L'}$ be a respectful tangle in $L'$ of order $\theta$.

Let $\rho'$ be the smallest nonnegative integer such that $(\Se_1,\Se_2)$ is a $(\kappa,\rho')$-witness.
Let $s= s'(\rho)$.
For every $0 \leq i \leq 2$, let $k_i = k_i'(\rho')$.
For every $i \in [2]$, let $\phi_i=\phi_i(\rho')$.
For every $i \in [3]$, let $\lambda_i = \lambda_i'(\rho')$.

Suppose that this lemma does not hold.

For every $(S,\Omega) \in \Se_2$, Theorem \ref{big zone contains ball} implies that 
	\begin{itemize}
		\item[(i)] there exists a $(\lambda_1+5)$-zone $\Lambda_S$ around some vertex in $\overline{\Omega}$ such that $\alpha(S,\Omega) \subseteq \Lambda_S$ and $x \subseteq \lambda_S$ for every atom $x$ with $m_\T(x,\overline{\Omega}) \leq \lambda_1$.
	\end{itemize}

By Lemma \ref{surface_ball_free}, there exist distinct vertices $w_1$ and $w_2$ of $L'$ such that the following hold:
	\begin{itemize}
		\item[(ii)] For every $i \in [2]$, there exist $2s+1$ disjoint cycles $D_{i,0},D_{i,1},D_{i,2},...,D_{i,2s}$ and $k_2$ disjoint paths $W_{i,1},W_{i,2},...,W_{i,k_2}$ from $V(D_{i,0})$ to $V(D_{i,2s})$ such that 
			\begin{itemize}
				\item $w_i \in \Delta_{i,0} \subseteq \Delta_{i,1} \subseteq ... \subseteq \Delta_{i,2s}$, where $\Delta_{i,j}$ is a closed disk in $\Sigma$ bounded by $D_{i,j}$ for each $j \in [2s] \cup \{0\}$,
				\item $m_{\T_{L'}}(w_i,V(D_{i,2s})) \leq \lambda_2$, 
				\item $\bigcup_{j=1}^{k_2}W_{i,j}$ is perpendicular to the nest $(D_{i,1},D_{i,2},...,D_{i,2s})$ for the presentation $(\Gamma_i,\Delta_{i,2s},\Delta_{i,0})$ of a rural neighborhood $(\Gamma_i,\Omega_{i,2s}, \Omega_{i,0})$, where $\Gamma_i$ is the drawing obtained from $L' \cap \Delta_{i,2s}$ by deleting the interior of $\Delta_{i,0}$, and $\Omega_{i,j}$ is a natural cyclic orderings of $V(D_{i,j})$ for every $j \in \{0,2s\}$, 
			\end{itemize}
		\item[(iii)] For every $i \in [2]$, $m_{\T_{L'}}(\Delta_{i,2s}, \bigcup_{(S,\Omega) \in \Se_2}\overline{\Lambda_S} \cup \bigcup_{j \in [2]-\{i\}}\Delta_{j,2s}) \geq (\phi_2+2k_2) + (\kappa+2+(\kappa+4k_2))(4(\lambda_2+\lambda_1+25)+2)$,
		\item[(iv)] For any $i \in [2]$ and $j \in [k_2]$, let $w_{i,j}$ be the end in $V(D_{i,s+1})$ of the maximal subpath of $W_{i,j}$ between $V(D_{i,s+1})$ and $V(D_{i,2s})$ (not necessarily internally disjoint from $V(D_{i,s+1})$).
			If $\Lambda_1',\Lambda_2',...,\Lambda_{\kappa+4k_2}'$ are $\kappa+4k_2$ many $(\lambda_1+25)$-zones disjoint from $\Delta_{1,2s} \cup \Delta_{2,2s}$ such that $m_\T(V(D_{1,s+1} \cup D_{2,s+1}),\bigcup_{j=1}^{\kappa+4k_2}\overline{\Lambda_j'}) > 2k_2 +(2+\kappa+4k_2)(4(\lambda_1+25+\lambda_2)+2)$, and if $\T'$ is a tangle obtained from $\T$ by clearing $\Lambda_1',\Lambda_2',...,\Lambda_{\kappa+4k_2}'$ and the two $\lambda_2$-zones which are the interior of $\Delta_{1,s+1},\Delta_{2,s+1}$, then for every $i \in [2]$, the set $\{w_{i,j}: j \in [k_2]\}$ is free with respect to $\T'$.
	\end{itemize}
For every $i \in [2]$, let $\Lambda_{w_i}$ be the $\lambda_2$-zone bounded by $D_{i,s+1}$, and let $W_i$ be the set $\{w_{i,j}: j \in [k_2]\}$. 

For every vertex $x$ of $\bigcup_{(S,\Omega) \in \Se_1}S$, let $(S_x,\Omega_x)$ be a member of $\Se_1$ such that $x \in V(S_x)$, and let $v_x$ be a vertex in $\overline{\Omega_x}$.

\medskip

\noindent{\bf Claim 1:} There exists $\Xi \in \{X,Y\}$ such that there exist distinct vertices $x_1,x_2,...,x_{k_2}$ in $\Xi-N_G^{\leq r+\ell}[Z]$ such that for every $i \in [k_2]$, $m_{\T_{L'}}(v_{x_i},\{v_{x_j}: j \in [k_2]-\{i\}\} \cup \bigcup_{(S,\Omega) \in \Se_2}\overline{\Omega} \cup \overline{\Lambda_{w_1}} \cup \overline{\Lambda_{w_2}}) > \lambda_3$.

\noindent{\bf Proof of Claim 1:}
For every $(S,\Omega) \in \Se_2$, let $v_S \in \overline{\Omega}$, and by Theorem \ref{big zone contains ball}, there exists a $(\lambda_3+7)$-zone $\Lambda_S$ around $v_S$ containing all atoms $y$ with $m_{\T_{L'}}(v_S,y) \leq \lambda_3+4$; note that it implies that $y \subseteq \Lambda_S$ for every atom $y$ with $m_{\T_{L'}}(\overline{\Omega},y) \leq \lambda_3+2$.
Similarly, for every $i \in [2]$, there exists a $(\lambda_2+\lambda_3+5)$-zone $\Lambda_{w_i}^+$ around $w_i$ such that $y \subseteq \Lambda_{w_i}^+$ for every atom $y$ with $m_{\T_{L'}}(\overline{\Lambda_{w_i}},y) \leq \lambda_3+2$.

So we are done if there exists $\Xi \in \{X,Y\}$ such that there are $k_2$ distinct vertices $x$ in $\Xi-(N_G^{\leq r+\ell}[Z] \cup \bigcup_{(S,\Omega) \in \Se_2}\overline{\Lambda_S} \cup \overline{\Lambda_{w_1}^+} \cup \overline{\Lambda_{w_2}^+})$ such that the $k_2$ corresponding $v_x$'s have pairwise distance in $m_{\T_{L'}}$ greater than $\lambda_3$. 

Hence we may assume that for every $\Xi \in \{X,Y\}$, there do not exist $k_2$ distinct vertices $x$ in $\Xi-(N_G^{\leq r+\ell}[Z] \cup \bigcup_{(S,\Omega) \in \Se_2}\overline{\Lambda_S} \cup \overline{\Lambda_{w_1}^+} \cup \overline{\Lambda_{w_2}^+})$ such that the $k_2$ corresponding $v_x$'s have pairwise distance in $m_{\T_{L'}}$ greater than $\lambda_3$.
This implies that for every $\Xi \in \{X,Y\}$, there exist $\Xi' \subseteq \{v_x: x \in \Xi - (N_G^{\leq r+\ell}[Z] \cup \bigcup_{(S,\Omega) \in \Se_2}\overline{\Lambda_S} \cup \overline{\Lambda_{w_1}^+} \cup \overline{\Lambda_{w_2}^+})\}$ with $|\Xi'| \leq k_2-1$ such that $m_{\T_{L'}}(v_y,\Xi') \leq \lambda_3$ for every $y \in \Xi - (N_G^{\leq r+\ell}[Z] \cup \bigcup_{(S,\Omega) \in \Se_2}\overline{\Lambda_S} \cup \overline{\Lambda_{w_1}^+} \cup \overline{\Lambda_{w_2}^+})$; so there exist $|\Xi'| \leq k_2-1$ many $(\lambda_3+5)$-zones such that their union $\Lambda_\Xi$ contains $\alpha(S,\Omega)$ for every $(S,\Omega) \in \Se_1$ with $V(S) \cap \Xi - (N_G^{\leq r+\ell}[Z] \cup \bigcup_{(S,\Omega) \in \Se_2}\overline{\Lambda_S} \cup \overline{\Lambda_{w_1}^+} \cup \overline{\Lambda_{w_2}^+}) \neq \emptyset$.
Therefore, $\bigcup_{(S,\Omega) \in \Se_2}\Lambda_S \cup \Lambda_{w_1}^+ \cup \Lambda_{w_2}^+ \cup \Lambda_X \cup \Lambda_Y$ is a union of at most $|\Se_2|+2+2(k_2-1) \leq \kappa+2k_2$ many $(\lambda_2+\lambda_3+5)$-zones such that its closure contains $\alpha(S,\Omega)$ for every $(S,\Omega) \in \Se_1$ with $V(S) \cap (X \cup Y)-N_G^{\leq r+\ell}[Z] \neq \emptyset$.
By Lemma \ref{sweeping_into_vortices_protected}, there exists a $\T_L$-central segregation $\Se^*$ of $L$ with a $(\kappa^*,\rho^*)$-witness $(\Se_1^*,\Se^*_2)$ such that $\Se_1^* \subseteq \Se_1$ and $\bigcup_{(S,\Omega) \in \Se_2^*}S \supseteq \bigcup_{(S,\Omega) \in \Se_2}S \cup \bigcup_{(S,\Omega) \in \Se_1, \alpha(S,\Omega) \subseteq \bigcup_{(S',\Omega') \in \Se_2}\overline{\Lambda_{S'}} \cup \overline{\Lambda_{w_1}^+} \cup \overline{\Lambda_{w_2}^+}\cup \overline{\Lambda_X} \cup \overline{\Lambda_Y}}S \supseteq (X \cup Y)-N_G^{\leq r+\ell}[Z]$, and there exists an $(8,r,\Se_1^*,\Se_2^*)$-protected $(\Sigma,\theta^*,\phi^*,\T_L)$-arrangement with respect to $(\Se_1^*,\Se_2^*)$ in $\Sigma$.
So the lemma holds, a contradiction.
$\Box$

\medskip

We may assume that the element $\Xi$ in Claim 1 equals $X$ by symmetry.

\medskip

\noindent{\bf Claim 2:} There do not exist distinct vertices $y_1,y_2,...,y_{k_2}$ in $Y-N_G^{\leq r+\ell}[Z]$ such that for every $i \in [k_2]$, $m_{\T_{L'}}(v_{y_i},\{v_{y_j}: j \in [k_2]-\{i\}\} \cup \{v_{x_j}: j \in [k_2]\} \cup \bigcup_{(S,\Omega) \in \Se_2}\overline{\Omega} \cup \overline{\Lambda_{w_1}} \cup \overline{\Lambda_{w_2}}) > \lambda_3$.

\noindent{\bf Proof of Claim 2:}
Suppose to the contrary that those distinct vertices $y_1,y_2,...,y_{k_2}$ in $Y-N_G^{\leq r+\ell}[Z]$ exist.
By Lemma \ref{2_free}, since $\lambda_3 \geq 70$, for every $a \in \{x_i,y_i: i \in [k_2]\}$, there exists a cycle $D_a$ in $L'$ bounding a 25-zone $\Lambda_a$ around $v_a$ with $N_{L'}[v_a] \subseteq \Lambda_a$, and there exists $U_a \subseteq V(D_a)$ with $|U_a|=2$ such that for every $u \in U_a$, there exists a path in $L'$ from $v_a$ to $u$ internally disjoint from $V(D_a)$.
Let $L''$ and $\T_{L''}$ be the drawing and the tangle, respectively, obtained from $L'$ and $\T_{L'}$ by clearing the $\kappa+2+2k_2$ many $(\lambda_1+\lambda_2+25)$-zones $\Lambda_S,\Lambda_{w_j},\Lambda_a$ for all $(S,\Omega) \in \Se_2$, $j \in [2]$ and $a \in \{x_i,y_i: i \in [k_2]\}$.
Note that Lemma \ref{2_free} implies that $U_a$ is free with respect to $\T_{L''}$ for every $a \in \{x_i,y_i: i \in [k_2]\}$.

We will apply Theorem \ref{linkage on surface} and we verify the conditions.

For every $a \in \{x_i,y_i: i \in [k_2]\}$, let $\Delta_a$ be a closed disk in $\Sigma$ contained in $\overline{\Lambda_a}$ such that $\Delta_a \cap U(L'') = U_a = V(L'') \cap \partial\Delta_a$.
For every $i \in [2]$, let $\Delta_{w_i}$ be a closed disk in $\Sigma$ contained in $\Delta_{i,s+1}$ such that $\Delta_{w_i} \cap U(L'') = W_i = V(L'') \cap \partial\Delta_{w_i}$. 

Now we show $W_1$ and $W_2$ are free with respect to $\T_{L''}$.
Note that $\Lambda_S$ and $\Lambda_v$ are $\max\{\lambda_1+5,25\}$-zones for any $(S,\Omega) \in \Se_2$ and $v \in \{x_i,y_i: i \in [k_2]\}$.
By (iv), it suffices to show $m_\T(V(D_{1,s+1} \cup D_{2,s+1}),\bigcup_{(S,\Omega) \in \Se_2}\overline{\Lambda_S} \cup \bigcup_{i=1}^{k_2}(\overline{\Lambda_{x_i}} \cup \overline{\Lambda_{y_i}})) > 2k_2 +(\kappa+4k_2+2)(4(\lambda_1+25+\lambda_2)+2)$.
In fact, it follows from the combination of (iii) and the fact that $\lambda_3-25 > 2k_2 +(\kappa+4k_2+2)(4(\lambda_1+25+\lambda_2)+2)$.

So $W_1$ and $W_2$ are free with respect to $\T_{L''}$.

Let $R = \bigcup_{i=1}^{k_2}(U_{x_i} \cup U_{y_i}) \cup W_1 \cup W_2$.
Note that $|R|=6k_2$.

Now we define a partition $\PP$ of $R$ into 2-element sets as follows:
	\begin{itemize}
		\item for each $i \in [k_2]$, some part consists of an element of $U_{x_i}$ and an element of $U_{y_i}$, and
		\item for each $i \in [k_2]$, some part consists of an element of $U_{x_i}$ and an element of $W_1$, and some part consists of an element of $U_{y_i}$ and an element of $W_2$.
	\end{itemize}
Clearly, $\PP$ satisfies the topological feasibility condition and $|\PP|=3k_2$.

For any distinct $a,b \in \{x_i,y_i: i \in [k_2]\}$, if $r_a,r_b$ are the regions of $L''$ meeting $\Delta_a$ and $\Delta_b$, respectively, then $m_{\T_{L''}}(r_a,r_b) \geq m_{\T_{L''}}(U_a,U_b)-2 \geq m_{\T_{L'}}(U_a,U_b)-2-(\kappa+2+2k_2)(4(\lambda_1+\lambda_2+25)+2) \geq (m_{\T_{L'}}(v_a,v_b) - 50)-2-(\kappa+2+2k_2)(4(\lambda_1+\lambda_2+25)+2) > \lambda_3 - (52+(\kappa+2+2k_2)(4(\lambda_1+\lambda_2+25)+2)) > \phi_2$.
For any $a \in \{x_i,y_i: i \in [k_2]\}$ and $i \in [2]$, if $r_a,r_i$ are the regions of $L''$ meeting $\Delta_a$ and $\Delta_{w_i}$, respectively, then $m_{\T_{L''}}(r_a,r_i) \geq m_{\T_{L''}}(U_a,W_i)-2 \geq m_{\T_{L'}}(U_a,W_i)-2-(\kappa+2+2k_2)(4(\lambda_1+\lambda_2+25)+2) \geq (m_{\T_{L'}}(v_a,w_1) - 25-\lambda_2)-2-(\kappa+2+2k_2)(4(\lambda_1+\lambda_2+25)+2) > \lambda_3 - (27+\lambda_2+(\kappa+2+2k_2)(4(\lambda_1+\lambda_2+25)+2)) > \phi_2$.
If $r_1,r_2$ are the regions of $L''$ meeting $\Delta_{w_1}$ and $\Delta_{w_2}$, respectively, then by (iii), 
	\begin{align*}
		m_{\T_{L''}}(r_1,r_2) & \geq m_{\T_{L''}}(W_1,W_2)-2 \\
		& \geq m_{\T_{L'}}(W_1,W_2)-2-(\kappa+2+2k_2)(4(\lambda_1+\lambda_2+25)+2) \\
		& \geq m_{\T_{L'}}(\Delta_{1,2s},\Delta_{2,2s})-2-(\kappa+2+2k_2)(4(\lambda_1+\lambda_2+25)+2) \\
		& \geq \phi_2+2k_2+(\kappa+4k_2+4)(4(\lambda_1+\lambda_2+25)+2) \\
		& \ \ \ \ \ \ \ - (2+(\kappa+2+2k_2)(4(\lambda_1+\lambda_2+25)+2)) \\
		& > \phi_2.
	\end{align*}

Hence by Theorem \ref{linkage on surface}, there exist disjoint paths $Q_{x_i},Q_{x_i,y_i},Q_{y_i}$ (for $i \in [k_2]$) in $L''$ such that the ends of $Q_{x_i}$ are an element of $U_{x_i}$ and an element of $W_1$, the ends of $Q_{x_i,y_i}$ are an element of $U_{x_i}$ and an element of $U_{y_i}$, and the ends of $Q_{y_i}$ are an element of $U_{y_i}$ and an element of $W_2$.
Note that for every $i \in [k_2]$, there is a $\Delta_{1,s+1}$-$\Delta_{2,s+1}$ line $L_i$ in $\Sigma - (\Delta_{1,s+1} \cup \Delta_{2,s+1})$ contained in $Q_{x_i} \cup Q_{x_i,y_i} \cup Q_{y_i} \cup D_{x_i} \cup D_{y_i}$.
By Lemma \ref{homotopic_counts_2}, $k_2/\rho_0 = k_1$ of them are pairwise homotopic in $\Sigma - (\Delta_{1,s+1} \cup \Delta_{2,s+1})$.

Without loss of generality, we may assume that the ends in $W_1$ of $Q_{x_1},Q_{x_2},...,Q_{x_{k_1}}$ appear in $V(D_{1,s+1})$ in the order listed, the ends in $W_2$ of $Q_{y_1},Q_{y_2},...,Q_{y_{k_1}}$ appear in $V(D_{2,s+1})$ in the order listed, and the union of $L_1 \cup L_{k_1}$, the subpath of $V(D_{1,s+1})$ disjoint from $Q_{x_2}$ and the subpath of $V(D_{2,s+1})$ disjoint from $Q_{y_2}$ bounds a disk in $\Sigma$.
By reindexing, we may assume that for every $i \in [k_1]$, the end of $Q_{x_i}$ in $W_1$ is $w_{1,i}$, and the end of $Q_{y_i}$ in $W_2$ is $w_{2,i}$.

For every $i \in [k_1]$, by the property of $U_{x_i}$ and $U_{y_i}$, there exists a path in $L'$ from $v_{x_i}$ to $D_{x_i}$ internally disjoint from $V(D_{x_i})$, and there exists a path in $L'$ from $v_{y_i}$ to $D_{y_i}$ internally disjoint from $V(D_{y_i})$; let $P^L_i$ be the union of these two paths and the path $Q_{x_i,y_i}$.
For every $i \in [k_1]$, let $P_i$ be a union of a path in $S_{x_i}$ from $x_i$ to $v_{x_i}$ internally disjoint from $\overline{\Omega_{x_i}}$, a lifting of $P^L_i$, and a path in $S_{y_i}$ from $y_i$ to $v_{y_i}$ internally disjoint from $\overline{\Omega_{y_i}}$; note that $P_i$ is a path in $L$.

Recall $k_1 \geq 9r(k+1)$.
For every $i \in [k]$, let $P^*_i = P_{9ri}$.

Note that for every $i \in [k]$, the ends of $P^*_i$ are in $L-N_G^{\leq r+\ell}[Z]$, so if there exists a path in $G$ with length at most $\ell$ between them, this path is in $L$ (since $N_G(V(L)) \subseteq Z$), so the natural projection of this path shows $m_{\T_{L'}}(v_{x_i},v_{y_i}) \leq \ell+2 < \lambda_3$, a contradiction.
So $P^*_i$ is an $(\ell,X,Y)$-path in $(L,G)$.

By assumption of this lemma, there do not exist $k$ $(\ell,X,Y)$-paths in $(L,G)$ pairwise at distance in $G$ at least $r$.
So there exist $1 \leq i_1<i_2 \leq k$ such that $\dist_G(V(P_{i_1}^*),V(P_{i_2}^*)) <r$.
Let $K$ be the radial drawing of $L'$.
Since $\Se$ is $(8,r,\Se_1,\Se_2)$-protected, there exists a path $Q$ in $L' \cup K$ of length at most $8r$ from $V(D_{x_{i_1}} \cup Q_{x_{i_1},y_{i_1}} \cup D_{y_{i_1}})$ to $V(D_{x_{i_2}} \cup Q_{x_{i_2},y_{i_2}} \cup D_{y_{i_2}})$. 
Since $s \geq 8r$, the existence of $D_{i,j}$ for $i \in [2]$ and $j \in [s]$ implies that $Q$ cannot intersect $D_{1,1} \cup D_{2,1}$.

By (ii), for every $i \in [k_1]$, there exists a subpath $W_{1,i}^-$ of $W_{1,i}$ from $V(D_{1,1})$ to $w_{1,i}$ internally disjoint from $V(D_{1,s+1})$, so $W_{1,i}^- \cup Q_{x_{i}}$ is a path in $L'$ from $V(D_{1,1})$ to $V(D_{x_{i}})$; similarly, there exists a subpath $W_{2,i}^-$ of $W_{2,i}$ from $V(D_{2,1})$ to $w_{2,i}$ internally disjoint from $V(D_{2,s+1})$, so $W_{2,i}^- \cup Q_{y_{i}}$ is a path in $L'$ from $V(D_{2,1})$ to $V(D_{y_{i}})$.
For every $i \in [k_1]$, let $A_i = (W_{1,i}^- \cup Q_{x_{i}}) \cup D_{x_{i}} \cup Q_{x_{i},y_{i}} \cup D_{y_{i}} \cup (W_{2,i}^- \cup Q_{y_{i}})$.
By planarity, since $Q$ cannot intersect $D_{1,1} \cup D_{2,1}$, $Q$ must intersects $A_i$ either for every $i$ with $9ri_1 < i < 9ri_2$ or for every $i$ with $1 \leq i < 9ri_1$.
So the length of $Q$ is at least $9r>8r$, a contradiction.
This proves the claim.
$\Box$

\medskip

\noindent{\bf Claim 3:} There exists a $\T_L$-central segregation $\Se^*$ of $L$ with a $(\kappa^*,\rho^*)$-witness $(\Se_1^*,\Se_2^*)$ such that Statements 1 and 2 of this lemma hold, and $Y-N_G^{\leq r+\ell}[Z] \subseteq \bigcup_{(S,\Omega) \in \Se_2^*}V(S)$.

\noindent{\bf Proof of Claim 3:}
By Theorem \ref{big zone contains ball}, there exists $\Lambda$ which is a union of at most $\kappa+2$ $(\lambda_1+5+\lambda_2+\lambda_3+3)$-zones such that $z \subseteq \Lambda$ for every atom $y$ of $L'$ with $m_{\T_{L'}}(z,\bigcup_{(S,\Omega) \in \Se_2}\Lambda_S \cup \overline{\Lambda_{w_1}} \cup \overline{\Lambda_{w_2}}) \leq \lambda_3$.
By Claim 2 and Theorem \ref{big zone contains ball}, there exist at most $2k_2-1$ $(\lambda_3+3)$-zones $\Lambda_1,...,\Lambda_{2k_2-1}$ such that at most $k_2$ of them are around $x_i$ for some $i \in [k_2]$, and at most $k_2-1$ of them are around vertices in $Y-N_G^{\leq r+\ell}[Z]$, and such that $\Lambda \cup \bigcup_{i=1}^{2k_2-1}\Lambda_i \supseteq Y-N_G^{\leq r+\ell}[Z]$.
So there exist at most $2k_2-1+\kappa+2$ many $2\lambda_3$-zones such that their union contains $Y-N_G^{\leq r+\ell}[Z]$.
Then this claim follows from Lemma \ref{sweeping_into_vortices_protected}. 
$\Box$

\medskip

By Statement 2 of Lemma \ref{hitting_all_vortices}, there exists a $(\xi^*,\eta^*)$-centered set $Z^*$ in $G$ with $Z^* \supseteq N_G^{\leq r+\ell}[Z]$ such that for every $(S,\Omega) \in \Se_2$, if $\Xi'$ is a member of $\{X,Y\}$ such that $Z^*$ does not intersect all $(\ell,X,Y)$-paths in $(L,G)$ having an end in $\Xi' \cap V(S)$, then there exist $2s$ disjoint cycles $D_{S,1},D_{S,2},...,D_{S,2s}$ in $L'$, and there exist $k_2$ disjoint paths $P^S_1,P^S_2,...,P^S_{k_2}$ in $L$ from $\Xi' \cap V(S)-N_G^{\leq r+\ell}[Z]$ to $V(D_{S,2s})$ such that  
	\begin{itemize}
		\item[(v)] $\alpha(S,\Omega) \subseteq \Delta_{S,1} \subseteq \Delta_{S,2} \subseteq ... \subseteq \Delta_{S,2s}$, where $\Delta_{S,i}$ is a closed disk in $\Sigma$ bounded by $D_{S,i}$,
		\item[(vi)] $\Delta_{S,2s}$ is disjoint from $\alpha(S',\Omega')$ for every $(S',\Omega') \in \Se_2-\{(S,\Omega)\}$,
		\item[(vii)] $m_{\T_{L'}}(\overline{\Omega},V(D_{S,2s})) \leq \lambda_1$, 
		\item[(viii)] the intersection of $\Gamma_S$ and the natural projection of $\bigcup_{i=1}^{k_2}P^S_i$ in $L'$ is perpendicular to the nest $(D_{S,s+1},D_{S,s+2},...,D_{S,2s})$ for the rural neighborhood $(\Gamma_S,\Delta_{S,2s}, \Delta_S)$ for some closed disk $\Delta_S$ bounded by a cycle $C_S$ in $\Gamma_S$ with $\Delta_{S,s} \subseteq \Delta_S \subseteq \Delta_{S,s+1}$, where $\Gamma_S$ is the drawing obtained from $L' \cap \Delta_{S,2s}$ by deleting the interior of $\Delta_S$, 
		\item[(ix)] $\{P^S_i \cap \bigcup_{(S',\Omega') \in \Se, \alpha(S',\Omega') \subseteq \overline{\Delta_{S,s}}}S': i \in [k_2]\}$ is a set of subgraphs of $L-N_G^{\leq r+\ell}[Z]$ pairwise at distance in $G$ at least $r$, 
		\item[(x)] $m_{\T_{L'}}(\overline{\Omega},V(C_S)) \geq s$ and $m_{\T_{L'}}(V(C_S),V(D_{S,2s})) \geq s$.
	\end{itemize}

Note that Claim 3 implies that Statements 1-3 of this lemma hold.
Since we supposed that this lemma does not hold, Statement 4 of this lemma must be violated.
By Claim 3, $Y-N_G^{\leq r+\ell}[Z] \subseteq \bigcup_{(S,\Omega) \in \Se_2^*}V(S)$.
So $X-N_G^{\leq r+\ell}[Z] \not \subseteq \bigcup_{(S,\Omega) \in \Se_2^*}V(S)$, but $Z^*$ does not intersect all $(\ell,X,Y)$-paths in $(L,G)$ having an end in $Y \cap \bigcup_{(S,\Omega) \in \Se_2}V(S)$.

Hence there exists $(S,\Omega) \in \Se_2$ such that $Z^*$ does not intersect all $(\ell,X,Y)$-paths in $(L,G)$ having an end in $Y \cap V(S)$.
So there exist $2s$ disjoint cycles $D_{S,1},D_{S,2},...,D_{S,2s}$ in $L'$, and there exist $k_2$ disjoint paths $P^S_1,P^S_2,...,P^S_{k_2}$ in $L$ from $Y \cap V(S)-N_G^{\leq r+\ell}[Z]$ to $V(D_{S,2s})$ such that (v)-(x) hold.

Note that (vii) implies that $C_S$ bounds a $(\lambda_1+5)$-zone $\Lambda_S'$.
By (i), $\Lambda_S' \subseteq \Lambda_S$.

By Claim 1 and Lemma \ref{2_free}, for every $a \in \{x_i: i \in [k_2]\}$, there exists a cycle $D_a$ in $L'$ bounding a 25-zone $\Lambda_a$ around $v_a$ with $N_{L'}[v_a] \subseteq \Lambda_a$, and there exists $U_a \subseteq V(D_a)$ with $|U_a|=2$ such that 
	\begin{itemize}
		\item[(xi)] for every $u \in U_a$, there exists a path in $L'$ from $v_a$ to $u$ internally disjoint from $V(D_a)$.
	\end{itemize}
Let $L''$ and $\T_{L''}$ be the drawing and the tangle, respectively, obtained from $L'$ and $\T_{L'}$ by clearing the $\kappa+1+k_2$ $(\lambda_1+\lambda_2+25)$-zones $\Lambda_S',\Lambda_{S'},\Lambda_{w_1},\Lambda_a$ for all $(S',\Omega') \in \Se_2-\{(S,\Omega)\}$ and $a \in \{x_i: i \in [k_2]\}$.
Note that Claim 1 and Lemma \ref{2_free} imply that $U_a$ is free with respect to $\T_{L''}$ for every $a \in \{x_i: i \in [k_2]\}$.

For every $i \in [k_2]$, let $Q_i$ be the intersection of $\Gamma_S$ and the natural projection of $P^S_i$ in $L'$, and let $u_i$ be the end of $Q_i$ in $C_S$.
Note that each $Q_i$ is also in $L''$ since $\alpha(S',\Omega')$ (for $(S',\Omega') \in \Se_2-\{(S,\Omega)\}$) and $x_1,...,x_{k_2}$ are far from $\alpha(S,\Omega)$ with respect to $m_{\T_{L'}}$ by the definition of $\phi$ and Claim 1.

Let $U_S = \{u_i: i \in [k_2]\}$.

\medskip

\noindent{\bf Claim 4:} $U_S$ is free with respect to $\T_{L''}$.

\noindent{\bf Proof of Claim 4:}
Suppose to the contrary that $U_S$ is not free.
Then there exists a separation $(A,B) \in \T_{L''}$ of order less than $|U_S|=k_2$ such that $U_S \subseteq V(A)$.
We choose $(A,B)$ such that $A$ is minimal.
By Lemma \ref{A distance set}, $m_{\T_{L''}}(U_S,a) \leq |V(A \cap B)|<k_2$ for every $a \in V(A)$. 
Since $Q_1,Q_2,...,Q_{k_2}$ are $k_2$ disjoint paths in $L''$ from $U_S$ to $V(D_{S,2s})$, some vertex $v \in V(D_{S,2s})$ is in $V(A)$.
So there exists a closed walk $W$ in the radial drawing $K''$ of $L''$ of length less than $2k_2$ such that $\ins(W)$ contains $v$ and some vertex $u \in U_S$.

If $W$ contains a vertex contained $\Lambda_S'$, then since $v \in V(D_{S,2s})$ is contained in $\ins(W)$, $W$ must intersect $V(D_{S,i})$ for every $s+1 \leq i \leq 2s$, so $W$ contains at least $s>2k_2$ vertices in the radial drawing $K''$, a contradiction.
So $W$ is disjoint from $\Lambda_S'$.

Since $\phi(\rho') \geq \lambda_3 \geq 2k_2 + (\kappa+1+k_2)(4(\lambda_1+\lambda_2+25)+2) + 25+(\lambda_1+5)+\lambda_2$, $W$ is disjoint from $\bigcup_{(S',\Omega') \in \Se_2-\{(S,\Omega)\}}\overline{\Lambda_{S'}} \cup \bigcup_{i=1}^{k_2}\overline{\Lambda_{x_i}} \cup \overline{\Lambda_{w_1}}$.
Since $u \in U_S \subseteq V(C_{S})$ is contained in $\ins(W)$, and $W$ is disjoint from $\Lambda_S'$, we know that $W$ is a closed walk in the radial drawing $K$ of $L'$ such that $V(C_{S}) \cup \overline{\Omega}$ is contained in $\ins(W)$.
So $m_{\T_{L'}}(V(C_S),\overline{\Omega})<k_2<s$, contradicting (x).
$\Box$

\medskip

\noindent{\bf Claim 5:} There exist $2k_2$ disjoint paths $Q_{x_1},Q_{x_2},...,Q_{x_{k_2}},Q'_{x_1},Q'_{x_2},...,Q'_{x_{k_2}}$ in $L'' \subseteq L'$ such that for every $i \in [k_2]$, $Q_{x_i}$ is from $U_{x_i}$ to $U_S$, and $Q'_{x_i}$ is from $U_{x_i}$ to $W_1$.

\noindent{\bf Proof of Claim 5:}
For every $a \in \{x_i: i \in [k_2]\}$, let $\Delta_a$ be a closed disk in $\Sigma$ contained in $\overline{\Lambda_a}$ such that $\Delta_a \cap U(L'') = U_a = V(L'') \cap \partial\Delta_a$.
Let $\Delta_{w_1}$ be a closed disk in $\Sigma$ contained in $\Delta_{1,s+1}$ such that $\Delta_{w_1} \cap U(L'') = W_1 = V(L'') \cap \partial\Delta_{w_1}$. 
Let $\Delta_S'$ be the closed disk in $\Sigma$ contained in $\Delta_S$ such that $\Delta_{S'} \cap U(L'') = U_S = V(L'') \cap \partial\Delta_S'$.

Let $R = \bigcup_{i=1}^{k_2}U_{x_i} \cup W_1 \cup U_S$.
Note that $|R|=4k_2$.

Now we define a partition $\PP$ of $R$ into 2-element sets such that for each $i \in [k_2]$, some part consists of an element of $U_{x_i}$ and an element of $W_1$, and some part consists of an element of $U_{x_i}$ and an element of $U_S$.
Clearly, $\PP$ satisfies the topological feasibility condition and $|\PP|=2k_2$.

For any distinct $a,b \in \{x_i: i \in [k_2]\}$, if $r_a,r_b$ are the regions of $L''$ meeting $\Delta_a$ and $\Delta_b$, respectively, then $m_{\T_{L''}}(r_a,r_b) \geq m_{\T_{L''}}(U_a,U_b)-2 \geq m_{\T_{L'}}(U_a,U_b)-2-(\kappa+1+k_2)(4(\lambda_1+\lambda_2+25)+2) \geq (m_{\T_{L'}}(v_a,v_b) - 50)-2-(\kappa+1+k_2)(4(\lambda_1+\lambda_2+25)+2) > \lambda_3 - (52+(\kappa+1+k_2)(4(\lambda_1+\lambda_2+25)+2)) > \phi_2$.
For any $a \in \{x_i: i \in [k_2]\}$, if $r_a,r_1$ are the regions of $L''$ meeting $\Delta_a$ and $\Delta_{w_1}$, respectively, then $m_{\T_{L''}}(r_a,r_1) \geq m_{\T_{L''}}(U_a,W_1)-2 \geq m_{\T_{L'}}(U_a,W_1)-2-(\kappa+1+k_2)(4(\lambda_1+\lambda_2+25)+2) \geq (m_{\T_{L'}}(v_a,w_1) - 25-\lambda_2)-2-(\kappa+1+k_2)(4(\lambda_1+\lambda_2+25)+2) > \lambda_3 - (27+\lambda_2+(\kappa+1+k_2)(4(\lambda_1+\lambda_2+25)+2)) > \phi_2$.
If $r_1,r_2$ are the regions of $L''$ meeting $\Delta_{w_1}$ and $\Lambda_S'$, respectively, then $m_{\T_{L''}}(r_1,r_2) \geq m_{\T_{L''}}(W_1,U_S)-2 \geq m_{\T_{L'}}(W_1,U_S)-2-(\kappa+1+k_2)(4(\lambda_1+\lambda_2+25)+2) \geq (m_{\T_{L'}}(w_1,\overline{\Omega}) - \lambda_2-\lambda_1)-2-(\kappa+1+k_2)(4(\lambda_1+\lambda_2+25)+2) > \lambda_3 - (\lambda_2+\lambda_1+2+(\kappa+1+k_2)(4(\lambda_1+\lambda_2+25)+2)) > \phi_2$.
Similarly, for any $a \in \{x_i: i \in [k_2]\}$, if $r_a,r_1$ are the regions of $L''$ meeting $\Delta_a$ and $\Lambda_S'$, respectively, then $m_{\T_{L''}}(r_a,r_1) > \phi_2$.

Recall that $U_a$ is free with respect to $\T_{L''}$ for every $a \in \{x_i: i \in [k_2]\}$.
By Claim 4, $U_S$ is free with respect to $\T_{L''}$.
By (iii) and Claim 1, we know that (iv) implies that $W_1$ is free with respect to $\T_{L''}$.
Then this claim follows from Theorem \ref{linkage on surface}.
$\Box$

\medskip

Note that for every $i \in [k_2]$, there is a $\Delta_{1,s+1}$-$\Delta_S$ line $L_i$ in $\Sigma - (\Delta_{1,s+1} \cup \Delta_{S})$ contained in $Q_{x_i} \cup Q_{x_i}' \cup D_{x_i}$.
By Lemma \ref{homotopic_counts_2}, $k_2/\rho_0 = k_1$ of them are pairwise homotopic in $\Sigma - (\Delta_{1,s+1} \cup \Delta_{S})$.

Without loss of generality, we may assume that the ends in $U_S$ of $Q_{x_1},Q_{x_2},...,Q_{x_{k_1}}$ appear in $V(C_S)$ in the order listed, the ends in $W_1$ of $Q_{x_1}',Q_{x_2}',...,Q_{x_{k_1}}'$ appear in $V(D_{1,s+1})$ in the order listed, and the union of $L_1 \cup L_{k_1}$, the subpath of $V(C_S)$ disjoint from $Q_{x_2}$ and the subpath of $V(D_{1,s+1})$ disjoint from $Q_{x_2}'$ bounds a disk in $\Sigma$.

By reindexing, we may assume that for every $i \in [k_1]$, the end of $Q_{x_i}$ in $U_S$ is $u_i$, and the end of $Q_{x_i}'$ in $W_1$ is $w_{1,i}$.
For every $i \in [k_1]$, let $W_{1,i}^-$ be the subpath of $W_{1,i}$ between $V(D_{1,1})$ and $w_{1,i}$ internally disjoint from $V(D_{1,1})$.

For every $i \in [k_1]$, by (xi), there exists a path in $L'$ from $v_{x_i}$ to the end of $Q_{x_i}$ in $D_{x_i}$ internally disjoint from $V(D_{x_i})$; let $P^L_i$ be the union of this path and the path $Q_{x_i}$.
For every $i \in [k_1]$, let $P_i$ be a union of a path in $S_{x_i}$ from $x_i$ to $v_{x_i}$ internally disjoint from $\overline{\Omega_{x_i}}$, a lifting of $P^L_i$, and the subpath of $P^S_i$ between $u_i$ and its end in $V(S)$; note that $P_i$ is a path in $L$.

Recall $k_1 = 9r(k_0+1)$.
For every $i \in [k_0]$, let $P^*_i = P_{9ri}$.

\medskip

\noindent{\bf Claim 6:} $P^*_i$ is an $(\ell,X,Y)$-path in $(L,G)$ for every $i \in [k_0]$.

\noindent{\bf Proof of Claim 6:}
Note that for every $i \in [k_0]$, the ends of $P^*_i$ are in $L-N_G^{\leq r+\ell}[Z]$.
So if there exists a path in $G$ with length at most $\ell$ between them, then this path is in $L$ (since $N_G(V(L)) \subseteq Z$), so the natural projection of this path shows $m_{\T_L}(v_{x_i},\overline{\Omega}) \leq \ell+1 < \lambda_3$, a contradiction.
$\Box$

\medskip

We can construct a 2-edge-coloring of the complete graph $K_{k_0}$ with $V(K_{k_0})=[k_0]$ by coloring each edge $ij$ according to whether $\dist_G(V(P_i^*),V(P_j^*))$ is at most $r$ or not.
Since $k_0 \geq R(8\rho'+8,k)$ and there do not exist $k$ indices $i_1,i_2,...,i_k$ such that $P^*_{i_1},P^*_{i_2},...,P^*_{i_k}$ have pairwise distance in $G$ at least $r$ (by Claim 6 and the assumption of this lemma), we know that there exist distinct $j_1<j_2<...<j_{8\rho'+8}$ in $[k_0]$ such that $P^*_{j_1},P^*_{j_2},...,P^*_{j_{8\rho'+8}}$ have pairwise distance in $G$ at most $r$. 

For every $i \in [k_0]$, let $I_{i} = P^*_i \cap \bigcup_{(S',\Omega') \in \Se, \alpha(S',\Omega') \subseteq \overline{\Delta_{S,s}}}S'$; note that $I_{i}$ is a subgraph of $P^S_{9ri} \cap \bigcup_{(S',\Omega') \in \Se, \alpha(S',\Omega') \subseteq \overline{\Delta_{S,s}}}S'$; let $M_i = P^*_i-V(I_i)$.

\medskip

\noindent{\bf Claim 7:} For any $\beta_1,\beta_2 \in \{j_i: i \in [8\rho'+8]\}$, if $a_1 \in V(P^*_{\beta_1})$ and $a_2 \in V(P^*_{\beta_2})$ with $\dist_G(a_1,a_2) \leq r$, then there exists $i \in [2]$ such that $a_i \in V(M_{\beta_i})$ and $a_{3-i} \in V(I_{\beta_{3-i}})$.

\noindent{\bf Proof of Claim 7:} 
By symmetry, we may assume $\beta_1<\beta_2$.
Let $K$ be the radial drawing of $L'$.
Since $\alpha$ is $(8,r,\Se_1,\Se_2)$-projected, there exists a path $P$ in $L' \cup K$ of length at most $8r$ from $\overline{\Omega_{a_1}}$ to $\overline{\Omega_{a_2}}$ for some $(S_{a_1},\Omega_{a_1}),(S_{a_2},\Omega_{a_2}) \in \Se$ with $a_1 \in V(S_{a_1})$ and $a_2 \in V(S_{a_2})$.

We first suppose that $a_1 \in V(M_{\beta_1})$ and $a_2 \in V(M_{\beta_2})$.
Since $s-1>8r$, by the existence of the $s$ disjoint cycles $D_{S,1},D_{S,2},...,D_{S,s}$ and the $s$ disjoint cycles $D_{1,1},D_{1,2},...,D_{1,s}$, we know that $P$ is disjoint from the vertices of $K$ corresponding to the regions of $L'$ intersecting $\alpha(S,\Omega) \cup \Delta_{1,1}$.
By the planarity, we know that $P$ intersects $Q_{x_i} \cup D_{x_i} \cup Q'_{x_i} \cup W_{1,i}^-$ either for each $i$ with $9r\beta_1 < i < 9r\beta_2$ or for each $i$ with $1 \leq i <9r\beta_1$.
Hence $P$ has length at least $9r>8r$, a contradiction.

So by symmetry, we may assume $a_1 \not \in V(M_{\beta_1})$.
Hence $a_1 \in V(I_{\beta_1})$.

Since $I_{i}$ is a subgraph of $P^S_{9ri} \cap \bigcup_{(S',\Omega') \in \Se, \alpha(S',\Omega') \subseteq \overline{\Delta_{S,s}}}S'$ for every $i \in [k_0]$, (ix) implies that $a_2 \not \in V(I_{\beta_2})$ since $a_1 \in V(I_{\beta_1})$.
So $a_2 \in V(M_{\beta_2})$.
This proves the claim.
$\Box$

\medskip

By Claim 7, we can define a tournament on $[8\rho'+8]$ such that for any distinct $\beta,\gamma$, $(\beta,\gamma)$ is an arc if there exist $v_\beta \in V(M_\beta)$ and $v_\gamma \in V(I_{\gamma})$ such that $\dist_G(v_\beta,v_\gamma) \leq r$.
Then some vertex has out-degree at least $4\rho'+4$ in this tournament.
That is, there exist distinct $i^*,\beta_1,\beta_2,...,\beta_{4\rho'+4} \in \{j_i: i \in [8\rho'+8]\}$ such that for every $i \in [4\rho'+4]$, there exist $a_i \in V(M_{i^*})$ and $b_i \in V(I_{\beta_i})$ such that $\dist_G(a_i,b_i) \leq r$.

We may further assume that there are at least $2\rho'+2$ elements of $\{\beta_i: i \in [4\rho'+4]\}$ greater than $i^*$, since the case that at least $2\rho'+2$ elements of $\{\beta_i: i \in [4\rho'+4]\}$ smaller than $i^*$ is analogous.
Without loss of generality we may assume that $i^* < \beta_1 < \beta_2 < ... < \beta_{2\rho'+2}$.

For every $i$, let $P^-_i$ be the subpath of $P^*_i$ between $\overline{\Omega}$ and $V(D_{x_{9ri}})$ internally disjoint from $\overline{\Omega} \cup V(D_{x_{9ri}})$.
Let $t_{i^*}$ be the end of $P^-_{i^*}$ in $\overline{\Omega}$.
For every $i \in [2\rho'+2]$, let $t_{\beta_i}$ be the end of $P^-_{\beta_i}$ in $\overline{\Omega}$.

Since $(S,\Omega)$ is a $\rho'$-vortex, there exists a vortical decomposition $(W,(X_t: t \in V(W))$ of $(S,\Omega)$ of adhesion at most $2\rho'$ such that $V(W)=\overline{\Omega}$ and $t \in X_t$ for every $t \in V(W)=\overline{\Omega}$. 
Let $t'$ and $t''$ be the neighbors of $t_{\beta_1}$ and $t_{\beta_{2\rho'+2}}$ in $W$, respectively, such that the open interval of $\Omega$ disjoint from $t_{i^*}$ contains $t_{\beta_i}$ for all $i \in [2\rho'+2]$.

Let $K$ be the radial drawing of $L'$.
Since $\alpha$ is $(8,r,\Se_1,\Se_2)$-projected, for every $i \in [2\rho'+2]$, there exists a path $R_i$ in $L' \cup K$ of length at most $8r$ from $\overline{\Omega_{a_i}}$ to $\overline{\Omega_{b_i}}$ for some $(S_{a_i},\Omega_{a_i}),(S_{b_i},\Omega_{b_i}) \in \Se$ with $a_i \in V(S_{a_i})$ and $b_i \in V(S_{b_i})$.
For every $i \in [2\rho'+2]-\{1\}$, since $a_i \in V(M_{i^*})$, by the existence of the $s$ disjoint cycles $D_{S,1},D_{S,2},...,D_{S,s}$ and the $s$ disjoint cycles $D_{1,1},D_{1,2},...,D_{1,s}$, we know that $R_i$ is disjoint from the vertices of $K$ corresponding to the regions of $L'$ intersecting $\alpha(S,\Omega) \cup \Delta_{1,1}$; and the planarity and the existence of $Q_{x_j} \cup D_{x_j} \cup Q'_{x_j} \cup W_{1,j}^-$ for $i$ with $9ri^* < j < 9r(i^*+1)$ and for $i$ with $1 \leq j <9ri^*$ imply that $P^*_{\beta_i}$ intersects $(X_{t'} \cap X_{t_{\beta_1}}) \cup (X_{t''} \cap X_{t_{\beta_{2\rho'+2}}})$.
So $(X_{t'} \cap X_{t_{\beta_1}}) \cup (X_{t''} \cap X_{t_{\beta_{2\rho'+2}}})| \geq 2\rho'+1$, contradicting that the adhesion of $(W,(X_t: t \in V(W)))$ is at most $2\rho'$.
This proves the lemma.
\end{pf}

\bigskip

Now we are ready to construct a hitting set for all $(\ell,X,Y)$-paths in $(L,G)$ if there are no apices.

\begin{lemma} \label{hitting_no_apex}
For any $k,r,\kappa,\rho \in {\mathbb N}$ and $\ell,g \in {\mathbb N}_0$, there exist $\phi,\theta \in {\mathbb N}$ such that for any $\xi,\eta \in {\mathbb N}$, there exist $\xi^*, \eta^* \in {\mathbb N}$ with $\eta^*=4\eta+6r+\ell$ such that the following hold. 

If $(G,L,Z,\xi,\eta,X,Y,\F,\ell,k, \allowbreak r,\T_L,\theta',\theta,\Se,\kappa,\rho,\Se_1,\Se_2,\Sigma,g,\phi,\alpha)$ is an interesting tuple, then there exists $Z^* \subseteq V(G)$ with $Z^* \supseteq N_G^{\leq r+\ell}[Z]$ such that $Z^*$ is $(\xi^*,\eta^*)$-centered in $G$ and intersects all $(\ell,X,Y)$-paths in $(L,G)$.
\end{lemma}

\begin{pf}
Let $k,r,\kappa,\rho \in {\mathbb N}$, and let $\ell,g \in {\mathbb N}_0$.
	\begin{itemize}
		\item Let $\phi_0$ be the nondecreasing function $\phi^*$ mentioned in Lemma \ref{hitting_all_vortices} by taking $(k,r,s,k',\ell,g) \allowbreak =(k,r,1,1,\ell,g)$. 
		\item Let $\kappa_1$ be the integer $\kappa^*=\kappa_1(\rho)$ mentioned in Lemma \ref{hitting_surface} by taking $(k,r,\kappa,g,\ell,\rho)=(k,r,\kappa,g,\ell,\rho)$.
		\item Let $\phi_1: {\mathbb Z} \rightarrow {\mathbb R}$ be the function $\phi$ mentioned in Lemma \ref{hitting_surface} by taking $(k,r,\kappa,g,\ell,\phi^*)=(k,r,\kappa_1,g,\ell,\phi_0)$. 
		\item Define $\phi$ be the function $\phi$ mentioned in Lemma \ref{hitting_surface} by taking $(k,r,\kappa,g,\ell,\phi^*)=(k,r,\kappa,g,\ell,\phi_1)$.  
		\item Let $\rho_2$ be the integer $\rho^*$ mentioned in Lemma \ref{hitting_surface} by taking $(k,r,\kappa,g,\ell,\phi^*,\rho)=(k,r,\kappa,g,\ell,\phi_1,\rho)$.
		\item Let $\kappa^*,\rho^*$ be the integers $\kappa^*,\rho^*$, respectively, mentioned in Lemma \ref{hitting_surface} by taking $(k,r,\kappa,g,\ell,\phi^*,\rho)=(k,r,\kappa_1,g,\ell,\phi_0,\rho_2)$. 
		\item Let $\theta_4$ be the integer $\theta$ mentioned in Lemma \ref{hitting_all_vortices} by taking $(k,r,s,k',\ell,g,\rho,\kappa)=(k,r,1,1,\ell,g,\rho^*,\kappa^*)$.
		\item Let $\theta_3$ be the integer $\theta$ mentioned in Lemma \ref{hitting_surface} by taking $(k,r,\kappa,g,\ell,\phi^*,\rho,\theta^*)=(k,r,\kappa_1,g,\ell,\phi_0,\rho_2,\theta_4)$.
		\item Let $\theta_2$ be the integer $\theta$ mentioned in Lemma \ref{hitting_surface} by taking $(k,r,\kappa,g,\ell,\phi^*,\rho,\theta^*)=(k,r,\kappa,g,\ell,\phi_1,\rho,\theta_3)$.
		\item Define $\theta = \theta_2+\theta_3+\theta_4$.
	\end{itemize}
Let $\xi,\eta \in {\mathbb N}$.
	\begin{itemize}
		\item Let $\xi_3,\eta_3$ be the integers $\xi^*,\eta^*$, respectively, mentioned in Lemma \ref{hitting_surface} by taking $(k,r,\kappa,g,\ell,\phi^*,\rho,\xi,\eta)=(k,r,\kappa_1,g,\ell,\phi_0,\rho_2,\xi,\eta)$. 
		\item Let $\xi_4,\eta_4$ be the integers $\xi^*,\eta^*$, respectively, mentioned in Lemma \ref{hitting_all_vortices} by taking $(k,r,s,k',\ell,g,\rho,\kappa,\xi,\eta)=(k,r,1,1,\ell,g,\rho^*,\kappa^*,\xi,\eta)$.
		\item Define $\xi^* = \xi_3+\xi_4$ and $\eta^* = \max\{\eta_3,\eta_4\}$.
	\end{itemize}
Note that $\eta_3=\eta_4=\eta^*=4\eta+6r+\ell$.

Let $G,L,Z,X,Y,\F,\T_L,\Se,\Se_1,\Se_2,\alpha,\Sigma$ be as stated in the lemma.

By Lemma \ref{hitting_surface} (with $(k,r,\kappa,g,\ell,\phi^*,\rho,\theta^*,\xi,\eta)=(k,r,\kappa,g,\ell,\phi_1,\rho,\theta_3,\xi,\eta)$), there exists a $\T_L$-central segregation $\Se'$ of $L$ with a $(\kappa_1,\rho_2)$-witness $(\Se_1',\Se_2')$ such that the following hold.
	\begin{itemize}
		\item[(i)] There exists an $(8,r,\Se_1',\Se_2')$-protected $(\Sigma,\theta_3,\phi_1,\T_L)$-arrangement of $\Se'$ with respect to $(\Se'_1,\Se'_2)$ in $\Sigma$.
		\item[(ii)] $X-N_G^{\leq r+\ell}[Z] \subseteq \bigcup_{(S,\Omega) \in \Se_2'}V(S)$ or $Y-N_G^{\leq r+\ell}[Z] \subseteq \bigcup_{(S,\Omega) \in \Se_2'}V(S)$.
	\end{itemize}
By applying Lemma \ref{hitting_surface} to the segregation $\Se'$ (with $(k,r,\kappa,g,\ell,\phi^*,\rho,\theta^*,\xi,\eta)=(k,r,\kappa_1,g,\ell, \allowbreak \phi_0,\rho_2, \allowbreak \theta_4,\xi,\eta)$), there exists $Z_3 \subseteq V(G)$ with $Z_3 \supseteq N_G^{\leq r+\ell}[Z]$ such that $Z_3$ is $(\xi_3,\eta_3)$-centered in $G$, and there exists a $\T_L$-central segregation $\Se^*$ of $L$ with a $(\kappa^*,\rho^*)$-witness $(\Se_1^*,\Se_2^*)$ such that the following hold.
	\begin{itemize}
		\item[(iii)] $\bigcup_{(S,\Omega) \in \Se_2^*}S \supseteq \bigcup_{(S,\Omega) \in \Se_2'}S$. 
		\item[(iv)] There exists an $(8,r,\Se_1^*,\Se_2^*)$-protected $(\Sigma,\theta_4,\phi_0,\T_L)$-arrangement of $\Se^*$ with respect to $(\Se^*_1,\Se^*_2)$ in $\Sigma$.
		\item[(v)] If there exists $\Xi \in \{X,Y\}$ such that $\Xi-N_G^{\leq r+\ell}[Z] \not \subseteq \bigcup_{(S,\Omega) \in \Se_2^*}V(S)$, then $Z_3$ intersects all $(\ell,X,Y)$-paths in $(L,G)$ having an end in $\Xi' \cap \bigcup_{(S,\Omega) \in \Se_2'}V(S)$, where $\Xi' \in \{X,Y\}-\{\Xi\}$.
	\end{itemize}

By (ii) and symmetry, we may assume $X-N_G^{\leq r+\ell}[Z] \subseteq \bigcup_{(S,\Omega) \in \Se_2'}V(S)$.
By (iii), $X-N_G^{\leq r+\ell}[Z] \subseteq \bigcup_{(S,\Omega) \in \Se_2^*}V(S)$.

Suppose that this lemma does not hold.

\medskip

\noindent{\bf Claim 1:} $Y-N_G^{\leq r+\ell}[Z] \subseteq \bigcup_{(S,\Omega) \in \Se_2^*}V(S)$.

\noindent{\bf Proof of Claim 1:}
Suppose to the contrary that $Y-N_G^{\leq r+\ell}[Z] \not \subseteq \bigcup_{(S,\Omega) \in \Se_2^*}V(S)$.
By (v), $Z_3$ intersects all $(\ell,X,Y)$-paths in $(L,G)$ having an end in $X \cap \bigcup_{(S,\Omega) \in \Se_2'}V(S)$.
Since $X-N_G^{\leq r+\ell}[Z] \subseteq \bigcup_{(S,\Omega) \in \Se_2'}V(S)$, we know that $Z_3$ intersects all $(\ell,X,Y)$-paths in $(L,G)$ having an end in $X-N_G^{\leq r+\ell}[Z]$.
Since $Z_3 \supseteq N_G^{\leq r+\ell}[Z]$, $Z_3$ intersects all $(\ell,X,Y)$-paths in $(L,G)$ having an end in $X$.
Since $Z_3$ is $(\xi_3,\eta_3)$-centered in $G$, it is also $(\xi^*,\eta^*)$-centered in $G$.
So this lemma holds, a contradiction.
$\Box$

\medskip

By (iii) and Claim 1, $(X \cup Y)-N_G^{\leq r+\ell}[Z] \subseteq \bigcup_{(S,\Omega) \in \Se_2^*}V(S)$.
By applying Lemma \ref{hitting_all_vortices} to the segregation $\Se^*$ (taking $(k,r,s,k',\ell,g,\rho,\kappa,\xi,\eta)=(k,r,1,1,\ell,g,\rho^*,\kappa^*,\xi,\eta)$), there exists $Z^* \subseteq V(G)$ with $Z^* \supseteq N_G^{\leq r+\ell}[Z]$ such that $Z^*$ is $(\xi_4,\eta_4)$-centered in $G$ such that there exists $\Xi \in \{X,Y\}$ such that $Z^*$ intersects all $(\ell,X,Y)$-paths in $(L,G)$ having an end in $\Xi \cap \bigcup_{(S,\Omega) \in \Se_2^*}V(S)$.
So $Z^*$ intersects all $(\ell,X,Y)$-paths in $(L,G)$ having an end in $\Xi-N_G^{\leq r+\ell}[Z]$.
Since $Z^* \supseteq N_G^{\leq r+\ell}[Z]$, $Z^*$ intersects all $(\ell,X,Y)$-paths in $(L,G)$ having an end in $\Xi$.
That is, $Z^*$ intersects all $(\ell,X,Y)$-paths in $(L,G)$.
Note that $Z^*$ is also $(\xi^*,\eta^*)$-centered in $G$.
This proves the lemma.
\end{pf}

\section{Minor-closed families}

In this section, we prove the weak coarse Menger property for minor-closed families.
We will prove the case when all graphs in the class are finite in Section \ref{subsec:finite_minor_closed}.
Then we use it to prove the case for locally finite graphs in Section \ref{subsec:infinite}.
In Section \ref{subsec:apex_minor_free}, we give a simpler proof for apex-minor free (finite) graphs with a better bound on the radius of the balls in the hitting set.

\subsection{Finite graphs} \label{subsec:finite_minor_closed}

\begin{lemma} \label{control_1}
For any non-planar graph $H$ and $k,r,\ell \in {\mathbb N}$, there exists $\theta \in {\mathbb N}$ such that for any $\xi,\eta \in {\mathbb N}$, there exist $\xi^*,\eta^* \in {\mathbb N}$ with $\eta^*=4\eta+(4g+14)r+\ell$, where $g$ is the largest nonnegative integer such that $H$ cannot be drawn in some surface of Euler genus $g$, such that the following holds.

Let $G$ be a graph, and let $L$ be an induced subgraph of $G$.
Let $Z \subseteq V(G)$ be a $(\xi,\eta)$-centered set in $G$ such that $N_G(V(L)) \subseteq Z$.
Let $X$ and $Y$ be subsets of $V(L)$. 
Let $\F$ be the set of all $(\ell,X,Y)$-paths in $(L,G)$. 
Assume that no $k$ members of $\F$ have pairwise distance in $G$ at least $r$.
Let $\T_L$ be the $(G,\F,r,\theta',Z)$-tangle in $L$ for some integer $\theta' \geq \theta$.
Let $\T_G$ be a tangle of order at least $\theta$ induced by $\T_L$.
If $\T_G$ does not control an $H$-minor, then there exists $Z^* \subseteq V(G)$ such that $Z^*$ is $(\xi^*,\eta^*)$-centered in $G$ and intersects all $(\ell,X,Y)$-paths in $(L,G)$.
\end{lemma}

\begin{pf}
Let $H$ be a non-planar graph, and let $k,r,\ell \in {\mathbb N}$.
	\begin{itemize}
		\item Let $\kappa_1,\rho_1,\xi_1,\theta_1$ be the integers $\kappa,\rho,\xi,\theta$, respectively, mentioned in Theorem \ref{minor_structual_thm} by taking $H=H$.
		\item Let $g$ be the largest nonnegative integer such that $H$ cannot be drawn in some surface of Euler genus $g$.
		\item Let $\kappa_2$ be the integer $\kappa^*$ mentioned in Lemma \ref{avoid_apex_new_2} by taking $(g,\kappa)=(g,\kappa_1)$.
		\item Let $\phi_1: {\mathbb Z} \rightarrow {\mathbb R}$ and $\theta_2': {\mathbb Z} \rightarrow {\mathbb R}$ be the functions such that for every $x \in {\mathbb Z}$, $\phi_1(x)$ equals the integer $\phi$ and $\theta'(x)$ equals the integer $\theta$, respectively, mentioned in Lemma \ref{hitting_no_apex} by taking $(k,r,\kappa,\rho,\ell,g)=(k,r,\kappa_2,\max\{x,1\},\ell,g)$.
		\item Let $\rho_2$ be the integer $\rho^*$ mentioned in Lemma \ref{avoid_apex_new_2} by taking $(g,\kappa,\rho,\phi^*) = (g,\kappa_1,\rho_1,\phi_1)$.
		\item Let $\theta_2 = \max_{0 \leq x \leq \rho_2}\theta'_2(x)$.
		\item Let $\theta_3$ be the integer $\theta$ mentioned in Lemma \ref{avoid_apex_new_2} by taking $(g,\kappa,\rho,\phi^*,k,r,\xi_{-1},\theta^*) = (g,\kappa_1,\rho_1,\phi_1,k,r,\xi_1,\theta_2)$.
		\item Define $\theta = \theta_1+\theta_2+\theta_3$.
	\end{itemize}
Let $\xi,\eta \in {\mathbb N}$.
	\begin{itemize}
		\item Let $\xi_3,\eta_3$ be the integers $\xi^*,\eta^*$, respectively, mentioned in Lemma \ref{avoid_apex_new_2} by taking $(g,\kappa,\rho,\phi^*,k,r,\xi_{-1},\theta^*,\xi,\eta) = (g,\kappa_1,\rho_1,\phi_1,k,r,\xi_1,\theta_2,\xi+\xi_1,\eta)$.
		\item For every positive integer $x$, let $\xi^*_x,\eta^*_x$ to be the integers $\xi^*,\eta^*$, respectively, mentioned in Lemma \ref{hitting_no_apex} by taking $(k,r,\kappa,\rho,\ell,g,\xi,\eta)=(k,r,\kappa_2,x,\ell,g,\xi_3,\eta_3)$.
		\item Let $\xi_4 = \max\{\xi_3,\max_{1 \leq x \leq \rho_2}\xi^*_x\}$ and $\eta_4 = \max_{1 \leq x \leq \rho_2}\eta^*_x$.
		\item Define $\xi^* = (k-1)\xi_4$ and $\eta^* = \max\{\eta_3,\eta_4\}$. 
	\end{itemize}
Note that $\eta_3 = \eta+(g+2)r$ and $\eta_4 = 4\eta_3+6r+\ell$.
So $\eta^*=4\eta+(4g+14)r+\ell$.

Let $G,L,Z,X,Y,\F,\T_L,\T_G$ be as stated in the lemma.

Since $\T_G$ does not control an $H$-minor, Theorem \ref{minor_structual_thm} implies that there exist $Z_1 \subseteq V(G)$ with $|Z_1| \leq \xi_1$ and a $(\T_G-Z)$-central $(\kappa_1,\rho_1)$-segregation $\Se$ of $G-Z$ such that there exists a proper arrangement of $\Se$ in a surface of Euler genus at most $g$.
Let $Z^+ = Z_1 \cup Z$.
So $Z^+ \supseteq Z_1 \cup N_G(V(L))$ and $Z^+$ is $(\xi+\xi_1,\eta)$-centered in $G$.
Since every $(\ell,X,Y)$-path is a connected subgraph of $L$, Lemma \ref{avoid_apex_new_2} implies that there exist a set $\W$ of pairwise disjoint induced subgraphs of $L-N_G^{\leq r}[Z^+]$ and a collection $(Z_W: W \in \W)$ of subsets of $V(G)$ such that 
	\begin{itemize}
		\item[(i)] $N_G^{\leq r}[\bigcup_{W \in \W}Z_W]$ intersects all members of $\F-(\F \cap \bigcup_{W \in \W}W)$, 
		\item[(ii)] for every $W \in \W$, we have $Z_W \supseteq N_G(V(W))$, $Z_W$ is $(\xi_3,\eta_3)$-centered in $G$, and $(\F \cap W)-N_G^{\leq r}[Z_W] \neq \emptyset$,
		\item[(iii)] for every $W \in \W$, either
			\begin{itemize}
				\item[(iiia)] there exists $Z_W^* \subseteq V(G)$ with $Z_W^* \supseteq N_G^{\leq r}[Z_W]$ such that $Z_W^*$ is $(\xi_3,\eta_3)$-centered in $G$ and intersects all members of $\F \cap W$, or
				\item[(iiib)] there exist a $(G, \F \cap W,r,\theta_3,Z_W)$-tangle $\T_W$ in $W$, a segregation $\Se^W$ of $W$ with a $(\kappa_2,\rho_2)$-witness $(\Se^W_1,\Se^W_2)$, and an $(8,r,\Se_1^W,\Se_2^W)$-protected $(\Sigma_W,\theta_2,\phi_1,\T_W)$-arrangement $\alpha_W$ of $\Se^W$ with respect to $(\Se^W_1,\Se^W_2)$ for some surface $\Sigma_W$ of Euler genus at most $g$.
			\end{itemize}
	\end{itemize}

Note that $Z_W^*$ is defined for every $W \in \W$ satisfying (iiia).
Now we define $Z_W^*$ for every $W \in \W$ not satisfying (iiia).
Let $W$ be a member of $\W$ not satisfying (iiia).
So $W$ satisfies (iiib).
Let $\rho_W$ be the smallest nonnegative integer such that $(\Se_1^W,\Se_2^W)$ is a $(\kappa_2,\rho_W)$-witness of $\Se^W$.
So $\alpha_W$ is an $(8,r,\Se_1^W,\Se_2^W)$-protected $(\Sigma_W,\theta_2,\phi_1(\rho_W),\T_W)$-arrangement and hence a $(\Sigma_W,\theta_2'(\rho_W),\phi_1(\rho_W), \allowbreak \T_W)$-arrangement. 
Hence $(G,W,Z_W,\xi_3,\eta_3,X \cap V(W),Y \cap V(W),\F \cap W,\ell,k,r,\T_W,\theta_3,\theta_2'(\rho_W), \allowbreak \Se^W,\kappa_2,\rho_W,\Se_1^W,\Se_2^W,\Sigma_W,g,\phi_1(\rho_W),\alpha_W)$ is an interesting tuple.
By Lemma \ref{hitting_no_apex}, there exists $Z_W^* \subseteq V(G)$ with $Z_W^* \supseteq N_G^{\leq r+\ell}[Z_W]$ such that $Z_W^*$ is $(\xi^*_{\rho_W},\eta^*_{\rho_W})$-centered in $G$ and intersects all $(\ell,X,Y)$-paths in $(W,G)$.

So we defined $Z_W^*$ for every $W \in \W$.
Note that for every $W \in \W$, $Z_W^*$ is $(\xi_4,\eta^*)$-centered in $G$ with $Z_W^* \supseteq N_G^{\leq r}[Z_W]$ and intersects all members of $\F \cap W$.

Since members of $\W$ are pairwise disjoint and $N_G(V(W)) \subseteq Z_W$ for every $W \in \W$ (by (ii)), every set consisting of a member of $(\F \cap W)-N_G^{\leq r}[Z_W]$ for each $W \in \W$ is a set of members of $\F$ pairwise at distance in $G$ at least $r$.
So (ii) implies that $|\W| \leq k-1$.

Let $Z^* = \bigcup_{W \in \W}Z_W^*$.
Then $Z^*$ is $(|\W|\xi_4,\eta^*)$-centered in $G$ and hence $(\xi^*,\eta^*)$-centered in $G$.
Since $Z^* \supseteq N_G^{\leq r}[\bigcup_{W \in \W}Z_W]$, we know that $Z^*$ intersects all members of $\F$ by (i).
This proves the lemma.
\end{pf}

\bigskip

Now we prove Theorem \ref{finite_minor_induction_intro}, which is the main result for finite graphs in a minor-closed family in this paper.
Recall that the planar $H$ case of Theorem \ref{finite_minor_induction_intro} follows from Lemma \ref{easy_tangle_1}.
The non-planar $H$ case follows from the following theorem, which is a cleaner version of Lemma \ref{control_1}.

\begin{theorem} \label{finite_minor_induction}
For any non-planar graph $H$, any $k,r \in {\mathbb N}$ and $\ell,\xi,\eta \in {\mathbb N}_0$, there exist $\xi^*,\eta^* \in {\mathbb N}$ with $\eta^*=4\eta+(4g+22)r+\ell$, where $g$ is the largest nonnegative integer such that $H$ cannot be drawn in some surface of Euler genus $g$, such that the following holds. 
Let $G$ be an $H$-minor free graph, and let $L$ be an induced subgraph of $G$.
Let $Z \subseteq V(G)$ be a $(\xi,\eta)$-centered set in $G$ such that $N_G(V(L)) \subseteq Z$.
Let $X$ and $Y$ be subsets of $V(L)$. 
If no $k$ $(\ell,X,Y)$-paths in $(L,G)$ have pairwise distance in $G$ at least $r$, then there exists $Z^* \subseteq V(G)$ such that $Z^*$ is $(\xi^*,\eta^*)$-centered in $G$ intersecting all $(\ell,X,Y)$-paths in $(L,G)$.
\end{theorem}

\begin{pf}
Let $H$ be a non-planar graph, and let $k,r \in {\mathbb N}$ and $\ell,\xi,\eta \in {\mathbb N}_0$.
Let $g$ be the largest nonnegative integer such that $H$ cannot be drawn in some surface of Euler genus $g$.
Let $\theta$ be the integer $\theta$ mentioned in Lemma \ref{control_1} by taking $(H,k,r,\ell)=(H,k,r,\ell)$.
Let $\xi_1 = \xi+\max\{2k-3,1\}(3\theta-3)$.
Let $\xi_2,\eta_2$ to be the integers $\xi^*,\eta^*$, respectively, mentioned in Lemma \ref{control_1} by taking $(H,k,r,\ell,\xi,\eta)=(H,k,r,\ell,\xi_1,\eta+2r)$.
Define $\xi^*=\xi_1+(k-1)\xi_2$ and $\eta^*=\max\{\eta+2r,\eta_2\}$.
Note that $\eta_2 = 4(\eta+2r)+(4g+14)r+\ell=4\eta+(4g+22)r+\ell$, so $\eta^*=\eta_2$.

Let $G,L,Z,X,Y$ be as stated in the lemma.
Let $\F$ be the set of $(\ell,X,Y)$-paths in $(L,G)$.
By Lemma \ref{easy_tangle_1}, there exist $Z_1 \subseteq V(G)$ and a set $\W$ of disjoint induced subgraphs of $L$ with $|\W| \leq k-1$ such that 
	\begin{itemize}
		\item $Z_1$ is $(\xi+\max\{2k-3,1\}(3\theta-3),\eta+2r\}$-centered set in $G$ and intersects all members of $\F-(\F \cap \bigcup_{W \in \W}W)$, and
		\item for every $W \in \W$, there exists a $(G,\F \cap W,r,\theta,Z_1)$-tangle $\T_W$ in $W$ and $N_G(V(W)) \subseteq Z_1$.
	\end{itemize}
For every $W \in \W$, let $\T_G^W$ be the tangle in $G$ of order $\theta$ induced by $\T_W$; note that $\T_W$ does not control an $H$-minor since $G$ is $H$-minor free, so by Lemma \ref{control_1} (with $(G,L,Z,X,Y,\F,\T_L,\T_G) = (G,W,Z_1,X \cap W,Y \cap W,\F \cap W,\T_W,\T_G^W)$), there exists $Z_W^* \subseteq V(G)$ such that $Z_W^*$ is $(\xi_2,\eta_2)$-centered in $G$ and intersects all members of $\F \cap W$.

Let $Z^*=Z_1 \cup \bigcup_{W \in \W}Z_W^*$.
Then $Z^*$ is $(\xi_1+(k-1)\xi_2,\max\{\eta+2r,\eta_2\})$-centered set in $G$ and intersects all members of $\F$.
\end{pf}

\subsection{Locally finite infinite graphs} \label{subsec:infinite}

In this section, we will prove a version of Theorem \ref{finite_minor_induction} for locally finite infinite graphs by using a compactness argument.
We use a result of Gottschalk \cite{g} about choice functions.
A \defn{choice function} of a family $(X_\alpha)_{\alpha \in I}$ of sets is a function $c$ such that $c(\alpha) \in X_\alpha$ for every $\alpha \in I$.

\begin{theorem}[\cite{g}] \label{compactness}
Let $I$ be a set.
Let $(X_\alpha)_{\alpha \in I}$ be a family of finite sets, let $\A$ be the class of all finite subsets of $I$.
For every $A \in \A$, let $c_A$ be a choice function of $(X_\alpha)_{\alpha \in A}$.
Then there exists a choice function $c$ of $(X_\alpha)_{\alpha \in I}$ such that for every $A \in \A$, there exists $B \in \A$ with $B \supseteq A$ such that $c(\alpha)=c_B(\alpha)$ for every $\alpha \in A$.
\end{theorem}

\begin{theorem} \label{infinite_minor}
For any (finite) graph $H$, any $k,r \in {\mathbb N}$ and $\ell \in {\mathbb N}_0$, there exist $\xi^*,\eta^* \in {\mathbb N}$ and $c \in {\mathbb R}_{>0}$ with $\eta^* \leq c(r+\ell)$ such that the following holds. 
Let $G$ be an $H$-minor free finite or locally finite infinite graph, and let $X$ and $Y$ be subsets of $V(G)$. 
If $k$ is a positive integer such that no $k$ $(\ell,X,Y)$-paths in $G$ have pairwise distance in $G$ at least $r$, then there exists $Z^* \subseteq V(G)$ such that $Z^*$ is $(\xi^*,\eta^*)$-centered in $G$ and intersects all $(\ell,X,Y)$-paths in $G$.
\end{theorem}

\begin{pf}
Let $H$ be a finite graph, and let $k,r \in {\mathbb N}$ and $\ell \in {\mathbb N}_0$.
By Theorem \ref{finite_minor_induction}, there exist functions $\xi,\eta$ such that for any finite $K_{|V(H)|+1}$-minor graph $G'$, induced subgraph $L'$ of $G'$ and $X',Y' \subseteq V(L')$, such that $N_{G'}(V(L')) \subseteq Z$ for some $(1,0)$-centered set $Z \subseteq V(G')$ in $G'$, if there exist no $k$ $(\ell,X',Y')$-paths in $(L',G')$ pairwise at distance in $G'$ at least $r$, then there exists a $(\xi,\eta)$-centered set in $G'$ intersecting all $(\ell,X',Y')$-paths in $(L',G')$.
Note that $\xi^* \geq \ell$. 
Define $\xi^*=\xi$ and $\eta^*=2\eta$.

Let $G,X,Y$ be as stated in the lemma.

\medskip

\noindent{\bf Claim 1:} For every finite induced subgraph $L$ of $G$, there exists $Z_L \subseteq N_G^{\leq 2\eta}[V(L)]$ such that $Z_L$ is $(\xi,\eta)$-centered in $G[N_G^{\leq 2\eta}[V(L)]]$ intersecting all $(\ell,X \cap V(L), Y \cap V(L))$-paths in $(L,G)$.

\noindent{\bf Proof of Claim 1:} 
Let $L' = G[N_G^{\leq 2\eta}[V(L)]]$.
Let $G'$ be the graph obtained from $L'$ by adding a new vertex adjacent to all vertices in $N_G(V(G)-V(L'))$. 
Since $G$ is locally finite, $L'$ and $G'$ are finite.
Since $L'$ is an induced subgraph of $G$, it is $H$-minor free.
So $G'$ is $K_{|V(H)|+1}$-minor free.
Moreover, $L'$ is an induced subgraph of $G'$ such that $N_{G'}(V(L'))$ is $(1,0)$-centered in $G'$.

Suppose that there exist $k$ $(\ell,X \cap V(L'), Y \cap V(L'))$-paths $P_1,P_2,...,P_k$ in $(L',G')$ pairwise at distance in $G'$ at least $r$.
Since $V(L') \subseteq V(G)$, we know that $P_1,P_2,...,P_k$ are contained in $G$. 
If there exists $i \in [k]$ such that $P_i$ is not an $(\ell,X,Y)$-path in $(L',G)$, then $\dist_G(a_i,b_i) <\ell$, where $a_i,b_i$ are the ends of $P_i$, so there exists a path $Q$ in $G$ of length at most $\ell$ from $a_i$ to $b_i$; by identifying all vertices in $V(Q)-V(L')$ into a vertex, we obtain a walk in $G'$ from $a_i$ to $b_i$ with length less than $\ell$, contradicting that $P_i$ is an $(\ell,X \cap V(L'), Y \cap V(L'))$-path in $(L',G')$.
So $P_1,P_2,...,P_k$ are $(\ell,X,Y)$-paths in $(L',G)$.
Since there exist no $k$ $(\ell,X,Y)$-paths in $G$ have pairwise distance in $G$ at least $r$, we may assume that $\dist_G(V(P_1),V(P_2)) < r$.
So there exists a path $P$ in $G$ of length less than $r$ from $V(P_1)$ to $V(P_2)$.
By identifying all vertices in $V(P)-V(L')$ into a vertex, we obtain a walk in $G'$ from $V(P_1)$ to $V(P_2)$ of length less than $r$, contradicting that $P_1$ and $P_2$ have distance in $G'$ at least $r$.

So there do not exist $k$ $(\ell,X \cap V(L'), Y \cap V(L'))$-paths $P_1,P_2,...,P_k$ in $(L',G')$ with pairwise distance in $G'$ at least $r$.
Hence there exists a $(\xi,\eta)$-centered set $Z$ in $G'$ intersecting all $(\ell,X \cap V(L'),Y \cap V(L'))$-paths in $(L',G')$.
So there exists $Z^- \subseteq V(G')$ with $|Z^-| \leq \xi$ such that $Z \subseteq N_{G'}^{\leq \eta}[Z^-]$.

Let $Z_L^- = Z^- \cap N_{G'}^{\leq \eta}[V(L)]$, and let $Z_L = N_{G'}^{\leq \eta}[Z_L^-]$.
Note that $Z_L$ is $(\xi,\eta)$-centered in $G'$.

For every $z \in Z_L^-$, we know $z \in N_{G'}^{\leq \eta}[V(L)]$, so $N_{G'}^{\leq \eta}[z] \subseteq N_{G'}^{\leq 2\eta}[V(L)] = N_G^{\leq 2\eta}[V(L)] \subseteq V(L')$. 
Hence $Z_L$ is also $(\xi,\eta)$-centered in $G[N_G^{\leq 2\eta}[V(L')]]$.

Suppose there exists an $(\ell,X \cap V(L), Y \cap V(L))$-paths $R$ in $(L,G)$ disjoint from $Z_L$.
Since $\eta \geq \ell$, $R$ is also an $(\ell,X \cap V(L'), Y \cap V(L'))$-path in $(L',G')$.
So $V(R) \cap Z \neq \emptyset$.
Hence there exists $z^* \in Z^-$ such that $V(R) \cap N_{G'}^{\leq \eta}[z^*] \neq \emptyset$.
Since $V(R) \subseteq V(L)$, $z^* \in Z_L^-$, so $V(R) \cap Z_L \supseteq V(R) \cap N_{G'}^{\leq \eta}[z^*] \neq \emptyset$, a contradiction.
This proves the claim.
$\Box$

\medskip

Let $I = V(G)$.
Let $\A$ be the set of all finite subsets of $I$.
For every $A \in \A$, we define $c_A: A \rightarrow [2]$ to be the function such that for every $v \in A$, $c_A(v)=1$ if $v \in A \cap Z_{G[A]}$, and $c_A(v)=2$ otherwise; so $c_A$ is a choice function of $([2])_{a \in A}$.
By Theorem \ref{compactness}, there exists a choice function $c$ of $([2])_{\alpha \in I}$ such that for every $A \in \A$, there exists $B \in \A$ with $B \supseteq A$ such that $c(\alpha)=c_B(\alpha)$ for every $\alpha \in A$.

Let $Z^* = \{v \in V(G)=I: c(v)=1\}$.

\medskip

\noindent{\bf Claim 2:} $Z^*$ is $(\xi,2\eta)$-centered in $G$.

\noindent{\bf Proof of Claim 2:} 
Suppose to the contrary that $Z^*$ is not $(\xi,2\eta)$-centered in $G$.
Then there exist $z_1,z_2,...,z_{\xi+1}$ in $Z^*$ such that $\dist_G(z_i,z_j)>2\eta$ for any $1 \leq i<j \leq \xi+1$.
Note that the existence of $z_1,z_2,...,z_{\xi+1}$ follows from the fact that for every $i \in [\xi]$, $N_G^{\leq 2\eta}[\{z_j: j \in [i]\}] \not \supseteq Z^*$, for otherwise $Z^*$ is $(i,\eta)$-centered in $G$.

For any $1 \leq i < j \leq \xi+1$, let $P_{i,j}$ be a path in $G$ from $z_i$ to $z_j$ with $|E(P_{i,j})|=\dist_G(z_i,z_j)$.
Let $S = N_G^{\leq 2\eta}[\{z_i: i \in [\xi+1]\} \cup \bigcup_{1 \leq i < j \leq \xi+1}V(P_{i,j})]$.
Since $G$ is locally finite, $S$ is a finite subset of $V(G)$, so $S \in \A$.
Hence there exists $B \in \A$ with $B \supseteq S$ such that $c(v)=c_B(v)$ for every $v \in S$.

So for every $i \in [\xi+1]$, we know $c_B(z_i)=c(z_i)=1$, so $z_i \in B \cap Z_{G[B]}$.
Hence $\{z_i: i \in [\xi+1]\}$ is contained in the set $B \cap Z_{G[B]}$ which is $(\xi,\eta)$-centered in $G[N_G^{\leq 2\eta}[B]]$.
So there exists $Z^- \subseteq N_G^{\leq 2\eta}[B]$ with $|Z^-| \leq \xi$ such that $B \cap Z_{G[B]} \subseteq N_{G[N_G^{\leq 2\eta}[B]]}^{\leq \eta}[Z^-]$.
Hence there exist distinct $i,j \in [\xi+1]$ and $z \in Z^-$ such that $\{z_i,z_j\} \subseteq N_{G[N_G^{\leq 2\eta}[B]]}^{\leq \eta}[z]$, so $\dist_{G[N_G^{\leq 2\eta}[B]]}(z_i,z_j) \leq 2\eta$.
But $P_{i,j} \subseteq G[B] \subseteq G[N_G^{\leq 2\eta}[B]]$, so $\dist_{G[N_G^{\leq 2\eta}[B]]}(z_i,z_j)=\dist_G(z_i,z_j) > 2\eta$, a contradiction.
$\Box$

\medskip

To prove this lemma, it suffices to show that $Z^*$ intersects all $(\ell,X,Y)$-paths in $G$.

Suppose to the contrary that there exists an $(\ell,X,Y)$-path $P$ in $G$ disjoint from $Z^*$.
So $c(v)=2$ for every $v \in V(P)$.
Since $V(P) \in \A$, there exists $B \in \A$ with $B \supseteq V(P)$ such that $c_B(v)=c(v)=2$ for every $v \in V(P)$.
So $P$ is an $(\ell, X \cap B, Y \cap B)$-path in $(G[B],G)$ disjoint from $Z_{G[B]}$, contradicting Claim 1.
\end{pf}

\subsection{Apex-minor free graphs} \label{subsec:apex_minor_free}

Recall that an \defn{apex-graph} is a graph that can be made planar by deleting at most one vertex.

Theorem \ref{minor_structual_thm} provided a structure theorem for $H$-minor free graphs for an arbitrary graph $H$.
If $H$ is an apex-graph, the structure can be further refined to allow us to provide a much simpler proof for the weak Menger property for $H$-minor free graphs with a much better bound.
Such a structure theorem is known in the literature, but it was not stated in the form that is completely ready to apply our machinery.
In this section we combine the known results to derive a form that we will use in later sections.

We show how to slightly revise the proof of \cite[Lemma 36]{dt} to obtain the following lemma. 

\begin{lemma} \label{dt_lemma_apex_minor}
For any apex-graph $H$, surface $\Sigma$ in which $H$ cannot be drawn and for any $\kappa,\xi \in {\mathbb N}_0$ and $\rho \in {\mathbb N}$, there exist $\kappa^*,\rho^* \in {\mathbb N}$ such that for every $\theta^* \in {\mathbb N}$, there exists $\theta \in {\mathbb N}$ such that the following holds.
Let $G$ be a graph, and let $Z \subseteq V(G)$ with $|Z| \leq \xi$.
Let $\T$ be a tangle in $G$ of order at least $\theta+|Z|$.
Let $\Se$ be an effective $(\T-Z)$-central segregation of $G-Z$ with a $(\kappa,\rho)$-witness $(\Se_1,\Se_2)$ such that there exists a $(\Sigma,\theta,\theta,\T-Z)$-arrangement of $\Se$ with respect to $(\Se_1,\Se_2)$.
If $H$ is not a minor of $G$, then there exists an effective $(\T-Z)$-central segregation $\Se^*$ of $G-Z$ with a $(\kappa^*,\rho^*)$-witness $(\Se^*_1,\Se^*_2)$ such that there exists a $(\Sigma,\theta^*,1,\T-Z)$-arrangement of $\Se^*$ with respect to $(\Se^*_1,\Se^*_2)$ such that $N_G(Z) \subseteq \bigcup_{(S,\Omega) \in \Se_2^*}S$.
\end{lemma}

\begin{pf}
Let $H$ be an apex-graph, and let $\Sigma$ be a surface in which $H$ cannot be drawn.
Let $\kappa,\xi \in {\mathbb N}_0$ and $\rho \in {\mathbb N}$.
Define $\kappa^* = \kappa+2\xi|E(H)|$.
Let $\lambda$ be the integer $\theta$ mentioned in the proof of \cite[Lemma 36]{dt}.
Define $\rho^*=4\rho+3\lambda$.
Let $\theta^* \in {\mathbb N}$.
Define $\theta = \theta^*+18\lambda+20\rho^*+2+(12\lambda+2)\kappa^*$.

Let $G,Z,\T,\Se,\Se_1,\Se_2$ be as stated in the lemma.
Then the second paragraph of the proof of \cite[Lemma 36]{dt} shows that there exists an effective $(\T-Z)$-central segregation $\Se^*$ (i.e. the segregation $S^n$ in the proof of \cite[Lemma 36]{dt}) of $G-Z$ with a $(\kappa^*,\rho^*)$-witness $(\Se^*_1,\Se^*_2)$ such that there exists a $(\Sigma, \theta-(12\lambda+2)\kappa^*,1,\T-Z)$-arrangement.
And the first two sentences in the third paragraph of the proof of \cite[Lemma 36]{dt} shows that $N_G(Z) \subseteq \bigcup_{(S,\Omega) \in \Se_2^*}S$.
\end{pf}

\bigskip

We will use the following restatement of \cite[Theorem 28]{dt}.

\begin{theorem}[{{\cite[Theorem 28]{dt}}}] \label{dvorak_thomas_structure}
For every graph $H$, there exist $\kappa,\rho \in {\mathbb N}_0$ such that for every nondecreasing function $\phi: {\mathbb Z} \rightarrow {\mathbb R}_{> 0}$, there exist $\xi,\theta \in {\mathbb N}_0$ with $\theta > \xi$ such that if $G$ is an $H$-minor free graph with a tangle $\T$ of order at least $\theta$, then there exist $Z \subseteq V(G)$ with $|Z| \leq \xi$ and an effective $(\T-Z)$-central segregation $\Se$ of $G-Z$ with respect to a $(\kappa,\rho)$-witness $(\Se_1,\Se_2)$ such that $\Se$ has a $(\Sigma,\phi(|Z|),\phi(|Z|),\T-Z)$-arrangement with respect to $(\Se_1,\Se_2)$ for some surface $\Sigma$ in which $H$ cannot be drawn.
\end{theorem}

Now we prove the structure theorem for apex-minor free graphs that we will use.

\begin{theorem} \label{apex_minor_structure}
For every apex-graph $H$, there exists $\kappa \in {\mathbb N}$ such that for any nondecreasing function $\phi: {\mathbb Z} \rightarrow {\mathbb R}_{> 0}$ and $r \in {\mathbb N}$, there exists $\rho \in {\mathbb N}_0$ such that for every $\theta^* \in {\mathbb N}$, there exist $\xi,\theta \in {\mathbb N}$ with $\theta > \xi$ such that if $G$ is an $H$-minor free graph with a tangle $\T$ of order at least $\theta$, then there exist $Z \subseteq V(G)$ with $|Z| \leq \xi$ and an effective $(\T-Z)$-central segregation $\Se$ of $G-Z$ with respect to a $(\kappa,\rho)$-witness $(\Se_1,\Se_2)$ such that $\Se$ has a $(\Sigma,\theta^*,\phi,\T-Z)$-arrangement with respect to $(\Se_1,\Se_2)$ for some surface $\Sigma$ in which $H$ cannot be drawn, and $N_G^{\leq r}[Z]-Z \subseteq \bigcup_{(S,\Omega) \in \Se_2}S$.
\end{theorem}

\begin{pf}
Let $H$ be an apex-graph.
Let $\kappa_1,\rho_1$ be the integers $\kappa,\rho$ mentioned in Theorem \ref{dvorak_thomas_structure} by taking $H=H$.
For every surface $\Sigma$ in which $H$ cannot be drawn, let $\kappa_\Sigma,\rho_\Sigma$ be the functions such that for every $x \in {\mathbb Z}$, the integers $\kappa_\Sigma(x),\rho_\Sigma(x)$, respectively, equal the integers $\kappa^*,\rho^*$ mentioned in Lemma \ref{dt_lemma_apex_minor} by taking $(H,\Sigma,\kappa,\xi,\rho)=(H,\Sigma,\kappa_1,\max\{x,0\},\rho_1)$.
Let $\kappa_2 = \max_\Sigma \kappa_\Sigma$ and $\rho_2 = \max_\Sigma \rho_\Sigma$, where the maxima are over all connected surfaces in which $H$ cannot be drawn.
For every $x \in {\mathbb Z}$, let $\kappa_3(x)$ be the integer $\kappa^*$ mentioned in Lemma \ref{sweeping balls into vortices} by taking $(\kappa,k)=(\kappa_2(x),0)$.
Define $\kappa=\max_{0 \leq x \leq \xi}\kappa_3(x)$.

Let $\phi: {\mathbb Z} \rightarrow {\mathbb R}_{> 0}$ be a nondecreasing function, and let $r \in {\mathbb N}$.
For every $x \in {\mathbb Z}$, let $\rho_3(x)$ be the integer $\rho^*$ mentioned in Lemma \ref{sweeping balls into vortices} by taking $(\kappa,k,\rho,\lambda,\phi)=(\kappa_2(x),0,\rho_2(x),r+3,\phi)$.
Define $\rho = \max_{0 \leq x \leq \xi}\rho_3(x)$.

Let $\theta^* \in {\mathbb N}$.
For every $x \in {\mathbb Z}$, let $\theta_2(x)$ be the integer $\theta$ mentioned in Lemma \ref{sweeping balls into vortices} by taking $(\kappa,k,\rho,\lambda,\phi,\theta^*)=(\kappa_2(x),0,\rho_2(x),r+3,\phi,\theta^*)$.
For every surface $\Sigma$ in which $H$ cannot be drawn, let $\theta_\Sigma$ be the function such that for every $x \in {\mathbb Z}$, $\theta_\Sigma(x)$ equals the integer $\theta$ mentioned in Lemma \ref{dt_lemma_apex_minor} by taking $(H,\Sigma,\kappa,\xi,\rho,\theta^*)=(H,\Sigma,\kappa_1,\max\{x,0\},\rho_1,\theta_2(x))$.
Let $\theta_3 = \max_\Sigma \theta_\Sigma$, where the maximum is over all connected surfaces in which $H$ cannot be drawn.
Define $\xi,\theta_4$ to be the integers $\xi,\theta$ mentioned in Theorem \ref{dvorak_thomas_structure} by taking $(H,\phi)=(H,\theta_3)$.
Define $\rho = \max_{0 \leq x \leq \xi}\rho_3(x)$ and $\theta = \max_{0 \leq x \leq \xi}\theta_3(x)+\theta_4+\xi$.

Let $G$ be an $H$-minor free graph with a tangle $\T$ of order at least $\theta$.
By Theorem \ref{dvorak_thomas_structure}, there exist $Z \subseteq V(G)$ with $|Z| \leq \xi$ and an effective $(\T-Z)$-central segregation $\Se^1$ of $G-Z$ with respect to a $(\kappa_1,\rho_1)$-witness $(\Se^1_1,\Se^1_2)$ such that $\Se^1$ has a $(\Sigma,\theta_3(|Z|),\theta_3(|Z|),\T-Z)$-arrangement with respect to $(\Se^1_1,\Se^1_2)$ for some surface $\Sigma$ in which $H$ cannot be drawn.
Note that $\Sigma$ is connected since the skeleton is 2-cell.
So $\theta_3(|Z|) \geq \theta_\Sigma(|Z|)$.

Since $\T$ has order at least $\theta \geq \theta_3(|Z|)+\xi \geq \theta_3(|Z|)+|Z|$, Lemma \ref{dt_lemma_apex_minor} implies that there exists an effective $(\T-Z)$-central segregation $\Se^2$ of $G-Z$ with respect to a $(\kappa_\Sigma(|Z|),\rho_\Sigma(|Z|))$-witness $(\Se^2_1,\Se^2_2)$ such that $\Se^2$ has a $(\Sigma,\theta_2(|Z|),1,\T-Z)$-arrangement $\alpha_2$ with respect to $(\Se^2_1,\Se^2_2)$ such that $N_G(Z) \subseteq \bigcup_{(S,\Omega) \in \Se_2^2}S$.
Since $\kappa_\Sigma(|Z|) \leq \kappa_2(|Z|)$ and $\rho_\Sigma(|Z|) \leq \rho_2(|Z|)$, $(\Se_1^2,\Se_2^2)$ is a $(\kappa_2(|Z|),\rho_2(|Z|))$-witness of $\Se^2$.

Since $N_G(Z) \subseteq \bigcup_{(S,\Omega) \in \Se_2^2}S$, for every $(S,\Omega) \in \Se_2^2$, there exists an $(r+3)$-zone $\Lambda_S$ such that $N_G^{\leq r}[Z]-Z \subseteq \bigcup_{(S',\Omega') \in \Se^2, \alpha_2(S',\Omega') \subseteq \bigcup_{(S'',\Omega'') \in \Se_2^2} \Lambda_{S''}}S''$.
Then by Lemma \ref{sweeping balls into vortices}, there exists a $(\T-Z)$-central segregation $\Se$ of $G-Z$ with a $(\kappa_3(|Z|),\rho_3(|Z|))$-witness $(\Se_1,\Se_2)$ such that $\Se_1 \subseteq \Se_1^2$, $\bigcup_{(S,\Omega) \in \Se_2}S \supseteq \bigcup_{(S,\Omega) \in \Se_2^2}S \cup \bigcup_{(S',\Omega') \in \Se^2_1, \alpha_2(S',\Omega') \subseteq \bigcup_{(S'',\Omega'') \in \Se_2^2} \Lambda_{S''}}S'' \supseteq N^{\leq r}_G[Z]-Z$ and $\Se$ has a $(\Sigma,\theta^*,\phi,\T-Z)$-arrangement with respect to $(\Se_1,\Se_2)$.
Since $\Se_1 \subseteq \Se_1^2$ and $\Se^2$ is effective with respect to $(\Se_1^2,\Se_2^2)$, we know $\Se$ is effective with respect to $(\Se_1,\Se_2)$.
Since $\kappa_3(|Z|) \leq \kappa$ and $\rho_3(|Z|) \leq \rho$, $(\Se_1,\Se_2)$ is a $(\kappa,\rho)$-witness of $\Se$.
\end{pf}

\bigskip

Now we give a proof for the weak coarse Menger property of apex-minor free graphs such that the bound for the radius of balls is an absolution constant.

\begin{theorem}
For any apex-graph $H$, any $k,r \in {\mathbb N}$ and $\ell \in {\mathbb N}_0$, there exists $\xi^* \in {\mathbb N}$ such that the following holds.
Let $G$ be an $H$-minor free graph.
Let $X,Y$ be subsets of $V(G)$.
If there do not exist $k$ $(\ell,X,Y)$-paths in $G$ pairwise at distance in $G$ at least $r$, then there exists $Z^* \subseteq V(G)$ such that $Z^*$ is $(\xi^*,14r+\ell)$-centered in $G$ and intersects all $(\ell,X,Y)$-paths in $G$.
\end{theorem}

\begin{pf}
Let $H$ be an apex-graph, and let $k,r \in {\mathbb N}$ and $\ell \in {\mathbb N}_0$.
We may assume that $H$ is non-planar.
	\begin{itemize}
		\item Let $g$ be the largest nonnegative integer $g$ such that $H$ cannot be drawn in some surface of Euler genus at most $g$.
		\item Let $\kappa_0,\rho_0$ be the integers $\kappa,\rho$, respectively, mentioned in Theorem \ref{apex_minor_structure} by taking $(H,\phi,r)=(H,1,r)$.
		\item For every $x \in {\mathbb Z}$, let $\phi_{\ref{hitting_no_apex}}(x)$ and $\theta_{\ref{hitting_no_apex}}(x)$ be the integers $\phi,\theta$, respectively, mentioned in Lemma \ref{hitting_no_apex} by taking $(k,r,\kappa,\rho,\ell,g) = (k,r,\kappa_0,\max\{x,1\},\ell,g)$.
		\item Let $\phi$ be the function $\phi_{\ref{hitting_no_apex}}+\theta_{\ref{hitting_no_apex}}$.
		\item Let $\theta_0=1$.
For every integer $x \geq 1$, 
			\begin{itemize}
				\item let $\kappa_x,\rho_x$ be the integers $\kappa^*,\rho^*$, respectively, mentioned in Lemma \ref{sweeping balls into vortices} by taking $(\kappa,k,\rho,\lambda,\phi) = (\kappa_0,\kappa_0+3\theta_{x-1}-3+2\sum_{i=2}^{x-1}(3\theta_i-3),\rho_0,4r+7,\phi)$,
				\item let $\gamma_x$ be the integer $\theta$ mentioned in Lemma \ref{segregation_respectful_central} by taking $(\theta^*,\kappa,\rho)=(\theta_{\ref{hitting_no_apex}}(\rho_x),\kappa_x, \allowbreak \rho_x)$,
				\item let $\kappa_x',\rho_x',\gamma_x'$ be the integers $\kappa^*,\rho^*,\theta$, respectively, mentioned in Lemma \ref{sweeping balls into vortices} by taking $(\kappa,k,\rho,\lambda,\phi,\theta^*) = (\kappa_x,\theta_{\ref{hitting_no_apex}}(\rho_x),\rho_x,5(\theta_{\ref{hitting_no_apex}}(\rho_x))^2,1,1)$,
				\item let $\theta_x'$ be the integer $\theta$ mentioned in Lemma \ref{sweeping balls into vortices} by taking $(\kappa,k,\rho,\lambda,\phi,\theta^*) = (\kappa_0,\kappa_0+3\theta_{x-1}-3+2\sum_{i=2}^{x-1}(3\theta_i-3),\rho_0,4r+7,\phi,\theta_{x-1}+\gamma_x+\gamma_x'+\phi(\rho_x))$,
				\item let $\xi_x$ and $\theta_x''$ be the integers $\xi$ and $\theta$, respectively, mentioned in Lemma \ref{apex_minor_structure} by taking $(H,\phi,r,\theta^*)=(H,1,r,\theta_x')$,  
				\item let $\theta_x = \theta_{x-1}+\theta_x''+\theta_x'''+\xi_x+4\rho_x+6+\theta_{\ref{hitting_no_apex}}(\rho_x)$, and
				\item let $\xi_x'$ and $\eta_x'$ be the integers $\xi^*,\eta^*$, respectively, mentioned in Lemma \ref{avoid_apex_new_0} by taking $(g,\kappa,\rho,k,r,\xi,\eta)=(g,\kappa_x,\rho_x,k,r,\xi_x+3\theta_{x-1}-3+2\sum_{i=2}^{x-1}(3\theta_i-3),2r)$,
				\item let $\xi_x''$ and $\eta_x''$ be the integers $\xi^*,\eta^*$, respectively, mentioned in Lemma \ref{avoid_apex_new_0} by taking $(g,\kappa,\rho,k,r,\xi,\eta)=(g,\kappa_x',\rho_x',k,r,\xi_x+3\theta_{x-1}-3+2\sum_{i=2}^{x-1}(3\theta_i-3),2r)$,
			\end{itemize}
		\item Let $\xi,\eta$ be the integers $\xi^*,\eta^*$ mentioned in Lemma \ref{hitting_no_apex} by taking $(k,r,\kappa,\rho,\ell,g,\xi,\eta)=(k,r,\kappa_0+\max_{1 \leq i \leq k}\kappa_i,\max_{1 \leq i \leq k}\rho_i,\ell,g,3\theta_1-3+2\sum_{i=2}^{k-1}(3\theta_i-3)+\max_{1 \leq i \leq k}\xi_i,2r)$.
		\item Let $\xi'=\max_{1 \leq i \leq k}(\xi_i+\xi_i'+\xi_x''+\theta_x)$ and $\eta'=\max_{1 \leq i \leq k}\{\eta_i',\eta_i''\}$.
		\item Define $\xi^* = (k-1)(\xi+\xi')+3\theta_1-3+2\sum_{i=2}^{k-1}(3\theta_i-3)$ and $\eta^* = \max\{\eta,\eta',3r\}$.
	\end{itemize}
Note that $\eta = 4 \cdot 2r+6r+\ell = 14r+\ell$ and $\eta'=3r$.
So $\eta^*=14r+\ell$.

Let $G,X,Y$ be as stated in the lemma.
Let $\F$ be the set of $(\ell,X,Y)$-paths in $G$.

By Lemma \ref{easy_tangle_multifold} (taking $(G,L,\F,k,(\theta_i: i \in [k]),r,r',\xi,\eta,Z)=(G,G,\F,(\theta_i: i \in [k]),r,r,0,0,\emptyset)$), there exist $Z_0 \subseteq V(G)$ and a set $\W$ of disjoint connected induced subgraphs of $G$ with $|\W| \leq k-1$ such that 
	\begin{itemize}
		\item for every $W \in \W$, there exist $i_W \in [k-1] \cup \{0\}$ and $Z_W \subseteq Z_0$ with $N_G(V(W)) \cup (Z_0 \cap V(W)) \subseteq Z_W$ such that 
			\begin{itemize}
				\item $Z_W$ is $(3\theta_1-3+2\sum_{i=2}^{i_W}(3\theta_i-3),2r)$-centered in $G$, and if $i_W=0$, then $Z_W=\emptyset$,
				\item there exists a $(G,\F \cap W,r,\theta_{i_W+1},Z_0)$-tangle $\T_W$ in $W$,
			\end{itemize}
		\item $Z_0$ is $(3\theta_1-3+2\sum_{i=2}^{k-1}(3\theta_i-3),2r)$-centered in $G$ and intersects all members of $\F-(\F \cap \bigcup_{W \in \W}W)$.
	\end{itemize}
To prove this theorem, it suffices to show that for every $W \in \W$, there exists $Z_W^* \subseteq V(G)$ such that $Z_W^*$ is $(\xi+\xi',\max\{\eta,\eta',3r\})$-centered in $G$ and $Z_W^* \cup N_G^{\leq r}[Z_0]$ intersects all members of $\F \cap W$, since the theorem follow from taking $Z^*=Z_0 \cup \bigcup_{W \in \W}Z_W^*$.

Suppose to the contrary that there exists $W \in \W$ such that no set $Z_W^* \subseteq V(G)$ that is $(\xi+\xi',\max\{\eta,\eta',3r\})$-centered in $G$ and satisfies that $Z_W^* \cup N_G^{\leq r}[Z_0]$ intersects all members of $\F \cap W$.

Let $\T$ be the tangle of order $\theta_{i_W+1}$ in $G$ induced by $\T_W$.
Since $\theta_{i_W+1} \geq \theta_{i_W+1}''$, by Lemma \ref{apex_minor_structure}, there exist $Z \subseteq V(G)$ with $|Z| \leq \xi_{i_W+1}$ and an effective $(\T-Z)$-central segregation $\Se^1$ of $G-Z$ with respect to a $(\kappa_0,\rho_0)$-witness $(\Se^1_1,\Se^1_2)$ such that $\Se^1$ has a $(\Sigma,\theta_{i_W+1}',1,\T-Z)$-arrangement $\alpha_1$ with respect to $(\Se^1_1,\Se^1_2)$ for some surface $\Sigma$ of Euler genus at most $g$ and such that 
	\begin{itemize}
		\item[(i)] $N_G^{\leq r}[Z]-Z \subseteq \bigcup_{(S,\Omega) \in \Se^1_2}S$.
	\end{itemize}

Let $G_1$ be the skeleton of $\Se^1$ with respect to $(\Se^1_1,\Se^1_2)$.
If $i_W=0$, then let $Z_W'=\emptyset$; otherwise, since $Z_W$ is $(3\theta_1-3+2\sum_{i=2}^{i_W}(3\theta_i-3),2r)$-centered in $G$, there exists $Z_W' \subseteq V(G)$ with $|Z_W'| \leq 3\theta_1-3+2\sum_{i=2}^{i_W}(3\theta_i-3)$ such that $Z_W \subseteq N_G^{\leq 2r}[Z_W']$.
For every $v \in Z_W'$, by Theorem \ref{big zone contains ball} and (i), there exist $v' \in V(G_1)$ and a $(4r+5)$-zone $\Lambda_v$ in $G_1$ around $v'$ such that $N_G^{\leq 2r+1}[v]-Z \subseteq \bigcup_{(S,\Omega) \in \Se^1,\alpha_1(S,\Omega) \subseteq \Lambda_v}S$.
Hence $N_G[Z_W]-Z \subseteq \bigcup_{(S,\Omega) \in \Se^1,\alpha_1(S,\Omega) \subseteq \bigcup_{v \in Z_W'}\Lambda_v}S$.
Similarly, there exist $\kappa_0$ $7$-zones $\Lambda_1,...,\Lambda_{\kappa_0}$ such that $N_G[\bigcup_{(S,\Omega) \in \Se^1_2}S]-Z \subseteq \linebreak \bigcup_{(S,\Omega) \in \Se^1,\alpha_1(S,\Omega) \subseteq \bigcup_{i=1}^{\kappa_0}\Lambda_i}S$.
Let $\Lambda = \bigcup_{v \in Z_W'}\Lambda_v \cup \bigcup_{i=1}^{\kappa_0}\Lambda_i$.
By Lemma \ref{sweeping balls into vortices}, there exists a $(\T-Z)$-central segregation $\Se^2$ of $G-Z$ with a $(\kappa_{i_w+1},\rho_{i_W+1})$-witness $(\Se^2_1,\Se^2_2)$-witness such that 
	\begin{itemize}
		\item[(ii)] $\Se^2_1 \subseteq \Se^1_1$, 
		\item[(iii)] $(N_G^{\leq r}[Z] \cup Z_W)-Z \subseteq \bigcup_{(S,\Omega) \in \Se^1_2}S \cup \bigcup_{(S,\Omega) \in \Se^1,\alpha_1(S,\Omega) \subseteq \Lambda}S \subseteq \bigcup_{(S,\Omega) \in \Se^2_2}(S-\overline{\Omega})$, and 
		\item[(iv)] $\Se^2$ has a $(\Sigma,\theta_{i_W}+\gamma_{i_W+1}+\gamma_{i_W+1}'+\phi(\rho_{i_W+1}),\phi,\T-Z)$-arrangement $\alpha_2$ with respect to $(\Se^2_1,\Se^2_2)$.
	\end{itemize}
Since $\Se^1$ is effective with respect to $(\Se^1_1,\Se^1_2)$, (ii) implies that $\Se^2$ is effective with respect to $(\Se^2_1,\Se^2_2)$.

Let $G_2$ be the skeleton of $\Se^2$ with respect to $(\Se^2_1,\Se^2_2)$.
Since $G_2$ is a 2-cell drawing in $\Sigma$ by (iv), $G_2$ is connected.

Let $Z^+=Z \cup Z_W$.
So $Z^+$ is $(\xi_{i_W+1}+3\theta_1-3+2\sum_{i=2}^{i_W}(3\theta_i-3),2r)$-centered in $G$ such that $Z \cup N_G(V(W)) \subseteq Z^+$; note that if $i_W=0$, then $Z_W=\emptyset$, so $Z^+=Z$ is $(\xi_{i_W+1},0)$-centered in $G$.

\medskip

\noindent{\bf Claim 1:} There exists $Z_W^* \subseteq V(G)$ with $Z_W^* \supseteq N_G^{\leq r}[Z^+] \supseteq Z \cup Z_W$ such that $Z_W^*$ is $(\xi_{i_W+1}',\eta_{i_W+1}')$-centered in $G$ intersecting all members of $\F \cap W$ contained in $\bigcup_{(S,\Omega) \in \Se^2_2}S$.

\noindent{\bf Proof of Claim 1:} 
Note that $\theta_{i_W+1} > \xi_{i_W+1}+4\rho_{i_W+1}+6 \geq |Z|+4\rho_{i_W+1}+6$.
By Lemma \ref{avoid_apex_new_0} (taking $(G,L,\F,Z,Z^+,\T_L,\T_G,\Se,\Se_1,\Se_2)=(G,W,\F \cap W,Z,Z^+,\T_W,\T,\Se^2,\Se^2_1,\Se^2_2)$), there exists $Z_W^* \subseteq V(G)$ with $Z_W^* \supseteq N_G^{\leq r}[Z^+] \supseteq Z \cup Z_W$ such that $Z_W^*$ is $(\xi_{i_W+1}',\eta_{i_W+1}')$-centered in $G$ intersecting all members of $\F \cap W$ contained in $\bigcup_{(S,\Omega) \in \Se^2_2}S$.
$\Box$

\medskip

Since $\xi_{i_W+1}'+\xi_{i_W+1}+\xi_x''+\theta_x \leq \xi'$ and $\eta_{i_W+1}' \leq \eta'$, 
	\begin{itemize}
		\item[(v)] there does not exist $Z_1' \subseteq V(G)$ such that $Z_1'$ is $(\xi+\xi_{i_W+1}+\xi_{i_W+1}''+\theta_{i_W+1},\max\{\eta,\eta',3r\})$-centered in $G$ such that $Z_1' \cup Z_W^* \cup N_G^{\leq r}[Z_0]$ intersects all members of $\F \cap W$.
	\end{itemize}

\medskip

\noindent{\bf Claim 2:} $V(G_2) \cap V(W) \neq \emptyset$.

\noindent{\bf Proof of Claim 2:} 
Suppose to the contrary that $V(G_2) \cap V(W)=\emptyset$.

Suppose that there exists $(S,\Omega) \in \Se^2_1$ such that $V(W) \cap V(S) \neq \emptyset$.
Then since $V(G_2) \cap V(W)=\emptyset$ and $W$ is connected, we have $W \subseteq S-\overline{\Omega}$ by (i)-(iii).
But it is a contradiction since $\T$ is induced by $\T_W$ and $\Se^2$ is $(\T-Z)$-central.

So $W-Z \subseteq \bigcup_{(S,\Omega) \in \Se^2_2}S$.
Since $V(G_2) \cap V(W)=\emptyset$, $W-Z \subseteq \bigcup_{(S,\Omega) \in \Se^2_2}(S-\overline{\Omega})$.
Hence every member of $(\F \cap W)-Z$ is contained in $\bigcup_{(S,\Omega) \in \Se^2_2}(S-\overline{\Omega})$.
Since $Z_W^* \supseteq Z$, we know that $Z_W^*$ is $(\xi_{i_W+1}',\eta_{i_W+1}')$-centered in $G$ intersecting all members of $\F \cap W$ by Claim 1, contradicting (v).
$\Box$

\medskip

\noindent{\bf Claim 3:} $V(G_2) \subseteq V(W)$ and $\Se^2[V(W)-Z]$ is a segregation of $W-Z$ effective with respect to $(\Se^2_1[V(W)-Z],\Se^2_2[V(W)]-Z)$.

\noindent{\bf Proof of Claim 3:}
By Claim 2, there exists $v \in V(G_2) \cap V(W)$.
Let $v' \in V(G_2)-\{v\}$.
Since $G_2$ is connected, there exists a path in $G_2$ between $v$ and $v'$.
Since $\Se^2$ is effective with respect to $(\Se^2_1,\Se^2_2)$, there exists a path $P$ in $(G-Z)-\bigcup_{(S,\Omega) \in \Se_2}E(S)$ between $v$ and $v'$.
By (iii), $P$ is disjoint from $Z_W \supseteq N_G(V(W))$.
Since $v \in V(W)$, we know $v' \in V(W)$. 
Hence $V(G_2) \subseteq V(W)$ and $\Se^2[V(W)-Z]$ is a segregation of $W-Z$ effective with respect to $(\Se^2_1[V(W)-Z],\Se^2_2[V(W)]-Z)$.
$\Box$

\medskip

\noindent{\bf Claim 4:} $\Se^2[V(W)-Z]$ is a $\T_W-Z$-central segregation of $W-Z$.

\noindent{\bf Proof of Claim 4:}
Suppose to the contrary that there exists $(A,B) \in \T_W-Z$ such that $B \subseteq S[V(W)-Z]$ for some $(S,\Omega) \in \Se^2$.
So there exists $(A',B') \in \T_W$ such that $V(A')=V(A) \cup Z$ and $V(B')=V(B) \cup Z \subseteq V(S) \cup Z$.

We first suppose that $(S,\Omega) \in \Se^2_1$.
Since $\Se^2$ is $(\T-Z)$-central, there exists $(A'',B'') \in \T$ with $V(A'')=V(S) \cup Z$ and $V(A'' \cap B'')=\overline{\Omega} \cup Z$.
So $V(B')=V(B) \cup Z \subseteq V(A'')$.
Since $\T$ is induced by $\T_W$, we know $(B'',A'') \in \T$, a contradiction.

So $(S,\Omega) \in \Se_2^2$.
Since $\T_W$ is a $(G,\F \cap W,r,\theta_{i_W+1},Z_0)$-tangle, we know that $A'-N_G^{\leq r}[V(A' \cap B')]$ does not contain any member of $(\F \cap W)-N_G^{\leq r}[Z_0]$.
So $A-N_G^{\leq r}[V(A \cap B) \cup Z]$ does not contain any member of $(\F \cap W)-N_G^{\leq r}[Z_0]$.
Hence every member of $(\F \cap W)-N_G^{\leq r}[Z_0 \cup Z \cup V(A \cap B)])$ is contained in $B \subseteq S$.
Since $(S,\Omega) \in \Se^2_2$, Claim 1 implies that $Z_W^* \cup N_G^{\leq r}[Z \cup V(A \cap B)] \cup N_G^{\leq r}[Z_0]=Z_W^* \cup Z \cup N_G^{\leq r}[Z_0]$ intersects all members of $\F \cap W$.
It contradicts (v) since $N_G^{\leq r}[Z \cup V(A \cap B)]$ is $(\xi_{i_W+1}+\theta_{i_W+1},r)$-centered in $G$.
$\Box$

\medskip

\noindent{\bf Claim 5:} $\alpha_2$ is an $(8,r,\Se^2_1[V(W)-Z],\Se^2_2[V(W)-Z])$-protected $(\Sigma,\theta_{\ref{hitting_no_apex}}(\rho_{i_W+1}),\phi,\T_W-Z)$-arrangement of $\Se^2[V(W)-Z]$ with respect to $(\Se^2_1[V(W)-Z],\Se^2_2[V(W)]-Z)$.

\noindent{\bf Proof of Claim 5:} 
By Claims 3 and 4 and (iv), it suffices to show that there exists a respectful tangle in $G_2$ of order at least $\theta_{\ref{hitting_no_apex}}(\rho_{i_W+1})$ conformal with $\T_W-Z$. 
By Claims 3 and 4 and Lemma \ref{segregation_respectful_central}, we are done if $\Sigma$ is the sphere.
So we may assume that $\Sigma$ is not a sphere.
Hence by (iv), there exists a unique respectful tangle $\T^+$ of order $\theta_{i_W}+\gamma_{i_W+1}+\gamma_{i_W+1}'+\phi(\rho_{i_W+1})$ in $G_2$ conformal with $\T-Z$. 
Let $\T'$ be the tangle in $G_2$ of order $\theta_{\ref{hitting_no_apex}}(\rho_{i_W+1})$ with $\T' \subseteq \T^+$.

By Claim 3, there exists a natural $G_2$-minor $\mu$ in $W-Z$.
Let $\T_W'$ be the tangle of order $\theta_{\ref{hitting_no_apex}}(\rho_{i_W+1})$ in $W-Z$ induced by $\T'$.
Note that $\T_W'$ is also induced by $\T^+$.
It suffices to show $\T_W' \subseteq \T_W-Z$.

Suppose that there exists $(A,B) \in \T_W'-(\T_W-Z)$.
Then there exists $(A',B') \in \T' \subseteq \T^+$ with $\mu(E(A')) = E(A) \cap \mu(E(G_2))$.
Since $G_2$ is connected, $A'$ is a union of at most $|V(A' \cap B')| \leq \theta_{\ref{hitting_no_apex}}(\rho_{i_W+1})$ components.
By Theorems \ref{A distance} and \ref{big zone contains ball}, $A'$ is contained in a union of at most $|V(A' \cap B')| \leq \theta_{\ref{hitting_no_apex}}(\rho_{i_W+1})$ many $2|V(A' \cap B')|^2+3 \leq 5(\theta_{\ref{hitting_no_apex}}(\rho_{i_W+1}))^2$-zones in $G_2$.
By Lemma \ref{sweeping balls into vortices}, there exists a $(\T-Z)$-central segregation $\Se^3$ with a $(\kappa'_{i_W+1},\rho'_{i_W+1})$-witness $(\Se^3_1,\Se^3_2)$ such that $\Se^3$ has a proper arrangement and $A \subseteq \bigcup_{(S,\Omega) \in \Se^3_2}S$.
By \ref{avoid_apex_new_0} (taking $(G,L,\F,Z,Z^+,\T_L,\T_G,\Se,\Se_1,\Se_2)=(G,W,\F \cap W,Z,Z^+,\T_W,\T,\Se^3,\Se^3_1,\Se^3_2)$), there exists $Z_W^3 \subseteq V(G)$ with $Z_W^3 \supseteq N_G^{\leq r}[Z^+] \supseteq Z \cup Z_W$ such that $Z_W^3$ is $(\xi_{i_W+1}'',\eta_{i_W+1}'')$-centered in $G$ intersecting all members of $\F \cap W$ contained in $\bigcup_{(S,\Omega) \in \Se^3_2}S$.
Since $(A,B) \in \T_W'-(\T_W-Z)$, $(B,A) \in \T_W-Z$.
So $B-N_G^{\leq r}[V(A \cap B)]$ does not contain any member of $(\F \cap W)-N_G^{\leq r}[Z_0]$.
Hence every member of $(\F \cap W)-N_G^{\leq r}[Z_0 \cup V(A \cap B)]$ is contained in $A$ and hence intersects $Z_W^3$.
So $Z_W^3 \cup N_G^{\leq r}[V(A \cap B)] \cup N_G^{\leq r}[Z_0]$ intersects all members of $\F \cap W$.
It contradicts (v) since $Z_W^3 \cup N_G^{\leq r}[V(A \cap B)]$ is $(\xi_{i_W+1}''+\theta_{i_W+1},\max\{\eta',r\})$-centered in $G$.
$\Box$

\medskip

Since $\phi \geq \phi_{\ref{hitting_no_apex}}$, Claims 5 implies that $\alpha_2$ is also an $(8,r,\Se^2_1[V(W)-Z],\Se^2_2[V(W)-Z])$-protected $(\Sigma,\theta_{\ref{hitting_no_apex}}(\rho_{i_W+1}),\phi_{\ref{hitting_no_apex}},\T_W-Z)$-arrangement of $\Se^2[V(W)-Z]$ with respect to $(\Se^2_1[V(W)-Z],\Se^2_2[V(W)-Z])$.

\medskip

\noindent{\bf Claim 6:} $\T_W-Z$ is a $(G,\F \cap (W-Z),r,\theta_{i_W+1}-\xi_{i_W+1},Z^+)$-tangle in $W-Z$.

\noindent{\bf Proof of Claim 6:} 
Let $(A,B) \in \T_W-Z$.
So there exists $(A',B') \in \T_W$ such that $A'-Z=A$, $B'-Z=B$ and $V(A' \cap B') = V(A \cap B) \cup Z$.
Since $\T_W$ is a $(G,\F \cap W,r,\theta_{i_W+1},Z_0)$-tangle in $W$, we know that $A'-N_G^{\leq r}[V(A' \cap B')]$ does not contain any member of $(\F \cap W)-N_G^{\leq r}[Z_0]$. 
It implies that $A-N_G^{\leq r}[V(A \cap B) \cup Z] \subseteq A'-N_G^{\leq r}[V(A' \cap B')]$ does not contain any member of $(\F \cap W)-N_G^{\leq r}[Z_0]$.
So $A-N_G^{\leq r}[V(A \cap B)]$ does not contain any member of $(\F \cap W)-N_G^{\leq r}[Z_0 \cup Z]$.

Note that $V(W) \cap N_G^{\leq r}[Z_0] \subseteq V(W) \cap (N_G^{\leq r}[Z_0 \cap V(W)] \cup N_G^{\leq r}[N_G(V(W))])$.
Since $Z_0 \cap V(W) \subseteq Z_W$ and $N_G(V(W)) \subseteq Z_W$, we know $V(W) \cap N_G^{\leq r}[Z_0] \subseteq V(W) \cap N_G^{\leq r}[Z_W]$.
So $(\F \cap W)-N_G^{\leq r}[Z_0 \cup Z] \supseteq (\F \cap W)-N_G^{\leq r}[Z_W \cup Z] = (\F \cap W)-N_G^{\leq r}[Z^+]$.
Hence $A-N_G^{\leq r}[V(A \cap B)]$ does not contain any member of $(\F \cap W)-N_G^{\leq r}[Z_0 \cup Z] \supseteq (\F \cap (W-Z))-N_G^{\leq r}[Z^+]$.

Suppose that $B-N_G^{\leq r}[V(A \cap B)]$ does not contain any member of $(\F \cap (W-Z))-N_G^{\leq r}[Z^+]$, then $N_G^{\leq r}[Z^+ \cup V(A \cap B)]$ intersects all members of $\F \cap W$ (since $Z \subseteq Z^+$).
Since $Z_W^* \supseteq N_G^{\leq r}[Z^+]$, we know $Z_W^* \cup N_G^{\leq r}[V(A \cap B)]$ intersects all members of $\F \cap W$, contradicting (v) since $N_G^{\leq r}[V(A \cap B)]$ is $(\theta_{i_W+1}-1,r)$-centered.

So $B-N_G^{\leq r}[V(A \cap B)]$ contains a member of $(\F \cap (W-Z))-N_G^{\leq r}[Z^+]$.
Therefore, $\T_W-Z$ is a $(G,\F \cap (W-Z),r,\theta_{i_W+1}-\xi_{i_W+1},Z^+)$-tangle in $W-Z$.
$\Box$

\medskip

Note that $\theta_{i_W+1}-\xi_{i_W+1} \geq \theta_{\ref{hitting_no_apex}}(\rho_{i_W+1})$.
By Claims 5 and 6, $(G,W-Z,Z^+,\xi_{i_W+1}+3\theta_1-3+2\sum_{i=2}^{i_W}(3\theta_i-3),2r,X \cap W-Z, Y \cap W-Z, \F \cap (W-Z), \ell,k,r, \T_W-Z,\theta_{i_W+1}-\xi_{i_W+1},\theta_{\ref{hitting_no_apex}}(\rho_{i_W+1}),\Se^2[V(W)-Z], \kappa_{i_W+1},\rho_{i_W+1}, \Se^2_1[V(W)-Z], \Se^2_2[V(W)-Z], \Sigma, g, \phi_{\ref{hitting_no_apex}}, \alpha_2)$ is an interesting tuple.
By Lemma \ref{hitting_no_apex}, there exists $Z_1 \subseteq V(G)$ with $Z_1 \supseteq N_G^{\leq r+\ell}[Z^+]$ such that $Z_1$ is $(\xi,\eta)$-centered in $G$ intersecting all members of $\F \cap (W-Z)$.
Since $Z_1 \supseteq Z^+ \supseteq Z$, we know $Z_1$ intersects all members of $\F \cap W$.
But $Z_1$ is $(\xi,\eta)$-centered, contradicting (v).
\end{pf}

\section{Quasi-isometry} \label{sec:quasi-isometry}

We prove Theorem \ref{strongest_intro} in this section.

Recall that we say that an ${\mathbb R}_{>0}$-weighted graph or a length space $(X,d_X)$ has the \defn{weak coarse Menger property} if there exist functions $f,g: {\mathbb N} \times {\mathbb R}_{>0} \rightarrow {\mathbb R}$ such that for any $k \in {\mathbb N},r \in {\mathbb R}_{> 0}$ and subsets $A,B$ of $X$, either there exist $k$ paths in $(X,d_X)$ from $A$ to $B$ pairwise at distance in $d_X$ at least $r$, or there exists $Z \subseteq X$ such that $Z$ is $(f(k,r),g(k,r))$-centered and intersects every path in $(X,d_X)$ from $A$ to $B$.
We call $(f,g)$ a \defn{witness} for the weak coarse Menger property.

Similarly, we say that an ${\mathbb R}_{>0}$-weighted graph or a length space $(X,d_X)$ has the \defn{remote weak coarse Menger property} if there exist functions $f,g: {\mathbb N} \times {\mathbb R}_{>0} \times {\mathbb R}_{\geq 0} \rightarrow {\mathbb R}$ such that for any subsets $A,B$ of $X$ and $k \in {\mathbb N},r \in {\mathbb R}_{>0},\ell \in {\mathbb R}_{\geq 0}$, either there exist $k$ $(\ell,A,B)$-paths in $(X,d_X)$ from $A$ to $B$ pairwise at distance in $d_X$ at least $r$, or there exists $Z \subseteq X$ such that $Z$ is $(f(k,r,\ell),g(k,r,\ell))$-centered and intersects every $(\ell,A,B)$-path in $(X,d_X)$.
We call $(f,g)$ a \defn{witness} for the remote weak coarse Menger property.

Moreover, we say that an ${\mathbb R}_{>0}$-weighted graph or a length space $(X,d_X)$ has the \defn{weak coarse Gallai property} if there exist functions $f,g: {\mathbb N} \times {\mathbb R}_{>0} \rightarrow {\mathbb R}$ such that for any $k \in {\mathbb N},r \in {\mathbb R}_{\geq 0}$ and subset $A$ of $X$, either there exist $k$ $A$-paths in $(X,d_X)$ pairwise at distance in $d_X$ at least $r$, or there exists $Z \subseteq X$ such that $Z$ is $(f(k,r),g(k,r))$-centered and intersects every $A$-path in $(X,d_X)$.
We call $(f,g)$ a \defn{witness} for the weak coarse Gallai property.

We first show that those properties are preserved under quasi-isometry.

\begin{lemma} \label{quasi_iso}
For any $m \geq 1$ and $a \geq 0$, there exist $c_1,c_2>0$ such that the following holds.
Let $(X,d_X)$ and $(Y,d_Y)$ be finite or infinite $(0,1]$-weighted graphs or length spaces.
Let $\iota$ be an $(m,a)$-quasi-isometry from $(X,d_X)$ to $(Y,d_Y)$ for some $m \geq 1$ and $a \geq 0$.
	\begin{enumerate}
		\item If $(Y,d_Y)$ has the weak coarse Menger property with witness $(f(k,r),g(k,r))$ for some functions $f,g: {\mathbb N} \times {\mathbb R}_{>0} \rightarrow {\mathbb R}$, then $(X,d_X)$ has the weak coarse Menger property with witness $(f(k,mr+c_1),2m \cdot g(k,mr+c_1)+c_2)$.
		\item If $(Y,d_Y)$ has the remote weak coarse Menger property with witness $(f(k,r,\ell),g(k,r,\ell))$ for some functions $f,g: {\mathbb N} \times {\mathbb R}_{>0} \times {\mathbb R}_{\geq 0} \rightarrow {\mathbb R}$, then $(X,d_X)$ has the remote weak coarse Menger property with witness $(f(k,mr+c_1,m\ell+3a)+k-1,2m \cdot g(k,mr+c_1,m\ell+3a)+2m^2\ell+r+c_2)$.
		\item If $(Y,d_Y)$ has the weak coarse Gallai property with witness $(f(k,r),g(k,r))$ for some functions $f,g: {\mathbb N} \times {\mathbb R}_{>0} \rightarrow {\mathbb R}$, then $(X,d_X)$ has the weak coarse Gallai property with witness $(f(k,mr+c_1),2m \cdot g(k,mr+c_1)+c_2)$.
	\end{enumerate}
\end{lemma}

\begin{pf}
We only prove Statement 2, since Statements 1 and 3 are analogous and simpler.

Let $f,g: {\mathbb N} \times {\mathbb R}_{>0} \times {\mathbb R}_{\geq 0}$.
Let $k \in {\mathbb N}$, and let $r>0$ and $\ell \geq 0$ be real numbers.
Let $r' = 2m^2(3a+1)+(2+r)m+3a$ and $\ell' = m\ell+3a$.
Let $\xi_1 = f(k,r',\ell')$ and $\eta_1 = g(k,r',\ell')$.
Let $\eta_2 = m(2\eta_1+3a)$.
Let $\eta_3 = \eta_2 + m(m+2a+1)$.
Let $\ell'' = (\ell'+a)m$.
Let $\eta_4 = 2\ell''+2+r$.
Define $f'(k,r,\ell) = \xi_1+k-1$ and $g'(k,r,\ell)= \max\{\eta_3,\eta_4\}$.

Note that $f'(k,r,\ell)=f(k,r',\ell')+k-1$ and $g'(k,r,\ell) \leq (2g(k,r',\ell')+m+8a+1+2m\ell)m+r+2$.
Let $c_1=2m^2(3a+1)+2m+3a$.
So $r'=mr+c_1$.
Let $c_2 = (m+8a+1)m+2$.
Then $g'(k,r,\ell) \leq 2m \cdot g(k,r',\ell')+2m^2\ell+r+c_2$.

Let $A,B$ be subsets of $X$.
We assume that there exist no $k$ $(\ell,A,B)$-paths in $(X,d_X)$ with pairwise distance in $d_X$ at least $r$. 
It suffices to show that there exists an $(f'(k,r,\ell),g'(k,r,\ell))$-centered set in $(X,d_X)$ intersecting all $(\ell,X,Y)$-paths in $(X,d_X)$.

Let $\iota$ be an $(m,a)$-quasi-isometry from $(X,d_X)$ to $(Y,d_Y)$.
Then for every $y \in Y$, there exists $x_y \in X$ such that $d_Y(\iota(x_y),y) \leq a$.
Note that we may assume that $\iota(x_y)=y$ and $x_y \in W$ for any $W=\{A,B\}$ and $y \in \iota(W)$.

\medskip

\noindent{\bf Claim 1:} There do not exist $k$ $(\ell',\iota(A),\iota(B))$-paths in $(Y,d_Y)$ with pairwise distance in $d_Y$ at least $r'$.

\noindent{\bf Proof of Claim 1:} 
Suppose to the contrary that there exist $k$ $(\ell',\iota(A),\iota(B))$-paths $P_1,P_2,..., \allowbreak P_k$ in $(Y,d_Y)$ with pairwise distance in $d_Y$ at least $r'$.

For every $i \in [k]$, let $a_i$ and $b_i$ be the ends of $P_i$.
Note that for every $i \in [k]$, there exists a finite sequence $p_{i,1},p_{i,2},...,p_{i,c_i}$ over $P_i$ for some positive integer $c_i$ such that $p_{i,1}=a_i$, $p_{i,c_i}=b_i$ and $d_Y(p_{i,j},p_{i,j+1}) \leq 1$ for every $j \in [c_i-1]$; we let $V(P_i) = \{p_{i,j}: j \in [c_i]\}$.
Since $a_i \in \iota(A)$ and $b_i \in \iota(B)$, we have $x_{a_i} \in A$ and $x_{b_i} \in B$.
For every $i \in [k]$, let $W_i = \{x_y \in X: y \in V(P_i)\}$; note that for any $j \in [c_i-1]$, $d_X(x_{p_{i,j}},x_{p_{i,j+1}}) \leq m(d_Y(\iota(x_{p_{i,j}}),\iota(x_{p_{i,j+1}}))+a) \leq m(d_Y(p_{i,j},p_{i,j+1})+2a+a) \leq m(3a+1)$, and there exists a path $Q_{i,j}$ in $(X,d_X)$ from $x_{p_{i,j}}$ to $x_{p_{i,j+1}}$ with length in $d_X$ at most $d_X(x_{p_{i,j}},x_{p_{i,j+1}}) +1 \leq m(3a+1)+1$; we let $Q_i = \bigcup_{j=1}^{c_i-1}Q_{i,j}$.
Note that each $Q_i$ contains a path in $(X,d_X)$ from $x_{a_i} \in A$ to $x_{b_i} \in B$.

Now we show that $Q_1,Q_2,...,Q_k$ have pairwise distance in $d_X$ at least $r$.
Suppose to the contrary that there exist $i_1<i_2$ such that $d_X(Q_{i_1},Q_{i_2})<r$.
Then there exist $j_1,j_2$ such that $d_X(Q_{i_1,j_1},Q_{i_2,j_2})<r$.
So $d_X(x_{p_{i_1,j_1}},x_{p_{i_2,j_2}}) < 2 \cdot (m(3a+1)+1) + r = 2m(3a+1)+2+r$.
Hence $d_Y(\iota(x_{p_{i_1,j_1}}),\iota(x_{p_{i_2,j_2}})) < m \cdot (2m(3a+1)+2+r)+a = 2m^2(3a+1)+(2+r)m+a$.
So $d_Y(P_{i_1},P_{i_2}) \leq d_Y(p_{i_1,j_1},p_{i_2,j_2}) \leq d_Y(\iota(x_{p_{i_1,j_1}}),\iota(x_{p_{i_2,j_2}}))+2a < 2m^2(3a+1)+(2+r)m+3a =r'$, a contradiction.

Hence $Q_1,Q_2,...,Q_k$ have pairwise at distance in $d_X$ at least $r$.
Note that for any $i \in [k]$, $d_X(x_{a_i},x_{b_i}) \geq \frac{1}{m} \cdot (d_Y(\iota(x_{a_i}),\iota(x_{b_i}))-a) \geq \frac{1}{m} \cdot (d_Y(a_i,b_i)-2a-a) \geq \frac{1}{m} \cdot (\ell'-3a) \geq \ell$.
So for every $i \in [k]$, $Q_i$ contains an $(\ell,X,Y)$-path in $(X,d_X)$.
Hence there exist $k$ $(\ell,X,Y)$-paths in $(X,d_X)$ with pairwise distance in $d_X$ at least $r$, a contradiction.
$\Box$

\medskip

By Claim 1, there exists $Z_1 \subseteq Y$ such that $Z_1$ is $(f(k,r',\ell'),g(k,r',\ell'))$-centered in $d_Y$ intersecting all $(\ell',\iota(A),\iota(B))$-paths in $(Y,d_Y)$.
Let $Z_2 = \{x_y \in X: y \in Z_1\}$.
Note that $Z_2$ is $(\xi_1, \eta_2)$-centered in $d_X$.

Let $Z_3 = N_{(X,d_X)}^{\leq m(m+2a+1)}[Z_2]$.
Then $Z_3$ is $(\xi_1, \eta_3)$-centered in $d_X$.

\medskip

\noindent{\bf Claim 2:} $Z_3$ intersects all $(\ell,A,B)$-paths $P$ in $(X,d_X)$ satisfying that $d_X(s,t) \geq \ell''$, where $s$ and $t$ are the ends of $P$.

\noindent{\bf Proof of Claim 2:}
Let $P$ be an $(\ell,A,B)$-path in $(X,d_X)$ from a vertex $s \in A$ to a vertex $t \in B$ such that $d_X(s,t) \geq \ell''$.
Note that there exists a finite sequence $p_1,p_2,...,p_c$ over $P$ for some positive integer $c$ such that $p_1=s$, $p_c=t$ and $d_X(p_{i},p_{i+1}) \leq 1$ for every $i \in [c-1]$.
Let $V(P) = \{p_{i}: i \in [c]\}$.
Let $W = \{\iota(p_i): i \in [c]\}$.
Note that $d_Y(\iota(s),\iota(t)) \geq \frac{1}{m} \cdot d_X(s,t) -a \geq \frac{\ell''}{m}-a \geq \ell'$.
For every $i \in [c-1]$, $d_Y(\iota(p_i),\iota(p_{i+1})) \leq m \cdot d_X(p_i,p_{i+1})+a \leq m+a$, and there exists a path $Q_i$ in $(Y,d_Y)$ with length in $d_Y$ at most $d_Y(\iota(p_i),\iota(p_{i+1}))+1 \leq m+a+1$ from $\iota(p_i)$ to $\iota(p_{i+1})$.
Let $Q = \bigcup_{i=1}^cQ_i$.
Then $Q$ contains an $(\ell',\iota(A),\iota(B))$ in $(Y,d_Y)$.
So $Z_1 \cap Q \neq \emptyset$.

Let $z \in Z_1 \cap Q$.
So there exists $i \in [c-1]$ such that $z \in Z_1 \cap Q_i$.
Hence $d_X(p_i,x_z) \leq m \cdot (d_Y(\iota(p_i),\iota(x_z))+a) \leq m((m+a+1)+a) = m(m+2a+1)$.
Since $x_z \in Z_2$, $p_i \in Z_3 \cap P$.
$\Box$

\medskip

Let $\C$ be a maximal collection of $(\ell,A,B)$-paths in $(X,d_X)$ satisfying that every member of $\C$ is a path in $(X,d_X)$ from some vertex $s \in A$ to some vertex $t \in B$ with $\ell \leq d_X(s,t) <\ell''$ such that its length in $d_X$ is at most $d_X(s,t)+1$, and satisfying that the members of $\C$ have pairwise distance in $d_X$ at least $r$. 
Note $|\C| \leq k-1$ since $(X,d_X)$ does not contain $k$ $(\ell,A,B)$-paths in $(X,d_X)$ with pairwise distance in $d_X$ at least $r$.

Let $Z_4 = N_{(X,d_X)}^{\leq r+\ell''+1}[\bigcup_{P \in \C}P]$. 
Note that $Z_4$ is $(|\C|, (\ell''+1)+r+\ell''+1)$-centered and hence $(k-1,\eta_4)$-centered in $(X,d_X)$.

If $Q$ is an $(\ell,A,B)$-path in $(X,d_X)$ disjoint from $Z_3$, then Claim 2 implies that it is a path from a vertex $s \in A$ to a vertex $t \in B$ with $\ell \leq d(s,t)<\ell''$, so the maximality of $\C$ implies that $d_X(Q',\bigcup_{P \in \C}P)<r$, where $Q'$ is a path in $(X,d_X)$ between $s$ and $t$ with length at most $d_X(s,t)+1 \leq \ell''+1$, and hence $d_X(s,\bigcup_{P \in \C}P) < r+(d_X(s,t)+1) < r+\ell''+1$, and hence $s \in Z_4$ and $V(Q) \cap Z_4 \neq \emptyset$.

Therefore, $Z_3 \cup Z_4$ intersects all $(\ell,A,B)$-paths in $(X,d_X)$.
Note that $Z_3 \cup Z_4$ is $(\xi_1+k-1, \max\{\eta_3,\eta_4\})$-centered in $(X,d_X)$. 
This proves the lemma since $f'(k,r,\ell)=\xi_1+k-1$ and $g'(k,r,\ell)=\max\{\eta_3,\eta_4\}$. 
\end{pf}

\begin{lemma} \label{main_quasi}
For any (finite) graph $H$, real numbers $m,a$ with $m \geq 1$ and $a \geq 0$, there exist functions $f: {\mathbb N} \times {\mathbb R}_{>0} \times {\mathbb R}_{\geq 0} \rightarrow {\mathbb R}$ and $g: {\mathbb R}_{>0} \times {\mathbb R}_{\geq 0} \rightarrow {\mathbb R}$ such that if $(W,d_W)$ is a length space that is $(m,a)$-quasi-isometric to an $H$-minor free finite or locally finite infinite graph, then
	\begin{enumerate}
		\item for any $k \in {\mathbb N}, r \in {\mathbb R}_{>0},\ell \in {\mathbb R}_{\geq 0}$ and subsets $X,Y$ of $W$, either there exist $k$ $(\ell,X,Y)$-paths in $(W,d_W)$ pairwise at distance in $d_W$ at least $r$, or there exists an $(f(k,r,\ell),g(r,\ell))$-centered set in $(W,d_W)$ intersecting all $(\ell,X,Y)$-paths in $(W,d_W)$, and  
		\item for any $k \in {\mathbb N}, r \in {\mathbb R}_{> 0}$ and subset $A$ of $W$, either there exist $k$ $A$-paths in $(W,d_W)$ pairwise at distance in $d_W$ at least $r$, or there exists an $(f(k,r,0),g(r,0))$-centered set in $(W,d_W)$ intersecting all $A$-paths in $(W,d_W)$.  
	\end{enumerate}
Moreover, $g$ is linear in the first variable.
\end{lemma}

\begin{pf}
Since the distance between any two vertices in a finite or locally finite infinite graph is at least 1, Theorem \ref{infinite_minor} implies that there exist functions $f_1: {\mathbb N} \times {\mathbb R}_{>0} \times {\mathbb R}_{\geq 0} \rightarrow {\mathbb R}$ and $g_1: {\mathbb R}_{>0} \times {\mathbb R}_{\geq 0} \rightarrow {\mathbb R}$ such that every $H$-minor free finite or locally finite infinite graph has the remote weak coarse Menger property with witness $(f_1,g_1)$.
Note that $g_1$ is linear in the first variable.
By considering $(0,A,A)$-paths, it implies that every $H$-minor free finite or locally finite infinite graph has the remote weak coarse Gallai property with witness $(f_2,g_2)$ for some functions $f_2: {\mathbb N} \times {\mathbb R}_{>0} \rightarrow {\mathbb R}$ and $g_2: {\mathbb R}_{>0} \rightarrow {\mathbb R}$.
Note that $g_2$ is linear.
Then this lemma immediately follows from Lemma \ref{quasi_iso}.
\end{pf}

\bigskip

Now we extend Lemma \ref{main_quasi} to length spaces quasi-isometric to locally finite $H$-minor free metric graphs.
We say that a metric graph is \defn{finite} if it is defined by a finite ${\mathbb R}_{>0}$-weighted graph.

\begin{lemma} \label{main_quasi_weighted}
For any (finite) graph $H$, real numbers $m,a$ with $m \geq 1$ and $a \geq 0$, there exist functions $f: {\mathbb N} \times {\mathbb R}_{>0} \times {\mathbb R}_{\geq 0} \rightarrow {\mathbb R}$ and $g: {\mathbb R}_{>0} \times {\mathbb R}_{\geq 0} \rightarrow {\mathbb R}$ such that if $(W,d_W)$ is a length space that is $(m,a)$-quasi-isometric to an $H$-minor free finite or locally finite infinite metric graph, then
	\begin{enumerate}
		\item for any $k \in {\mathbb N}, r \in {\mathbb R}_{>0},\ell \in {\mathbb R}_{\geq 0}$ and subsets $X,Y$ of $W$, either there exist $k$ $(\ell,X,Y)$-paths in $(W,d_W)$ pairwise at distance in $d_W$ at least $r$, or there exists an $(f(k,r,\ell),g(r,\ell))$-centered set in $(W,d_W)$ intersecting all $(\ell,X,Y)$-paths in $(W,d_W)$, and  
		\item for any $k \in {\mathbb N}, r \in {\mathbb R}_{> 0}$ and subset $A$ of $W$, either there exist $k$ $A$-paths in $(W,d_W)$ pairwise at distance in $d_W$ at least $r$, or there exists an $(f(k,r,0),g(r,0))$-centered set in $(W,d_W)$ intersecting all $A$-paths in $(W,d_W)$. 
	\end{enumerate}
Moreover, $g$ is linear in the first variable.
\end{lemma}

\begin{pf}
Let $H$ be a graph, and let $m,a$ be real numbers with $m \geq 1$ and $a \geq 0$.
Let $H'$ be a 3-connected graph such that $H \subseteq H'$.
By Theorem \ref{01-weighted_quasi}, there exist real numbers $m_1,a_1$ with $m_1 \geq 1$ and $a_1 \geq 0$ such that every $H'$-minor free locally finite $(0,1]$-weighted graph is $(m_1,a_1)$-quasi-isometric to an $H'$-minor free locally finite graph.

We prove Statement 1 of this lemma; the proof of Statement 2 is analogous, and we omit it.

Let $(W,d_W)$ be a length space that is $(m,a)$-quasi-isometric to an $H$-minor free finite or locally finite infinite metric graph $(W_0,d_0)$.
Let $(W',\phi')$ be the ${\mathbb R}_{>0}$-weighted graph that defines metric graph $(W_0,d_0)$.
Let $(W'',\phi'')$ be a $(0,1]$-weighted graph obtained from subdividing edges of $W'$ such that $\dist_{(W'',\phi'')}(u,v)=\dist_{(W',\phi')}(u,v)$ for any $u,v \in V(W')$.
Since $W'$ is $H$-minor free, $W'$ is also $H'$-minor free.
Since $H'$ is 3-connected, $W''$ is $H'$-minor free.
By Theorem \ref{01-weighted_quasi}, there exists an $H'$-minor free locally finite graph $W'''$ such that $(W'',\phi'')$ is $(m_1,a_1)$-quasi-isometric to $W'''$.

Let $(W_1,d_1)$ be the metric graph defined by the $(0,1]$-weighted graph $(W'',\phi'')$.
Then $(W_1,d_1)$ is a length space $(1,1)$-quasi-isometric to $(W'',\phi'')$.
So $W_1$ is $(m_1,a_1+1)$-quasi-isometric to $W'''$.
By Lemma \ref{main_quasi}, there exist functions $f_1: {\mathbb N} \times {\mathbb R}_{>0} \times {\mathbb R}_{\geq 0} \rightarrow {\mathbb R}$ and $g_1: {\mathbb R}_{>0} \times {\mathbb R}_{\geq 0} \rightarrow {\mathbb R}$ only dependent on $H',m_1,a_1$ (and hence only dependent on $H,m,a$) such that for any $k \in {\mathbb N}, \ell \in {\mathbb R}_{>0},r \in {\mathbb R}_{\geq 0}$ and subsets $X,Y$ of $W_1$, either there exist $k$ $(\ell,X,Y)$-paths in $(W_1,d_1)$ pairwise at distance in $d_1$ at least $r$, or there exists an $(f_1(k,r,\ell),g_1(r,\ell))$-centered set in $(W_1,d_1)$ intersecting all $(\ell,X,Y)$-paths in $(W_1,d_1)$.
Moreover, $g_1$ is linear in the first variable $r$.

Since $W_0 \subseteq W_1$, we know that for any $k \in {\mathbb N}, r \in {\mathbb R}_{>0},\ell \in {\mathbb R}_{\geq 0}$ and subsets $X,Y$ of $V(W_0)$, either there exist $k$ $(\ell,X,Y)$-paths in $(W_0,d_0)$ pairwise at distance in $d_0$ at least $r$, or there exists an $(f_1(k,r,\ell),g_1(r,\ell))$-centered set in $(W_0,d_0)$ intersecting all $(\ell,X,Y)$-paths in $(W_0,d_0)$.
Since $(W,d_W)$ is $(m,a)$-quasi-isometric to the length space $(W_0,d_0)$, Lemma \ref{quasi_iso} implies that there exist functions $f: {\mathbb N} \times {\mathbb R}_{>0} \times {\mathbb R}_{\geq 0} \rightarrow {\mathbb R}$ and $g: {\mathbb R}_{>0} \times {\mathbb R}_{\geq 0} \rightarrow {\mathbb R}$ only dependent on $f_1,g_1,m,a$ (and hence only dependent on $H,m,a$) such that this lemma holds.
\end{pf}

\bigskip

Lemma \ref{main_quasi_weighted} allows us to extend our results to weighted graphs from length spaces.

\begin{lemma} \label{main_weighted}
For any (finite) graph $H$, there exist functions $f: {\mathbb N} \times {\mathbb R}_{>0} \times {\mathbb R}_{\geq 0} \rightarrow {\mathbb R}$ and $g: {\mathbb R}_{>0} \times {\mathbb R}_{\geq 0} \rightarrow {\mathbb R}$ such that if $(G,\phi)$ is an $H$-minor free finite or locally finite infinite ${\mathbb R}_{>0}$-weighted graph, then
	\begin{enumerate}
		\item for any $k \in {\mathbb N}, r \in {\mathbb R}_{>0},\ell \in {\mathbb R}_{\geq 0}$ and subsets $X,Y$ of $W$, either there exist $k$ $(\ell,X,Y)$-paths in $(G,\phi)$ pairwise at distance in $(G,\phi)$ at least $r$, or there exists an $(f(k,r,\ell), \allowbreak g(r,\ell))$-centered set in $(W,d_W)$ intersecting all $(\ell,X,Y)$-paths in $(G,\phi)$, and 
		\item for any $k \in {\mathbb N}, r \in {\mathbb R}_{> 0}$ and subset $A$ of $W$, either there exist $k$ $A$-paths in $(G,\phi)$ pairwise at distance in $(G,\phi)$ at least $r$, or there exists an $(f(k,r,0),g(r,0))$-centered set in $(W,d_W)$ intersecting all $A$-paths in $(G,\phi)$. 
	\end{enumerate}
Moreover, $g$ is linear in the first variable.
\end{lemma}

\begin{pf}
Let $H'$ be a 3-connected graph such that $H \subseteq H'$.
Let $(G',\phi')$ be the $(0,1]$-weighted graph obtained from $(G,\phi)$ by subdividing edges of $G$ such that $\dist_{(G',\phi')}(u,v)=\dist_{(G,\phi)}(u,v)$ for any $u,v \in V(G)$.
Since $G$ is $H$-minor free, $G$ is $H'$-minor free, and hence $G'$ is $H'$-minor free.
Let $(W,d_W)$ be the metric graph defined by $(G',\phi')$.
Then $(G',\phi')$ is $(1,1)$-quasi-isometric to $(W,d_W)$.
So Lemma \ref{main_quasi_weighted} implies that there exist functions $f: {\mathbb N} \times {\mathbb R}_{>0} \times {\mathbb R}_{\geq 0} \rightarrow {\mathbb R}$ and $g: {\mathbb R}_{>0} \times {\mathbb R}_{\geq 0} \rightarrow {\mathbb R}$ only dependent on $H'$ (and hence only dependent on $H$) such that the conclusions of this lemma hold for $(G',\phi')$.
Since $V(G) \subseteq V(G')$, this lemma holds for $(G,\phi)$.
\end{pf}

\bigskip

Now we introduce a trick that allows us to remove the dependency of $r$ from the function $f$ in the witness for all our previous theorems.
In order to do so, we consider our properties for classes of length spaces or weighted graphs.
We say that a class $\F$ of length spaces or ${\mathbb R}_{>0}$-weighted graphs has the \defn{weak coarse Menger property} (and the \defn{remote weak coarse Menger property} and \defn{weak coarse Gallai property}, respectively) if there exist functions $f,g$ such that every member of $\F$ has the weak coarse Menger property (and the remote weak coarse Menger property and weak coarse Gallai property, respectively) with witness $(f,g)$.
We call $(f,g)$ the \defn{witness} for $\F$.

A set $\F$ of metric spaces is \defn{scaling-closed} if $(X,r \cdot d_X) \in \F$ for any $(X,d_X) \in \F$ and real number $r>0$.

\begin{lemma} \label{scaling_trick}
Let $\F$ be a scaling-closed set of length spaces or ${\mathbb R}_{>0}$-weighted graphs.
    \begin{enumerate}
        \item If $f,g: {\mathbb N} \times {\mathbb R} \rightarrow {\mathbb R}$ are functions such that every member of $\F$ has the weak coarse Menger property with witness $(f(k,r),g(k,r))$, then every length space in $\F$ has the weak coarse Menger property with witness $(f(k,1),g(k,1) \cdot r)$. 
	\item If $f,g: {\mathbb N} \times {\mathbb R} \rightarrow {\mathbb R}$ are functions such that every member of $\F$ has the weak coarse Gallai property with witness $(f(k,r),g(k,r))$, then every length space in $\F$ has the weak coarse Gallai property with witness $(f(k,1),g(k,1) \cdot r)$. 
        \item If $f,g: {\mathbb N} \times {\mathbb R}_{>0} \times {\mathbb R}_{\geq 0} \rightarrow {\mathbb R}$ are functions such that every member of $\F$ has the remote weak coarse Menger property with witness $(f(k,r,\ell),g(k,r,\ell))$, then every length space in $\F$ has the remote weak coarse Menger property with witness $(f(k,1,\frac{\ell}{r}),g(k,1,\frac{\ell}{r}) \cdot r)$. 
    \end{enumerate}
\end{lemma}

\begin{pf}
We only prove Statement 3 because the proofs of the first two statements are analogous and simpler.

Let $(X,d_X) \in \F$, and let $k \in {\mathbb N}$, $r>0$ and $\ell \geq 0$.
Let $A,B \subseteq X$.
Since $\F$ is scaling-closed, $(X,\frac{1}{r} \cdot d_X) \in \F$.
Note that every $(\ell,A,B)$-path in $(X,d_X)$ is an $(\frac{\ell}{r},A,B)$-path in $(X,\frac{1}{r} \cdot d_X)$, and vice versa.
If there exist $k$ $(\frac{\ell}{r},A,B)$-paths in $(X,\frac{1}{r} \cdot d_X)$ pairwise at distance in $(X,\frac{1}{r} \cdot d_X)$ at least $1$, then there exist $k$ $(\ell,A,B)$-paths in $(X,d_X)$ pairwise at distance in $(X,d_X)$ at least $r$.
Otherwise, since $(X,\frac{1}{r} \cdot d_X) \in \F$, there exists a $(f(k,1,\frac{\ell}{r}),g(k,1,\frac{\ell}{r}))$-centered set $Z$ in $(X,\frac{1}{r} \cdot d_X)$ intersecting all $(\frac{\ell}{r},A,B)$-paths in $(X,\frac{1}{r} \cdot d_X)$, so $Z$ is a $(f(k,1,\frac{\ell}{r}),g(k,1,\frac{\ell}{r}) \cdot r)$-centered set in $(X,d_X)$ intersecting all $(\ell,A,B)$-paths in $(X,d_X)$.
Therefore, $(X,d_X)$ has the remote coarse Menger property with the desired witness. 
\end{pf}

\bigskip

Now we prove a special case of Theorem \ref{strongest_intro} for which the quasi-isometry has zero additive distortion.

\begin{lemma} \label{strongest_form_prep}
Let $H$ be a graph, and let $m \geq 1$ be a real number.
Let $\F_1$ be the set of finite or locally finite infinite $H$-minor free ${\mathbb R}_{>0}$-weighted graphs.
Let $\F_2$ be the set of finite or locally finite infinite $H$-minor free metric graphs.
Let $\F_3$ be the set of length spaces that are $(m,0)$-quasi-isometric to a member of $\F_2$.
Let $\F=\F_1 \cup \F_2 \cup \F_3$.
Then the following hold:
    \begin{enumerate}
	    \item There exist a function $f: {\mathbb N} \rightarrow {\mathbb R}$ and a real number $c$ such that every member of $\F$ has the weak coarse Menger property with witness $(f, cx)$.
	\item There exist a function $f: {\mathbb N} \rightarrow {\mathbb R}$ and a real number $c$ such that every member of $\F$ has the weak coarse Gallai property with witness $(f, cx)$.
	\item For every $\alpha \geq 0$, there exist a function $f: {\mathbb N} \rightarrow {\mathbb R}$ and a real number $c$ such that if $(W,d_W) \in \F$, then for any $k \in {\mathbb N}$, real number $r>0$ and subsets $X$ and $Y$ of $W$, either there exist $k$ $(\alpha r,X,Y)$-paths in $(W,d_W)$ pairwise at distance in $d_W$ at least $r$, or there exists an $(f(k),cr)$-centered set intersecting all $(\alpha r, X,Y)$-paths in $(W,d_W)$.
    \end{enumerate}
\end{lemma}

\begin{pf}
Let $H,m,\F_1,\F_2,\F_3,\F$ be as stated in the lemma.
Obviously, $\F_1 \cup \F_2$ is scaling-closed.
This implies that $\F_3$ is scaling-closed.
So $\F$ is scaling-closed.

Lemmas \ref{main_quasi_weighted} and \ref{main_weighted} imply that there exist functions $f_1: {\mathbb N} \times {\mathbb R}_{>0} \times {\mathbb R}_{\geq 0} \rightarrow {\mathbb R}, g_1: {\mathbb R}_{>0} \times {\mathbb R}_{\geq 0} \rightarrow {\mathbb R},f_2: {\mathbb N} \times {\mathbb R}_{>0} \rightarrow {\mathbb R}$ and $g_2: {\mathbb R}_{>0} \rightarrow {\mathbb R}$ such that $\F$ satisfies the remote weak coarse Menger property with witness $(f_1(k,r,\ell),g_1(r,\ell))$ and satisfies the weak coarse Gallai property with witness $(f_2(k,r),g_2(r))$, and $g_1,g_2$ are linear in their first variable.
Taking $\ell=0$, we know that $\F$ satisfies the weak coarse Menger property with witness $(f_1(k,r,0),g_1(r,0))$.
Note that $f_1,g_1,f_2,g_2$ only depend on $H$.

Since $\F$ is scaling-closed, Lemma \ref{scaling_trick} implies that
	\begin{itemize}
		\item $\F$ has the weak coarse Menger property with $(f_1(k,1,0),g_1(1,0) \cdot r)$,
		\item $\F$ has the weak coarse Gallai property with $(f_2(k,1),g_2(1) \cdot r)$, and
		\item $\F$ has the remote weak coarse Menger property with witness $(f_1(k,1,\frac{\ell}{r}),g_1(1,\frac{\ell}{r}) \cdot r)$.
	\end{itemize}
Statements 1 and 2 follows from the first two bullets, respectively.
For every $\alpha \geq 0$, let $f: {\mathbb N} \rightarrow {\mathbb R}$ be the function such that $f(x)=f_1(x,1,\alpha)$ for every $x \in {\mathbb N}$, and let $c=g_1(1,\alpha)$.
Then Statement 3 of this lemma holds.
\end{pf}

\bigskip

Now we are ready to prove Theorem \ref{strongest_intro}.
The following is an equivalent statement.

\begin{theorem} \label{strongest_form}
Let $H$ be a (finite) graph.
Let $m \geq 1,a \geq 0$ be real numbers. 
Let $\F_1$ be the set of finite or locally finite infinite $H$-minor free ${\mathbb R}_{>0}$-weighted graphs.
Let $\F_2$ be the set of finite or locally finite infinite $H$-minor free metric graphs.
Let $\F_3$ be the set of length spaces that are $(m,a)$-quasi-isometric to a member of $\F_2$.
Let $\F_4$ be the set of length spaces that are $(m,a)$-quasi-isometric to a $(0,1]$-weighted graph in $\F_1$.
Let $\F=\F_1 \cup \F_2 \cup \F_3 \cup \F_4$.
Then the following hold:
    \begin{enumerate}
	    \item There exist a function $f: {\mathbb N} \rightarrow {\mathbb R}$ and real numbers $c,c' \geq 0$ such that if $(W,d_W) \in \F$, then for any $k \in {\mathbb N}$, real number $r>0$ and subsets $X$ and $Y$ of $W$, either there exist $k$ $(X,Y)$-paths in $(W,d_W)$ pairwise at distance in $d_W$ at least $r$, or there exists an $(f(k),cr+c')$-centered set intersecting all $(X,Y)$-paths in $(W,d_W)$.  
	\item There exist a function $f: {\mathbb N} \rightarrow {\mathbb R}$ and real numbers $c,c' \geq 0$ such that if $(W,d_W) \in \F$, then for any $k \in {\mathbb N}$, real number $r>0$ and subset $A$ of $W$, either there exist $k$ $A$-paths in $(W,d_W)$ pairwise at distance in $d_W$ at least $r$, or there exists an $(f(k),cr+c')$-centered set intersecting all $A$-paths in $(W,d_W)$.
        \item For every $\alpha \geq 0$, there exist a function $f: {\mathbb N} \rightarrow {\mathbb R}$ and real numbers $c,c' \geq 0$ such that if $a=0$ and $(W,d_W) \in \F_1 \cup \F_2 \cup \F_3$, then for any $k \in {\mathbb N}$, real number $r>0$ and subsets $X$ and $Y$ of $W$, either there exist $k$ $(\alpha r,X,Y)$-paths in $(W,d_W)$ pairwise at distance in $d_W$ at least $r$, or there exists an $(f(k),cr+c')$-centered set intersecting all $(\alpha r, X,Y)$-paths in $(W,d_W)$.
    \end{enumerate}
Moreover, if $a=0$ and $(W,d_W) \in \F_1 \cup \F_2 \cup \F_3$, then $c'$ can be chosen to be $0$.
\end{theorem}

\begin{pf}
Let $H,m,a,\F_1,\F_2,\F_3,\F_4,\F$ be as stated in the theorem.
The case $(W,d_W) \in \F_1 \cup \F_2$ and the case $a=0$ with $(W,d_W) \in \F_1 \cup \F_2 \cup \F_3$ have been proved in Lemma \ref{strongest_form_prep}.
So we may assume $a>0$, and it suffices to prove Statements 1 and 2 for members $(W,d_W) \in \F-(\F_1 \cup \F_2)$.
Let $(W,d_W) \in \F-(\F_1 \cup \F_2) \subseteq \F_3 \cup \F_4$.
So $(W,d_W)$ is a length space $(m,a)$-quasi-isometric to a member of $\F_2$ or a $(0,1]$-weighted graph in $\F_1$.
Hence Statements 1 and 2 for $(W,d_W)$ follow from Lemma \ref{quasi_iso} and Statements 1 and 2 of Lemma \ref{strongest_form_prep}.
\end{pf}

\section{Concluding remarks} \label{sec:concluding_remarks}

A special case of the main result of this paper states that for any graph $H$ and real numbers $m \geq 1,a \geq 0$, there exist a function $f: {\mathbb N} \rightarrow {\mathbb R}$ and a real number $c>0$ such that if $G$ is a graph that is $(m,a)$-quasi-isometric to an $H$-minor free graph, then for any $k \in {\mathbb N}, r \in {\mathbb R}_{>0}$ and $X,Y \subseteq V(G)$, either there exist $k$ paths from $X$ to $Y$ pairwise at distance at least $r$, or there exists an $(f(k),cr)$-centered set intersecting all paths in $G$ from $X$ to $Y$.
Like asymptotic dimension of metric spaces can be used to distinguish metric spaces up to quasi-isometry, so does this result.
Georgakopoulos and Papasoglu \cite{gp} conjectured the following coarse analog of Wagner's characterization for planar graphs:

\begin{conjecture}[{{\cite[Conjecture 9.1]{gp}}}] \label{conj_coarse_Kuratowski}
A length space is quasi-isometric to a planar graph if and only if it has no asymptotic $K_5$ or $K_{3,3}$ minor.
\end{conjecture}

We refer interested readers to \cite{gp} for the definition of asymptotic minors of length spaces.
By our aforementioned result, a negative answer to the following question would disprove Conjecture \ref{conj_coarse_Kuratowski}.

\begin{question} \label{conj_fat_planar_menger}
For every $\ell \in {\mathbb N}$, do there exist a function $f: {\mathbb N} \rightarrow {\mathbb R}$ and a real number $c>0$ such that if $G$ does not contain an $\ell$-fat $K_5$-minor model nor an $\ell$-fat $K_{3,3}$-minor model, then for any $k \in {\mathbb N}, r \in {\mathbb R}_{>0}$ and $X,Y \subseteq V(G)$, either there exist $k$ paths from $X$ to $Y$ pairwise at distance at least $r$, or there exists an $(f(k),cr)$-centered set intersecting all paths in $G$ from $X$ to $Y$?
\end{question}

Many problems in combinatorial optimization can be stated as a packing problem or a covering problem.
The packing problem and the covering problem can be formulated as integer programming problems, and they are dual to each other.
Menger's theorem states that the optimal value for packing $X$-$Y$ paths equals the optimal value for covering $X$-$Y$ paths.
In general, the optimal values for packing and covering are not necessarily related, except for the weak duality that the optimal value for covering is at least the optimal value for packing.
Erd\H{o}s-P\'{o}sa property captures the case that the optimal values of the integer programming problems for packing and for covering are tied to each other.

Paths between two prescribed sets $X$ and $Y$ are exactly rooted subdivisions of $K_2$.
It is well known that the set of rooted subdivisions of $P_3$ does not have the Erd\H{o}s-P\'{o}sa property.
On the other hand, Liu \cite{l} proved that for every graph $H$, the set of rooted subdivisions of $H$ has the half-integral Erd\H{o}s-P\'{o}sa property. 
Roughly speaking, the half-integral Erd\H{o}s-P\'{o}sa property captures the case that the optimal values for packing and for covering are tied to each other if we allow the solution for the packing problem to be an half-integral vector.

We propose the following conjecture, which is the half-integral weakening of Conjecture \ref{weak_coarse_Menger_conj}. 

\begin{conjecture}
There exist functions $f,g: {\mathbb N} \times {\mathbb R}_{>0} \rightarrow {\mathbb R}$ such that for any $k \in {\mathbb N}, r \in {\mathbb R}_{>0}$, graph $G$ and subsets $X,Y$ of $V(G)$, either 
	\begin{enumerate}
		\item there exist $k$ $X$-$Y$ paths $P_1,P_2,...,P_k$ in $G$ such that every ball of radius $r$ in $G$ intersects at most two of $P_1,P_2,...,P_k$, or
		\item there exists an $(f(k,r),g(k,r))$-centered set intersecting all $X$-$Y$ paths in $G$.
	\end{enumerate}
\end{conjecture}

\bigskip

\noindent{\bf Acknowledgement:} The author thanks James Davies for clarification of results in \cite{d_string,ddh} and thanks Micha\l{} Pilipczuk for informing him of \cite{bpp}.

\end{document}